\documentclass[
11pt,  reqno]{amsart}
\usepackage[margin=1.2in,marginparwidth=1.5cm, marginparsep=0.5cm]{geometry}
\pdfoutput=1

\usepackage{booktabs} 
\usepackage{microtype}
\usepackage{amssymb}
\usepackage{mathrsfs}
\usepackage{stmaryrd}

\usepackage{upgreek}

\usepackage{color}
\usepackage[implicit=true]{hyperref}

\usepackage{xsavebox}

\usepackage{cases}

\allowdisplaybreaks[3]

\definecolor{gr}{rgb}   {0.,   0.69,   0.23 }
\definecolor{bl}{rgb}   {0.,   0.5,   1. }
\definecolor{mg}{rgb}   {0.85,  0.,    0.85}
\definecolor{yl}{rgb}   {0.8,  0.7,   0.}
\definecolor{or}{rgb}  {0.7,0.2,0.2}

\usepackage{tikz}

\usepackage{marginnote}
\usepackage{scalerel} 

\usetikzlibrary{shapes.misc}
\usetikzlibrary{shapes.symbols}
\usetikzlibrary{decorations}
\usetikzlibrary{decorations.markings}

\tikzset{
	dot/.style={circle,fill=black,draw=black,inner sep=0pt,minimum size=0.5mm},
	>=stealth,
	}

\tikzset{
	dot2/.style={circle,fill=black,draw=black,inner sep=0pt,minimum size=0.2mm},
	>=stealth,
	}

\tikzset{
	ddot/.style={circle,fill=black,draw=black,inner sep=0pt,minimum size=0.8mm},
	>=stealth,
	}

\tikzset{decision/.style={ 
        draw,
        diamond,
        aspect=1.5
    }}

\tikzset{dia2/.style
={diamond,fill=white,draw=black,inner sep=0pt,minimum size=1mm},
	>=stealth,
	}

\tikzset{dia/.style
={star,fill=black,draw=black,inner sep=0pt,minimum size=1mm},
	>=stealth,
	}

\tikzset{dia/.style
={diamond,fill=black,draw=black,inner sep=0pt,minimum size=1.3mm},
	>=stealth,
	}

\makeatletter
\def\DeclareSymbol#1#2#3{\xsavebox{#1}{\tikz[baseline=#2,scale=0.15]{#3}}}
\def\<#1>{\xusebox{#1}}
\makeatother

\newsavebox{\peA}
\newsavebox{\pneA}
\newsavebox{\plA}
\newsavebox{\pgA}
\newsavebox{\pleA}
\newsavebox{\pgeA}
\newsavebox{\pezA}

\savebox{\peA}{\tikz \draw (0,0) node[shape=circle,draw,inner sep=0pt,minimum size=8.5pt] {\scriptsize  $=$};}
\savebox{\pneA}{\tikz \draw (0,0) node[shape=circle,draw,inner sep=0pt,minimum size=8.5pt] {\footnotesize $\neq$};}
\savebox{\plA}{\tikz \draw (0,0) node[shape=circle,draw,inner sep=0pt,minimum size=8.5pt] {\scriptsize $<$};}
\savebox{\pgA}{\tikz \draw (0,0) node[shape=circle,draw,inner sep=0pt,minimum size=8.5pt] {\scriptsize $>$};}
\savebox{\pleA}{\tikz \draw (0,0) node[shape=circle,draw,inner sep=0pt,minimum size=8.5pt] {\scriptsize $\leqslant$};}
\savebox{\pgeA}{\tikz \draw (0,0) node[shape=circle,draw,inner sep=0pt,minimum size=8.5pt] {\scriptsize $\geqslant$};}
\savebox{\pezA}{\tikz \draw (0,0) node[shape=circle,draw,
fill=white, 
inner sep=0pt,minimum size=8.5pt]{} ;}

\def \peB{\mathchoice
{\scalebox{.7}{{\usebox{\peA}}}}
{\scalebox{.7}{{\usebox{\peA}}}}
{\scalebox{.7}{{\usebox{\peA}}}}
{}
}

\def \pezB{\mathchoice
{\scalebox{.7}{{\usebox{\pezA}}}}
{\scalebox{.7}{{\usebox{\pezA}}}}
{\scalebox{.7}{{\usebox{\pezA}}}}
{}
}

\newcommand{\pe}{\mathbin{{\peB}}}

\newcommand{\pez}{\mathbin{{\pezB}}}

\usetikzlibrary{decorations.pathmorphing, patterns,shapes}

\DeclareSymbol{dot}{-2.7}{\draw (0,0) node[dot]{};}

 \DeclareSymbol{1}{0}{\clip (-.7,0) rectangle (.7,1.6); \draw (0,0)  -- (0,1.2) node[dot] {};}

\DeclareSymbol{2}{0}{\draw (-0.5,1.2) node[dot] {} -- (0,0) -- (0.5,1.2) node[dot] {};}
\DeclareSymbol{3}{0}{\draw (0,0) -- (0,1.2) node[dot] {}; \draw (-.7,1) node[dot] {} -- (0,0) -- (.7,1) node[dot] {};}
\DeclareSymbol{30}{-3}{\draw (0,0) -- (0,-1); \draw (0,0) -- (0,1.2) node[dot] {}; \draw (-.7,1) node[dot] {} -- (0,0) -- (.7,1) node[dot] {};}
\DeclareSymbol{31}{-3}{\draw (0,0) -- (0,-1) -- (1,0) node[dot] {}; \draw (0,0) -- (0,1.2) node[dot] {}; \draw (-.7,1) node[dot] {} -- (0,0) -- (.7,1) node[dot] {};}
\DeclareSymbol{32}{-3}{\draw (0,0) -- (0,-1) -- (1,0) node[dot] {}; \draw (0,0) -- (0,-1) -- (-1,0) node[dot] {}; \draw (0,0) -- (0,1.2) node[dot] {}; \draw (-.7,1) node[dot] {} -- (0,0) -- (.7,1) node[dot] {};}
\DeclareSymbol{20}{-3}{\draw (0,0) -- (0,-1);\draw (-.7,1) node[dot] {} -- (0,0) -- (.7,1) node[dot] {};}
\DeclareSymbol{22}{-3}{\draw (0,0.3) -- (0,-1) -- (1,0) node[dot] {}; \draw (0,0.3) -- (0,-1) -- (-1,0) node[dot] {};\draw (-.7,1) node[dot] {} -- (0,0.3) -- (.7,1) node[dot] {};}
\DeclareSymbol{31p}{-3}{\draw (0,0) -- (0,-1) -- (1,0) node[dot] {}; \draw (0,0) -- (0,1.2) node[dot] {}; \draw (-.7,1) node[dot] {} -- (0,0) -- (.7,1) node[dot] {};
\draw (0,-1) node{\scaleobj{0.5}{\pez}};
\draw (0,-1) node{\scaleobj{0.5}{\pe}}; }
\DeclareSymbol{32p}{-3}{\draw (0,0) -- (0,-1) -- (1,0) node[dot] {}; \draw (0,0) -- (0,-1) -- (-1,0) node[dot] {}; \draw (0,0) -- (0,1.2) node[dot] {}; \draw (-.7,1) node[dot] {} -- (0,0) -- (.7,1) node[dot] {};
\draw (0,-1) node{\scaleobj{0.5}{\pez}};
\draw (0,-1) node{\scaleobj{0.5}{\pe}};}
\DeclareSymbol{22p}{-3}{\draw (0,0.3) -- (0,-1) -- (1,0) node[dot] {}; \draw (0,0.3) -- (0,-1) -- (-1,0) node[dot] {};\draw (-.7,1) node[dot] {} -- (0,0.3) -- (.7,1) node[dot] {};
\draw (0,-1) node{\scaleobj{0.5}{\pez}};
\draw (0,-1) node{\scaleobj{0.5}{\pe}};}
\DeclareSymbol{21p}{-3}{\draw (0,0.3) -- (0,-1) -- (1,0) node[dot] {};
\draw (-.7,1) node[dot] {} -- (0,0.3) -- (.7,1) node[dot] {};
\draw (0,-1) node{\scaleobj{0.5}{\pez}};
\draw (0,-1) node{\scaleobj{0.5}{\pe}};}

\DeclareSymbol{21}{-3}{\draw (0,0.3) -- (0,-1) -- (1,0) node[dot] {};
\draw (-.7,1) node[dot] {} -- (0,0.3) -- (.7,1) node[dot] {};
}

\DeclareSymbol{70}{-3}{
\draw (0,-0.4) -- (0,-1.2);
\draw (0,0.8) node[dot] {} -- (0,-0.4) -- (0.8,0.6);
\draw (0.8, .6) -- (1.4, 1.35)  node[dot] {};
\draw (0.8, .6) -- (1.8, .8)  node[dot] {};
\draw (0.8, .6) -- (0.85, 1.6)  node[dot] {};
\draw (0,.6) -- (0,-0.4) -- (-0.8,0.6) ;
\draw (-0.8, 0.6) -- (-1.4, 1.35)  node[dot] {};
\draw (-0.8, .6) -- (-1.8, .8)  node[dot] {};
\draw (-0.8, .6) -- (-0.85, 1.6)  node[dot] {};
}

\DeclareSymbol{7}{-3}{
\draw (0,0.2) node[dot] {} -- (0,-1) -- (0.8,0);
\draw (0.8, 0) -- (1.4, 0.75)  node[dot] {};
\draw (0.8, 0) -- (1.8, 0.2)  node[dot] {};
\draw (0.8, 0) -- (0.85, 1)  node[dot] {};
\draw (0,0) -- (0,-1) -- (-0.8,0) ;
\draw (-0.8, 0) -- (-1.4, 0.75)  node[dot] {};
\draw (-0.8, 0) -- (-1.8, 0.2)  node[dot] {};
\draw (-0.8, 0) -- (-0.85, 1)  node[dot] {};
}

\DeclareSymbol{11p}{-3}{\draw (0,0) node[dot2] {} -- (0,-1) -- (1,0) node[dot] {};
\draw (0,0) -- (0,1.2) node[dot] {};  
\draw (0,-1) node{\scaleobj{0.5}{\pez}};
\draw (0,-1) node{\scaleobj{0.5}{\pe}}; }
\DeclareSymbol{100}{-3}
{\draw[white] (-.4,0) -- (.4,0);
\draw (0,0) node[dot2] {} -- (0,-1)  ;
\draw (0,0) -- (0,1.2) node[dot] {};}  

\DeclareSymbol{320}{-3}{
\draw (0,0.4) -- (0,-1.2);
\draw (0,0.6) -- (0,-0.4) -- (1,0.6) node[dot] {};
\draw (0,0.6) -- (0,-0.4) -- (-1,0.6) node[dot] {};
\draw (0,0.6) -- (0,1.8) node[dot] {};
 \draw (-.7, 1.6) node[dot] {} -- (0,0.6) -- (.7,1.6) node[dot] {};}

\usetikzlibrary{mindmap,backgrounds,trees,arrows,decorations.pathmorphing,decorations.pathreplacing,positioning,shapes.geometric,shapes.misc,decorations.markings,decorations.fractals,calc,patterns}

\tikzset{>=stealth',
         cvertex/.style={circle,draw=black,inner sep=1pt,outer sep=3pt},
         vertex/.style={circle,fill=black,inner sep=1pt,outer sep=3pt},
         star/.style={circle,fill=yellow,inner sep=0.75pt,outer sep=0.75pt},
         tvertex/.style={inner sep=1pt,font=\scriptsize},
         gap/.style={inner sep=0.5pt,fill=white}}
\tikzstyle{mybox} = [draw=black, fill=blue!10, very thick,
    rectangle, rounded corners, inner sep=10pt, inner ysep=20pt]
\tikzstyle{boxtitle} =[fill=blue!50, text=white,rectangle,rounded corners]

\tikzstyle{decision} = [diamond, draw, fill=blue!20,
    text width=4.5em, text badly centered, node distance=2.5cm, inner sep=0pt]
\tikzstyle{block} = [rectangle, draw, fill=blue!20,
    text width=2.7cm, text centered, rounded corners, minimum height=3em]

\tikzstyle{line} = [draw, very thick, color=black!50, -latex']

\tikzstyle{cloud} = [draw, ellipse,fill=red!40,
node distance=2.5cm,
    minimum height=1em]

\tikzstyle{cloud2} = [draw, ellipse,fill=red!30, text=white,text width=10em, node distance=2.5cm, text centered, minimum height=4em]

\tikzstyle{cloud3} = [draw, ellipse, fill=cyan!30,
node distance=2.5cm,
    minimum height=3em,
    text width=1.7cm]

\tikzstyle{cloud4} = [draw, ellipse,fill=orange!70, node distance=2.5cm,
   text width=2.1cm,
           minimum height=3em]

\tikzstyle{cloud5} = [draw, ellipse,fill=red!20, node distance=2.5cm,
    text centered,
    text width=3.0cm,
    minimum height=3em]

\tikzstyle{cloud6} = [draw, ellipse,fill=red!20, node distance=2.5cm,
    text width=1.5cm,
    text centered,
    minimum height=3em]

\tikzset{
    position/.style args={#1:#2 from #3}{
        at=(#3.#1), anchor=#1+180, shift=(#1:#2)
    }
}

\newtheorem{theorem}{Theorem} [section]

\newtheorem{lemma}[theorem]{Lemma}
\newtheorem{proposition}[theorem]{Proposition}
\newtheorem{remark}[theorem]{Remark}

\newtheorem{definition}[theorem]{Definition}

\DeclareMathOperator{\Id}{Id}

\newcommand{\II}{\text{I \hspace{-2.8mm} I} }

\newcommand{\noi}{\noindent}
\newcommand{\Z}{\mathbb{Z}}
\newcommand{\R}{\mathbb{R}}
\newcommand{\T}{\mathbb{T}}

\newcommand{\wick}[1]{:\!{#1}\!:}

\newcommand{\E}{\mathbb{E}}

\renewcommand{\L}{\mathcal{L}}

\newcommand{\F}{\mathcal{F}}

\newcommand{\al}{\alpha}
\newcommand{\be}{\beta}
\newcommand{\dl}{\delta}

\newcommand{\nb}{\nabla}

\newcommand{\Dl}{\Delta}
\newcommand{\eps}{\varepsilon}

\newcommand{\g}{\gamma}
\newcommand{\G}{\Gamma}
\newcommand{\ld}{\lambda}
\newcommand{\Ld}{\Lambda}
\newcommand{\s}{\sigma}

\newcommand{\ft}{\widehat}

\newcommand{\wt}{\widetilde}
\newcommand{\cj}{\overline}

\newcommand{\dt}{\partial_t}

\newcommand{\ud}{\underline}

\newcommand{\ta}{\theta}

\renewcommand{\l}{\ell}
\renewcommand{\o}{\omega}
\newcommand{\Om}{\Omega}

\newcommand{\les}{\lesssim}

\newcommand{\jb}[1]
{\langle #1 \rangle}

\newcommand{\deff}{\stackrel{\textup{def}}{=}}

\newcommand{\jbb}[1]
{[\hspace{-0.6mm}[ #1 ]\hspace{-0.6mm}]}

\newcommand{\ind}{\mathbf 1}

\newcommand{\N}{\mathbb{N}}

\renewcommand{\H}{\mathcal{H}}

\newtheorem*{ackno}{Acknowledgements}

\newcommand{\I}{\mathcal{I}}

\newcommand{\If}{\mathfrak{I}}

\newcommand{\C}{\mathcal{C}}
\numberwithin{equation}{section}
\numberwithin{theorem}{section}

\newcommand{\PP}{\mathbb{P}}

\DeclareMathOperator{\Law}{Law}

\newcommand{\ZZ}{\mathfrak{Z}}

\newcommand{\Dr}{\Theta}

\newcommand{\dia}{\diamond}

\newcommand{\Ups}{\Upsilon}
\newcommand{\UUps}{{\ud \Upsilon}}

\usepackage{comment}

\makeatletter
\@namedef{subjclassname@2020}{%
  \textup{2020} Mathematics Subject Classification}
\makeatother

\begin{document}
\baselineskip = 13.5pt

\title[Weak universality of global dynamics for fractional hyperbolic $\Phi^4_3$-model]
{Weak universality of global dynamics for the fractional hyperbolic $\Phi^4_3$-model}

\author[R.~Liu]{Ruoyuan Liu}

\address{
Ruoyuan Liu,  School of Mathematics\\
The University of Edinburgh\\
and The Maxwell Institute for the Mathematical Sciences\\
James Clerk Maxwell Building\\
The King's Buildings\\
Peter Guthrie Tait Road\\
Edinburgh\\ 
EH9 3FD\\
United Kingdom\\
and Mathematical Institute\\
University of Bonn\\
Endenicher Allee 60\\
53115\\
Bonn\\
Germany}

\email{ruoyuanl@math.uni-bonn.de}

\subjclass[2020]{60H15, 81T08, 35L71, 35R60}
\keywords{weak universality; stochastic fractional nonlinear wave equation; fractional nonlinear wave equation; Gibbs measure}
%


\begin{abstract}
We continue the study of weak universality for the fractional hyperbolic $\Phi^4_3$-model initiated by Tzvetkov, Wang, and the author (2023). In this work, we establish weak universality in a regime where the associated Gibbs measure is mutually singular with respect to the underlying Gaussian measure. We first investigate the weak convergence of the Gibbs measures associated with macroscopic models to the fractional $\Phi^4_3$-measure constructed by Tzvetkov, Wang, and the author (2023). We then obtain a structured representation for these Gibbs measures via the stochastic quantization method. Finally, we prove dynamical weak universality for the fractional hyperbolic $\Phi^4_3$-model by using an extended second order expansion and exploiting the invariance of the Gibbs measures.
\end{abstract}

\maketitle

\tableofcontents

\section{Introduction}
\label{SEC:intro}

\subsection{Overview}
\label{SUB:intro}
In this paper, we continue our investigation in \cite{LTzW} on weak universality of the cubic stochastic damped fractional nonlinear wave equation on the three-dimensional torus $\T^3 = (\R / 2 \pi \Z)^3$, formally written as:
\begin{align}
\dt^2 u + \dt u + (1 - \Dl)^\al u + u^3 = \sqrt{2} \xi,
\label{SdNLW0}
\end{align}

\noi
where $\al \geq 1$ and $\xi$ is a space-time white noise on $\R_+ \times \T^3$. From our previous study \cite{LTzW} on \eqref{SdNLW0}, the equation is also referred to as the {\it fractional hyperbolic $\Phi^4_3$-model}.

Let us first briefly discuss the history of weak universality. The term {\it weak universality} refers to the phenomenon that a given model arises as a universal scaling limit of microscopic processes describing the large-scale behavior of solutions to a broad class of more general equations. In particular, although the microscopic processes may involve many parameters, only a few of them remain nontrivial in the limiting model. This is due to the fact that, under an appropriate scaling, most nonlinear interactions become ``weak'' as a result of the scaling factors, which explains the term ``weak universality''. An especially interesting case arises when the model is a singular stochastic PDE, i.e.~a stochastic PDE whose solution is not a function but merely a distribution due to the roughness of the stochastic forcing. In this setting, a renormalization procedure is required to give meaning to the corresponding nonlinear interactions, leading to additional mathematical challenges.

The first weak universality result for singular stochastic PDEs was established by Hairer-Quastel~\cite{HQ} for the Kardar-Parisi-Zhang equation. In this work, the authors studied the convergence of solutions to microscopic models with even polynomial nonlinearities under two different scaling regimes, using the theory of regularity structures developed by Hairer in \cite{Hai13, Hai14-1, Hai14-2, Hai15}. Subsequently, Hairer-Xu~\cite{HX18} applied the theory of regularity structures to prove weak universality for three-dimensional reaction-diffusion equations, including the well-studied parabolic $\Phi^4_3$-model (i.e.~the three-dimensional stochastic quantization equation). In particular, the authors in \cite{HX18} considered the weakly nonlinear regime (corresponding to small nonlinearity) for polynomial nonlinearities and the weak noise regime (corresponding to small noise) for more general nonlinearities with sufficient regularity. The case of more general nonlinearities in the weakly nonlinear regime for the parabolic $\Phi^4_3$-model was then studied by Furlan-Gubinelli in \cite{FG}, which is the first weak universality result in the framework of paracontrolled distributions developed by Gubinelli-Imkeller-Perkowski~\cite{GIP} and Catellier-Chouk~\cite{CC}. See also \cite{GP16, GP17, SX, HX19, MP, Duch21, EX, ZZ, KWX, Yang} for further results on weak universality of singular stochastic parabolic equations in various settings and generalizations.

Compared to parabolic equations, weak universality for (stochastic) hyperbolic equations is more challenging due to the weaker smoothing properties of the underlying equations. Recently, there has been several important developments in the study of weak universality for (stochastic) hyperbolic equations. In the weak noise regime, Gubinelli-Koch-Oh \cite{GKO} established a weak universality result for the two-dimensional cubic stochastic nonlinear wave equation (NLW) as a consequence of their local well-posedness argument. More precisely, under an appropriate choice of scaling parameters, they showed local-in-time convergence of solutions to microscopic models with a general odd, bounded nonlinearity possessing a sufficient number of bounded derivatives. See also \cite{GKO2} for a similar weak universality result for the three-dimensional quadratic stochastic NLW.

Later, in the setting of the two-dimensional cubic fractional NLW with order $\al < 1$ (corresponding to weaker dispersion than that of the usual NLW with order $\al = 1$), Sun-Tzvetkov-Xu \cite{STzX} established a global-in-time weak universality result with Gaussian initial data. In this work, the authors considered the weakly nonlinear regime with an arbitrary even polynomial satisfying certain structural conditions. A remarkable feature of the setting in \cite{STzX} is that the high-degree polynomial nonlinearity is energy super-critical.
An additional challenge, compared to parabolic equations, is the lack of $L^\infty$-based estimates (or, more generally, $L^p$-based estimates for $p \gg 1$), which are crucial for controlling the high-degree terms in the nonlinearities.
To overcome these difficulties, the authors of \cite{STzX} exploited the invariance of the Gibbs measures in two essential ways. Firstly, by using the invariance of the Gibbs measure associated with the limiting model, they obtained an a priori bound on the growth of the limiting solution in the relevant solution space, which allowed them to establish global-in-time convergence via a bootstrap argument. Secondly, inspired by Bourgain-Bulut~\cite{BB14}, they exploited the invariant Gibbs measures of the macroscopic models to obtain an a priori bound on the $L^p$-norms of the solutions for $p \gg 1$, thereby overcoming the aforementioned difficulty in controlling the high-degree nonlinearities.

The result in \cite{STzX} was subsequently extended to the three-dimensional setting and to the stochastic setting in \cite{LTzW}, where Tzvetkov, Wang, and the author established a weak universality result for the three-dimensional stochastic cubic stochastic (damped) fractional NLW, i.e.~the fractional hyperbolic $\Phi^4_3$-model \eqref{SdNLW0} with order $\al > \frac 98$, corresponding to stronger dispersion than that of the usual NLW with order $\al = 1$. The main novelty in \cite{LTzW} is the use of space-time analysis based on the $X^{s,b}$-spaces (see Subsection~\ref{SUB:Xsb}), which allowed the authors to establish weak universality throughout the entire regime in which the  Gibbs measure of the limiting model is equivalent to its underlying Gaussian measure, namely, for all $\al > \frac 98$.  
We emphasize that, in both \cite{STzX} and \cite{LTzW}, the Gibbs measures of the limiting models are equivalent to their underlying Gaussian measures. Moreover, the convergence of the Gibbs measures associated with the macroscopic models is sufficiently strong to allow for some flexibility in switching between Gibbs measures and the corresponding underlying Gaussian measures.

The situation becomes drastically different and more challenging when the Gibbs measure of the limiting model is singular with respect to its underlying Gaussian measure. This is the main focus of this paper. More precisely, we prove weak universality for the fractional hyperbolic $\Phi^4_3$-model with order $1 < \al \leq \frac 98$. As proved in \cite[Theorem~1.1~(ii)]{LTzW} (see also \cite[Corollary~4.4]{HKN}), the Gibbs measure and the underlying Gaussian measure are mutually singular in the range $1 < \al \leq \frac 98$. Moreover, as we discuss in Subsection~\ref{SUB:Gibbs} below, we only have weak convergence of the Gibbs measures associated with the macroscopic models to the limiting Gibbs measure. This prevents us from directly using the more convenient underlying Gaussian measures to obtain a priori bounds for the solutions. The main goal of this paper is to overcome these difficulties. To the best of our knowledge, this is the first work on weak universality for (stochastic) hyperbolic equations in a regime where the Gibbs measure is singular with respect to the underlying Gaussian measures. Although the range $1 < \al \leq \frac 98$ does not include the case of the usual NLW ($\al = 1$), we believe that the methods developed in this paper provide a step towards this ultimate goal; see Remark~\ref{RMK:meas1} and Remark~\ref{RMK:al1} below.

\subsection{Weak universality for the fractional hyperbolic $\Phi^4_3$-model}
\label{SUB:wuwave}
Let us now be more precise on the models that we work on in this paper. We first recall that in the fractional hyperbolic $\Phi^4_3$-model \eqref{SdNLW0}, the damping term $\dt u$ is required for the invariance of the associated Gibbs measure. This is due to the fact that \eqref{SdNLW0} can be viewed as the superposition of the fractional NLW dynamics for the component $u$: 
\begin{align}
\dt^2 u + (1 - \Dl)^\al u + u^3 = 0
\label{fNLW0}
\end{align}

\noi
and the Ornstein-Uhlenbeck process for the component $\dt u$:
\begin{align*}
\dt (\dt u) + \dt u = \sqrt{2} \xi,
\end{align*}

\noi
both of which preserve the Gibbs measure; see the discussion in \cite[Section~6.1]{LTzW}. For fractional NLW \eqref{fNLW0}, we see that one derivative in time corresponds to $\al$ derivatives in space. Thus, after fixing the correct scaling parameters for the damping term and the stochastic forcing, we consider
the microscopic process $U_N$ satisfying the following stochastic damped fractional NLW on the dilated torus $\T_N^3 = (\R / 2 \pi N \Z)^3$:
\begin{align}
\dt^2 U_N + N^{- \al} \dt U_N + (1 - \Dl_N)^\al U_N + N^{- \ta} \Pi_N V' (\Pi_N U_N) = \sqrt{2} N^{- \frac{\al}{2}} \xi_N.
\label{micro}
\end{align}

\noi
Here, $1 < \al \leq \frac 98$, $\ta > 0$ is a parameter to be fixed later, $\xi_N$ is a space-time white noise on $\R_+ \times \T_N^3$, $\Pi_N$ is the Fourier projection on $\T_N^3$ given by
\begin{align*}
\Pi_N f (x) \deff \frac{1}{(2 \pi N)^{\frac 32}} \sum_{\substack{ n \in \Z^3 \\ |n| \leq N}} (\F_N f) (n) e^{i n \cdot \frac{x}{N}} = \frac{1}{(2 \pi N)^3} \sum_{\substack{ n \in \Z^3 \\ |n| \leq N}} \bigg( \int_{\T_N^3} f(x) e^{- i n \cdot \frac{x}{N}} dx \bigg) e^{i n \cdot \frac{x}{N}},
\end{align*}

\noi
and the differential operator $(1 - \Dl_N)^\al$ is defined by (see also Remark~\ref{RMK:Lap} below)
\begin{align}
\F_N \big( (1 - \Dl_N)^\al f \big) (n) \deff \Big( \frac{1 + |n|^2}{N^2} \Big)^\al (\F_N f) (n),
\label{dif_op}
\end{align}

\noi
and $V$ is an even polynomial of the form
\begin{align*}
V (z) = \sum_{j = 1}^m a_j z^{2 j}
\end{align*}

\noi
with $m \geq 2$. 
To avoid confusion of all the parameters appearing in \eqref{micro}, one can view it as the superposition of the deterministic fractional NLW for the component $U_N$:
\begin{align}
\dt^2 U_N + (1 - \Dl_N)^\al U_N + N^{- \ta} \Pi_N V' (\Pi_N U_N) = 0
\label{micro0}
\end{align}

\noi
and the Ornstein-Uhlenbeck process for the component $\dt U_N$:
\begin{align*}
\dt (\dt U_N) + N^{-\al} \dt U_N = \sqrt{2} N^{- \frac{\al}{2}} \xi_N.
\end{align*}

\noi
Although the model \eqref{micro} suggests that we are in the weak noise regime, we refer to \eqref{micro} as the weakly nonlinear regime just as in \cite{STzX}. The main reason is that the same analysis in this paper applies to the deterministic fractional NLW \eqref{micro0} in the weakly nonlinear regime.
Nevertheless, in this paper, we choose to focus on the stochastic model \eqref{micro} in order to keep consistency with our previous study in \cite{LTzW}.

A direct but tedious computation yields that the dynamics of the linear part of the equation \eqref{micro}:
\begin{align*}
\dt^2 \Phi_N + N^{- \al} \dt \Phi_N + (1 - \Dl_N)^{\al} \Phi_N = \sqrt{2} N^{- \frac{\al}{2}} \xi_N
\end{align*}

\noi
preserves the Gaussian measure
\begin{align*}
\text{``}Z_N^{-1} e^{- \frac 12 \| (\Phi, \dt \Phi) \|_{H^{\al} \times L^2 (\T_N^3)} } d \Phi d (\dt \Phi)\text{''}
\end{align*}

\noi
induced by the random distributions $(\Phi_{0, N}^{\o}, \Phi_{1, N}^{\o})$ given by
\begin{align*}
\Phi_{0, N}^{\o} (x) &\deff \frac{N^\al}{(2 \pi N)^{\frac 32}} \sum_{n \in \Z^3} \frac{g_n (\o)}{\jb{n}^{\al}} e^{i n \cdot \frac{x}{N}}, \\
\Phi_{1, N}^{\o} (x) &\deff \frac{1}{(2 \pi N)^{\frac 32}}  \sum_{n \in \Z^3} h_n (\o) e^{i n \cdot \frac{x}{N}}.
\end{align*}

\noi
Here, $\{g_n\}_{n \in \Z^3}$ and $\{h_n\}_{n \in \Z^3}$ are independent standard complex Gaussian random variables conditioned such that $g_{-n} = \cj{g_n}$ and $h_{-n} = \cj{h_n}$ for all $n \in \Z^3$ and are independent with the space-time white noise $\xi_N$.
Let us define the variance
\begin{align}
\begin{split}
\wt{\s}_N &\deff \E \big[ |\Pi_N \Phi_{0, N}^{\o} (x)|^2 \big] = \frac{1}{(2 \pi)^{3} N^{3 - 2 \al}} \sum_{\substack{n \in \Z^3 \\ |n| \leq N}} \frac{1}{\jb{n}^{2 \al}} \\
&\,= \frac{1}{(2 \pi)^3} \int_{\{|\xi| \leq 1\}} \frac{1}{|\xi|^{2 \al}} d \xi + O(N^{- (3 - 2 \al - \eps)}),
\end{split}
\label{sigmaN}
\end{align}

\noi
for any $\eps > 0$ arbitrarily small and we define
\begin{align}
\wt{\s} \deff \frac{1}{(2 \pi)^3} \int_{\{|\xi| \leq 1\}} \frac{1}{|\xi|^{2 \al}} d \xi.
\label{sigmat}
\end{align}

\noi
Inspired by \cite{HX18, STzX}, we define the averaged potential (or the effective potential) $\jb{V}$ as 
\begin{align*}
\jb{V} (z) \deff \frac{1}{\sqrt{2 \pi \wt{\s}}} \int_\R V (z + y) e^{- \frac{y^2}{2 \wt{\s}}} dy,
\end{align*}

\noi
so that $\jb{V}$ is of the form
\begin{align*}
\jb{V} (z) = \sum_{j = 0}^m \cj{a}_j z^{2j},
\end{align*}

\noi
where one can compute that
\begin{align}
\cj{a}_{j} = \frac{1}{(2j)!} \E \big[ V^{(2j)} (\mathcal{N} (0, \wt{\s})) \big] = \frac{1}{(2j)!} \sum_{k = j}^m \frac{(2k)!}{(2k - 2j)!!} a_k \wt{\s}^{(k - j)}
\label{aj}
\end{align}

\noi
for each $0 \leq j \leq m$ with $\mathcal{N} (0, \wt{\s})$ being the centered Gaussian distribution with variance $\wt{\s}$. In this paper, we make the following assumptions on $V$:
\begin{enumerate}
\item[(i)] $\jb{V}'' (0) = 0$ or equivalently 
\begin{align}
\cj a_1 = 0;
\label{a1}
\end{align}

\item[(ii)] $\jb{V} (z) - \jb{V} (0) > 0$ for all $z \neq 0$ or equivalently (along with assumption (i))
\begin{align}
\sum_{j = 2}^m \cj{a}_j z^{2 (j - 2)} > 0
\label{pos0}
\end{align}

\noi
for all $z \neq 0$; \smallskip

\item[(iii)] $\jb{V}^{(4)} (z) - \jb{V}^{(4)} (0) > 0$ for all $z \neq 0$ or equivalently
\begin{align}
\sum_{j = 3}^m \cj{a}_j z^{2 (j - 2)} > 0
\label{pos}
\end{align}

\noi
for all $z \neq 0$.
\end{enumerate}

\noi
We will mention in detail in the later part of the introduction how these assumptions work. Briefly speaking, the assumption~(i) ensures that the divergent parts of the renormalization terms in both the dynamics and the measures are cancelled out (see \eqref{a1N} below); the assumption~(ii) is an almost necessary condition for the boundedness of the densities for the truncated Gibbs measures (see Proposition~\ref{PROP:nonconv} below); the assumption~(iii) guarantees that the truncated Gibbs measure converges to the desired ultimate Gibbs measure (see the discussion below Proposition~\ref{PROP:nonconv}). Note that if we fix any $a_2, \dots, a_m > 0$, we can always find $a_1 < 0$ using \eqref{aj} such that our assumptions on $V$ are satisfied.

We now perform the following rescaling:
\begin{align*}
u_N (t, x) \deff N^{\frac 32 - \al} U_N (N^\al t, N x)
\end{align*}

\noi
for $(t, x) \in \R_+ \times \T^3$, which leads to the following macroscopic model on $\T^3$:
\begin{align*}
\dt^2 u_N + \dt u_N + (1 - \Dl)^\al u_N + N^{\frac 32 + \al - \ta} \pi_N V' \Big( \frac{\pi_N u_N}{N^{\frac 32 - \al}} \Big) = \sqrt{2} \xi,
\end{align*}

\noi
where $\xi$ is a space-time white noise on $\R_+ \times \T^3$ and $\pi_N$ denotes the frequency projection onto frequencies $\{ |n| \leq N \}$. In order for $u_N$ to converge to the solution of a cubic equation as $N \to \infty$, we need $N^{\frac 32 + \al - \ta}$ to cancel out $N^{-3 (\frac 32 - \al)}$. This gives $\frac 32 + \al - \ta = 3 (\frac 32 - \al)$, so that we choose $\ta = 4 \al - 3$. For each $N \in \N$, we define the polynomial $V_N$ by
\begin{align}
V_N (z) \deff N^{4 (\frac 32 - \al)} V \Big( \frac{z}{N^{\frac 32 - \al}} \Big) = \sum_{j = 0}^m \cj{a}_{j, N} N^{-(2j - 4) (\frac 32 - \al)} H_{2j} (z; N^{3 - 2 \al} \wt{\s}_N),
\label{defVN}
\end{align}

\noi
where $H_k (z; \s)$ is the Hermite polynomial of degree $k$ with variance parameter $\s > 0$ as defined in \eqref{hermite} below, $\wt{\s}_N$ is as defined in \eqref{sigmaN}, and the coefficients $\cj{a}_{j, N}$ can be explicitly computed as
\begin{align}
\cj{a}_{j, N} = \frac{1}{(2j)!} \E \big[ V^{(2j)} (\mathcal{N} (0, \wt{\s}_N)) \big] = \frac{1}{(2j)!} \sum_{k = j}^m \frac{(2k)!}{(2k - 2j)!!} a_k \wt{\s}_N^{(k - j)}.
\label{ajN}
\end{align}

\noi
We will disregard the 0-degree term (i.e.~the term with $j = 0$) in $V_N(z)$ in our later discussion in this paper, since it does not play any important role in our analysis. From \eqref{sigmaN} and \eqref{sigmat}, we have
\begin{align}
\wt{\s}_N - \wt{\s} = O(N^{-(3 - 2 \al - \eps)}),
\label{sigma_diff}
\end{align}

\noi
so that for $2 \leq j \leq m$, we have $\cj{a}_{j, N} \to \cj{a}_j$ as $N \to \infty$ in view of \eqref{aj} and \eqref{ajN}. For $j = 1$, we can use similar steps from \cite[Proposition~1.4]{STzX} to obtain the following more precise approximation
\begin{align}
\cj{a}_{1, N} N^{2 (\frac 32 - \al)} = \cj{a}_1 N^{2 (\frac 32 - \al)} + \frac{a}{2} + O(N^{2 - 2 \al}) + O(N^{-2 (\frac 32 - \al)})
\label{a1N0}
\end{align}

\noi
for some $a \in \R$. It is in \eqref{a1N0} that we see the necessity of the assumption $\cj{a}_1 = 0$ in \eqref{a1}, which cancels out the divergent part and leads us to 
\begin{align}
\cj{a}_{1, N} N^{2 (\frac 32 - \al)} = \frac{a}{2} + O(N^{2 - 2 \al}) + O(N^{-2 (\frac 32 - \al)}).
\label{a1N}
\end{align}

Using the notation $V_N$ above, we now write the macroscopic model for $u_N$ as
\begin{align}
\dt^2 u_N + \dt u_N + (1 - \Dl)^\al u_N + \pi_N V_N' (\pi_N u_N) = \sqrt{2} \xi.
\label{macro}
\end{align}

\noi
Our goal in this paper is to show that, under suitable initial conditions and in some suitable function space, $u_N$ converges as $N \to \infty$ to a process $u$ which formally satisfies the following fractional hyperbolic $\Phi^4_3$-model:
\begin{align}
\dt^2 u + \dt u + (1 - \Dl)^\al u + a u + 4 \cj{a}_2 \wick{u^3} \, = \sqrt{2} \xi.
\label{SdfNLW}
\end{align}

\noi
In fact, the process $u$ is obtained as a limit of the process $u_N^\dagger$ as $N \to \infty$, where $u_N^\dagger$ satisfies the following truncated fractional hyperbolic $\Phi^4_3$-model:
\begin{align}
\dt^2 u_N^\dagger + \dt u_N^\dagger + (1 - \Dl)^\al u_N^\dagger + a \pi_N u_N^\dagger + 4 \cj{a}_2 \pi_N \big( \wick{ (\pi_N u_N^\dagger)^3 } \big) = \sqrt{2} \xi.
\label{SdfNLW_N}
\end{align}

\noi
The precise meaning of the Wick power $\wick{ (\pi_N u_N^\dagger)^3 }$ will be provided in Subsection~\ref{SUB:intro_dyn} below.

\begin{remark} \rm
\label{RMK:Lap}
In this paper, we choose to work with the differential operator $(1 - \Dl_N)^\al$ of the form \eqref{dif_op} in order to be consistent with our previous work \cite{LTzW}. Alternatively, one can start with a homogeneous operator $(- \Dl_N)^\al$ (the Fourier multiplier with symbol $\frac{|n|^{2 \al}}{N^{2 \al}}$) as in \cite{STzX} and insert a linear correction term coming from the nonlinearity. In this setting, the proof of weak universality follows from essentially the same way as the proof we present in this paper.
\end{remark}

\subsection{Gibbs measure for the macroscopic model}
\label{SUB:Gibbs}
As mentioned in Subsection~\ref{SUB:intro}, our method relies on convergence and invariance of the corresponding Gibbs measures, which we define as follows. 

Let us first start with the base Gaussian measures. Given $s \in \R$, we let $\mu_s$ be a Gaussian measure with the Cameron-Martin space $H^s (\T^3)$, formally written as
\begin{align*}
d \mu_s = Z_s^{-1} e^{- \frac 12 \| u \|_{H^s}^2} du = Z_s^{-1} \prod_{n \in \Z^3} e^{- \frac 12 \jb{n}^{2 s} |\ft u (n)|^2} d \ft u (n),
\end{align*}

\noi
where $\ft u$ denotes the Fourier transform of $u$. In this paper, we are mainly concerned with the case when $s = \al$ with $\al \geq 1$ and $s = 0$, where in the later case $\mu_0$ is the white noise measure on $\T^3$. For a more precise definition of $\mu_s$, we let $\{g_n\}_{n \in \Z^3}$ and $\{h_n\}_{n \in \Z^3}$ be two sequences of independent standard complex-valued Gaussian random variables conditioned such that $g_{-n} = \cj{g_n}$ and $h_{-n} = \cj{h_n}$ for all $n \in \Z^3$. In particular, $g_0$ and $h_0$ are real-valued Gaussian random variables. We then define the random distributions
\begin{align}
u_0^{\o} = \frac{1}{(2 \pi)^{\frac 32}} \sum_{n \in \Z^3} \frac{g_n (\o)}{\jb{n}^\al} e^{in \cdot x} \quad \text{and} \quad u_1^{\o} = \frac{1}{(2 \pi)^{\frac 32}} \sum_{n \in \Z^3} h_n (\o) e^{in \cdot x},
\label{gauss_init}
\end{align}

\noi
and so we can now define
\begin{align}
\mu_\al \deff \Law (u_0^\o) \quad \text{and} \quad \mu_0 \deff \Law (u_1^\o).
\label{mu_al}
\end{align}

\noi
Note that the measure $\vec \mu_\al = \mu_\al \otimes \mu_0$ is supported on
\begin{align*}
\vec H^s (\T^3) \deff H^s (\T^3) \times H^{s - \al} (\T^3)
\end{align*}

\noi
for $s < \al - \frac 32$ but not for $s \geq \al - \frac 32$; see \cite[Lemma~B.1]{BTz08}.

We now restrict our attention to $1 < \al \leq \frac 98$ and move on to the Gibbs measures, starting with the fractional hyperbolic $\Phi_3^4$-measure $\rho$ corresponding to the fractional hyperbolic $\Phi_3^4$-model. Given $1 < \al \leq \frac 98$, a typical element in the support of the Gaussian measure $\mu_\al$ is merely a distribution, so that a proper renormalization is needed to make sense of powers of the element in the support of the Gibbs measures. Given $N \in \N$, we define
\begin{align}
\s_N = \E \big[ |\pi_N u_0^\o (x)|^2 \big] = \frac{1}{(2 \pi)^3} \sum_{\substack{n \in \Z^3 \\ |n| \leq N}} \frac{1}{\jb{n}^{2 \al}}
\label{sigma}
\end{align}

\noi
for any $x \in \T^3$, where $\pi_N$ denotes the frequency projection onto $\{ |n| \leq N \}$. For any distribution $u$ and $k \in \N$, we define the Wick power
\begin{align*}
\wick{(\pi_N u)^k} \, \deff H_k (\pi_N u; \s_N),
\end{align*}

\noi
where $H_k (x; \s)$ is the Hermite polynomial of degree $k$ with variance parameter $\s > 0$ as defined in \eqref{hermite} below.
We now define the renormalized potential energy
\begin{align*}
R_N (u) \deff  \int_{\T^3} \Big( \frac{a}{2} \wick{(\pi_N u)^2} +  \cj{a}_2  \wick{(\pi_N u)^{4}} \Big) dx,
\end{align*}

\noi
and we want to construct the fractional hyperbolic $\Phi_3^4$-measure $\rho$ as the limit of the truncated measures
\begin{align*}
\text{``} d \rho_N (u) = Z_N^{-1} e^{- R_N (u)} d \mu_\al (u) \text{''}
\end{align*}

\noi
as $N \to \infty$. This is essentially achieved in \cite[Theorem~1.1~(ii)]{LTzW},\footnote{In \cite[Theorem~1.1~(ii)]{LTzW}, the authors only considered the case when $a = 0$. Nevertheless, the renormalized quadratic term in the case $a \neq 0$ does not pose any additional difficulty.} which allows us to define
\begin{align}
\rho = \substack{\text{{\small w-lim}} \\ N \to \infty} \, \rho_N
\label{rho}
\end{align}

\noi
with
\begin{align}
d\rho_N (u) \deff Z_N^{-1} e^{- R_N^\dia (u)} d \mu_\al (u),
\label{rhoN}
\end{align}

\noi
where w-lim refers to the weak convergence of measures, the potential energy $R_N^\dia$ is defined by
\begin{align}
R_N^\dia (u) \deff R_N (u) + \al_N' = \int_{\T^3} \Big( \frac{a}{2} \wick{(\pi_N u)^2} +  \cj{a}_2  \wick{(\pi_N u)^{4}} \Big) dx + \al_N'
\label{deffRN}
\end{align}

\noi
for some diverging constant $\al_N' > 0$, and
\begin{align}
Z_N \deff \int e^{- R_N^\dia (u)} d \mu_\al (u).
\label{deffZN}
\end{align}

We now consider the Gibbs measures for the macroscopic models \eqref{macro}. Recalling the polynomial $V_N$ in \eqref{defVN}, we define the following potential energy:
\begin{align}
\mathcal{R}_N (u) \deff \int_{\T^3} V_N (\pi_N u) dx.
\label{defRN}
\end{align}

\noi
Then, similar to \eqref{rhoN}, we define the following truncated Gibbs measure $\varrho_N$ associated to the macroscopic model \eqref{macro}:
\begin{align}
d \varrho_N (u) = \mathcal{Z}_N^{-1} e^{- \mathcal{R}^\dia_N (u)} d \mu_\al (u),
\label{nuN}
\end{align}

\noi
where
\begin{align}
\mathcal{R}_N^\dia (u) \deff \mathcal{R}_N (u) + \al_N = \int_{\T^3} V_N (\pi_N u) dx + \al_N
\label{defRNd}
\end{align}

\noi
for some diverging constant $\al_N > 0$ defined in \eqref{alphaN} below, and
\begin{align}
\mathcal{Z}_N \deff \int e^{- \mathcal{R}_N^\dia (u)} d \mu_\al (u).
\label{defZN}
\end{align}

\noi
Our first result in this paper is about the convergence of the truncated Gibbs measure $\varrho_N$ to the fractional $\Phi_3^4$-measure $\rho$ as $N \to \infty$.
\begin{theorem}
\label{THM:meas}
Let $1 < \al \leq \frac 98$. Then, there exists $0 < C < \infty$ such that
\begin{align}
\sup_{N \in \N} \big\| e^{- \mathcal{R}_N^\dia (u)} \big\|_{L^1 (\mu_\al)} \leq C.
\label{L1bdd}
\end{align}

\noi
Moreover, the truncated Gibbs measure $\varrho_N$ in \eqref{nuN} converges weakly to the fractional $\Phi^4_3$-measure $\rho$ in \eqref{rho} as $N \to \infty$.
\end{theorem}

One may compare the situation in Theorem~\ref{THM:meas} with that in \cite[Proposition~4.3]{LTzW}, where the authors proved the convergence of the truncated Gibbs measure in the regime $\frac 98 < \al < \frac 32$. The main difference in the case $\frac 98 < \al < \frac 32$ is that (i) the diverging constant $\al_N$ is not needed in the potential energy \eqref{defRNd}, (ii) the uniform bound \eqref{L1bdd} holds in $L^p (\mu_\al)$ for any finite $p \geq 1$, and (iii) the convergence of the truncated Gibbs measure $\varrho_N$ to the fractional $\Phi_3^4$-measure $\rho$ holds in total variation. In addition, the fractional $\Phi_3^4$-measure $\rho$ is singular with respect to the base Gaussian measure $\mu_\al$, whereas when $\frac 98 < \al < \frac 32$, the two measures $\rho$ and $\mu_\al$ are equivalent; see \cite[Theorem~1.1~(ii)]{LTzW}. 

Our proof of Theorem~\ref{THM:meas} is based on the variational approach by Barashkov-Gubinelli \cite{BG} and is presented in Section~\ref{SEC:meas}. Similar to \cite{OOT1, OOT2, LTzW}, we identify a divergent stochastic term and perform a change of variable, where the divergent part is canceled out by the diverging constant $\al_N$; see Subsection~\ref{SUB:meas}. This change of variable then allows us to show the tightness of the truncated Gibbs measures $\{\varrho_N\}_{N \in \N}$ and then prove the convergence of $\varrho_N$ to $\rho$; see Subsection~\ref{SUB:tight} and Subsection~\ref{SUB:meas_conv}, respectively. Note that the appearance of this divergent stochastic term is the main reason that causes the phase transition of the behaviors of the Gibbs measures at $\al = \frac 98$.

The main difficulty in our setting compared to \cite{OOT1, OOT2, LTzW} is that the potential energy $\mathcal{R}_N^\dia$ and the divergent stochastic term (see \eqref{defZZ}) has a more general form containing high-degree nonlinearities, which make it highly non-trivial to show the desired regularity properties of certain complicated products of stochastic objects. Nevertheless, we are able to control all the high order terms effectively when showing the $L^1$-boundedness of the densities of the truncated Gibbs measures $\{\varrho_N\}_{N \in \N}$, thanks to the negative powers of $N$ in front of high-degree terms and, more importantly, the positivity condition~\eqref{pos0}. In fact, we show that the positivity condition~\eqref{pos0} is almost a necessary assumption for showing the $L^1$-boundedness of the densities of $\{\varrho_N\}_{N \in \N}$ as in \eqref{L1bdd}, in a spirit similar to \cite[Proposition~1.12]{STzX}.

\begin{proposition}
\label{PROP:nonconv}
Let $1 < \al \leq \frac 98$. Suppose that there exists $z' \in \R$ such that
\begin{align*}
\sum_{j = 2}^m \cj{a}_j (z')^{2 (j - 2)} < 0.
\end{align*}

\noi
Then, there exists a constant $C > 0$ such that $\log \mathcal{Z}_N > C N^{4 (\frac 32 - \al)}$ for all $N \in \N$. Consequently, the densities $\{ e^{- \mathcal{R}_N^\dia (u)} \}_{N \in \N}$ for the truncated Gibbs measures $\{\varrho_N\}_{N \in \N}$ is not bounded in $L^1 (\mu_\al)$.
\end{proposition}

We present the proof of Proposition~\ref{PROP:nonconv} in Subsection~\ref{SUB:nonconv}.

We encounter another difficulty when showing that the truncated Gibbs measure $\varrho_N$ actually converges to the fractional $\Phi^4_3$-measure $\rho$, which is again due to the high-degree terms in the potential energy $\mathcal{R}_N^\dia$. Briefly speaking, we need to show that the quadratic term and the quadruple term in $\mathcal{R}_N^\dia$ converges to corresponding renormalized powers, while the higher order terms should simply vanish. However, the assumption \eqref{pos0} does not seem strong enough to separate the two distinct convergence behaviors. This justifies the need of the stronger positivity assumption \eqref{pos}, which allows us to show that all terms with degree higher than the quadruple term indeed vanishes as $N \to \infty$.

\begin{remark} \rm
\label{RMK:alN}
The renormalization constant $\al_N$ in \eqref{defRNd} appears only at the level of the Gibbs measure and does not affect the dynamical problem \eqref{macro}, which is going to be discussed in Subsection~\ref{SUB:intro_dyn}.
\end{remark}

\begin{remark} \rm
\label{RMK:meas1}
Let us briefly discuss the difficulty for extending the measure convergence result in Theorem~\ref{THM:meas} to the case $\al = 1$. In this case, the limiting Gibbs measure is the well-studied $\Phi^4_3$-measure constructed in \cite{GJ, BG}, which was one of the early achievements in the constructive quantum field theory. See also the introductions in \cite{AK, GH21}. Compared to the case $\al > 1$, the construction of the $\Phi^4_3$-measure for $\al = 1$ requires an additional renormalization of the form
\begin{align*}
\g_N \int_{\T^3} \wick{(\pi_N u)^2} dx
\end{align*}

\noi
in the potential energy $R_N^\dia$ in \eqref{deffRN}, where $\g_N > 0$ diverges logarithmically as $N \to \infty$; see~\cite{BG}. Thus, when $\al = 1$, in order to obtain the convergence of the truncated Gibbs measure $\varrho_N$ in \eqref{nuN} to the $\Phi^4_3$-measure, we expect that a similar renormalization is also needed in the potential energy $\mathcal{R}_N^\dia$ in \eqref{defRNd}. This requires more delicate analysis on the high order terms coming from $\mathcal{R}_N^\dia$.
\end{remark}

\subsection{Representation of the Gibbs measure initial data}
\label{SUB:rep_intro}
As mentioned in previous subsections, the main difficulty in our setting, compared to weak universality in \cite{STzX, LTzW}, is the singularity of the Gibbs measure $\rho$ in \eqref{rho} with respect to the base Gaussian measure $\mu_\al$ in \eqref{mu_al}. Thus, in this case, to exploit the invariance of the Gibbs measure for the macroscopic model~\eqref{macro}, we cannot simply choose the initial data to be the Gaussian distributions $(u_0^\o, u_1^\o)$ in \eqref{gauss_init} as in \cite{STzX, LTzW}, but we need to work directly with initial data distributed by the truncated Gibbs measure $\varrho_N$ in \eqref{nuN}.

The singularity issue of the Gibbs measure with respect to the underlying Gaussian measure has already been considered in (stochastic) wave-type equations with various nonlinearities; see \cite{Bring2, OOT2, BDNY, LTzW}. In these works, the main goal was to prove global well-posedness of the corresponding models and invariance of the associated Gibbs measures. In particular, in the context of the (deterministic) hyperbolic $\Phi^4_3$-model, Bringmann-Deng-Nahmod-Yue \cite{BDNY} represented the truncated Gibbs measure initial data as a structured solution of the truncated parabolic $\Phi^4_3$-model. This is valid thanks to the fact that the parabolic $\Phi^4_3$-model is the Langevin dynamics for the $\Phi^4_3$-measure and thus leaves this measure invariant. The advantage of this structured solution is that it separates the Gaussian part and the remaining part, where the Gaussian part can be easily dealt with using tools from stochastic analysis and the remaining part has good enough regularity for the dynamical problem. See also \cite{OTWZ, BS} on similar uses of this strategy, which is referred to as the stochastic quantization method.

In this paper, we follow the idea from \cite{BDNY, OTWZ} and show a useful representation of the Gibbs measures $\varrho_N$ and $\rho$. Our result reads as follows.

\begin{theorem}
\label{THM:rep}
Let $1 < \al \leq \frac 98$. Then, there exist a probability space $(\Sigma, \mathcal{G}, \mathbb{Q})$ and random variables $z_0, v_{0, N}, v_0$ on $(\Sigma, \mathcal{G}, \mathbb{Q})$ such that
\begin{align*}
\Law (z_0) = \mu_\al, \quad \Law (z_0 + v_{0, N}) = \varrho_N, \quad \Law (z_0 + v_0) = \rho,
\end{align*}

\noi
where $\mu_\al$ is the Gaussian measure in \eqref{mu_al}, $\varrho_N$ is the truncated Gibbs measure in \eqref{nuN}, and $\rho$ is the fractional $\Phi_3^4$-measure in \eqref{rho}. Moreover, there exists $\eps > 0$ small such that $v_{0, N}$ converges in probability to $v_0$ in the space $\mathcal{C}^{\al - \frac 12 + \eps}$, and for any $\eta > 0$, there exists $K = K(\eta) \geq 1$ such that
\begin{align}
\mathbb{Q} \big( \| v_0 \|_{\mathcal{C}^{\al - \frac 12 + \eps}} > K \big) < \eta.
\label{v0bdd_goal}
\end{align}
\end{theorem}

We present the proof of Theorem~\ref{THM:rep} in Section~\ref{SEC:heat}, which relies on the following stochastic fractional nonlinear heat equation on $\T^3$:
\begin{align}
\partial_t X_N + (1 - \Dl)^\al X_N + \pi_N V_N' (\pi_N X_N) = \sqrt{2} \xi^h
\label{SNLHN0}
\end{align}

\noi
and the following fractional parabolic $\Phi^4_3$-model:
\begin{align}
\partial_t X + (1 - \Dl)^\al X + a X + 4 \cj{a}_2 \wick{X^3} \, = \sqrt{2} \xi^h.
\label{SNLH0}
\end{align}

\noi
Here, $\xi^h$ is a space-time white noise independent of $\xi$ in \eqref{macro} and the Wick power $\wick{X^3}$ can be viewed as a limit of the Hermite polynomial $H_3 (\pi_N X_N; \s_N)$ as $N \to \infty$, where $\s_N$ is as defined in \eqref{sigma}; see also the discussion after Proposition~\ref{PROP:GWP_h}. 
The equations \eqref{SNLHN0} and \eqref{SNLH0} are precisely the Langevin dynamics for the Gibbs measures $\varrho_N$ and $\rho$, respectively.
Briefly speaking, we want to solve the equations \eqref{SNLHN0} and \eqref{SNLH0} from time $t = 0$ to time $t = 1$ and the solutions at time $t = 1$ have the desired laws, structures, and regularities.
The stochastic fractional nonlinear heat equations \eqref{SNLHN0} and \eqref{SNLH0} have been previously studied in \cite{Duch21}, where Duch introduced the flow equation approach and proved convergence of $X_N$ to $X$ in a suitable function space. In this paper, however, in order to establish all the probabilistic properties in Theorem~\ref{THM:rep} in a more unified manner and not to complicate our presentation with more technical notations, we choose to proceed in a way that is more self-contained and rely on tools that are more compatible to those in this paper.

At a formal level, the fractional $\Phi^4_3$-measure $\rho$ is invariant under the flow of \eqref{SNLH0}. Our first task is to show global well-posedness of \eqref{SNLH0} and verify the invariance of $\rho$ under the flow of \eqref{SNLH0}. For well-posedness of \eqref{SNLH0}, we use a second order expansion of the Da~Prato-Debussche trick (see \cite{McKean, Bour96, DPD}), which separates out a linear stationary solution and a cubic stationary solution. The linear stationary solution contributes to the random variable $z_0$ in Theorem~\ref{THM:rep}, while the cubic stationary solution has good enough regularity and is incorporated in the remainder term $v_0$ in Theorem~\ref{THM:rep}. This gives the structure $X = z_0 + v_0$ for the solution to \eqref{SNLH0}. To prove global well-posedness of \eqref{SNLH0}, instead of following the deterministic approach as in \cite{MW2d, MW3d} for solving stochastic parabolic equations (see Remark~\ref{RMK:heatMW}), we choose to proceed in a  probabilistic approach similar to that in \cite{OOT2, LTzW} exploiting the invariance of the truncated fractional $\Phi^4_3$-measure $\rho_N$ in \eqref{rhoN} under the flow of a truncated version of \eqref{SNLH0}. This approach is more adapted for later purposes of proving weak universality and also yields interesting comparisons between the fractional parabolic $\Phi^4_3$-model and the fractional hyperbolic $\Phi^4_3$-model; see Subsection~\ref{SUB:SNLH} (in particular, Remark~\ref{RMK:heat_wave}) for more details.

Similarly, for the equation \eqref{SNLHN0} for $X_N$, we separate out a linear stationary solution $z_0$ and a generalized stochastic term which has the same regularity property as the cubic stationary solution mentioned above. This generalized stochastic term is incorporated in the remainder term $v_{0, N}$ in Theorem~\ref{THM:rep}, and so we have the structure $X_N = z_0 + v_{0, N}$. To prove the convergence of $v_{0, N}$ to $v_0$, we establish a key stability estimate and exploit the invariance of the truncated Gibbs measure $\varrho_N$ in \eqref{nuN} under the flow of \eqref{SNLHN0} and also the invariance of the fractional $\Phi^4_3$-measure $\rho$ under the flow of \eqref{SNLH0}; see Subsection~\ref{SUB:heat_wu} for more details. The proof of Theorem~\ref{THM:rep} is then given in Subsection~\ref{SUB:rep}.

\begin{remark} \rm
As mentioned above, the idea of the representation in Theorem~\ref{THM:rep} comes from \cite{BDNY} and \cite{OTWZ}. Let us compare the differences between the representation in Theorem~\ref{THM:rep} and those in \cite{BDNY} and \cite{OTWZ}.

In \cite{BDNY}, Bringmann-Deng-Nahmod-Yue proved global well-posedness and invariance of the Gibbs measure for the (deterministic) hyperbolic $\Phi^4_3$-model (i.e. formally the equation \eqref{fNLW0} with $\al = 1$), where they essentially used the local-in-time theory of the frequency truncated parabolic $\Phi^4_3$-model. 
With the frequency truncated, the models are known to be globally well-posed and have the invariance property from standard tools on finite dimensional spaces. In our case in this paper, however, we need global well-posedness of the parabolic model \eqref{SNLH0}, i.e.~the equation without the frequency truncation. This is due to the fact that, for the later purpose of proving dynamical weak universality, we need more concrete information from the convergence of Gibbs measures $\varrho_N$ in \eqref{nuN} to $\rho$ in \eqref{rho}, which requires us to work on the limiting parabolic model \eqref{SNLH0}. Moreover, we need to prove the invariance of $\rho$ under the flow of \eqref{SNLH0}, for which a random local existence time of the equation will not be useful. In this sense, our representation in Theorem~\ref{THM:rep} is closer to that in \cite{OTWZ} in the study of the hyperbolic $\Phi^{k + 1}_2$-model ($k \in 2 \mathbb{N} + 1$) on the plane (i.e.~formally the equation \eqref{fNLW0} on $\R^2$ with $\al = 1$ and $u^3$ replaced by $u^k$).
\end{remark}

\begin{remark} \rm
\label{RMK:heatMW}
In \cite{MW2d} and \cite{MW3d}, the authors studied (a renormalized version of) the following the parabolic $\Phi^4_2$-model and the parabolic $\Phi^4_3$-model, respectively:
\begin{align}
\dt X - \Dl X + X^3 = \xi,
\label{heat}
\end{align}

\noi
where $\xi$ denotes a space-time white noise on $\R_+ \times \T^d$ for $d = 2, 3$. To prove global well-posedness of \eqref{heat}, they tested the equation \eqref{heat}\footnote{More precisely, in \cite{MW2d, MW3d}, the testing was performed on another nonlinear heat equation with a smoother unknown. However, we omit details here and only discuss ideas.} against $X^{p - 1}$ for some suitable even integer $p \geq 2$ and derived an energy inequality to control the growth of $\| X \|_{L^p}$, which allowed them to iterate the local well-posedness argument. In particular, by using an integration by parts, the Laplacian term $- \Dl X$ in \eqref{heat} contributes to a term $\| X^{p - 2} |\nb X|^2 \|_{L^1}$ with a good sign, which is essential for establishing an a priori bound for $\| X \|_{L^p}$.

For the fractional parabolic $\Phi^4_3$-model \eqref{SNLH0} in our setting, however, the same method in \cite{MW2d, MW3d} leads to a term $\jb{\nb}^\al (X^{p - 1})$, which creates some subtleties since $1 < \al \leq \frac 98$ is not an integer. While it may be possible to perform some analysis on $\jb{\nb}^\al (X^{p - 1})$, we choose to solve \eqref{SNLH0} using a more probabilistic approach in this paper.
\end{remark}

\subsection{Dynamical convergence}
\label{SUB:intro_dyn}

After discussing the convergence of the truncated Gibbs measures and the representation of the Gibbs measures, we are now ready to state our dynamical convergence result for the macroscopic model \eqref{macro}.

Let us first introduce the limiting equation \eqref{SdfNLW} (i.e.~the fractional hyperbolic $\Phi^4_3$-model) in a more rigorous manner. As mentioned earlier, the process $u$ satisfying \eqref{SdfNLW} is obtained as a limit of the process $u_N^\dagger$ as $N \to \infty$, where $u_N^\dagger$ satisfies the truncated fractional hyperbolic $\Phi^4_3$-model \eqref{SdfNLW_N}. The Wick power $\wick{(\pi_N u_N^\dagger)^3}$ in \eqref{SdfNLW_N} is defined by the Hermite polynomial:
\begin{align*}
\wick{(\pi_N u_N^\dagger)^3} \, \deff H_3 (\pi_N u_N^\dagger; \s_N) = (\pi_N u_N^\dagger)^3 - 3 \s_N \pi_N u_N^\dagger
\end{align*}

\noi
with $\s_N$ being defined as in \eqref{sigma}.
The precise statement is given by the following theorem in \cite{LTzW}. Note that in \cite{LTzW}, the theorem was proved in the case $a = 0$, but the case $a \neq 0$ does not pose any additional difficulty. Here, we recall the fractional $\Phi^4_3$-measure $\rho$ in \eqref{rho} and the Gaussian white noise measure $\mu_0$ in \eqref{mu_al}.
\begin{theorem}[Theorem~1.3 in \cite{LTzW}]
\label{THM:GWP}
Let $1 < \al \leq \frac 98$ and $a \in \R$. Then, the fractional hyperbolic $\Phi^4_3$-model \eqref{SdfNLW} is almost surely globally well-posed with initial data distributed by the Gibbs measure $\vec \rho = \rho \otimes \mu_0$, and the Gibbs measure $\vec \rho$ is invariant under the resulting dynamics.
More precisely, for any $\eps > 0$, there exists a non-trivial stochastic process $\vec u = (u, \dt u) \in C(\R_+; \vec H^{\al - \frac 32 - \eps} (\T^3))$ such that the solution $\vec u_N^\dagger = (u_N^\dagger, \dt u_N^\dagger)$ to the truncated fractional hyperbolic $\Phi^4_3$-model \eqref{SdfNLW_N} converges $\vec \rho$-almost surely to $\vec u$ in $C(\R_+; \vec H^{\al - \frac 32 - \eps} (\T^3))$ (endowed with the compact-open topology in time). Furthermore, we have $\Law (\vec u (t)) = \vec \rho$ for any $t \in \R_+$.
\end{theorem}

The proof of Theorem~\ref{THM:GWP} was based on the following second order expansion:
\begin{align}
u_N^\dagger = \<1> - 4 \cj{a}_2 \<30>_N + v_N^\dagger.
\label{sec0}
\end{align}

\noi
Here, $\<1>$ is the solution to the linear stochastic damped wave equation:
\begin{align}
\begin{cases}
\dt^2 \<1> + \dt \<1> + (1 - \Dl)^\al \<1> = \sqrt{2} \xi \\
(\<1>, \dt \<1>)|_{t = 0} = (u_0, u_1)
\end{cases}
\label{lin_eqn}
\end{align}

\noi
with initial data $(u_0, u_1)$ distributed by the Gibbs measure $\vec \rho = \rho \otimes \mu_0$. Denoting $\I$ as the Duhamel integral operator of the damped wave equation (see \eqref{defI} below), the cubic object $\<30>_N$ is defined by
\begin{align}
\<30>_N \deff \pi_N \I \big( \wick{\<1>_N^3} \big),
\label{def30}
\end{align}

\noi
where $\<1>_N = \pi_N \<1>$ and the Wick power $\wick{\<1>_N^k}$ for any integer $k \geq 1$ is defined by
\begin{align*}
\wick{\<1>_N^k} \, \deff H_k (\<1>_N; \s_N)
\end{align*}

\noi
with $H_k (x; \s)$ being the Hermite polynomial of degree $k$ with variance parameter $\s > 0$ as defined in \eqref{hermite} below and $\s_N$ being defined in \eqref{sigma}. From \cite[Proposition~6.4]{LTzW}, we know that $\<1>$ has regularity $\al - \frac 32 - \eps$ and $\<30>_N$ has regularity $3 \al - 3 - \eps$ for any $\eps > 0$. The remainder term $v_N^\dagger$ then satisfies a nonlinear damped wave equation with more complicated but smoother terms, and $v_N^\dagger$ has better regularity than $\<1>$ and $\<30>_N$. As $N \to \infty$, it has been shown in \cite{LTzW} that $v_N^\dagger$ converges to $v$, which satisfies the equation \eqref{SdNLWv4b} below.

We now switch our attention back to the macroscopic model \eqref{macro}. Let us state our main result in this paper in the following theorem.

\begin{theorem}
\label{THM:dyn}
Let $1 < \al \leq \frac 98$. Given $N \in \N$, let $\vec u_N = (u_N, \dt u_N)$ be the solution to the macroscopic model \eqref{macro} with initial data distributed by $\varrho_N \otimes \mu_0$, where $\varrho_N$ is the truncated Gibbs measure in \eqref{nuN} and $\mu_0$ is the Gaussian white noise measure in \eqref{mu_al}. Then, there exists $a \in \R$ such that $\vec u_N$ converges in probability as $N \to \infty$ in the space $C(\R_+; \vec H^{\al - \frac 32 - \eps} (\T^3))$ for any $\eps > 0$ (endowed with the compact-open topology in time) to the stochastic process $\vec u$ given by Theorem~\ref{THM:GWP}.
\end{theorem}

To prove Theorem~\ref{THM:dyn}, we first use the representation of Gibbs measures provided by Theorem~\ref{THM:rep}. Let $(\Sigma, \mathcal{G}, \mathbb{Q})$ be the probability space and $z_0, v_{0, N}, v_0$ be random variables given by Theorem~\ref{THM:rep}, and we let $z_1$ be a random variable on $(\Sigma, \mathcal{G}, \mathbb{Q})$ independent with $z_0, v_{0, N}, v_0$ and satisfying $\Law (z_1) = \mu_0$.
As before, we define $\<1>$ as the solution to linear stochastic damped wave equation \eqref{lin_eqn}, but now with Gaussian initial data $(u_0, u_1) = (z_0, z_1)$ instead of the Gibbs measure initial data in \eqref{lin_eqn}. Here, we notice that the Gaussian part of the Gibbs measure can be separated out thanks to the representation of Gibbs measures provided by Theorem~\ref{THM:rep}.

From \cite[Lemma~7.2]{LTzW}, we know that the following term
\begin{align*}
N^{- (2j - 4) (\frac 32 - \al)} \I \big( \wick{\<1>_N^{2 j - 1}} \big)
\end{align*}

\noi
for $2 \leq j \leq m$ has regularity $3 \al - 3 - \eps$, which is the same as $\<30>_N$ mentioned in \eqref{def30}. Thus, inspired by the second order expansion \eqref{sec0}, we proceed with the following ansatz
\begin{align}
u_N = \<1> - \Xi_N^{(3)} + v_N,
\label{sec_exp}
\end{align}

\noi
where $\Xi_N^{(3)}$ is defined by the following with correct coefficients:
\begin{align}
\Xi_N^{(3)} \deff \sum_{j = 2}^m 2j \cj{a}_{j, N} N^{-(2j - 4) (\frac 32 - \al)} \pi_N \I \big( :\! \<1>_N^{2j - 1} \!: \big).
\label{cubic}
\end{align}

\noi
The equation for the remainder term $v_N$, endowed with initial data $v_{0, N}$, then involves many complicated products of stochastic terms; see Subsection~\ref{SUB:ansatz} for details. Note that the separation of the Gaussian part performed in the previous paragraph largely simplifies the analysis required for various stochastic terms formed by $\<1>$; see Subsection~\ref{SUB:reg} below as compared to \cite[Subsection~6.2]{LTzW}. The regularity properties of these stochastic terms are then reduced to certain counting estimates, which are presented in Section~\ref{SEC:count}. 

Our goal then becomes showing that $v_N$ converges as $N \to \infty$ in some suitable function space to some process $v$. The convergence relies on some key probabilistic a priori bounds for $v_N$ using invariance of the truncated Gibbs measure $\varrho_N$ in \eqref{rhoN}, especially an $L^\infty$-based bound which cannot be provided by the wave dynamics. The idea is similar to that in \cite{STzX} and goes back to Bourgain-Bulut \cite{BB14}. However, due the singularity of the fractional $\Phi^4_3$-measure $\rho$ and the base Gaussian measure $\mu_\al$ in \eqref{mu_al}, we cannot proceed as in \cite{STzX, LTzW} by directly using the Gaussian measure $\mu_\al$ to establish the probabilistic a priori bounds for $v_N$. This issue is again resolved by our representation of the (truncated) Gibbs measures in Theorem~\ref{THM:rep}, which provides us with a common underlying probability space for the probabilistic a priori bounds. 
See also Remark~\ref{RMK:OTWZ} below.

\medskip
We conclude the introduction by stating several remarks.

\begin{remark} \rm
\label{RMK:OTWZ}
In this paper, we show the convergence of the invariant dynamics of the macroscopic model \eqref{macro} to the invariant dynamics of the fractional hyperbolic $\Phi^4_3$-model \eqref{SdfNLW}. This resembles the situation in \cite{OTWZ}, where the authors also only have weak convergence of the relevant Gibbs measures. In \cite{OTWZ}, the authors resolved this issue by studying a stronger convergence property for the so-called enhanced Gibbs measures, i.e.~the distributions of the stochastic objects. However, the approach in \cite{OTWZ} does not seem to be applicable to our setting in this paper. This is due to the fact that the function spaces for the stochastic objects in \cite{OTWZ} are separable (so that the spaces are polish), whereas in our paper, certain stochastic objects are equipped with the operator-norms, which are in general not separable. Although one may consider the weak topology for these operator-norms and proceed in a similar way to \cite{OTWZ}, we choose not to pursue this approach in this paper.
\end{remark}

\begin{remark} \rm
\label{RMK:al1}
Let us briefly discuss the situation for the dynamical problem in the case $\al = 1$. In \cite{BDNY}, Bringmann-Deng-Nahmod-Yue proved global well-posedness and invariance of the Gibbs measure for the (deterministic) hyperbolic $\Phi^4_3$-model (i.e. formally the equation \eqref{fNLW0} with $\al = 1$). To establish well-posedness, compared to the expansion in \eqref{sec0}, the authors separated out an additional quintic stochastic object and further imposed a paracontrolled structure on the remainder term. This ansatz led to almost a hundred nonlinear terms, and the authors in \cite{BDNY} grouped together similar terms and estimated all of them appropriately with some new tools: bilinear random operator estimates and molecule counting estimates. For the hyperbolic $\Phi^4_3$-model \eqref{SdfNLW} (with $\al = 1$) with a space-time white noise forcing, we expect that a similar paracontrolled ansatz is required to study the well-posedness problem.

As for the dynamical weak universality for the hyperbolic $\Phi^4_3$-model with $\al = 1$, due to the more general nonlinearity in the macroscopic model \eqref{macro}, we need to proceed with an extended version of the paracontrolled ansatz in \cite{BDNY}. Even with the help of scaling factors (i.e.~the negative powers of $N$), analyzing all these general nonlinear terms or stochastic terms becomes very challenging. Nevertheless, we believe that the steps in this paper (along with our previous work \cite{LTzW}) can provide a guidance on how to adapt all the analysis on those general terms in the case $\al = 1$.
\end{remark}

\begin{remark} \rm
At this point, we only consider weak universality (in the weakly nonlinear regime) for hyperbolic equations with an even polynomial nonlinearity and a Gaussian space-time white noise forcing. It would be of interest to study such weak universality for hyperbolic equations with a more general nonlinearity or a more general stochastic forcing as in parabolic equations \cite{SX, FG, HX19, ZZ, KWX}. 

For weak universality of local-in-time dynamics for hyperbolic equations, one can, for example, adapt the local well-posedness argument to general nonlinearities and general stochastic forcings as in \cite{GKO, GKO2, OTh}. When the equation gets more singular, one may need to make some efforts in making sense of the general nonlinearities (such as introducing renormalizations). For global-in-time dynamics, however, since the currently known methods rely heavily on the invariance of Gibbs measures, the adaptations to general nonlinearities and general stochastic forcings become more challenging and remain to be explored. 
\end{remark}

\begin{remark} \rm
Except for wave equations, weak universality problems for other dispersive equations, such as the Schr\"odinger equations or the Korteweg-de Vries equations, remain widely open. For these equations, smoothing from the dispersion is sometimes manifested in different manners. The analysis that deals with a general polynomial nonlinearity, such as that in \cite{ORTh, CK, DNY3, DNY2}, can be helpful for establishing weak universality for these models. It would be of interest to see if weak universality can be shown given different smoothing ingredients from various models.
\end{remark}

\section{Notations and preliminary lemmas}

In this section, we go over notations and some preliminary lemmas.

For two positive quantities $a$ and $b$, we use $a \les b$ to denote $a \leq C b$ for some constant $C > 0$ independent of the set where $a$ and $b$ are allowed to vary. We also use subscripts such as $\les_{s, t}$ to emphasize dependence on parameters. We write $a \sim b$ if $a \les b$ and $b \les a$. We also write $a \ll b$ if $a < cb$ for some small $c > 0$ independent of the set where $a$ and $b$ are allowed to vary. If $N \geq 2$ is a dyadic number, we write $|n| \sim N$ if $\frac{N}{2} < |n| \leq N$; if $N = 1$, we write $|n| \sim N$ if $|n| \leq 1$. For a dyadic number $N \geq 1$, we denote $\pi_N$ to be the frequency projection onto $\{|n| \leq N\}$.

For a space-time function $u$, we denote by $\ft u$ the spatial Fourier transform or the space-time Fourier transform of $u$. To avoid ambiguity, we use $\tau$ and $n$ as the temporal and spatial frequency variables, respectively; and we use $t$ and $x$ as the temporal and spatial physical variables, respectively. We use $\F^{-1}$ to denote the inverse Fourier transform. Unless specified, we use the notation $\vec u = (u, \dt u)$, which is used frequently in the discussion of fractional wave dynamics.

Given Banach spaces $B_1$ and $B_2$, we denote by $\L (B_1; B_2)$ the space of bounded linear operators from $B_1$ to $B_2$. When writing space-time function spaces, we sometimes use short-hand notations such as $L_T^q L_x^r = L^q ([0, T]; L^r (\T^3))$ for simplicity.


\subsection{Sobolev and H\"older-Besov spaces}
Let us first define Sobolev spaces. Let $d \in \N$, $s \in \R$, and $1 \leq p \leq \infty$. We define the $L^p$-based Sobolev space $W^{s, p} (\T^d)$ by the norm
\begin{align*}
\| f \|_{W^{s, p}} \deff \| \jb{\nb}^s f \|_{L^p} = \big\| \F^{-1} [\jb{n}^s \ft f (n)] \big\|_{L^p},
\end{align*}

\noi
where $\jb{\cdot} = (1 + |\cdot|^2)^{\frac 12}$.
When $p = 2$, we denote $H^s (\T^d) \deff W^{s, 2} (\T^d)$. We also denote $\vec H^{s} (\T^d) = H^s (\T^d) \times H^{s - \al} (\T^d)$.

We now define H\"older-Besov spaces. Let $\phi : \R \to [0, 1]$ be a smooth cutoff function such that $\phi \equiv 1$ on $[- \frac 54, \frac 54]$ and $\phi \equiv 0$ outside of $[- \frac 85, \frac 85]$. For $\xi \in \R^d$, we set $\varphi_0 (\xi) = \phi (|\xi|)$ and
\begin{align*}
\varphi_j (\xi) = \phi \big( \tfrac{|\xi|}{2^j} \big) - \phi \big( \tfrac{|\xi|}{2^{j - 1}} \big)
\end{align*}

\noi
for $j \in \N$. Then, for $j \in \Z_{\geq 0} \deff \N \cup \{0\}$, we define the Littlewood-Paley projector $\mathbf{P}_j$ as the Fourier multiplier operator with the symbol $\varphi_j$. We now define the H\"older-Besov space $\C^s (\T^d)$ by the norm
\begin{align*}
\| f \|_{\C^s} \deff \big\| 2^{s j} \| \mathbf{P}_j f \|_{L_x^\infty (\T^d)} \big\|_{\l_j^\infty (\Z_{\geq 0})}.
\end{align*}

We now recall some basic properties of Sobolev and H\"older-Besov spaces. The proofs of these estimates can be found in, for example, \cite{BCD, GOTW}.
\begin{lemma}
\label{LEM:Bes}
Let $d \in \N$. The following properties hold.

\smallskip \noi
\textup{(i)}
Let $s_1, s_2 \in \R$ with $s_1 < s_2$ and $1 \leq p \leq \infty$. Then, we have the following embedding
\begin{align*}
\| f \|_{W^{s_1, p}} \les \| f \|_{\C^{s_2}}.
\end{align*}

\smallskip \noi
\textup{(ii)} Let $s_1, s_2 \in \R$ be such that $s_2 \geq s_1 + \frac{d}{2}$. Then, we have the following Besov embedding
\begin{align*}
\| f \|_{\C^{s_1}} \les \| f \|_{H^{s_2}}.
\end{align*}

\smallskip \noi
\textup{(iii)} 
Let $s, s_1, s_2 \in \R$ be such that $s = \ta s_1 + (1 - \ta) s_2$ for some $0 < \ta < 1$. Then, we have the following interpolation
\begin{align*}
\| f \|_{\mathcal{C}^{s}} \les \| f \|_{\mathcal{C}^{s_1}}^\ta \| f \|_{\mathcal{C}^{s_2}}^{1 - \ta}.
\end{align*}

\smallskip \noi
\textup{(iv)}
Let $s_1, s_2 \in \R$ with $s_1 + s_2 > 0$ and let $s = \min \{ s_1, s_2, s_1 + s_2 \}$. Then, we have the following product estimate
\begin{align*}
\| f g \|_{\mathcal{C}^s} \les \| f \|_{\mathcal{C}^{s_1}} \| g \|_{\mathcal{C}^{s_2}}.
\end{align*}
\end{lemma}

We now recall the following product estimates in Sobolev spaces. For proofs, see, for example, \cite[Lemma~3.4]{GKO} (see also \cite{BOZ} for the endpoint case).
\begin{lemma}
\label{LEM:prod}
Let $d \in \N$ and $s > 0$.

\smallskip \noi
\textup{(i)} Let $1 < p_1, p_2, q_1, q_2 \leq \infty$ and $1 \leq r \leq \infty$ be such that $\frac{1}{r} = \frac{1}{p_1} + \frac{1}{q_1} = \frac{1}{p_2} + \frac{1}{q_2}$. Then, we have
\begin{align*}
\| \jb{\nb}^s (f g) \|_{L^r (\T^d)} \les \| \jb{\nb}^s f \|_{L^{p_1} (\T^d)} \| g \|_{L^{q_1} (\T^d)} + \| f \|_{L^{p_2} (\T^d)} \| \jb{\nb}^s g \|_{L^{q_2} (\T^d)}.
\end{align*}

\smallskip \noi
\textup{(ii)} Let $1 < p, q \leq \infty$ and $1 < r \leq p$ be such that $\frac{s}{d} + \frac 1r = \frac 1p + \frac 1q \leq 1$. Then, we have
\begin{align*}
\| \jb{\nb}^{-s} (f g) \|_{L^r (\T^d)} \les \| \jb{\nb}^{-s} f \|_{L^p (\T^d)} \| \jb{\nb}^s g \|_{L^q (\T^d)}.
\end{align*}
\end{lemma}

We will also need the following Schauder estimate. The proof follows from minor modifications of \cite[Lemma~A.5]{GIP}.
\begin{lemma}
\label{LEM:Schau}
Let $\al \geq 1$, $t > 0$, and $s_1, s_2 \in \R$ be such that $s_1 > s_2$. Then, we have
\begin{align*}
\big\| e^{- t \jb{\nb}^{2 \al}} f \big\|_{\mathcal{C}^{s_1}} \les t^{- \frac{s_1 - s_2}{2 \alpha}} \| f \|_{\mathcal{C}^{s_2}}.
\end{align*}
\end{lemma}

\subsection{On convolution inequalities}
We first record the following discrete convolution inequalities.

\begin{lemma}
\label{LEM:conv}
Let $d \in \N$.

\smallskip \noi
\textup{(i)} Let $\be, \g \in \R$ be such that $\be + \g > d$ and $\be < d$. Then, for any $n \in \Z^d$, we have
\begin{align*}
\sum_{n_1 \in \Z^d} \frac{1}{\jb{n_1}^\be \jb{n - n_1}^\g} \les \jb{n}^{- \be + \ld},
\end{align*}

\noi
where $\ld = \max (d - \g, 0)$ when $\g \neq d$ and $\ld = \eps$ when $\g = d$ for any $\eps > 0$.

\smallskip \noi
\textup{(ii)} Let $k \in \N$ and $\be \in \R$ be such that
\begin{align*}
\frac{(k - 1)d}{k} < \be < d.
\end{align*}

\noi
Then, we have
\begin{align*}
\sum_{\substack{n_1, \dots, n_k \in \Z^d \\ n = n_1 + \cdots + n_k}} \frac{1}{\jb{n_1}^\be \cdots \jb{n_k}^\be} \les \jb{n}^{(k - 1) d - k \be}.
\end{align*}
\end{lemma}

For a proof of part (i), see, for example, \cite[Lemma~4.2]{GTV} and \cite[Lemma~4.1]{MWX}. Part (ii) follows from repetitive applications of part (i).

We also record the following convolution inequality with integration. For a proof, see \cite[Lemma~4.2]{GTV}.
\begin{lemma}
\label{LEM:conv_int}
Let $0 \leq \be \leq \g$ be such that $\g > 1$. Then, for any $y \in \R$, we have
\begin{align*}
\int_\R \frac{1}{\jb{x}^\be \jb{x - y}^\gamma} dx \les \frac{1}{\jb{y}^\be}.
\end{align*}
\end{lemma}

\subsection{Fourier restriction norm method}
\label{SUB:Xsb}
In this section, we recall the Fourier restriction norm method, which was first introduced by Klainerman-Machedon \cite{KM93} for wave equations and by Bourgain \cite{Bour93} for Schr\"odinger equations. For this purpose, we consider the following inhomogeneous linear damped fractional wave equation:
\begin{align}
\begin{cases}
\dt^2 u + \dt u + (1 - \Dl)^\al u = F \\
(u, \dt u)|_{t = 0} = (u_0, u_1).
\end{cases}
\label{SdLW}
\end{align}

\noi
The solution to \eqref{SdLW} is given by
\begin{align*}
u (t) = S(t) (u_0, u_1) + \int_0^t \mathcal{D} (t - t') F(t') dt'
\end{align*}

\noi
for $t \geq 0$. Here, with
\begin{align*}
\jbb{\nb} = \sqrt{(1 - \Dl)^\al - \frac 14},
\end{align*}

\noi
the linear operators $\mathcal{D} (t)$ and $S(t)$ are defined by
\begin{align}
\mathcal{D} (t) \deff e^{- \frac{t}{2}} \frac{\sin (t \jbb{\nb})}{\jbb{\nb}}, 
\label{defD} 
\end{align}
\begin{align}
S(t) (f, g) \deff \dt \mathcal{D} (t) f + \mathcal{D} (t) (f + g).
\label{defS}
\end{align}

\noi
We also define the Duhamel integral operator
\begin{align}
\I_{t_0} (F) (t) = \int_{t_0}^t \mathcal{D} (t - t') F(t') dt' = \int_{t_0}^t e^{- \frac{t - t'}{2}} \frac{\sin ((t - t') \jbb{\nb})}{\jbb{\nb}} F (t') dt'
\label{defI}
\end{align}

\noi
given $0 \leq t_0 \leq t$, and we write $\I (F) \deff \I_{0} (F)$. 
Moreover, we denote
\begin{align}
\vec S (t) (f, g) = (S (t) (f, g), \dt S (t) (f, g)) \quad \text{and} \quad \vec \I_{t_0} (F) = (\I_{t_0} (F), \dt \I_{t_0} (F)).
\label{vecSI}
\end{align}

As in \cite{LTzW}, we define the $X^{s, b}$-space for the fractional wave equation via the norm
\begin{align*}
\| u \|_{X^{s, b}} \deff \| \jb{n}^s \jb{|\tau| - \jbb{n}}^b \ft u (\tau, n) \|_{\l_n^2 L_\tau^2}.
\end{align*}

\noi
We also denote $\vec X^{s, b} = X^{s, b} \times X^{s - \al, b}$. 
Given an interval $I \subset \R$, we define the local-in-time version of the $X^{s, b}$-space as
\begin{align*}
\| u \|_{X_I^{s, b}} \deff \inf \big\{ \| v \|_{X^{s, b}} : v|_I = u \big\}.
\end{align*}

\noi
If $I = [0, T]$, we write $X_T^{s, b} = X_{[0, T]}^{s, b}$.

We now recall some useful linear estimates for the $X_I^{s, b}$-norm, starting with the following  embedding lemma. See, for example, \cite[Corollary~2.10]{Tao}.
\begin{lemma}
\label{LEM:HsXsb}
Let $s \in \R$, $b > \frac 12$, and $I \subset \R$ be a closed interval. Then, we have the following embedding:
\begin{align*}
\| u \|_{C_I H^s} \les \| u \|_{X^{s, b}_I}.
\end{align*}
\end{lemma}

We then record the following homogeneous linear estimate. For a proof, see \cite[Lemma~2.8]{LTzW}.
\begin{lemma}
\label{LEM:homo}
Let $s \in \R$, $b < \frac 32$, and $I \subset \R$ be a closed interval. Then, we have
\begin{align*}
\big\| \vec S(t) (f, g) \big\|_{\vec X_I^{s, b}} \les (1 + |I|) \| (f, g) \|_{\vec H^s}.
\end{align*}
\end{lemma}

Next, we recall the following inhomogeneous linear estimate. For a proof, see \cite[Lemma~2.9]{LTzW}.
\begin{lemma}
\label{LEM:nhomo}
Let $\al \geq 0$, $s \in \R$, $\frac 12 < b < 1$, and $I = [t_0, t_1] \subset \R$ be a closed interval with $t_0 \geq 0$. Then, we have
\begin{align*}
\big\| \vec \I_{t_0} (F) \big\|_{\vec X_I^{s, b}} \les \| F \|_{X_I^{s - \al, b - 1}}.
\end{align*}
\end{lemma}

We also record the following time localization estimate. For a proof, see \cite[Lemma~2.11]{Tao}.
\begin{lemma}
\label{LEM:time}
Let $s \in \R$, $- \frac 12 < b_1 < b_2 < \frac 12$, and $I \subset \R$ be a closed interval with $0 < |I| \leq 1$. Then, we have
\begin{align*}
\| u \|_{X_I^{s, b_1}} \les |I|^{b_2 - b_1} \| u \|_{X_I^{s, b_2}}.
\end{align*}
\end{lemma}

The following gluing lemma will be useful for establishing some global-in-time bounds. For a proof, see \cite[Lemma~4.5]{Bring2}.
\begin{lemma}
\label{LEM:glue}
Let $s \in \R$ and $\frac 12 < b < 1$. Let $I_1, I_2 \subset \R$ be two closed intervals such that $|I_1 \cap I_2| > 0$. Then, we have
\begin{align*}
\| u \|_{X_{I_1 \cup I_2}^{s, b}} \les |I_1 \cap I_2|^{\frac 12 - b} \Big( \| u \|_{X_{I_1}^{s, b}} + \| u \|_{X_{I_2}^{s, b}} \Big).
\end{align*}
\end{lemma}

Lastly, we record the following lemma on the continuity of the $X_T^{s, b}$-norm with respect to $T$. For a proof, see \cite[Lemma~2.7]{BLLZ}
\begin{lemma}
\label{LEM:time_cty}
Let $s \in \R$, $b \in \R$, and $u \in X^{s, b}$. Then, the function
\begin{align*}
t \mapsto \| u \|_{X_t^{s, b}}
\end{align*}

\noi
is continuous on $\R_+$.
\end{lemma}

\subsection{Strichartz estimate and multilinear estimates}
In this subsection, we recall the Strichartz estimate and prove some useful multilinear estimates. We first record the following Strichartz estimate from \cite[Lemma 2.12]{LTzW}, whose proof is based on the $\l^2$-decoupling theorem in \cite{BD15}. See also \cite[Proposition 1.1]{Sch}.
\begin{lemma}
\label{LEM:Str}
Let $1 < \al \leq 2$, $p \geq \frac{10}{3}$, and $s \geq 0$, and $I \subset \R$ be a closed interval. Suppose that
\begin{align*}
s > \frac 32 - \frac{\al + 3}{p}.
\end{align*}

\noi
Then, there exists $\dl > 0$ small such that
\begin{align*}
\| u \|_{L_I^p L_x^p} \les (1 + |I|)^{\frac{1}{p}} \| u \|_{X_I^{s, \frac 12 - \dl}}.
\end{align*}
\end{lemma}

Let us now mention some useful multilinear estimates. We start by recording the following trilinear estimate. For a proof, see \cite[Lemma~2.14]{LTzW}.
\begin{lemma}
\label{LEM:Str_mul1}
Let $1 < \al \leq \frac 98$ and $I \subset \R_+$ be a closed interval. Then, there exists $r \geq \frac{10}{3}$ satisfying $\al - \frac 12 > \frac 32 - \frac{\al + 3}{r}$ and $\eps, \dl > 0$ small such that the following estimate holds: 
\begin{align}
\| u v^2 \|_{X_I^{- \frac 12 + \eps, - \frac 12 + 2 \dl}} \les (1 + |I|) \| u \|_{L_I^\infty W_x^{\al - \frac 32 - \eps, \infty}} \| v \|_{L_I^r L_x^r} \| v \|_{X_I^{\al - \frac 12, \frac 12 + \dl}}.
\label{str_mul1-1}
\end{align}

\noi
Furthermore, we have
\begin{align}
\| u_1 u_2 u_3 \|_{X_I^{- \frac 12 + \eps, - \frac 12 + 2 \dl}} \les (1 + |I|) \| u_1 \|_{L_I^\infty W_x^{\al - \frac 32 - \eps, \infty}} \| u_2 \|_{X_I^{\al - \frac 12 - \eps, \frac 12 + \dl}} \| u_3 \|_{X_I^{\al - \frac 12 - \eps, \frac 12 + \dl}}.
\label{str_mul1-2}
\end{align}
\end{lemma}

We now show the following multilinear estimates.
\begin{lemma}
\label{LEM:Str_mul2}
Let $1 < \al \leq \frac 98$ and $I \subset \R_+$ be a closed interval. Then, there exist $\eps, \dl > 0$ small such that
\begin{align}
\| u_1 u_2 u_3 v \|_{X_I^{- \frac 12 + \eps, - \frac 12 + 2 \dl}} \les (1 + |I|)^{\frac 14} \bigg( \prod_{j = 1}^3 \| u_j \|_{L_I^4 W_x^{\eps, 4}} \bigg) \| v \|_{L_I^\infty W_x^{-\eps, \infty}}.
\label{str00}
\end{align}

\noi
Furthermore, for any $r > 4$ satisfying $\frac 12 > \frac 32 - \frac{\al + 3}{r}$, there exist $\dl, \ta > 0$ small such that
\begin{align}
\begin{split}
\| u_1^2 &u_2 u_3 v \|_{X_I^{- \frac 12, - \frac 12 + 2 \dl}} \\
&\les (1 + |I|)^{\frac{1 - 2 \ta}{r}} \| u_1 \|_{L_I^r L_x^r}^{1 + \ta} \| u_1 \|_{L_I^{\frac{r (1 - \ta)}{\ta}} L_x^{\frac{r (1 - \ta)}{\ta}}}^{1 - \ta} \| u_2 \|_{L_I^r L_x^r} \| u_3 \|_{L_I^r L_x^r} \| v \|_{L_I^{\frac{r}{r - 4}} L_x^{\frac{r}{r - 4}}} .
\end{split}
\label{str0}
\end{align}
\end{lemma}

\begin{proof}
Let us first consider \eqref{str00}. We let $\eps > 0$ be small enough such that
\begin{align}
\frac 12 - 2 \eps > \frac 32 - \frac{\al + 3}{4}.
\label{eps_bdd}
\end{align}

\noi
By duality, H\"older's inequalities, Lemma \ref{LEM:prod} (i), and Lemma \ref{LEM:Str} along with \eqref{eps_bdd}, we have
\begin{align*}
\| u_1 u_2 u_3 v \|_{X_I^{- \frac 12 + \eps, - \frac 12 + 2 \dl}} &\leq \sup_{\| w \|_{X^{\frac 12 - \eps, \frac 12 - 2 \dl}} \leq 1} \bigg| \int_I \int_{\T^3} u_1 u_2 u_3 v w dx dt \bigg| \\
&\leq \sup_{\| w \|_{X^{\frac 12 - \eps, \frac 12 - 2 \dl}} \leq 1} \int_I \int_{\T^3} |\jb{\nb}^\eps (u_1 u_2 u_3 w)| dx dt \cdot \| v \|_{L_I^\infty W_x^{- \eps, \infty}} \\
&\les \bigg( \prod_{j = 1}^3 \| \jb{\nb}^\eps u_j \|_{L_I^4 L_x^4} \bigg) \| v \|_{L_I^\infty W_x^{- \eps, \infty}} \sup_{\| w \|_{X^{\frac 12 - \eps, \frac 12 - 2 \dl}} \leq 1} \| \jb{\nb}^\eps w \|_{L_I^4 L_x^4} \\
&\les (1 + |I|)^{\frac 14} \bigg( \prod_{j = 1}^3 \| u_j \|_{L_I^4 W_x^{\eps, 4}} \bigg) \| v \|_{L_I^\infty W_x^{-\eps, \infty}},
\end{align*}

\noi
which proves the desired inequality.

For \eqref{str0}, we let $\ta > 0$ be small enough such that
\begin{align}
\frac 12 > \frac 32 - \frac{(\al + 3) (1 - 2 \ta)}{r}.
\label{nu_bdd}
\end{align}

\noi
By duality, H\"older's inequalities, and Lemma \ref{LEM:Str} along with \eqref{nu_bdd}, we have
\begin{align*}
\| u_1^2 &u_2 u_3 v \|_{X_I^{- \frac 12, - \frac 12 + 2 \dl}} \\
&\leq \sup_{\| w \|_{X^{\frac 12, \frac 12 - 2 \dl}} \leq 1} \bigg| \int_I \int_{\T^3} u_1^2 u_2 u_3 v w dx dt \bigg| \\
&\leq \| u_1^2 \|_{L_I^{\frac{r}{1 + 2 \ta}} L_x^{\frac{r}{1 + 2 \ta}}} \| u_2 \|_{L_I^r L_x^r} \| u_3 \|_{L_I^r L_x^r} \| v \|_{L_I^{\frac{r}{r - 4}} L_x^{\frac{r}{r - 4}}} \\
&\quad \times \sup_{\| w \|_{X^{\frac 12, \frac 12 - 2 \dl}} \leq 1} \| w \|_{L_I^{\frac{r}{1 - 2 \ta}} L_x^{\frac{r}{1 - 2 \ta}}} \\
&\les (1 + |I|)^{\frac{1 - 2 \ta}{r}} \| u_1 \|_{L_I^r L_x^r}^{1 + \ta} \| u_1 \|_{L_I^{\frac{r (1 - \ta)}{\ta}} L_x^{\frac{r (1 - \ta)}{\ta}}}^{1 - \ta} \| u_2 \|_{L_I^r L_x^r} \| u_3 \|_{L_I^r L_x^r} \| v \|_{L_I^{\frac{r}{r - 4}} L_x^{\frac{r}{r - 4}}},
\end{align*}

\noi
as desired.
\end{proof}

\subsection{Tools from stochastic analysis}
\label{SUB:sto}

In this subsection, we recall some useful lemmas from stochastic analysis. 

We define the Hermite polynomials $H_k (x ; \s)$ with variance $\s > 0$ via the following generating function:
\begin{align}
e^{tx - \frac 12 \s t^2} = \sum_{k = 0}^\infty \frac{t^k}{k!} H_k (x ; \s)
\label{hermite}
\end{align}

\noi
for $t, x \in \R$. From \eqref{hermite}, we have the following binomial expansion
\begin{align}
H_k (x + y; \s_1 + \s_2) = \sum_{\l = 0}^k \binom{k}{\l} H_{k - \l} (x ; \s_1) H_\l (y ; \s_2).
\label{her_binom}
\end{align}

Let us first record the following lemma about the orthogonality relation for the Hermite polynomials; see \cite[Lemma 1.1.1]{Nua}.
\begin{lemma}
\label{LEM:wick}
Let $X$ and $Y$ be jointly Gaussian random variables with mean zero and variances $\s_X$ and $\s_Y$, respectively. Then, for any $k, \l \in \N$, we have
\begin{align*}
\E \big[ H_k (X; \s_X) H_\l (Y; \s_Y) \big] = \ind_{\{k = \l\}} k! (\E [XY])^k.
\end{align*}
\end{lemma}

We now consider an abstract Wiener space $(H, B, \mu)$, where $\mu$ is a Gaussian measure on a separable Banach space $B$ with $H \subset B$ as its Cameron-Martin space. Given a complete orthonormal system $\{e_j\}_{j \in \N} \subset B^*$ of $H^* = H$, we define a polynomial chaos of order $k$ to be an element of the form $\prod_{j = 1}^\infty H_{k_j} (e_j (x))$, where $x \in B$, $H_{k_j}$ is the Hermite polynomial of degree $k_j$, $k_j \neq 0$ for only finitely many $j$'s, and $k = \sum_{j = 1}^\infty k_j$. We then denote by $\H_k$ the closure of polynomial chaoses of order $k$ under $L^2 (B, \mu)$. We also set
\begin{align*}
\H_{\leq k} = \bigoplus_{j = 0}^k \H_j
\end{align*}

\noi
for $k \in \N$.

Let us record the Wiener chaos estimate \cite[Theorem I.22]{Simon}, which is a consequence of the hypercontractivity of the Ornstein-Uhlenbeck semigroup due to Nelson \cite{Nel}. See also \cite[Proposition 2.4]{TTz}.
\begin{lemma}
\label{LEM:hyp}
Let $k \in \N$. Then, for any $2 \leq p < \infty$ and $X \in \H_{k}$, we have
\begin{align*}
\| X \|_{L^p (\Om)} \leq (p - 1)^{\frac{k}{2}} \| X \|_{L^2 (\Om)}.
\end{align*}
\end{lemma}

Lastly, we consider regularities of stochastic objects. A stochastic process $X : \R \to \mathcal{D}' (\T^d)$ is \textit{stationary in space} if $\{X(t, \cdot)\}_{t \in \R}$ and $\{X(t, x_0 + \cdot)\}_{t \in \R}$ have the same law for any $x_0 \in \T^d$. Given $t_0 \in \R$, we define the difference operator $\dl_{t_0}$ by
\begin{align}
\dl_{t_0} X (t) = X(t + t_0) - X(t).
\label{dh}
\end{align}

\begin{lemma}
\label{LEM:reg}
Let $d \in \N$. Given $N \in \N$, let $X_N$ and $X$ be stochastic processes from $\R$ to $\mathcal{D}' (\T^d)$ that are stationary in space. Suppose that there exists $k \in \N$ such that for each $t \in \R$, $X_N (t)$ and $X(t)$ belong to $\H_{\leq k}$.

\smallskip \noi
\textup{(i)} Let $t \in \R$. If there exists $s \in \R$ such that
\begin{align*}
\sum_{n \in \Z^d} \E \big[ | \jb{n}^s \ft{X} (t, n)|^2 \big] < \infty ,
\end{align*}

\noi
then we have almost surely $X(t) \in W^{s, \infty} (\T^d)$ and the tail estimate
\begin{align*}
\mathbb{P} \big( \| X(t) \|_{W^{s, \infty}} > \ld \big) < C \exp (-c \ld^{\frac{2}{k}})
\end{align*}

\noi
for any $\ld > 0$ and some constants $C, c > 0$.

\smallskip \noi
\textup{(ii)} Let $t \in \R$. If there exists $s \in \R$ and $\g > 0$ such that
\begin{align*}
\sum_{n \in \Z^d} \E \big[ |\jb{n}^s \ft{X}_M (t, n) - \jb{n}^s \ft{X}_N (t, n) |^2 \big] \les N^{- \g}
\end{align*}

\noi
for any $M \geq N \geq 1$, then $\{ X_N (t) \}_{N \in \N}$ is almost surely a Cauchy sequence in $W^{s, \infty} (\T^d)$. Furthermore, we have the tail estimate
\begin{align*}
\mathbb{P} \big( N^{\frac{\gamma}{2}} \| X_M (t) - X_N (t) \|_{W^{s, \infty}} > \ld \big) < C \exp (-c \ld^{\frac{2}{k}})
\end{align*}

\noi
for any $\ld > 0$ and some constants $C, c > 0$.

\smallskip \noi
\textup{(iii)} Let $I \subset \R$ be a closed interval. If there exist $0 < s' < 1$ and $s \in \R$ such that
\begin{align*}
\sum_{n \in \Z^d} \E \big[ |\jb{n}^s \dl_{t_0} \ft{X} (t, n)|^2 \big] \les |t_0|^{s'}
\end{align*}

\noi
for any $t, t+t_0 \in I$, then we have almost surely $X \in C(I; W^{s, \infty} (\T^d))$ and the tail estimate
\begin{align*}
\mathbb{P} \big( \| X \|_{C_I W_x^{s, \infty}} > \ld \big) < C(I) \exp (-c (I) \ld^{\frac{2}{k}})
\end{align*}

\noi
for any $\ld > 0$ and some constants $C(I), c(I) > 0$.

\smallskip \noi
\textup{(iv)} Let $I \subset \R$ be a closed interval. If there exist $0 < s' < 1$, $s \in \R$, and $\g > 0$ such that
\begin{align*}
\sum_{n \in \Z^d} \E \big[ | \jb{n}^s \dl_{t_0} \ft X_M (t, n) - \jb{n}^s \dl_{t_0} \ft X_N (t, n) |^2 \big] \les N^{- \g} |t_0|^{s'}
\end{align*}

\noi
for any $t, t + t_0 \in I$ and $M \geq N \geq 1$, then $\{ X_N \}_{N \in \N}$ is almost surely a Cauchy sequence in $C (I; W^{s, \infty} (\T^d))$. Furthermore, we have the tail estimate
\begin{align*}
\mathbb{P} \big( N^{\frac{\gamma}{2}} \| X_M - X_N \|_{C_I W_x^{s, \infty}} > \ld \big) < C(I) \exp (- c(I) \ld^{\frac{2}{k}}).
\end{align*}
\end{lemma}

The regularity properties in Lemma~\ref{LEM:reg} follows from a direct application of the Wiener chaos estimate (Lemma \ref{LEM:hyp}). For a proof of the almost sure convergence, see \cite[Proposition 3.6]{MWX} and \cite[Appendix]{OOTz}. For a proof of the tail estimates, see \cite[Lemma~2.3]{GKOT}. Note that from Lemma~\ref{LEM:Bes}~(i), the $W^{s, \infty}$-norm in the statement of Lemma~\ref{LEM:reg} can be replaced by the $\mathcal{C}^s$-norm. 

\medskip
We will also need the Prokhorov theorem. Let $\mathcal{J}$ be an index set and $\mathcal{M}$ be a metric space. A family $\{\mu_n\}_{n \in \mathcal{J}}$ of probability measures on $\mathcal{M}$ is said to be {\it tight} if, for any $\eps > 0$, there exists a compact set $K_\eps \subset \mathcal{M}$ such that $\sup_{n \in \mathcal{J}} \mu_n (K_\eps^c) \leq \eps$. The following theorem can be found in \cite{Bill}.
\begin{lemma}[Prokhorov theorem]
\label{LEM:prok}
Let $\mathcal{J}$ be an index set, $\mathcal{M}$ be a metric space, and $\{\mu_n\}_{n \in \mathcal{J}}$ be a family of probability measures on $\mathcal{M}$. If $\{\mu_n\}_{n \in \mathcal{J}}$ is tight, then it is relatively compact, i.e. every sequence in $\{\mu_n\}_{n \in \mathcal{J}}$ contains a weakly convergent subsequence.
\end{lemma}

Lastly, we recall the following Skorokhod representation theorem. See \cite[Chapter~31]{Bass}.
\begin{lemma}[Skorokhod representation theorem]
\label{LEM:skorok}
Let $\mathcal{M}$ be a complete separable metric space (i.e.~a Polish space). Suppose that a sequence $\{\mu_n\}_{n \in \N}$ of probability measures on $\mathcal{M}$ converges weakly to a probability measure $\mu$ as $n \to \infty$. Then, there exist a probability space $(\wt \Om, \wt \F, \wt{\mathbb{P}})$ and random variables $X_n, X: \wt \Om \to \mathcal{M}$ such that
\begin{align*}
\Law (X_n) = \mu_n \quad \text{and} \quad \Law (X) = \mu,
\end{align*}

\noi
and $X_n$ converges $\wt{\mathbb{P}}$-almost surely to $X$ as $n \to \infty$.
\end{lemma}

\subsection{Multiple stochastic integrals}
\label{SUB:mult}
In this subsection, we recall basic definitions and properties of multiple stochastic integrals.
For the notations, we mainly follow \cite[Appendix~B]{OWZ}. See \cite{Nua} and also \cite{Bring2, OWZ, BDNY} for further discussions.

Let $\{B_n\}_{n \in \Z^3}$ be a sequence of Gaussian processes on a probability space $(\Om, \F, \mathbb{P})$ with the following properties:
\begin{enumerate}
\item[(i)] $B_0$ is a real-valued two-sided standard Brownian motion and $B_n$ is a complex-valued two-sided standard Brownian motion for $n \in \Z^3 \setminus \{0\}$.  Specifically, for all $t \in \R$ and $n \in \Z^3$, we have $\E [|B_n(t)|^2] = |t|$.

\item[(ii)] For all $n_1, n_2 \in \Z^3$ with $n_1 \neq \pm n_2$, the processes $B_{n_1}$ and $B_{n_2}$ are independent.

\item[(iii)] For all $n \in \Z^3$, $\cj{B_n} = B_{-n}$.
\end{enumerate}

\noi
Although in this subsection we only present lemmas for two-sided Brownian motions, analogous results also hold for one-sided Brownian motions.

We let $\nu$ be the measure on $\mathsf{Z} \deff \R \times \Z^3$ defined by $d \nu = d t d n$, where $d t$ denotes the Lebesgue measure on $\R$ and $d n$ denotes the counting measure on $\Z^3$. 
For $k \in \N$, we write $\nu^{\otimes k} = \bigotimes_{j=1}^k \nu$ and $L^2 (\mathsf{Z}^k) = L^2((\R \times \Z^3)^k, \nu^{\otimes k})$. Given $f \in L^2 (\mathsf{Z}^k)$, we define the multiple stochastic integral $I_k[f]$ by
\begin{align}
I_k [f] \deff \sum_{n_1, \dots, n_k \in \Z^3} \int_{\R^k} f(t_1, n_1, \dots, t_k, n_k) d B_{n_1} (t_1) \cdots d B_{n_k} (t_k).
\label{defIk}
\end{align} 

\noi
Given $f \in L^2 (\mathsf{Z}^k)$, we also define the symmetrization $\text{Sym} (f)$ by
\begin{align*}
\text{Sym} (f) (z_1, \dots, z_k) = \frac{1}{k!} \sum_{\phi \in S_k} f (z_{\phi(1)}, \dots, z_{\phi (k)}),
\end{align*}

\noi
where $z_j = (t_j, n_j)$ and $S_k$ is the symmetric group on $\{ 1, \dots, k \}$. By the Cauchy-Schwarz inequality, we have
\begin{align}
| \text{Sym} (f) (z_1, \dots, z_k) |^2 \leq \frac{1}{k!} \sum_{\phi \in S_k} | f (z_{\phi(1)}, \dots, z_{\phi (k)}) |^2.
\label{jen}
\end{align}

\noi
We say that $f$ is symmetric if $\text{Sym} (f) = f$. 
If $f$ is not symmetric, we define $I_k [f] \deff I_k [\text{Sym} (f)]$.

Let us now recall some basic properties of multiple stochastic integrals.

\begin{lemma}
\label{LEM:sto_ito}
Let $k, \l \in \N$. Let $f \in L^2(\mathsf{Z}^k)$ and $g \in L^2 (\mathsf{Z}^\l)$.

\smallskip \noi
\textup{(i)} $I_k : L^2 (\mathsf{Z}^k) \to \H_k \subset L^2(\Om)$ is linear, where $\H_k$ denotes the $k$th Wiener chaos as defined in Subsection \ref{SUB:sto}. Moreover, for any $h \in L^2 (\mathsf{Z})$ with $\s_h = \| h \|_{L^2 (\mathsf{Z})}^2$, we have
\begin{align}
I_k [h^{\otimes k}] = H_k ( I_1[h]; \s_h ),
\label{IkHk}
\end{align}

\noi
where $h^{\otimes k}$ denotes the $k$-tensor product of $h$ with itself and $H_k$ is the $k$th Hermite polynomial as in \eqref{hermite}. 

\smallskip \noi
\textup{(ii)} The following Ito isometry holds:
\begin{align*}
\E \big[ I_k [f] \cj{I_\l [g]} \big] = \ind_{\{ k = \l \}} \cdot k! \int_{\mathsf{Z}^k} \textup{Sym} (f) \cj{\textup{Sym} (g)} d \nu^{\otimes k}.
\end{align*}
\end{lemma}

For proofs of the above properties, see \cite[Chapter~1]{Nua}. Note that \eqref{IkHk} is a rescaled version of \cite[Proposition~1.1.4]{Nua}.

We now record the following Fubini's theorem for multiple stochastic integrals. For a proof, see \cite[Lemma~B.2]{OWZ} and also \cite[Theorem~4.33]{DaZ}.
\begin{lemma}
\label{LEM:sto_fub}
Let $k \in \N$. Given a finite $T > 0$, let $f \in L^2 ([0, T] \times ( [0, T] \times \Z^3)^k, dt \otimes d \nu^{\otimes k} )$. Then, we have
\begin{align*}
\int_0^T I_k [f (t, \cdot)] dt = I_k \bigg[ \int_0^T f(t, \cdot) dt \bigg]
\end{align*}

\noi
in $L^2 (\Om)$.
\end{lemma}

Lastly, we record the product formula for multiple stochastic integrals. We recall that, for $k, \l, r \in \N$ with $0 \leq r \leq \min(k, \l)$ and $f \in L^2 (\mathsf{Z}^k)$ and $g \in L^2 (\mathsf{Z}^\l)$, the contraction $f \otimes_r g$ of $r$ indices is defined by
\begin{align*}
(f \otimes_r g) (z_1, \dots, z_{k + \l - 2r}) &\deff \sum_{m_1, \dots, m_r \in \Z^3} \int_{\R^r} f(z_1, \dots, z_{k-r}, \zeta_1, \dots, \zeta_r) \\
&\quad \times g (z_{k + 1 - r}, \dots, z_{k + \l - 2r}, \wt{\zeta}_1, \dots, \wt{\zeta}_r) ds_1 \dots ds_r,
\end{align*}

\noi
where $\zeta_j = (m_j, s_j)$ and $\wt{\zeta}_j = (- m_j, s_j)$. Whenever two elements $m_1, m_2 \in \Z^3$ are contracted, we call $(m_1, m_2)$ a \textit{pair}, for which we must have $m_1 + m_2 = 0$. 

For a proof of the following product formula, see \cite[Proposition~1.1.3]{Nua}.
\begin{lemma}
\label{LEM:sto_prod}
Let $k, \l \in \N$. Let $f \in L^2 (\mathsf{Z}^k)$ and $g \in L^2 (\mathsf{Z}^\l)$ be symmetric functions. Then, we have
\begin{align*}
I_k[f] I_\ell[g] = \sum_{r = 0}^{\min \{k, \l\}} r! \binom{k}{r} \binom{\l}{r} I_{k + \l - 2r} [f \otimes_r g].
\end{align*}
\end{lemma}

\subsection{Bou\'e-Dupuis variational formula}
In this subsection, we recall the Bou\'e-Dupuis variational formula in \cite{BD}, which will play a central role in proving the convergence of the Gibbs measures. 

Let $W (t)$ be a Gaussian process (with respect to the underlying probability measure $\mathbb{P}$) defined by
\begin{align}
W (t) = \frac{1}{(2 \pi)^{\frac 32}} \sum_{n \in \Z^3} \frac{B_n (t)}{\jb{n}^\al} e^{in \cdot x},
\label{defWt}
\end{align}

\noi
where $\{ B_n \}_{n \in \Z^3}$ is a family of Brownian motions satisfying the three properties listed at the beginning of Subsection~\ref{SUB:mult}. Note that we have $\Law (W (1)) = \mu_\al$, where $\mu_\al$ is the Gaussian measure defined in \eqref{mu_al}. Given $N \in \N$, we set $W_N = \pi_N W$.

We also define the spaces of drifts. For $s \in \R$, let $\mathbb{H}_a^s$ be the space of progressively measurable processes belonging to $L^2 ([0, 1]; H^s (\T^3))$, $\mathbb{P}$-almost surely. When $s = 0$, we set $\mathbb{H}_a = \mathbb{H}_a^0$.

We can now state the Bou\'e-Dupuis variational formula \cite{BD, Ust}; see \cite[Theorem 7]{Ust} and also \cite[Theorem 2]{BG}.

\begin{lemma}
\label{LEM:BD}
Let $N \in \N$ and $W$ be a Gaussian process as defined in \eqref{defWt}. Let $F: C^\infty (\T^3) \to \R$ be a measurable function such that $\E \big[ |F (W_N (1))|^p \big] < \infty$ and $\E \big[ |e^{- F(W_N (1))}|^q \big] < \infty$ for some $1 < p, q < \infty$ with $\frac 1p + \frac 1q = 1$. Thus, we have
\begin{align*}
- \log \E \big[ e^{- F(W_N (1))} \big] = \inf_{\ta \in \mathbb{H}_a} \E \bigg[ F(W_N (1) + \pi_N I(\ta) (1)) + \frac 12 \int_0^1 \| \ta (t) \|_{L_x^2}^2 dt \bigg],
\end{align*}

\noi
where $I (\ta)$ is defined by
\begin{align}
I (\ta) (t) \deff \int_0^t \jb{\nb}^{- \al} \ta (t') dt'.
\label{defIta}
\end{align}
\end{lemma}

Given $k \in \N$, we define
\begin{align*}
\wick{W_N^k (t)} \deff H_k (W_N (t); t \s_N)
\end{align*}

\noi
with $\s_N = \E [|W_N (1)|^2]$ given in \eqref{sigma}.
We now show the following lemma on some useful properties and regularity estimates.
\begin{lemma}
\label{LEM:Yreg}
Let $1 < \al \leq \frac 98$ and $T > 0$. Let $W$ be a Gaussian process as defined in \eqref{defWt}.

\smallskip \noi
\textup{(i)} Let $0 < s < \frac 32$, $\be \geq 0$, and $k \geq 1$ be an integer satisfying
\begin{align}
s > k \Big( \frac 32 - \al \Big) - \be > 0.
\label{skab}
\end{align}

\noi
Then, for any $1 \leq p < \infty$ and $t \in [0, T]$, there exists a constant $C > 0$ such that
\begin{align}
\sup_{N \geq 1} \Big( N^{-p \be} \E \big[ \big\| \wick{W_N^k (t)} \big\|_{\mathcal{C}^{-s}}^p \big] \Big) < C.
\label{YNk_bdd}
\end{align}

\smallskip \noi
\textup{(ii)} Let $k = 1, 2, 3$. Then, for any $1 \leq p < \infty$, $t \in [0, T]$, and $\eps > 0$, there exists a small number $\g > 0$ and a constant $C > 0$ such that
\begin{align*}
\sup_{M > N \geq 1} \Big( N^{p \g} \E \big[ \big\| \wick{W_M^k (t)} - \wick{W_N^k (t)} \big\|_{\mathcal{C}^{k (\al - \frac 32 - \eps)}}^p \big] \Big) < C.
\end{align*}

\smallskip \noi
\textup{(iii)} For any integer $k \geq 1$ and $N \in \N$, we have
\begin{align*}
\E \bigg[ \int_{\T^3} \wick{W_N^k (1)} dx \bigg] = 0.
\end{align*}

\smallskip \noi
\textup{(iv)} For any $\ta \in \mathbb{H}_a$, we have
\begin{align*}
\| I (\ta) (1) \|_{H^\al}^2 \leq \int_0^1 \| \ta (t) \|_{L^2}^2 dt.
\end{align*}
\end{lemma}

\begin{proof}
(i) By Lemma~\ref{LEM:wick} and Lemma \ref{LEM:conv}~(ii) along with the condition \eqref{skab}, we have that for any $n \in \Z^3$,
\begin{align*}
\E \big[ \big| N^{- \be} \F_x \big( \wick{W_N^k} \big) (t, n) \big|^2 \big] 
&\les N^{-2 \be} \sum_{\substack{|n_1|, \dots, |n_k| \leq N \\ n_1 + \cdots + n_k = n}} \frac{1}{\jb{n_1}^{2 \al} \cdots \jb{n_k}^{2 \al}} \\
&\les \sum_{\substack{n_1, \dots, n_k \in \Z^3 \\ n_1 + \cdots + n_k = n}} \frac{1}{\jb{n_1}^{2\al + \frac{2\be}{k}} \cdots \jb{n_k}^{2\al + \frac{2\be}{k}}} \\
&\les \jb{n}^{(3 - 2 \al) k - 2 \be - 3}.
\end{align*}

\noi
Thus, we can conclude the bound \eqref{YNk_bdd} by Lemma \ref{LEM:reg} along with the condition \eqref{skab}.

\smallskip \noi
(ii) The proof is similar to part (i). The only additional argument is that in view of the open condition in Lemma~\ref{LEM:conv}~(ii), we can use
\begin{align*}
\ind_{\{N < |n_j| \leq M\}} \frac{1}{\jb{n_j}^{2 \al}} \les N^{-\g} \frac{1}{\jb{n_j}^{2 \al - \g}}
\end{align*}

\noi
to create the $N^{-\g}$ term, where $\g > 0$ is sufficiently small so that one can still apply Lemma~\ref{LEM:conv}~(ii).

\smallskip \noi
(iii) This follows directly from the definition $\wick{W_N^k (1)} \, = H_k (W_N (1); \s_N)$ with $\s_N$ defined in \eqref{sigma} and Lemma~\ref{LEM:wick} with $\l = 0$.

\smallskip \noi
(iv) The inequality follows directly from Minkowski's integral inequality and the Cauchy-Schwarz inequality.
\end{proof}

\section{Integer lattice-counting estimates}
\label{SEC:count}

In this section, we recall and present some integer lattice-counting estimates. These counting estimates are essential for showing the regularity properties of the stochastic objects in Subsection \ref{SUB:reg} below. In the following, we denote
\begin{align*}
\jbb{n} = \jbb{n}_\al \deff \sqrt{\jb{n}^{2 \al} - \frac 14}
\end{align*}

\noi
for any $n \in \Z^3$ and $\al \geq 0$. Note that we have $\jbb{n} \sim \jb{n}^\al$.

\subsection{Basic counting estimates}
In this subsection, we recall some counting estimates from previous literature. For convenience, we will use the notation $n_{12\cdots k} = n_1 + n_2 + \cdots + n_k$ given $k \in \N$ and $n_1, \dots, n_k \in \Z^3$.

We first recall the following ``two-ball'' counting lemma. For a proof, see \cite[Lemma~3.2]{LTzW} and \cite[Lemma~4.17]{Bring2}.
\begin{lemma}
\label{LEM:count_twoball}
Let $1 < \al \leq 2$ and $N, A, B \geq 1$ be dyadic numbers. Let $a \in \Z^3$ satisfying $|a| \sim A$. Then, we have
\begin{align*}
\sup_{\zeta \in \R} \, &\# \big\{ n \in \Z^3 : |n| \sim N, |n + a| \sim B, \big| \jbb{a + n} \pm \jbb{n} - \zeta \big| \les 1 \big\} \\
&\les \min \{A, B, N\}^{-1} \min \{B, N\}^3.
\end{align*}
\end{lemma}

We also record the following summation counting estimate. For a proof, see \cite[Lemma 3.5]{LTzW}.
\begin{lemma}
\label{LEM:count_sum}
Let $k \geq 3$ be an integer, $1 < \al \leq \frac 32$, $s \geq \al - 1$, and $N_1, \dots, N_k \geq 1$ be dyadic numbers. Let $\eps_{\textup{sum}}, \eps_1, \dots, \eps_k \in \{\pm 1\}$ and define the phase function
\begin{align*}
\kappa = \kappa (n_1, \dots, n_k) \deff \eps_{\textup{sum}} \bigg\llbracket \sum_{j = 1}^k n_j \bigg\rrbracket + \sum_{j = 1}^k \eps_j \jbb{n_j}.
\end{align*}

\noi
Then, we have
\begin{align*}
\sup_{\zeta \in \R} &\sum_{n_1, \dots, n_k \in \Z^3} \bigg[ \bigg( \prod_{j = 1}^k \ind_{\{|n_j| \sim N_j\}} \jbb{n_j}^{-2} \bigg) \bigg\langle \sum_{j = 1}^k n_j \bigg\rangle^{2(s - \al)} \ind_{\{|\kappa - \zeta| \leq 1\}} \bigg] \\
&\les N_{\textup{max}}^{(k - 3) (3 - 2 \al) + 2s - 6 \al + 6},
\end{align*}

\noi
where $N_{\textup{max}} = \max \{ N_1, \dots, N_k \}$.
\end{lemma}

We will also need the following frequency-scale estimate. For a proof, see \cite[Lemma A.1]{Bring2}.
\begin{lemma}
\label{LEM:count_scale}
Let $N_2, N_3, N_{12}, N_{123} \geq 1$ be dyadic numbers and let $n_1, n_2, n_3 \in \Z^3$ be such that
\begin{align*}
\ind_{\{|n_2| \sim N_2\}} \cdot \ind_{\{|n_3| \sim N_3\}} \cdot \ind_{\{|n_{12}| \sim N_{12}\}} \cdot \ind_{\{|n_{123}| \sim N_{123}\}} \not\equiv 0.
\end{align*}

\noi
Then, we have
\begin{align*}
\frac{\min \{N_3, N_{123}\}^2 \min \{N_2, N_{12}\}}{\min \{N_{123}, N_{12}, N_3\}} \les N_3 \cdot N_{123}.
\end{align*}
\end{lemma}

\subsection{Extended non-resonant and resonant counting estimates}

In this subsection, we show some counting estimates that are useful for the extended quintic stochastic terms and the extended cubic-linear random operators in Subsection \ref{SUB:reg} below. We first show the following extended non-resonant counting estimate, which is an extended version of \cite[Lemma~3.7]{LTzW}.
\begin{lemma}
\label{LEM:count_quin}
Let $1 < \al \leq \frac 98$. Let $k_1 \geq 3$, $k_2 \geq 1$ be integers, and $k_{12} = k_1 + k_2$. Let $N_1, \dots, N_{k_{12}} \geq 1$ be dyadic numbers. Furthermore, we let $\eps_{\textup{sum}_1}, \eps_{\textup{sum}}, \eps_1, \dots, \eps_{k_{12}} \in \{\pm 1\}$ and define three phase functions by 
\begin{align}
\kappa_1 (n_1, \dots, n_{k_1}) &\deff \eps_{\textup{sum}_1} \bigg\llbracket \sum_{j = 1}^{k_1} n_j \bigg\rrbracket + \sum_{j = 1}^{k_1} \eps_j \jbb{n_j}, \label{kappa1} \\
\kappa_2 (n_1, \dots, n_{k_{12}}) &\deff \eps_{\textup{sum}} \bigg\llbracket \sum_{j = 1}^{k_{12}} n_j \bigg\rrbracket + \eps_{\textup{sum}_1} \bigg\llbracket \sum_{j = 1}^{k_1} n_j \bigg\rrbracket + \sum_{j = k_1 + 1}^{k_{12}} \eps_j \jbb{n_j}, \notag \\
\kappa_3 (n_1, \dots, n_{k_{12}}) &\deff \eps_{\textup{sum}} \bigg\llbracket \sum_{j = 1}^{k_{12}} n_j \bigg\rrbracket + \sum_{j = 1}^{k_{12}} \eps_j \jbb{n_j}. \notag
\end{align}

\smallskip \noi
\textup{(i)} If $k_2 \geq 2$, we have
\begin{align}
\begin{split}
\sup_{\zeta, \zeta' \in \R} &\sum_{n_1, \dots, n_{k_{12}} \in \Z^3} \bigg[ \bigg( \prod_{j = 1}^{k_{12}} \ind_{\{|n_j| \sim N_j\}} \jbb{n_j}^{-2} \bigg) \bigg\langle \sum_{j = 1}^{k_{12}} n_j \bigg\rangle^{-1}  \bigg\llbracket \sum_{j = 1}^{k_1} n_j \bigg\rrbracket^{-2} \bigg] \\
&\times \ind_{\{ |\kappa_1 - \zeta| \leq 1 \}} \cdot \big( \ind_{\{|\kappa_2 - \zeta'| \leq 1\}} + \ind_{\{|\kappa_3 - \zeta'| \leq 1\}} \big) \bigg] \les N_{\textup{max}}^{(k_{12} - 5) (3 - 2\al) - \ta}
\end{split}
\label{count_nr1}
\end{align}

\noi
for some $\ta > 0$, where $N_{\textup{max}} = \max \{N_1, \dots, N_{k_{12}}\}$.

\smallskip \noi
\textup{(ii)} If $k_2 \geq 1$, then for any $\eps > 0$, we have
\begin{align}
\begin{split}
\sup_{\zeta, \zeta' \in \R} &\sum_{n_1, \dots, n_{k_{12}} \in \Z^3} \bigg[ \bigg( \prod_{j = 1}^{k_{12}} \ind_{\{|n_j| \sim N_j\}} \jbb{n_j}^{-2} \bigg) \bigg\langle \sum_{j = 1}^{k_{12}} n_j \bigg\rangle^{2 \al - 3 - \eps} \bigg\llbracket \sum_{j = 1}^{k_1} n_j \bigg\rrbracket^{-2} \bigg] \ind_{\{ |\kappa_1 - \zeta| \leq 1 \}}  \\
&\les N_{\textup{max}}^{(k_{12} - 4) (3 - 2\al) - \ta}
\end{split}
\label{count_nr2}
\end{align}

\noi
for some $\ta > 0$.
\end{lemma}

\begin{proof}
We first consider \eqref{count_nr1}. For convenience, we let $n_{\text{sum}} = \sum_{j = 1}^{k_{12}} n_j$. Without loss of generality, we can assume that $N_{k_{12}} = \max \{ N_{k_1 + 1}, \dots, N_{k_{12}} \}$. Let $N_{\text{sum}}$ be the dyadic size of $n_{\text{sum}}$ and let $N_{\text{sum}}'$ be the dyadic size of $n_{\text{sum}} - n_{k_{12}}$. By using Lemma \ref{LEM:count_twoball} for the summation in $n_{k_{12}}$ and then summing over $n_{k_{12} - 1}, \dots, n_{k_1 + 1}$, we obtain
\begin{align}
\begin{split}
&\sum_{n_1, \dots, n_{k_{12}} \in \Z^3} \bigg[ \bigg( \prod_{j = 1}^{k_{12}} \ind_{\{|n_j| \sim N_j\}} \jbb{n_j}^{-2} \bigg) \ind_{\{|n_{\text{sum}}| \sim N_{\text{sum}}\}} \ind_{\{|n_{\text{sum}} - n_{k_{12}}| \sim N_{\text{sum}}'\}} \jb{n_{\text{sum}}}^{-1} \\
&\quad \times \bigg\llbracket \sum_{j = 1}^{k_1} n_j \bigg\rrbracket^{-2} \ind_{\{ |\kappa_1 - \zeta| \leq 1 \}} \cdot \big( \ind_{\{|\kappa_2 - \zeta'| \leq 1\}} + \ind_{\{|\kappa_3 - \zeta'| \leq 1\}} \big) \bigg] \\
&\les N_{\text{sum}}^{-1} N_{k_{12}}^{-2 \al} \min \{N_{\text{sum}}, N_{\text{sum}}', N_{k_{12}}\}^{-1} \min \{N_{\text{sum}}, N_{k_{12}}\}^3 \\
&\quad \times \sum_{n_1, \dots, n_{k_{12} - 1} \in \Z^3} \bigg( \prod_{j = 1}^{k_{12} - 1} \ind_{\{|n_j| \sim N_j\}} \jbb{n_j}^{-2} \bigg) \ind_{\{|n_{\text{sum}} - n_{k_{12}}| \sim N_{\text{sum}}'\}} \bigg\langle \sum_{j = 1}^{k_1} n_j \bigg\rangle^{-2} \ind_{\{|\kappa_1 - \zeta| \leq 1\}} \\
&\les N_{\text{sum}}^{-1} N_{k_{12} - 1}^{-2\al} N_{k_{12}}^{-2\al} \min \{N_{\text{sum}}, N_{\text{sum}}', N_{k_{12}}\}^{-1} \min \{N_{\text{sum}}, N_{k_{12}}\}^3 \min \{N_{\text{sum}}', N_{k_{12} - 1}\}^3 \\
&\quad \times N_{\text{max}}^{(k_2 - 2) (3 - 2\al)} \sum_{n_1, \dots, n_{k_1} \in \Z^3} \bigg( \prod_{j = 1}^{k_1} \ind_{\{|n_j| \sim N_j\}} \jbb{n_j}^{-2} \bigg) \bigg\langle \sum_{j = 1}^{k_1} n_j \bigg\rangle^{-2} \ind_{\{|\kappa_1 - \zeta| \leq 1\}}.
\end{split}
\label{count_quin1}
\end{align}

\noi
By Lemma \ref{LEM:count_sum} with $k = k_1$ and $s = \al - 1$ and Lemma \ref{LEM:count_scale}, we obtain that
\begin{align*}
\eqref{count_quin1} &\les N_{\text{sum}}^{-1} N_{k_{12} - 1}^{-2\al} N_{k_{12}}^{-2\al} \frac{\min \{N_{\text{sum}}, N_{k_{12}}\}^3 \min \{N_{\text{sum}}', N_{k_{12} - 1}\}^3}{\min \{N_{\text{sum}}, N_{\text{sum}}', N_{k_{12}}\}} \\ 
&\quad \times N_{\text{max}}^{(k_{12} - 5) (3 - 2\al)} \max \{N_1, \dots, N_{k_1}\}^{-2\ta} \\
&\les N_{k_{12} - 1}^{-2\al} N_{k_{12}}^{1-2\al} \min \{N_{\text{sum}}, N_{k_{12}}\} \min \{N_{\text{sum}}', N_{k_{12} - 1}\}^2 \\ 
&\quad \times N_{\text{max}}^{(k_{12} - 5) (3 - 2\al)} \max \{N_1, \dots, N_{k_1}\}^{-2\ta} \\
&\les N_{\text{max}}^{(k_{12} - 5) (3 - 2\al)} N_{k_{12} - 1}^{2 - 2 \al} N_{k_{12}}^{2 - 2 \al} \max \{N_1, \dots, N_{k_1}\}^{-2\ta}
\end{align*}

\noi
for some $\ta > 0$. Since $\al > 1$, we obtain the desired estimate \eqref{count_nr1} by summing up $1 \leq N_{\text{sum}}, N_{\text{sum}}' \les N_{\text{max}}$.

We now consider \eqref{count_nr2}. We again let $n_{\text{sum}} = \sum_{j = 1}^{k_{12}} n_j$ and assume that $N_{k_{12}} = \max \{ N_{k_1 + 1}, \dots, N_{k_{12}} \}$. Let $N_{\text{sum}}$ be the dyadic size of $n_{\text{sum}}$. By summing over $n_{k_{12} - 1}, \dots, n_{k_1 + 1}$ and using Lemma~\ref{LEM:count_sum} with $k = k_1$ and $s = \al - 1$, we obtain
\begin{align*}
&\sum_{n_1, \dots, n_{k_{12}} \in \Z^3} \bigg[ \bigg( \prod_{j = 1}^{k_{12}} \ind_{\{ |n_j| \sim N_j \}} \jbb{n_j}^{-2} \bigg) \ind_{\{ |n_{\text{sum}}| \sim N_{\text{sum}} \}} \jb{n_{\text{sum}}}^{2 \al - 3 - \eps} 
\bigg\llbracket \sum_{j = 1}^{k_1} n_j \bigg\rrbracket^{-2} \ind_{\{ |\kappa_1 - \zeta| \leq 1 \}} \bigg] \\
&\les N_{\text{sum}}^{2 \al - 3 - \eps} N_{k_{12}}^{-2 \al} \min \{N_{\text{sum}}, N_{k_{12}}\}^3 N_{\text{max}}^{(k_2 - 1) (3 - 2 \al)} \\
&\quad \times \sum_{n_1, \dots, n_{k_1} \in \Z^3} \bigg( \prod_{j = 1}^{k_1} \ind_{\{ |n_j| \sim N_j \}} \jbb{n_j}^{-2} \bigg) \bigg\langle \sum_{j = 1}^{k_1} n_j \bigg\rangle^{-2} \ind_{\{ |\kappa_1 - \zeta| \leq 1 \}} \\
&\les N_{\text{max}}^{(k_{12} - 4)(3 - 2 \al) - \ta} 
\end{align*}

\noi
for some $\ta > 0$, which gives the desired estimate \eqref{count_nr2}.
\end{proof}

We now show some extended resonant estimates. For this purpose, we need to recall the following definitions from \cite[Definition 4.30]{Bring2}.

\begin{definition} \rm
\label{DEF:pair}
Let $k \geq 1$. We call a set of unordered pairs $\mathcal{P} \subset \{1, \dots, k\}^2$ \textit{a pairing} if
for each $1 \leq j_1 \leq k$, $\{j_1, j_2\} \in \mathcal{P}$ for at most one $1 \leq j_2 \leq k$ with $j_2 \neq j_1$.
If $\{(j_1, j_2\} \in \mathcal{P}$, the tuple $\{j_1, j_2\}$ is called \textit{a pair}. If $1 \leq j \leq k$ is contained in a pair, we say that $j$ is \textit{paired}. With a slight abuse of notation, we also write $j \in \mathcal{P}$ if $j$ is paired. If $j$ is not paired, we also say that $j$ is \textit{unpaired} and write $j \notin \mathcal{P}$. Furthermore, given a partition $\mathcal{A} = \{ A_\l \}_{\l = 1}^L$ of $\{1, \dots, k\}$, we say that $\mathcal{P}$ \textit{respects} $\mathcal{A}$ if $\mathcal{P}$ does not pair elements of the same set $A_\l \in \mathcal{A}$, i.e.~$j_1, j_2 \in A_\l$ for some $1 \leq \l \leq L$ implying that $\{j_1, j_2\} \notin \mathcal{P}$. We say that $(n_1, \dots, n_k) \in (\Z^3)^k$ is \textit{admissible} if $\{j_1, j_2\} \in \mathcal{P}$ implies that $n_{j_1} + n_{j_2} = 0$.
\end{definition}

We can now state and prove our extended resonant estimates, which are extended versions of \cite[Lemma 3.9 and Lemma 3.10]{LTzW}. These estimates will be used when dealing with a product of two stochastic terms which contains pairings among the frequencies. 
\begin{lemma}
\label{LEM:count_quin2}
Let $1 < \al \leq \frac 98$ and $\eps > 0$. Let $k_1 \geq 3$ and $k_2 \geq 1$ be integers and denote $k_{12} = k_1 + k_2$. Let $N_1, \dots, N_{k_{12}} \geq 1$ be dyadic numbers. Let $\mathcal{P}$ be a pairing of $\{1, \dots, k_{12}\}$ which respects the partition $\mathcal{K}_1 = \{1, \dots, k_1\}$ and $\mathcal{K}_2 = \{k_1 + 1, \dots, k_{12}\}$, and we assume that $\mathcal{P} \neq \varnothing$. We also let $\eps_{\textup{sum}}, \eps_1, \dots, \eps_{k_1} \in \{\pm 1\}$ and define 
\begin{align*}
\kappa \deff \eps_{\textup{sum}} \bigg\llbracket \sum_{j \in \mathcal{K}_1} n_j \bigg\rrbracket + \sum_{j \in \mathcal{K}_1} \eps_j \jbb{n_j}.
\end{align*}

\smallskip \noi
\textup{(i)} If $k_2 \geq 2$ and $k_{12} > 2 |\mathcal P|$, we have
\begin{align}
\begin{split}
&\sup_{\zeta \in \R} \sum_{(n_j)_{j \notin \mathcal{P}}} \bigg[ \bigg( \prod_{j \notin \mathcal{P}} \ind_{\{|n_j| \sim N_j\}} \jbb{n_j}^{-2} \bigg) \bigg\langle \sum_{j \notin \mathcal{P}} n_j \bigg\rangle^{2 \al - 3 - \eps} \\
&\quad \times \bigg| \sum_{(n_j)_{j \in \mathcal{P}}} \ind_{\substack{(n_1, \dots, n_{k_{12}}) \\ \textup{admissible}}} \bigg( \prod_{j \in \mathcal{P}} \ind_{\{ |n_j| \sim N_j \}} \jbb{n_j}^{-1} \bigg) \bigg\llbracket \sum_{j \in \mathcal{K}_1} n_j \bigg\rrbracket^{-1} \ind_{\{|\kappa - \zeta| \leq 1\}} \bigg|^2 \bigg] \\
&\les N_{\textup{max}}^{(k_{12} - 5) (3 - 2 \al) - \ta}
\end{split}
\label{count_quin}
\end{align}

\noi
for some $\ta > 0$, where $N_{\textup{max}} = \max \{ N_1, \dots, N_{k_{12}} \}$. If $k_2 \geq 2$ and $k_1 = k_2 = |\mathcal P|$, we have
\begin{align}
\begin{split}
&\sup_{\zeta \in \R} \bigg| \sum_{(n_j)_{j \in \mathcal{P}}} \ind_{\substack{(n_1, \dots, n_{k_{12}}) \\ \textup{admissible}}} \bigg( \prod_{j \in \mathcal{P}} \ind_{\{ |n_j| \sim N_j \}} \jbb{n_j}^{-1} \bigg) \bigg\llbracket \sum_{j \in \mathcal{K}_1} n_j \bigg\rrbracket^{-1} \ind_{\{|\kappa - \zeta| \leq 1\}} \bigg| \\
&\les N_{\textup{max}}^{(|\mathcal P| - 2) (3 - 2 \al) - \ta}
\end{split}
\label{count_quin-}
\end{align}

\noi
for some $\ta > 0$.

\smallskip \noi
\textup{(ii)} If $k_2 = 1$, we have
\begin{align}
\begin{split}
&\sup_{\zeta \in \R} \sum_{(n_j)_{j \notin \mathcal{P}}} \bigg[ \bigg( \prod_{j \notin \mathcal{P}} \ind_{\{|n_j| \sim N_j\}} \jbb{n_j}^{-2} \bigg) \bigg\langle \sum_{j \notin \mathcal{P}} n_j \bigg\rangle^{2\al - 3 - \eps} \\
&\quad \times \bigg| \sum_{(n_j)_{j \in \mathcal{P}}} \ind_{\substack{(n_1, \dots, n_{k_{12}}) \\ \textup{admissible}}} \bigg( \prod_{j \in \mathcal{P}} \ind_{\{ |n_j| \sim N_j \}} \jbb{n_j}^{-1} \bigg) \bigg\llbracket \sum_{j \in \mathcal{K}_1} n_j \bigg\rrbracket^{-1} \ind_{\{|\kappa - \zeta| \leq 1\}} \bigg|^2 \bigg] \\
&\les N_{\textup{max}}^{(k_1 - 3) (3 - 2 \al) - \ta}
\end{split}
\label{count_quin+}
\end{align}

\noi
for some $\ta > 0$.
\end{lemma}

\begin{proof}
We first show the proof of \eqref{count_quin} by considering the following two cases.

\medskip \noi
\textbf{Case 1:} $|\mathcal{P}| = 1$.

By symmetry, we can assume without loss of generality that $\mathcal{P} = \{ \{1, k_{12}\} \}$. 
We claim that
\begin{align}
\sum_{n_1 \in \Z^3} \bigg[ \Big( \ind_{\{|n_1| \sim N_1\}} \jbb{n_1}^{-2} \Big) \bigg\llbracket \sum_{j \in \mathcal{K}_1} n_j \bigg\rrbracket^{-1} \ind_{\{|\kappa - \zeta| \leq 1\}} \bigg] \les \bigg\langle \sum_{j \in \mathcal{K}_1 \setminus \{1\}} n_j \bigg\rangle^{-\al} N_1^{- \ta}
\label{count_quin2}
\end{align}

\noi
for some $\ta > 0$. Let us first take \eqref{count_quin2} for granted and prove \eqref{count_quin}. We can assume without loss of generality that $N_2 = \max_{j \in \mathcal{K}_1 \setminus \{1\}} \{ N_j \}$ and $N_{k_{12} - 1} = \max_{j \in \mathcal{K}_2 \setminus \{k_{12}\}} \{ N_j \}$. We can then use \eqref{count_quin2} and Lemma~\ref{LEM:conv}~(i) three times to obtain that the left-hand-side of \eqref{count_quin} is bounded by
\begin{align*}
&N_1^{- 2\ta} \sum_{(n_j)_{j \notin \mathcal{P}}} \bigg( \prod_{j \notin \mathcal{P}} \ind_{\{|n_j| \sim N_j\}} \jbb{n_j}^{-2} \bigg) \bigg\langle \sum_{j \notin \mathcal{P}} n_j \bigg\rangle^{2 \al - 3 - \eps} \bigg\langle \sum_{j \in \mathcal{K}_1 \setminus \{1\}} n_j \bigg\rangle^{-2 \al} \\
&\les N_1^{- 2\ta} N_{k_{12} - 1}^{- \ta} \sum_{(n_j)_{j \notin \mathcal{P} \wedge j \neq k_{12} - 1}} \bigg( \prod_{j \notin \mathcal{P} \wedge j \neq k_{12} - 1} \ind_{\{|n_j| \sim N_j\}} \jbb{n_j}^{-2} \bigg) \bigg\langle \sum_{j \in \mathcal{K}_1 \setminus \{1\}} n_j \bigg\rangle^{-2 \al} \\
&\les N_1^{- 2\ta} N_{k_{12} - 1}^{- \ta} N_2^{- \ta} \sum_{(n_j)_{j \notin \mathcal{P} \wedge j \neq 2, k_{12} - 1}} \bigg( \prod_{j \notin \mathcal{P} \wedge j \neq 2, k_{12} - 1} \ind_{\{|n_j| \sim N_j\}} \jbb{n_j}^{-2} \bigg) \bigg\langle \sum_{j \in \mathcal{K}_1 \setminus \{1, 2\}} n_j \bigg\rangle^{- \al} \\
&\les N_{\text{max}}^{(k_{12} - 5)(3 - 2 \al) - \ta}
\end{align*}

\noi
for some $\ta > 0$, where we used $\al > 1$. This proves \eqref{count_quin}.

It remains to show \eqref{count_quin2}. For simplicity, we write $n_{\text{sum}}^1 = \sum_{j \in \mathcal{K}_1 \setminus \{1\}} n_j$. We also dyadically localize $|n_{\text{sum}}^1 + n_1| \sim N_{\text{sum}}$. By Lemma \ref{LEM:count_twoball}, we have
\begin{align*}
\sum_{n_1 \in \Z^3} &\bigg[ \Big( \ind_{\{|n_1| \sim N_1\}} \jbb{n_1}^{-2} \Big) \jbb{n_{\text{sum}}^1 + n_1}^{-1} \ind_{\{ |n_{\text{sum}}^1 + n_1| \sim N_{\text{sum}} \}} \ind_{\{|\kappa - \zeta| \leq 1\}} \bigg] \\
&\les N_1^{-2 \al} N_{\text{sum}}^{-\al} \min \{\jb{n_{\text{sum}}^1}, N_{\text{sum}}, N_1\}^{-1} \min \{N_{\text{sum}}, N_1\}^3.
\end{align*}

\noi
If $N_{\text{sum}} \ll N_1$, then $\jb{n_{\text{sum}}^1} \sim N_1$, so that the left-hand-side of \eqref{count_quin2} is bounded by
\begin{align*}
\sum_{\substack{N_{\text{sum}} \ll N_1 \\ N_{\text{sum}} \geq 1 \text{ dyadic}}} N_1^{-2 \al} N_{\text{sum}}^{2-\al} \les \jb{n_{\text{sum}}^1}^{-\al} N_1^{2 - 2 \al}.
\end{align*}

\noi
If $N_{\text{sum}} \sim N_1$, then $\jb{n_{\text{sum}}^1} \les N_{\text{sum}} \sim N_1$, so that the left-hand-side of \eqref{count_quin2} is bounded by a constant multiple of
\begin{align*}
\jb{n_{\text{sum}}^1}^{-1} N_1^{3 - 3\al} \les \jb{n_{\text{sum}}^1}^{-\al} N_1^{2 - 2 \al} ,
\end{align*}

\noi
where we used the fact that $\al > 1$. If $N_{\text{sum}} \gg N_1$, then $\jb{n_{\text{sum}}^1} \sim N_{\text{sum}} \gg N_1$, so that the left-hand-side of \eqref{count_quin2} is bounded by
\begin{align*}
\sum_{\substack{N_{\text{sum}} \sim \jb{n_{\text{sum}}^1} \\ N_{\text{sum}} \geq 1 \text{ dyadic}}} N_{\text{sum}}^{- \al} N_1^{2 - 2 \al} \les \jb{n_{\text{sum}}^1}^{-\al} N_1^{2 - 2 \al}.
\end{align*}

\noi
In all cases, we obtain \eqref{count_quin2} since $\al > 1$.

\medskip \noi
\textbf{Case 2:} $|\mathcal{P}| \geq 2$.

By symmetry, we can assume without loss of generality that $\{1, k_{12}\}, \{2, k_{12} - 1\} \in \mathcal{P}$. For convenience, we let $n_{\text{sum}} = \sum_{j \in \mathcal{K}_1} n_j$ and $n_{\text{sum}}^1 = \sum_{j \in \mathcal{K}_1 \setminus \{1\}} n_j$. Let $N_{\text{sum}}$ be the dyadic size of $n_{\text{sum}}$ and let $N_{\text{sum}}^1$ be the dyadic size of $n_{\text{sum}}^1$. Also, we can assume that $N_1 = \max_{j \in \mathcal{P}} \{ N_j \}$. By Lemma \ref{LEM:count_twoball} for the summation in $n_1$, a summation over $(n_j)_{j \in \mathcal{P} \setminus \{\{1, k_{12}\}\}}$, and Lemma \ref{LEM:count_scale}, we obtain
\begin{align}
\begin{split}
&\sum_{(n_j)_{j \in \mathcal{P}}} \bigg[ \ind_{\substack{(n_1, \dots, n_{k_{12}}) \\ \textup{admissible}}} \bigg( \prod_{j \in \mathcal{P}} \ind_{\{ |n_j| \sim N_j \}} \jbb{n_j}^{-1} \bigg) \jbb{n_{\text{sum}}}^{-1} \\
&\quad \times \ind_{\{|n_{\text{sum}}| \sim N_{\text{sum}}\}} \ind_{\{|n_{\text{sum}}^1| \sim N_{\text{sum}}^1\}} \ind_{\{|\kappa - \zeta| \leq 1\}} \bigg] \\
&\les N_{\text{sum}}^{-\al} N_1^{-2 \al} \min \{N_{\text{sum}}, N_{\text{sum}}^1, N_1\}^{-1} \min \{ N_{\text{sum}}, N_1 \}^3 \\
&\quad \times \sum_{(n_j)_{j \in \mathcal{P} \setminus \{\{1, k_{12}\}\}}} \ind_{\substack{(n_2, \dots, n_{k_{12} - 1}) \\ \textup{admissible}}} \bigg( \prod_{j \in \mathcal{P} \setminus \{\{1, k_{12}\}\}} \ind_{\{ |n_j| \sim N_j \}} \jbb{n_j}^{-1} \bigg) \ind_{\{ |n_{\text{sum}}^{1}| \sim N_{\text{sum}}^{1} \}} \\
&\les N_{\text{sum}}^{-\al} N_{1}^{-2 \al} N_{2}^{-2 \al} \frac{\min \{ N_{\text{sum}}, N_{1} \}^3 \min \{N_{\text{sum}}^{1}, N_{2}\}^3}{\min \{N_{\text{sum}}, N_{\text{sum}}^{1}, N_{1}\}} N_{\text{max}}^{(|\mathcal{P}| - 2) (3 - 2\al)} \\
&\les N_{\text{sum}}^{1 - \al} N_{1}^{2 - 2 \al} N_{2}^{2 - 2 \al} N_{\text{max}}^{(|\mathcal{P}| - 2) (3 - 2\al)}.
\end{split}
\label{count_quin3}
\end{align}

\noi
To deal with the summation over $(n_j)_{j \notin \mathcal{P}}$, we use Lemma~\ref{LEM:conv}~(i) and the fact that $1 < \al \leq \frac 98$ to obtain
\begin{align}
\begin{split}
\sum_{(n_j)_{j \notin \mathcal{P}}} \bigg( \prod_{j \notin \mathcal{P}} \ind_{\{|n_j| \sim N_j\}} \jbb{n_j}^{-2} \bigg) \bigg\langle \sum_{j \notin \mathcal{P}} n_j \bigg\rangle^{2 \al - 3 - \eps} \les \max_{j \notin \mathcal{P}} \{N_j\}^{-\ta} N_{\text{max}}^{(k_{12} - 2|\mathcal{P}| - 1) (3 - 2 \al)}
\end{split}
\label{count_quin4}
\end{align}

\noi
for some $\ta > 0$. The desired estimate then follows from \eqref{count_quin3} and \eqref{count_quin4} given that $\al > 1$.

Combining the above two cases, we obtain \eqref{count_quin}. For \eqref{count_quin-}, we simply use \eqref{count_quin3} in Case 2 above along with the fact that $\al > 1$ to conclude. For \eqref{count_quin+}, we just need to use similar steps in Case 1 above to get the desired estimate. Note that $\mathcal{K}_2 \setminus \{k_{12}\} = \varnothing$, so that there is no frequency in $\mathcal{K}_2 \setminus \{k_{12}\}$ to be summed up and we can simply bound $\jb{\sum_{j \notin \mathcal{P}} n_j}^{2 \al - 3 - \eps}$ by 1, which is the only difference we need for proving \eqref{count_quin+}.
\end{proof}

\subsection{Extended septic counting estimate}

In this subsection, we show an extended septic counting estimate, which is relevant to the extended septic stochastic terms in Subsection~\ref{SUB:reg} below. Recalling Definition \ref{DEF:pair}, we state and prove the following lemma, which is an extended version of \cite[Lemma 3.12]{LTzW}.

\begin{lemma}
\label{LEM:count_sep}
Let $1 < \al \leq \frac 98$. For $k \geq 3$, we define
\begin{align*}
\Ld^k (n_1, \dots, n_k) \deff \sum_{\eps_{\textup{sum}}, \eps_1, \dots, \eps_k \in \{\pm 1\}} \sum_{\zeta \in \Z} \jb{\zeta}^{-1}
\bigg\llbracket \sum_{j = 1}^k n_j \bigg\rrbracket^{-1} \Big( \prod_{j = 1}^k \jbb{n_j}^{-1} \Big) \ind_{\{|\kappa_k - \zeta| \leq 1\}},
\end{align*}

\noi
where
\begin{align*}
\kappa_k \deff \eps_{\textup{sum}} \bigg\llbracket \sum_{j = 1}^k n_j \bigg\rrbracket + \sum_{j = 1}^k \eps_j \jbb{n_j}.
\end{align*}

\noi
Let $k_1 \geq 3$, $k_2 \geq 3$, and $k_3 \geq 1$ be integers and denote $k_{12} = k_1 + k_2$, $k_{23} = k_2 + k_3$, and $k_{123} = k_1 + k_2 + k_3$. Let $N_1, \dots, N_{k_{123}} \geq 1$ be dyadic numbers. Furthermore, we let $\mathcal{P}$ be a pairing of $\{1, \dots, k_{123}\}$ which respects the partition $\mathcal{K}_1 = \{1, \dots, k_1\}$, $\mathcal{K}_2 = \{k_1 + 1, \dots, k_{12}\}$, $\mathcal{K}_3 = \{k_{12} + 1, \dots, k_{123}\}$.
Then, we have
\begin{align}
\begin{split}
\sum_{(n_j)_{j \notin \mathcal{P}}} &\bigg\langle \sum_{j \notin \mathcal{P}} n_j \bigg\rangle^{-1} \bigg( \sum_{(n_j)_{j \in \mathcal{P}}} \ind_{\substack{(n_1, \dots, n_{k_{123}}) \\ \textup{admissible}}} \Big( \prod_{j = 1}^{k_{123}} \ind_{\{|n_j| \sim N_j\}} \Big) \Ld^{k_1} (n_1, \dots, n_{k_1}) \\
&\times \Ld^{k_2} (n_{k_1 + 1}, \dots, n_{k_{12}}) \prod_{j \in \mathcal{K}_3} \jbb{n_j}^{-1}  \bigg)^2 \les N_{\textup{max}}^{(k_{123} - 7)(3 - 2 \al) - \ta}
\end{split}
\label{count_sep}
\end{align}

\noi
for some $\ta > 0$, where $N_{\textup{max}} = \max \{ N_1, \dots, N_{k_{123}} \}$.
\end{lemma}

\begin{proof}
By the Cauchy-Schwarz inequality in $\zeta$, H\"older's inequality in $\zeta$, and Lemma \ref{LEM:count_sum} with $s = \al - 1$ and $k = k_1$, we know that 
\begin{align}
\begin{split}
&\sum_{n_1, \dots, n_{k_1} \in \Z^3} \Big( \prod_{j \in \mathcal{K}_1} \ind_{\{|n_j| \sim N_j\}} \Big) \Ld^{k_1} (n_1, \dots, n_{k_1})^2 \\
&\leq \Big( \sum_{\zeta \in \Z} \jb{\zeta}^{-1} \ind_{\{ |\zeta| \les \max \{ N_1, \dots, N_{k_1} \}^\al \}} \Big) \sum_{n_1, \dots, n_{k_1} \in \Z^3} \Big( \prod_{j \in \mathcal{K}_1} \ind_{\{|n_j| \sim N_j\}} \Big) \sum_{\zeta \in \Z} \jb{\zeta}^{-1}  \\
&\quad \times \bigg( \sum_{\eps_{\textup{sum}}, \eps_1, \dots, \eps_{k_1} \in \{\pm 1\}}
\bigg\llbracket \sum_{j = 1}^{k_1} n_j \bigg\rrbracket^{-1} \Big( \prod_{j = 1}^{k_1} \jbb{n_j}^{-1} \Big) \ind_{\{|\kappa_{k_1} - \zeta| \leq 1\}} \bigg)^2 \\
&\les \max \{N_1, \dots, N_{k_1}\}^{(k_1 - 3)(3 - 2 \al) - 3\ta} 
\end{split}
\label{count_sep1}
\end{align}

\noi
for some $\ta > 0$, where we used the fact that $\alpha > 1$. Similarly, we have
\begin{align}
\begin{split}
&\sum_{n_{k_1 + 1}, \dots, n_{k_{12}} \in \Z^3} \Big( \prod_{j \in \mathcal{K}_2} \ind_{\{|n_j| \sim N_j\}} \Big) \Ld^{k_2} (n_{k_1 + 1}, \dots, n_{k_{12}})^2 \\
&\les \max \{N_{k_1 + 1}, \dots, N_{k_{12}}\}^{(k_2 - 3)(3 - 2 \al) - 3\ta}
\end{split}
\label{count_sep2}
\end{align}

\noi
for some $\ta > 0$.

We now consider the following cases.

\smallskip \noi
\textbf{Case 1:} None of the indices in $\mathcal{K}_3$ is paired.

We can assume without loss of generality that $N_{k_{123}} = \max \{ N_{k_{12} + 1}, \dots, N_{k_{123}} \}$. By the Cauchy-Schwarz inequality in $(n_j)_{j \in \mathcal{P}}$, Lemma~\ref{LEM:conv}~(i) for the summation in $n_{k_{123}}$ (along with $\al > 1$), \eqref{count_sep1}, and \eqref{count_sep2}, we obtain that the left-hand-side of \eqref{count_sep} is bounded by
\begin{align*}
&\sum_{(n_j)_{j \notin \mathcal{P}}} \bigg\langle \sum_{j \notin \mathcal{P}} n_j \bigg\rangle^{-1} \Big( \prod_{j \in \mathcal{K}_3} \ind_{\{|n_j| \sim N_j\}} \jbb{n_j}^{-2} \Big) \\
&\quad \times \bigg( \sum_{(n_j)_{j \in \mathcal{P} \wedge j \in \mathcal{K}_1}} \Big( \prod_{j \in \mathcal{K}_1} \ind_{\{|n_j| \sim N_j\}} \Big) \Ld^{k_1} (n_1, \dots, n_{k_1})^2 \bigg) \\
&\quad \times \bigg( \sum_{(n_j)_{j \in \mathcal{P} \wedge j \in \mathcal{K}_2}} \Big( \prod_{j \in \mathcal{K}_2} \ind_{\{|n_j| \sim N_j\}} \Big) \Ld^{k_2} (n_{k_1 + 1}, \dots, n_{k_{12}})^2 \bigg) \\
&\les N_{k_{123}}^{- \ta} N_{\text{max}}^{(k_3 - 1) (3 - 2 \al)} \bigg( \sum_{(n_j)_{j \in \mathcal{K}_1}} \Big( \prod_{j \in \mathcal{K}_1} \ind_{\{|n_j| \sim N_j\}} \Big) \Ld^{k_1} (n_1, \dots, n_{k_1})^2 \bigg) \\
&\quad \times \bigg( \sum_{(n_j)_{j \in \mathcal{K}_2}} \Big( \prod_{j \in \mathcal{K}_2} \ind_{\{|n_j| \sim N_j\}} \Big) \Ld^{k_2} (n_{k_1 + 1}, \dots, n_{k_{12}})^2 \bigg) \\
&\les N_{\text{max}}^{(k_{123} - 7) (3 - 2 \al) - \ta}, 
\end{align*}

\noi
which proves \eqref{count_sep}.

\smallskip \noi
\textbf{Case 2:} At least one index in $\mathcal{K}_3$ is paired.

By symmetry, we can assume without loss of generality that $\{1, k_{123}\} \in \mathcal{P}$. By using \eqref{count_quin2}, we have
\begin{align}
\begin{split}
\sum_{n_{1} \in \Z^3} &\ind_{\{ |n_{1}| \sim N_{1} \}} \jbb{n_{1}}^{-1} \Ld^{k_1} (n_1, \dots, n_{k_1}) \\
&\les N_{1}^{- \ta} N_{\text{max}}^{\frac{\ta}{2}} \Big\langle \sum_{j \in \mathcal{K}_1 \setminus \{1\}} n_j \Big\rangle^{- \al} \Big( \prod_{j \in \mathcal{K}_1 \setminus \{1\}} \jbb{n_j}^{-1} \Big)
\end{split}
\label{count_sep3}
\end{align}

\noi
for some $\ta > 0$ sufficiently small. We let $\mathcal{P}'$ be the pairing obtained by removing the pair $\{1, k_{123}\}$ from $\mathcal{P}$. 
When $k_3 = 1$, by \eqref{count_sep3}, the Cauchy-Schwarz inequality in $(n_j)_{j \in \mathcal{P}'}$, Lemma~\ref{LEM:conv}~(i) twice (along with $\al > 1$), and \eqref{count_sep2}, we have that the left-hand-side of \eqref{count_sep} is bounded by
\begin{align*}
&N_{1}^{-2 \ta} N_{\text{max}}^{\ta} \sum_{(n_j)_{j \notin \mathcal{P}}} \bigg( \sum_{(n_j)_{j \in \mathcal{P}'}} \ind_{\substack{(n_2, \dots, n_{k_{123} - 1}) \\ \textup{admissible}}} \Big( \prod_{j \notin \{1, k_{123}\}} \ind_{\{|n_j| \sim N_j\}} \Big)  \\
&\quad \times \Big\langle \sum_{j \in \mathcal{K}_1 \setminus \{1\}} n_j \Big\rangle^{- \al} \Big( \prod_{j \in \mathcal{K}_1 \setminus \{1\}} \jbb{n_j}^{-1} \Big) \Ld^{k_2} (n_{k_1 + 1}, \dots, n_{k_{12}}) \bigg)^2 \\
&\leq N_1^{- 2 \ta} N_{\text{max}}^\ta \sum_{(n_j)_{j \notin \mathcal{P}}} \bigg( \sum_{(n_j)_{j \in \mathcal{P}' \wedge j \in \mathcal{K}_1 \setminus \{1\}}} \Big\langle \sum_{j \in \mathcal{K}_1 \setminus \{1\}} n_j \Big\rangle^{- 2 \al} \Big( \prod_{j \in \mathcal{K}_1 \setminus \{1\}} \ind_{\{|n_j| \sim N_j\}} \jbb{n_j}^{-2} \Big) \bigg) \\
&\quad \times \bigg( \sum_{(n_j)_{j \in \mathcal{P}' \wedge j \in \mathcal{K}_2}} \Big( \prod_{j \in \mathcal{K}_2} \ind_{\{|n_j| \sim N_j\}} \Big) \Ld^{k_2} (n_{k_1 + 1}, \dots, n_{k_{12}})^2 \bigg) \\
&\les N_1^{-2 \ta} N_{\text{max}}^\ta \max \{ N_2, \dots, N_{k_1} \}^{(k_1 - 3) (3 - 2 \al) - 3 \ta} \max \{ N_{k_1 + 1}, \dots, N_{k_{12}} \}^{(k_2 - 3) (3 - 2 \al) - 3 \ta} \\
&\les N_{\text{max}}^{(k_{12} - 6) (3 - 2 \al) - \ta} ,
\end{align*}

\noi
which gives the desired estimate.

When $k_3 \geq 2$, we assume  $N_{k_{123} - 1} = \max \{ N_{k_{12} + 1}, \dots, N_{k_{123} - 1} \}$ without loss of generality. If $k_{123} - 1 \in \mathcal{P}'$, then we have $N_{k_{123} - 1} \leq \max \{ N_2, \dots, N_{k_{12}} \}$, so that by \eqref{count_sep3}, the Cauchy-Schwarz inequality in $(n_j)_{j \in \mathcal{P}'}$,
Lemma~\ref{LEM:conv}~(i) twice (along with $\al > 1$), \eqref{count_sep2}, and the fact that $N_{k_{123} - 1} \leq \max \{ N_2, \dots, N_{k_{12}} \}$, we have that the left-hand-side of \eqref{count_sep} is bounded by
\begin{align*}
&N_{1}^{-2 \ta} N_{\text{max}}^{\ta} \sum_{(n_j)_{j \notin \mathcal{P}}} \bigg( \sum_{(n_j)_{j \in \mathcal{P}'}} \ind_{\substack{(n_1, \dots, n_{k_{123}}) \\ \textup{admissible}}} \Big( \prod_{j \notin \{1, k_{123}\}} \ind_{\{|n_j| \sim N_j\}} \Big)  \\
&\quad \times \Big\langle \sum_{j \in \mathcal{K}_1 \setminus \{1\}} n_j \Big\rangle^{- \al} \Big( \prod_{j \in \mathcal{K}_1 \setminus \{1\}} \jbb{n_j}^{-1} \Big) \Big( \prod_{j \in \mathcal{K}_3 \setminus \{k_{123}\}} \jbb{n_j}^{-1} \Big) \Ld^{k_2} (n_{k_1 + 1}, \dots, n_{k_{12}}) \bigg)^2 \\
&\les N_{1}^{-2 \ta} N_{\text{max}}^{\ta} \sum_{(n_j)_{j \notin \mathcal{P}}} \bigg( \sum_{(n_j)_{j \in \mathcal{P}' \wedge j \in \mathcal{K}_1 \setminus \{1\}}} \Big\langle \sum_{j \in \mathcal{K}_1 \setminus \{1\}} n_j \Big\rangle^{- 2 \al} \Big( \prod_{j \in \mathcal{K}_1 \setminus \{1\}} \ind_{\{|n_j| \sim N_j\}} \jbb{n_j}^{-2} \Big) \bigg) \\
&\quad \times \bigg( \sum_{(n_j)_{j \in \mathcal{P}' \wedge j \in \mathcal{K}_2}} \Big( \prod_{j \in \mathcal{K}_2} \ind_{\{|n_j| \sim N_j\}} \Big) \Ld^{k_2} (n_{k_1 + 1}, \dots, n_{k_{12}})^2 \bigg) \\
&\quad \times \bigg( \sum_{(n_j)_{j \in \mathcal{P}' \wedge j \in \mathcal{K}_3 \setminus \{k_{123}\}}} \Big( \prod_{j \in \mathcal{K}_3 \setminus \{k_{123}\}} \ind_{\{|n_j| \sim N_j\}} \jbb{n_j}^{-2} \Big) \bigg) \\
&\les N_{1}^{-2 \ta} N_{\text{max}}^{\ta} \max\{ N_2, \dots, N_{k_1} \}^{(k_1 - 3) (3 - 2 \al) - 3 \ta} \\
&\quad \times \max \{ N_{k_1 + 1}, \dots, N_{k_{12}} \}^{(k_2 - 3) (3 - 2 \al) - 3 \ta} N_{k_{123} - 1}^{(k_3 - 1) (3 - 2\al)} \\
&\les N_{\text{max}}^{(k_{123} - 7) (3 - 2 \al) - \ta},
\end{align*}

\noi
which gives the desired estimate. If $k_{123} - 1 \notin \mathcal{P}'$, we need to use the term $\jb{\sum_{j \notin \mathcal{P}} n_j}^{-1}$ and Lemma~\ref{LEM:conv}~(i) to obtain
\begin{align}
\sum_{n_{k_{123} - 1} \in \Z^3} \bigg\langle \sum_{j \notin \mathcal{P}} n_j \bigg\rangle^{-1} \jbb{n_{k_{123} - 1}}^{-2} \ind_{\{ |n_{k_{123} - 1}| \sim N_{k_{123} - 1} \}} \les N_{k_{123} - 1}^{- 2\ta} .
\label{count_sep4}
\end{align}

\noi
Thus, by \eqref{count_sep3}, \eqref{count_sep4}, the Cauchy-Schwarz inequality in $(n_j)_{j \in \mathcal{P}'}$,
Lemma~\ref{LEM:conv}~(i) twice (along with $\al > 1$), and \eqref{count_sep2}, we have that the left-hand-side of \eqref{count_sep} is bounded by
\begin{align*}
&N_{1}^{-2 \ta} N_{\text{max}}^{\ta} N_{k_{123} - 1}^{- 2 \ta} \sum_{(n_j)_{j \notin \mathcal{P}}} \bigg( \sum_{(n_j)_{j \in \mathcal{P}'}} \ind_{\substack{(n_1, \dots, n_{k_{123}}) \\ \textup{admissible}}} \Big( \prod_{j \notin \{1, k_{123} - 1, k_{123}\}} \ind_{\{|n_j| \sim N_j\}} \Big)  \\
&\quad \times \Big\langle \sum_{j \in \mathcal{K}_1 \setminus \{1\}} n_j \Big\rangle^{- \al} \Big( \prod_{j \in \mathcal{K}_1 \setminus \{1\}} \jbb{n_j}^{-1} \Big) \Big( \prod_{j \in \mathcal{K}_3 \setminus \{k_{123} - 1, k_{123}\}} \jbb{n_j}^{-1} \Big) \Ld^{k_2} (n_{k_1 + 1}, \dots, n_{k_{12}}) \bigg)^2 \\
&\les N_{1}^{-2 \ta} N_{\text{max}}^{\ta} N_{k_{123} - 1}^{- 2 \ta} \sum_{(n_j)_{j \notin \mathcal{P}}} \bigg( \sum_{(n_j)_{j \in \mathcal{P}' \wedge j \in \mathcal{K}_1 \setminus \{1\}}} \Big\langle \sum_{j \in \mathcal{K}_1 \setminus \{1\}} n_j \Big\rangle^{- 2 \al} \Big( \prod_{j \in \mathcal{K}_1 \setminus \{1\}} \ind_{\{|n_j| \sim N_j\}} \jbb{n_j}^{-2} \Big) \bigg) \\
&\quad \times \bigg( \sum_{(n_j)_{j \in \mathcal{P}' \wedge j \in \mathcal{K}_2}} \Big( \prod_{j \in \mathcal{K}_2} \ind_{\{|n_j| \sim N_j\}} \Big) \Ld^{k_2} (n_{k_1 + 1}, \dots, n_{k_{12}})^2 \bigg) \\
&\quad \times \bigg( \sum_{(n_j)_{j \in \mathcal{P}' \wedge j \in \mathcal{K}_3 \setminus \{k_{123} - 1, k_{123}\}}} \Big( \prod_{j \in \mathcal{K}_3 \setminus \{k_{123} - 1, k_{123}\}} \ind_{\{|n_j| \sim N_j\}} \jbb{n_j}^{-2} \Big) \bigg) \\
&\les N_{1}^{-2 \ta} N_{\text{max}}^{\ta} N_{k_{123} - 1}^{- 2 \ta} \max\{ N_2, \dots, N_{k_1} \}^{(k_1 - 3) (3 - 2 \al) - 3 \ta} \\
&\quad \times \max \{ N_{k_1 + 1}, \dots, N_{k_{12}} \}^{(k_2 - 3) (3 - 2 \al) - 3 \ta} N_{k_{123} - 1}^{(k_3 - 2) (3 - 2\al)} \\
&\les N_{\text{max}}^{(k_{123} - 7) (3 - 2 \al) - \ta} ,
\end{align*}

\noi
which gives the desired estimate.
\end{proof}

\section{Convergence of the Gibbs measures}
\label{SEC:meas}
In this section, we prove Theorem \ref{THM:meas} for the $L^1$-boundedness of density of the truncated Gibbs measure $\varrho_N$ in \eqref{nuN} and its convergence to the fractional $\Phi^4_3$-measure $\rho$ in \eqref{rho} as $N \to \infty$. 

\subsection{Main setup and estimates}
\label{SUB:meas}
Let us first do some setup steps for proving the convergence of $\varrho_N$. As mentioned in Section \ref{SEC:intro}, we need to introduce a further renormalization for the potential energy $\mathcal{R}_N$ in \eqref{defRN}, which we explain in the following.

By the Bou\'e-Dupuis variational formula (Lemma \ref{LEM:BD}), we have
\begin{align*}
- \log \big\| e^{- \mathcal{R}_N (u)} \big\|_{L^1 (\mu_\al)} = \inf_{\ta \in \mathbb{H}_a} \E \bigg[ \mathcal{R}_N (W_N + \Dr_N) + \frac 12 \int_0^1 \| \ta (t) \|_{L^2_x}^2 dt \bigg],
\end{align*}

\noi
where $W_N = W_N (1) = \pi_N W (1)$ with $W$ being a Gaussian process as defined in \eqref{defWt} and $\Dr_N = \Dr_N (1) = \pi_N I (\ta) (1)$ with $I$ being defined in \eqref{defIta}. Note that
\begin{align}
\mathcal{R}_N (W_N + \Dr_N) = \int_{\T^3} \sum_{j = 1}^m \cj{a}_{j, N} N^{- (2j - 4) (\frac 32 - \al)} \sum_{k = 0}^{2j} \binom{2j}{k} \wick{W_N^k} \Dr_N^{2j - k} dx.
\label{RN_exp}
\end{align}

\noi
Let us consider the following term in $\mathcal{R}_N (W_N + \Dr_N)$:
\begin{align}
\int_{\T^3} \bigg( \sum_{j = 2}^m 2j \cj{a}_{j, N} N^{- (2j - 4) (\frac 32 - \al)} \wick{W_N^{2j - 1}} \bigg) \Dr_N dx.
\label{problem}
\end{align}

\noi
From Lemma~\ref{LEM:Yreg}~(i), we know that $\sum_j 2j \cj{a}_{j, N} N^{- (2j - 4) (\frac 32 - \al)} \wick{W_N^{2j - 1}}$ has regularity $3 (\al - \frac 32) - \eps$ with $\eps > 0$ arbitrarily small, uniformly in $N \in \N$. Also, the drift term $\Dr_N$ has regularity $\al$ uniformly in $N \in \N$. Thus, by standard product analysis in Besov spaces, the product of these two terms converges as $N \to \infty$ only when 
\begin{align*}
3 \Big( \al - \frac 32 \Big) - \eps + \al > 0,
\end{align*}

\noi
which gives $\al > \frac 98$. Since we assumed $1 < \al \leq \frac 98$, one cannot make sense of the limit of \eqref{problem} as $N \to \infty$ in a classical manner.

To overcome this issue, we define $\ZZ^N$ with $\ZZ^N (0) = 0$ by its time derivative:
\begin{align}
\dot{\ZZ}^N (t) \deff \sum_{j = 2}^m 2j \cj{a}_{j, N} N^{- (2j - 4) (\frac 32 - \al)} \jb{\nb}^{-2 \al} \wick{W_N^{2j - 1} (t)},
\label{defZZ}
\end{align}

\noi
and we set $\ZZ_N = \ZZ_N (1) = \pi_N \ZZ^N (1)$. We now perform a change of variable
\begin{align}
\dot{\Ups}^N (t) \deff \dot{\Dr} (t) + \dot{\ZZ}_N (t)
\label{defUps}
\end{align}

\noi
and we set $\Ups_N = \pi_N \Ups^N = \pi_N \Ups^N (1)$. By the Ito product formula along with the fact that $\Dr_N (t)$ is a finite variation process, \eqref{defZZ}, and \eqref{defUps}, we have
\begin{align}
\begin{split}
&\E \bigg[ \int_{\T^3} \bigg( \sum_{j = 2}^m 2j \cj{a}_{j, N} N^{- (2j - 4) (\frac 32 - \al)} \wick{W_N^{2j - 1}} \bigg) \Dr_N dx + \frac 12 \int_0^1 \| \ta (t) \|_{L_x^2}^2 dt \bigg] \\
&= \E \bigg[ \int_0^1 \int_{\T^3} \bigg( \sum_{j = 2}^m 2j \cj{a}_{j, N} N^{- (2j - 4) (\frac 32 - \al)} \wick{W_N^{2j - 1} (t)} \bigg) \dot{\Dr}_N (t) dx dt + \frac 12 \int_0^1 \| \ta (t) \|_{L_x^2}^2 dt \bigg] \\
&= \frac 12 \E \bigg[ \int_0^1 \big\| \dot{\Ups}^N (t) \big\|_{H_x^\al}^2 dt \bigg] - \al_N,
\end{split}
\label{RN_div}
\end{align}

\noi
where the constant $\al_N$ is defined by
\begin{align}
\al_N \deff \frac 12 \E \bigg[ \int_0^1 \big\| \dot{\ZZ}_N (t) \big\|_{H_x^\al}^2 dt \bigg].
\label{alphaN}
\end{align}

\noi
In view of the spatial regularity $5\al - \frac 92 - \eps$ of $\dot{\ZZ}_N (t)$ (Lemma \ref{LEM:Yreg} (i)) and the fact that $5\al - \frac 92 - \eps < \al$ for $1 < \al \leq \frac 98$, we see that $\al_N \to \infty$ as $N \to \infty$. Thus, we consider the renormalized potential energy $R_N^\dia$ as defined in \eqref{defRNd}. As mentioned in Remark~\ref{RMK:alN}, this extra renormalization appears only at the level of the measure and will not change its invariance property as stated in Lemma \ref{LEM:invN} below.

Let us now replace $\mathcal{R}_N$ by the renormalized potential energy $\mathcal{R}_N^\dia$. By the Bou\'e-Dupuis variational formula (Lemma \ref{LEM:BD}), \eqref{RN_exp}, the change of variable \eqref{defUps}, \eqref{RN_div}, and Lemma~\ref{LEM:Yreg}~(iii), we obtain
\begin{align}
\begin{split}
- \log \big\| e^{- \mathcal{R}_N^\dia (u)} \big\|_{L^1 (\mu_\al)} &= \inf_{\theta \in \mathbb{H}_a} \E \bigg[ \mathcal{R}_N (W_N + \Dr_N) + \al_N + \frac 12 \int_0^1 \| \ta (t) \|_{L_x^2}^2 dt \bigg] \\
&= \inf_{\dot{\Ups}^N \in \mathbb{H}_a^\al} \E \bigg[ \wt{\mathcal{R}}_N^\dia (W_N, \Ups_N - \ZZ_N) + \frac 12 \int_0^1 \big\| \dot{\Ups}^N (t) \big\|_{H_x^\al}^2 dt \bigg],
\end{split}
\label{BD_RNtil}
\end{align}

\noi
where
\begin{align}
\begin{split}
\wt{\mathcal{R}}_N^\dia (W_N, \Ups_N - \ZZ_N) &= \int_{\T^3} \sum_{j = 2}^m \cj{a}_{j, N} N^{- (2j - 4) (\frac 32 - \al)} \sum_{k = 0}^{2j - 2} \binom{2j}{k} \wick{W_N^k} \Dr_N^{2j - k} dx \\
&\quad + \int_{\T^3} \cj{a}_{1, N} N^{2 (\frac 32 - \al)} \Dr_N^2 dx + 2 \int_{\T^3} \cj{a}_{1, N} N^{2 (\frac 32 - \al)} W_N \Dr_N dx
\end{split}
\label{RNtil}
\end{align}

\noi
with $\Dr_N = \Ups_N - \ZZ_N$.

\medskip
We now show some estimates that deal with each term in $\wt{\mathcal{R}}_N^\dia$, starting with the case when $j \geq 2$ and there is no $\wick{W_N^k}$ term.
\begin{lemma}
\label{LEM:RNbdd1}
Let $1 < \al \leq \frac 98$. Then, with $\Dr_N = \Ups_N - \ZZ_N$, there exist small enough $\eps, \eps', \eps'' > 0$ and constants $C, c > 0$ such that
\begin{align}
\begin{split}
\int_{\T^3} &\sum_{j = 2}^m \cj{a}_{j, N} N^{- (2j - 4) (\frac 32 - \al)} \Dr_N^{2j} dx \\
&\geq c  \sum_{j = 2}^m N^{- (2j - 4) (\frac 32 - \al)} \| \Ups_N \|_{L^{2j}}^{2j} - C \sum_{j = 2}^m N^{- (2j - 4) (\frac 32 - \al)} \| \ZZ_N \|_{\mathcal{C}^{5 \al - \frac 92 - \eps}}^{2j} .
\end{split}
\label{RNbdd1-1}
\end{align}

\noi
Also, for any $\dl > 0$, there exist small enough $\eps, \eps', \eps'' > 0$ and a constant $C > 0$ such that
\begin{align}
\begin{split}
\int_{\T^3} &\sum_{j = 3}^m \cj{a}_{j, N} N^{- (2j - 4) (\frac 32 - \al)} \Dr_N^{2j} dx \\
&\geq - \dl \sum_{j = 3}^m N^{- (2j - 4) (\frac 32 - \al) - \eps'} \| \Ups_N \|_{L^{2j}}^{2j} - C \sum_{j = 3}^m N^{- (2j - 4) (\frac 32 - \al) + \eps''} \| \ZZ_N \|_{\mathcal{C}^{5 \al - \frac 92 - \eps}}^{2j}.
\end{split}
\label{RNbdd1-1.5}
\end{align}

\noi
Moreover, there exist $\eps > 0$ and a constant $C > 0$ such that
\begin{align}
\begin{split}
\bigg| \int_{\T^3} &\sum_{j = 2}^m \cj{a}_{j, N} N^{- (2j - 4) (\frac 32 - \al)} \Dr_N^{2j} dx \bigg| \\
&\leq C \sum_{j = 2}^m N^{- (2j - 4) (\frac 32 - \al)} \| \Ups_N \|_{L^{2j}}^{2j} + C \sum_{j = 2}^m N^{- (2j - 4) (\frac 32 - \al)} \| \ZZ_N \|_{\mathcal{C}^{5 \al - \frac 92 - \eps}}^{2j} 
\end{split}
\label{RNbdd1-2}
\end{align}

\noi
uniformly in sufficiently large $N \in \N$.
\end{lemma}

\begin{proof}
We first consider \eqref{RNbdd1-1}. As in \cite[Proposition 3.5]{STzX}, by the positivity condition \eqref{pos0} and the convergence $\cj{a}_{j, N} \to \cj{a}_j$, we obtain
\begin{align}
\sum_{j = 2}^m \cj{a}_{j, N} N^{- (2j - 4) (\frac 32 - \al)} \Ups_N^{2j} \geq c_0 \sum_{j = 2}^m N^{- (2j - 4) (\frac 32 - \al)} \Ups_N^{2j} 
\label{RNbd1}
\end{align}

\noi
for some constant $c_0 > 0$ and sufficiently large $N \in \mathbb{N}$. We also observe that for any $1 \leq \l \leq 2j$, by H\"older's inequality, Young's inequality, and the embedding $\mathcal{C}^{5 \al - \frac 92 - \eps} \subset L^{2j}$ (Lemma~\ref{LEM:Bes}~(i)), we have
\begin{align}
\int_{\T^3} \Ups_N^{2j - \l} \ZZ_N^\l dx \leq \| \Ups_N \|_{L^{2j}}^{2j - \l} \| \ZZ_N \|_{L^{2j}}^\l \leq c_1 \| \Ups_N \|_{L^{2j}}^{2j} + C(c_1) \| \ZZ_N \|_{\mathcal{C}^{5 \al - \frac 92 - \eps}}^{2j}
\label{RNbd2}
\end{align}

\noi
for some $c_1 > 0$ arbitrarily small and some constant $C(c_1) > 0$. The estimate \eqref{RNbdd1-1} then follows directly from \eqref{RNbd1} and \eqref{RNbd2} by choosing $c_1 = c_1 (c_0)$ sufficiently small.

As for \eqref{RNbdd1-1.5}, we use the positivity condition \eqref{pos} to obtain directly
\begin{align*}
\sum_{j = 3}^m \cj{a}_{j, N} N^{- (2j - 4) (\frac 32 - \al)} \Ups_N^{2j} \geq 0
\end{align*}

\noi
for sufficiently large $N$. Then, similar to \eqref{RNbd2}, we obtain that for any $1 \leq \l \leq 2j$,
\begin{align*}
\int_{\T^3} \Ups_N^{2j - \l} \ZZ_N^\l dx 
\leq \dl N^{- \eps'} \| \Ups_N \|_{L^{2j}}^{2j} + C(\dl) N^{\eps''} \| \ZZ_N \|_{\mathcal{C}^{5 \al - \frac 92 - \eps}}^{2j}
\end{align*}

\noi
for some small enough $\eps', \eps'' > 0$. This proves \eqref{RNbdd1-1.5}. 

The proof for \eqref{RNbdd1-2} is similar and so we omit details.
\end{proof}

We now deal with the case when $j = 1$.

\begin{lemma}
\label{LEM:RNbdd1.5}
Let $1 < \al \leq \frac 98$ and $\eps, \dl > 0$ be small. Then, with $\Dr_N = \Ups_N - \ZZ_N$, there exists a constant $C > 0$ such that
\begin{align*}
\bigg| \int_{\T^3} \cj{a}_{1, N} N^{2 (\frac 32 - \al)} \Dr_N^2 dx \bigg| &\leq C + C \| \ZZ_N \|_{\mathcal{C}^{5 \al - \frac 92 - \eps}}^2 + \dl \| \Ups_N \|_{L^4}^4, \\
\bigg| \int_{\T^3} \cj{a}_{1, N} N^{2 (\frac 32 - \al)} W_N \Dr_N dx \bigg| &\leq C \| W_N \|_{\mathcal{C}^{\al - \frac 32 - \eps}}^2 + C \| \ZZ_N \|_{\mathcal{C}^{5 \al - \frac 92 - \eps}}^2 + \dl \| \Ups_N \|_{H^\al}^2
\end{align*}

\noi
uniformly in sufficiently large $N \in \N$.
\end{lemma}
\begin{proof}
The estimates follow easily from \eqref{a1N}, duality, Lemma~\ref{LEM:Bes}~(i), the fact that $(\al - \frac 32 - \eps) + (5 \al - \frac 92 - \eps) > 0$ and $(\al - \frac 32 - \eps) + \al > 0$ given $\al > 1$ and $\eps > 0$ sufficiently small, and the Cauchy-Schwarz inequalities.
\end{proof}

We now need to deal with the terms
\begin{align}
\int_{\T^3} \sum_{j = 2}^m N^{- (2j - 4) (\frac 32 - \al)} \wick{W_N^k} \Dr_N^{2j - k} dx
\label{RNmixed}
\end{align}

\noi
with $1 \leq k \leq 2j - 2$. This is the main task of following four lemmas.

\begin{lemma}
\label{LEM:RNbdd2}
Let $1 < \al \leq \frac 98$, $\eps, \dl > 0$ be small, and $j, k, k_1, k_2 \geq 0$ be integers satisfying $2 \leq j \leq m$, $1 \leq k \leq 2j - 2$, $k_2 \geq 1$, and $k + k_1 + k_2 = 2j$. Let $\eps' = 0$ if $j = 2$ and $\eps' > 0$ be sufficiently small if $j \geq 3$. 

\smallskip \noi
\textup{(i)} For $k_2 = 1, 2, 3$, there exist constants $C, c > 0$ such that
\begin{align*}
\bigg| &\int_{\T^3} \cj{a}_{j, N} N^{- (2j - 4) (\frac 32 - \al) + \eps'} \wick{W_N^k} \ZZ_N^{k_1} \Ups_N^{k_2} dx \bigg| \\
&\leq C \big( N^{- (2j - 4) (\frac 32 - \al) + \eps'} \big\| \wick{W_N^k} \ZZ_N^{k_1} \big\|_{\mathcal{C}^{2\al - 3 - \eps}} \big)^{c} + \dl \big( \| \Ups_N \|_{H^\al}^2 + \| \Ups_N \|_{L^4}^4 \big).
\end{align*}

\smallskip \noi
\textup{(ii)} For $k_2 \geq 4$, we let $k_2' = \lfloor \frac{k_2}{2} \rfloor + 1$, $s > 0$, and $\g \in \R$ be such that
\begin{align*}
\frac{2 k_2' s}{(2 k_2' - k_2) \al + s} < 2 \quad \text{and} \quad \g = \frac{(k_2' - 2) (k_2 \al - s) (\frac 32 - \al)}{k_2' \al}.
\end{align*}

\noi
Then, there exist constants $C, c > 0$ such that
\begin{align*}
\begin{split}
\bigg| &\int_{\T^3} \cj{a}_{j, N} N^{- (2j - 4) (\frac 32 - \al) + \eps'} \wick{W_N^k} \ZZ_N^{k_1} \Ups_N^{k_2} dx \bigg| \\
&\leq C \big( N^{- (2j - 4) (\frac 32 - \al) + \g + \eps'} \big\| \wick{W_N^k} \ZZ_N^{k_1} \big\|_{\mathcal{C}^{-s}} \big)^c + \dl \bigg( \| \Ups_N \|_{H^\al}^2 + \sum_{\l' = 2}^{k_2'} N^{- (2\l' - 4) (\frac 32 - \al)} \| \Ups_N \|_{L^{2 \l'}}^{2 \l'} \bigg).
\end{split}
\end{align*}
\end{lemma}

\begin{proof}
For the proof of part (i), see \cite[Lemma~4.5]{LTzW}. For the proof of part (ii), see \cite[Lemma~4.6]{LTzW}. Note that the proofs are valid for the range $1 < \al \leq \frac 98$.
\end{proof}

We now deal with the stochastic terms appearing in the statement of Lemma~\ref{LEM:RNbdd2}.
\begin{lemma}
\label{LEM:RNsto1}
Let $1 < \al \leq \frac 98$, $\eps > 0$ be small, $1 \leq p < \infty$, and $j, k, k_1, k_2 \geq 0$ be integers satisfying $2 \leq j \leq m$, $1 \leq k \leq 2j - 2$, $k_2 \geq 1$, and $k + k_2 \leq 2 j$. Let $\eps' = 0$ if $j = 2$ and $\eps' > 0$ be sufficiently small if $j \geq 3$. 

\smallskip \noi
\textup{(i)} For $k_2 = 1, 2, 3$, we have
\begin{align*}
\sup_{N \geq 1} \E \Big[ \big( N^{- (2j - 4) (\frac 32 - \al) + \eps'} \big\| \wick{W_N^k} \ZZ_N^{2j - k - k_2} \big\|_{\mathcal{C}^{2 \al - 3 - \eps}} \big)^p \Big] < C
\end{align*}

\noi
for some constant $C > 0$.

\smallskip \noi
\textup{(ii)} For $k_2 \geq 4$, we let $k_2' = \lfloor \frac{k_2}{2} \rfloor + 1$. Then, there exist $s > 0$ and $\g \in \R$ satisfying
\begin{align}
\frac{2 k_2' s}{(2 k_2' - k_2) \al + s} < 2 \quad \text{and} \quad \g = \frac{(k_2' - 2) (k_2 \al - s) (\frac 32 - \al)}{k_2' \al}
\label{cond0}
\end{align}

\noi
such that 
\begin{align*}
\sup_{N \geq 1} \E \Big[ \big( N^{- (2j - 4) (\frac 32 - \al) + \g + \eps'} \big\| \wick{W_N^k} \ZZ_N^{2j - k - k_2} \big\|_{\mathcal{C}^{-s}} \big)^p \Big] < C
\end{align*}

\noi
for some constant $C > 0$.
\end{lemma}
\begin{proof}
(i) When $k = 1$, by the product estimate (Lemma~\ref{LEM:Bes}~(iv)), the fact that $- \al + \frac 32 + \eps < 5 \al - \frac 92 - \eps$ given $\al > 1$ and $\eps > 0$ sufficiently small, and Lemma~\ref{LEM:Yreg}~(i), we have
\begin{align*}
\big\| W_N \ZZ_N^{2j - 1 - k_2} \big\|_{\mathcal{C}^{\al - \frac 32 - \eps}} \les \| W_N \|_{\mathcal{C}^{\al - \frac 32 - \eps}} \| \ZZ_N \|_{\mathcal{C}^{5 \al - \frac 92 - \eps}}^{2j - 1 - k_2} \leq C (\o),
\end{align*}

\noi
where $C(\o) > 0$ has finite $p$th moment. Since $2 \al - 3 - \eps < \al - \frac 32 - \eps$ given $1 < \al \leq \frac 98$, this yields the desired estimate in this case. 

We now consider the case when $k \geq 2$. If $k_2 = 2$ or $k_2 = 3$, by the product estimate (Lemma~\ref{LEM:Bes}~(iv)), Bernstein's inequality, the fact that $- \al + \frac 32 + 3 \eps < 5 \al - \frac 92 - \eps$, and Lemma~\ref{LEM:Yreg}~(i), we have
\begin{align*}
&N^{- (2j - 4) (\frac 32 - \al) + \eps'} \big\| \wick{W_N^k} \ZZ_N^{2j - k - k_2} \big\|_{\mathcal{C}^{2 \al - 3 - \eps}} \\
&\quad \les N^{- (2j - 4) (\frac 32 - \al) + \eps'} \| \wick{W_N^k} \|_{\mathcal{C}^{2 \al - 3 - \eps}} \| \ZZ_N \|_{\mathcal{C}^{- 2 \al + 3 + 2 \eps}}^{2j - k - k_2} \\
&\quad \les N^{- (2j - 4) (\frac 32 - \al) + \eps' + (k - 2) (\frac 32 - \al - \frac{\eps}{k - 1})} \big( N^{- (k - 2) (\frac 32 - \al - \frac{\eps}{k - 1})} \| \wick{W_N^k} \|_{\mathcal{C}^{2 \al - 3 - \eps}} \big) \\
&\quad \quad \times N^{(2j - k - k_2) (\frac 32 - \al - \eps)} \| \ZZ_N \|_{\mathcal{C}^{- \al + \frac 32 + 3 \eps}}^{2j - k - k_2} \\
&\quad \les C(\o)
\end{align*}

\noi
with $C(\o) > 0$ having finite $p$th moment, where we used
\begin{align*}
- (2j - 4) \Big( \frac 32 - \al \Big) + \eps' + (k - 2) \Big( \frac 32 - \al - \frac{\eps}{k - 1} \Big) + (2j - k - k_2) \Big( \frac 32 - \al - \eps \Big) \leq 0
\end{align*}

\noi
given $\eps' > 0$ sufficiently small. This yields the desired estimate in this case.

It remains to consider the case when $k \geq 2$ and $k_2 = 1$. Since $k \leq 2j - 2$, we must have $2j - k - 1 \geq 1$. Let us write
\begin{align*}
\wick{W_N^k} \ZZ_N = \big( \wick{W_N^k} \ZZ_N - \E [\wick{W_N^k} \ZZ_N] \big) + \E [\wick{W_N^k} \ZZ_N].
\end{align*}

\noi
For the $\wick{W_N^k} \ZZ_N - \E [\wick{W_N^k} \ZZ_N]$ term, by using similar steps as in previous cases along with Lemma~\ref{LEM:RNstoh}~(i) below, we obtain
\begin{align*}
&N^{- (2j - 4) (\frac 32 - \al) + \eps'} \big\| \big( \wick{W_N^k} \ZZ_N - \E [\wick{W_N^k} \ZZ_N] \big) \ZZ_N^{2j - k - 2} \big\|_{\mathcal{C}^{2 \al - 3 - \eps}} \\
&\quad \les N^{- (2j - 4) (\frac 32 - \al) + \eps'} \big\| \wick{W_N^k} \ZZ_N - \E [\wick{W_N^k} \ZZ_N] \big\|_{\mathcal{C}^{2 \al - 3 - \eps}} \| \ZZ_N \|_{\mathcal{C}^{- 2 \al + 3 + 2 \eps}}^{2j - k - 2} \\
&\quad \les N^{- (2j - 4) (\frac 32 - \al) + \eps' + (k - 2) (\frac 32 - \al - \frac{\eps}{k - 1})} \cdot N^{(2j - k - 2) (\frac 32 - \al - \eps)} \| \ZZ_N \|_{\mathcal{C}^{- \al + \frac 32 + 3 \eps}}^{2j - k - 2} \\
&\quad \quad  \times N^{- (k - 2) (\frac 32 - \al - \frac{\eps}{k - 1})} \big\| \wick{W_N^k} \ZZ_N - \E [\wick{W_N^k} \ZZ_N] \big\|_{\mathcal{C}^{2 \al - 3 - \eps}} \\
&\quad \les \| \ZZ_N \|_{\mathcal{C}^{5 \al - \frac 92 - \eps}}^{2j - k - 2}  \big( N^{- (k - 2) (\frac 32 - \al - \frac{\eps}{k - 1})} \big\| \wick{W_N^k} \ZZ_N - \E [\wick{W_N^k} \ZZ_N] \big\|_{\mathcal{C}^{2 \al - 3 - \eps}} \big) \\
&\quad \leq C(\o)
\end{align*}

\noi
with $C(\o) > 0$ having finite $p$th moment. For the $\E [\wick{W_N^k} \ZZ_N]$ term, we note from \eqref{defZZ}, the product formula in Lemma~\ref{LEM:sto_prod}, and Lemma~\ref{LEM:sto_ito}~(ii) that $\E [\wick{W_N^k} \ZZ_N] = 0$ unless $k$ is odd, so that we only need to consider $k \geq 3$. Since $k \leq 2j - 2$ and $2j - 2$ is even, we must have $j \geq 3$ and $k \leq 2j - 3$. Thus, by Lemma~\ref{LEM:RNstoh}~(ii) below, the product estimate (Lemma~\ref{LEM:Bes}~(iv)), and Lemma~\ref{LEM:Yreg}~(i), we obtain
\begin{align*}
&\big| N^{- (2j - 4) (\frac 32 - \al) + \eps'} \E [\wick{W_N^k} \ZZ_N] \big| \cdot \big\| \ZZ_N^{2j - k - 2} \big\|_{\mathcal{C}^{2 \al - 3 - \eps}} \les \| \ZZ_N \|_{\mathcal{C}^{5 \al - \frac 92 - \eps}}^{2j - k - 2} \leq C(\o)
\end{align*}

\noi
with $C(\o) > 0$ having finite $p$th moment. This finishes the proof of part (i).

\smallskip \noi
(ii) We first note that the case $2j - k - k_2 = 0$ has been covered by \cite{LTzW} (see the proof of Proposition~4.3 after Lemma~4.6 in \cite{LTzW}). Let us assume that $2j - k - k_2 \geq 1$ and impose the condition $s < 5 \al - \frac 92 - \eps$. Then, by the product estimate (Lemma~\ref{LEM:Bes}~(iv)), we have
\begin{align*}
&N^{- (2j - 4) (\frac 32 - \al) + \g + \eps'} \big\| \wick{W_N^k} \ZZ_N^{2j - k - k_2} \big\|_{\mathcal{C}^{-s}} \\
&\quad \les N^{- (2j - 4) (\frac 32 - \al) + \g + \eps'} \| \wick{W_N^k} \|_{\mathcal{C}^{- s}} \| \ZZ_N \|_{\mathcal{C}^{5 \al - \frac 92 - \eps}}^{2j - k - k_2}.
\end{align*}

\noi
In order to apply Lemma~\ref{LEM:Yreg}~(i) for the $\wick{W_N^k}$ term to obtain the desired estimate, we require
\begin{align}
s > k \Big( \frac 32 - \al \Big) - (2j - 4) \Big( \frac 32 - \al \Big) + \g + \eps' > 0
\label{cond1}
\end{align}

\noi
and
\begin{align}
(2j - 4) \Big( \frac 32 - \al \Big) - \g > 0.
\label{cond2}
\end{align}

\noi
As long as the above conditions are satisfied, we obtain the desired moment bound by Lemma~\ref{LEM:Yreg}~(i).

We first consider the condition \eqref{cond1}. By plugging in the value of $\g$, from conditions \eqref{cond0} and \eqref{cond1} along with $s < 5 \al - \frac 92 - \eps$, we require 
\begin{align}
\frac{(k_2' (k - 2j + k_2 + 4) - 2 k_2) \al (\frac 32 - \al)}{\frac 32 k_2' - 2 (\frac 32 - \al)} < s < \min \Big\{ \frac{(2 k_2' - k_2) \al}{k_2' - 1}, 5 \al - \frac 92 - \eps \Big\}.
\label{cond3}
\end{align}

\noi
Since $k - 2j + k_2 < 0$, we can easily verify
\begin{align*}
\frac{(k_2' (k - 2j + k_2 + 4) - 2 k_2) \al (\frac 32 - \al)}{\frac 32 k_2' - 2 (\frac 32 - \al)} < \frac{(2 k_2' - k_2) \al}{k_2' - 1}
\end{align*}

\noi
given $\al > 1 > \frac 34$. Recalling $k_2' = \lfloor \frac{k_2}{2} \rfloor + 1$, we know that if $k_2 \geq 5$ is odd or $k_2 \geq 8$ is even, we have
\begin{align*}
\frac{(2 k_2' - k_2) \al}{k_2' - 1} \leq \frac{\al}{2} < 5 \al - \frac 92 - \eps
\end{align*}

\noi
given $\al > 1$ and $\eps > 0$ sufficiently small. Thus, to show that \eqref{cond3} can be satisfied, we only need to show
\begin{align*}
\frac{(k_2' (k - 2j + k_2 + 4) - 2 k_2) \al (\frac 32 - \al)}{\frac 32 k_2' - 2 (\frac 32 - \al)} < 5 \al - \frac 92 - \eps
\end{align*}

\noi
for $k_2 = 4$ or $k_2 = 6$, which can be checked easily due to $k - 2j + k_2 \leq - 1$ and $1 < \al \leq \frac 98$.

It remains to consider \eqref{cond2}, which is equivalent to
\begin{align*}
s > \frac{(4 k_2' - 2 k_2 - (2j - k_2) k_2' ) \al}{k_2' - 2}.
\end{align*}

\noi
This can be satisfied easily since
\begin{align*}
\frac{(4 k_2' - 2 k_2 - (2j - k_2) k_2' ) \al}{k_2' - 2} < 0
\end{align*}

\noi
in view of $2j - k_2 \geq k + 1 \geq 2$ and $k_2' = \lfloor \frac{k_2}{2} \rfloor + 1 \leq k_2 - 1 < k_2$. Thus, we have finished the proof.
\end{proof}

The following lemma was used in the proof of part~(i) in Lemma~\ref{LEM:RNsto1}.

\begin{lemma}
\label{LEM:RNstoh}
Let $1 < \al \leq \frac 98$, $\eps, \eps' > 0$ be small, $1 \leq p < \infty$, and $k \geq 2$.  

\smallskip \noi
\textup{(i)} We have
\begin{align*}
\sup_{N \geq 1} \E \Big[ \big( N^{- (k - 2) (\frac 32 - \al - \eps')} \big\| \wick{W_N^k} \ZZ_N - \E [\wick{W_N^k} \ZZ_N] \big\|_{\mathcal{C}^{2 \al - 3 - \eps}} \big)^p \Big] < C
\end{align*}

\noi
for some constant $C > 0$.

\smallskip \noi
\textup{(ii)} If $k \geq 3$ is odd, we have
\begin{align*}
\sup_{N \geq 1} N^{- (k - 1) (\frac 32 - \al) + \eps'} \big| \E [\wick{W_N^k} \ZZ_N] \big| < C
\end{align*}

\noi
for some constant $C > 0$.
\end{lemma}
\begin{proof}
(i) Recalling \eqref{defZZ}, we need to deal with the term
\begin{align}
\begin{split}
&N^{- (k - 2) (\frac 32 - \al - \eps') - (2 k' - 4) (\frac 32 - \al)} \\
&\quad \times \big( \wick{W_N^k} \jb{\nb}^{-2 \al} \wick{W_N^{2 k' - 1}} - \,\E [\wick{W_N^k} \jb{\nb}^{-2 \al} \wick{W_N^{2 k' - 1}}] \big)
\end{split}
\label{YYN}
\end{align}

\noi
for any $2 \leq k' \leq m$, where we ignore the time variable for simplicity. Let $n \in \Z^3$ be fixed, and we denote $n_1, \dots, n_k$ to be the frequencies of the nodes of $\wick{W_N^k}$ and $n_{k + 1}, \dots, n_{k + 2k' - 1}$ to be the frequencies of the nodes of $\wick{W_N^{2 k' - 1}}$. In view of the product formula in Lemma~\ref{LEM:sto_prod}, we need to discuss the number of pairs between $\wick{W_N^k}$ and $\wick{W_N^{2 k' - 1}}$. 

If there are no pairs, by the Ito isometry (Lemma~\ref{LEM:sto_ito}~(ii)), \eqref{jen}, and Lemma~\ref{LEM:conv}~(ii) and then (i) along with the fact that $1 < \al \leq \frac 98$, we the following contribution for the variance of \eqref{YYN} at frequency $n$:
\begin{align*}
&N^{- (k - 2) (3 - 2 \al - 2 \eps') - (2 k' - 4) (3 - 2 \al)} \\
&\quad \times \sum_{\substack{|n_1|, \dots, |n_{k + 2 k' - 1}| \leq N \\ n_1 + \cdots + n_{k + 2 k' - 1} = n}} \bigg( \prod_{\l = 1}^k \frac{1}{\jb{n_\l}^{2 \al}} \cdot \frac{1}{\jb{ \sum_{\l' = k + 1}^{k + 2 k' - 1} n_{\l'} }^{4 \al}} \prod_{\l' = k + 1}^{k + 2k' - 1} \frac{1}{\jb{n_{\l'}}^{2 \al}} \bigg) \\
&\les \sum_{\substack{n_1, \dots, n_{k + 2 k' - 1} \in \Z^3 \\ n_1 + \cdots + n_{k + 2 k' - 1} = n}} \bigg( \prod_{\l = 1}^k \frac{1}{\jb{n_\l}^{2 \al + \frac{k - 2}{k} (3 - 2 \al - 2 \eps')}} \cdot \frac{1}{\jb{ \sum_{\l' = k + 1}^{k + 2k' - 1} n_{\l'} }^{4 \al}} \\
&\quad \times \prod_{\l' = k + 1}^{k + 2 k' - 1} \frac{1}{\jb{n_{\l'}}^{2 \al + \frac{2 k' - 4}{2k' - 1} (3 - 2 \al)}} \bigg) \\
&\les \sum_{\substack{n', n'' \in \Z^3 \\ n' + n'' = n}} \frac{1}{\jb{n'}^{4 \al - 3 - 2 (k - 2) \eps'} \jb{n''}^{10 \al - 6}} \\
&\les \jb{n}^{3 - 4 \al + 2 (k - 2) \eps'},
\end{align*}

\noi
so that the desired estimate follows from Lemma~\ref{LEM:reg} given $\eps' > 0$ sufficiently small. 

If there are pairs, recalling Definition~\ref{DEF:pair}, we let $\mathcal P \neq \varnothing$ be a pairing of $\{ 1, \dots, k + 2 k' - 1 \}$ which respects the partition $\{1, \dots, k\}$ and $\{ k + 1, \dots, k + 2 k' - 1 \}$. We consider the following cases.

\smallskip \noi
\textbf{Case 1:} $|\mathcal P| = 1$. 

By the Ito isometry (Lemma~\ref{LEM:sto_ito}~(ii)), \eqref{jen}, and Lemma~\ref{LEM:conv}~(ii) and then (i) along with the fact that $1 < \al \leq \frac 98$, we obtain the following contribution for the variance of \eqref{YYN} at frequency $n$:
\begin{align*}
&\sum_{\substack{(n_\l)_{\l \notin \mathcal P} \\ \sum_{\l \notin \mathcal P} n_\l = n}}  \prod_{\substack{\l \in \{1, \dots, k\} \\ \l \notin \mathcal P}} \frac{1}{\jb{n_\l}^{2 \al + \frac{k - 2}{k - 1} (3 - 2 \al - \eps'')}} \prod_{\substack{\l' \in \{k + 1, \dots, k + 2k' - 1\} \\ \l' \notin \mathcal P}} \frac{1}{\jb{n_{\l'}}^{2 \al + \frac{2k' - 4}{2k' - 2} (3 - 2 \al - \eps'')}} \\
&\quad \times \bigg( \sum_{(n_\l)_{\l \in \mathcal P}} \ind_{\substack{(n_1, \dots, n_{k + 2k' - 1}) \\ \text{admissible}}} \frac{1}{\jb{\sum_{\l' = k + 1}^{k + 2k' - 1} n_{\l'}}^{2 \al}} \prod_{\l \in \mathcal P} \frac{1}{\jb{n_\l}^{\al}} \bigg)^2 \\
&\les \sum_{\substack{(n_\l)_{\l \notin \mathcal P} \\ \sum_{\l \notin \mathcal P} n_\l = n}}  \prod_{\substack{\l \in \{1, \dots, k\} \\ \l \notin \mathcal P}} \frac{1}{\jb{n_\l}^{2 \al + \frac{k - 2}{k - 1} (3 - 2 \al - \eps'')}} \prod_{\substack{\l' \in \{k + 1, \dots, k + 2k' - 1\} \\ \l' \notin \mathcal P}} \frac{1}{\jb{n_{\l'}}^{2 \al + \frac{2k' - 4}{2k' - 2} (3 - 2 \al - \eps'')}} \\
&\quad \times \frac{1}{\jb{\sum_{\l' = k + 1}^{k + 2k' - 1} \ind_{\{\l' \notin \mathcal P\}} n_{\l'}}^{8 \al - 6}} \\
&\les \sum_{\substack{n', n'' \in \Z^3 \\ n' + n'' = n}} \frac{1}{\jb{n'}^{2 \al - (k - 2) \eps''}\jb{n''}^{12 \al - 9 - (2k' - 4) \eps''}} \\
&\les \jb{n}^{- 2 \al + (k - 2) \eps''}
\end{align*}

\noi
for some small $\eps'' > 0$, which is enough by Lemma~\ref{LEM:reg}.

\smallskip \noi
\textbf{Case 2:} $2 \leq |\mathcal P| \leq k - 1$ and $2 \leq |\mathcal P| \leq 2k' - 1$.

By the Ito isometry (Lemma~\ref{LEM:sto_ito}~(ii)), \eqref{jen}, and Lemma~\ref{LEM:conv}~(ii) and then (i) along with the fact that $1 < \al \leq \frac 98$, we obtain the following contribution for the variance of \eqref{YYN} at frequency $n$:
\begin{align*}
&\sum_{\substack{(n_\l)_{\l \notin \mathcal P} \\ \sum_{\l \notin \mathcal P} n_\l = n}}  \prod_{\substack{\l \in \{1, \dots, k\} \\ \l \notin \mathcal P}} \frac{1}{\jb{n_\l}^{2 \al + \frac{k - 1 - |\mathcal P|}{k - |\mathcal P|} (3 - 2 \al - \eps'')}} \prod_{\substack{\l' \in \{k + 1, \dots, k + 2k' - 1\} \\ \l' \notin \mathcal P}} \frac{1}{\jb{n_{\l'}}^{3 - \eps''}} \\
&\quad \times \bigg( \sum_{(n_\l)_{\l \in \mathcal P}} \ind_{\substack{(n_1, \dots, n_{k + 2k' - 1}) \\ \text{admissible}}} \frac{1}{\jb{\sum_{\l' = k + 1}^{k + 2k' - 1} n_{\l'}}^{2 \al}} \prod_{\l \in \mathcal P} \frac{1}{\jb{n_\l}^{\al + \frac{|\mathcal P| - 2}{2 |\mathcal P|} (3 - 2 \al - \eps'')}} \bigg)^2 \\
&\les \sum_{\substack{n', n'' \in \Z^3 \\ n' + n'' = n}} \frac{1}{\jb{n'}^{2 \al - (k - 1 - |\mathcal P|) \eps''} \jb{n''}^{12 \al - 9 - (2 k' - 5 + |\mathcal P|) \eps''}} \\
&\les \jb{n}^{- 2 \al + (k - 1 - |\mathcal P|) \eps''}
\end{align*}

\noi
for some small $\eps'' > 0$, which is enough by Lemma~\ref{LEM:reg}. 

\smallskip \noi
\textbf{Case 3:} $|\mathcal P| = k$ and $2 \leq |\mathcal P| \leq 2k' - 2$.

By the Ito isometry (Lemma~\ref{LEM:sto_ito}~(ii)), \eqref{jen}, and Lemma~\ref{LEM:conv}~(ii) and then (i) along with the fact that $1 < \al \leq \frac 98$, we obtain the following contribution for the variance of \eqref{YYN} at frequency $n$:
\begin{align*}
&\sum_{\substack{(n_\l)_{\l \notin \mathcal P} \\ \sum_{\l \notin \mathcal P} n_\l = n}}   \prod_{\substack{\l' \in \{k + 1, \dots, k + 2k' - 1\} \\ \l' \notin \mathcal P}} \frac{1}{\jb{n_{\l'}}^{2 \al + \frac{2k' - 2 - |\mathcal P|}{2 k' - 1 - |\mathcal P|} (3 - 2 \al - \eps'')}} \\
&\quad \times \bigg( \sum_{(n_\l)_{\l \in \mathcal P}} \ind_{\substack{(n_1, \dots, n_{k + 2k' - 1}) \\ \text{admissible}}} \frac{1}{\jb{\sum_{\l' = k + 1}^{k + 2k' - 1} n_{\l'}}^{2 \al}} \prod_{\l \in \mathcal P} \frac{1}{\jb{n_\l}^{\al + \frac{|\mathcal P| - 2}{2 |\mathcal P|} (3 - 2 \al - \eps'')}} \bigg)^2 \\
&\les \frac{1}{\jb{n}^{2 \al - (2 k' - 2 - |\mathcal P|) \eps''} \jb{n}^{12 \al - 12 - (2 |\mathcal P| - 4) \eps''}} \\
&\les \jb{n}^{- 14 \al + 12 + (2k' - 6 - |\mathcal P|) \eps''}
\end{align*}

\noi
for some small $\eps'' > 0$, which is enough by Lemma~\ref{LEM:reg}.

\smallskip \noi
\textbf{Case 4:} $|\mathcal P| = k = 2k' - 1$.

In this case, all frequencies are paired, which is canceled by the expectation in \eqref{YYN}. Thus, we have finished the proof of part (i).

\smallskip \noi
(ii) From \eqref{defZZ} and the Ito isometry (Lemma~\ref{LEM:sto_ito}~(ii)), we only need to deal with the following non-zero term
\begin{align*}
N^{- (k - 1) (\frac 32 - \al) + \eps' - (k - 3) (\frac 32 - \al)} \E [\wick{W_N^k} \jb{\nb}^{-2 \al} \wick{W_N^{k}}].
\end{align*}

\noi
By Lemma~\ref{LEM:conv}~(ii) along with the fact that $1 < \al \leq \frac 98$, we obtain
\begin{align*}
&N^{- (k - 1) (\frac 32 - \al - \eps') - (k - 3) (\frac 32 - \al)} \E [\wick{W_N^k} \jb{\nb}^{-2 \al} \wick{W_N^{k}}] \\
&\quad \les \sum_{n_{1}, \dots, n_{k} \in \Z^3} \frac{1}{\jb{\sum_{\l = 1}^{k} n_\l}^{2 \al}} \prod_{\l = 1}^{k} \frac{1}{\jb{n_\l}^{2 \al + \frac{k - 2}{k} (3 - 2 \al - \eps'')}} \\
&\quad \les \sum_{n' \in \Z^3} \frac{1}{\jb{n'}^{6 \al - 3 - (k - 2) \eps''}} \\
&\quad \les 1
\end{align*}

\noi
for some $\eps'' > 0$ small. This finishes the proof.
\end{proof}

The following lemma deals with the remaining terms in \eqref{RNmixed} that have not been covered by Lemma~\ref{LEM:RNbdd2} and Lemma~\ref{LEM:RNsto1}.

\begin{lemma}
\label{LEM:RNsto2}
Let $1 < \al \leq \frac 98$, $\eps > 0$ be small, $2 \leq j \leq m$, and $1 \leq k \leq 2j - 2$. Let $\eps' = 0$ if $j = 2$ and $\eps' > 0$ be sufficiently small if $j \geq 3$. Then, we have
\begin{align*}
\sup_{N \geq 1} \bigg| \E \bigg[ N^{- (2j - 4) (\frac 32 - \al) + \eps'} \int_{\T^3} \wick{W_N^k} \ZZ_N^{2j - k} dx \bigg] \bigg| < C
\end{align*}

\noi
for some constant $C > 0$.
\end{lemma}
\begin{proof}
When $k = 1$, since $- \al + \frac 32 + \eps < 5 \al - \frac 92 - \eps$ given $\al > 1$ and $\eps > 0$ small enough, by duality, the embedding in Lemma~\ref{LEM:Bes}~(i), the product estimate in Lemma~\ref{LEM:Bes}~(iv), the Cauchy-Schwarz inequality, and Lemma~\ref{LEM:Yreg}~(i), we have
\begin{align*}
\bigg| \E \bigg[ \int_{\T^3} W_N \ZZ_N^{2j - 1} dx \bigg] \bigg| \les  \E \Big[ \| W_N \|_{\mathcal{C}^{\al - \frac 32 - \eps}} \| \ZZ_N \|_{\mathcal{C}^{5 \al - \frac 92 - \eps}}^{2j - 1}  \Big] \leq C
\end{align*}

\noi
for some constant $C > 0$. This yields the desired estimate in this case. 

We now consider the case when $k \geq 2$. If $2 \leq k \leq 2j - 3$, we use similar steps as in the previous case and Lemma~\ref{LEM:Yreg}~(i) to obtain
\begin{align*}
\bigg| &\E \bigg[ N^{- (2j - 4) (\frac 32 - \al) + \eps'} \int_{\T^3} \wick{W_N^k} \ZZ_N^{2j - k} dx \bigg] \bigg| \\
&\les \E \Big[ N^{- (2j - 4) (\frac 32 - \al) + \eps'} \| \wick{W_N^k} \|_{\mathcal{C}^{\al - \frac 32 - \eps}} \| \ZZ_N \|_{\mathcal{C}^{5 \al - \frac 92 - \eps}}^{2j - k} \Big] \\
&\les \E \Big[ \big( N^{- (2j - 4) (\frac 32 - \al) + \eps'} \| \wick{W_N^k} \|_{\mathcal{C}^{\al - \frac 32 - \eps}} \big)^2 \Big]^{\frac 12} \E \Big[ \| \ZZ_N \|_{\mathcal{C}^{5 \al - \frac 92 - \eps}}^{2 (2j - k)} \Big]^{\frac 12} \\
&\leq C
\end{align*}

\noi
for some constant $C > 0$, which is valid since
\begin{align*}
\frac 32 - \al + \eps > k \Big( \frac 32 - \al \Big) - (2j - 4) \Big( \frac 32 - \al \Big) + \eps'
\end{align*}

\noi
thanks to $k \leq 2j - 3$ and $\eps' > 0$ sufficiently small.

It remains to consider the case when $k = 2j - 2$. Recalling \eqref{defZZ}, we need to estimate
\begin{align}
\begin{split}
&N^{- (2j - 4) (\frac 32 - \al) + \eps' - (2 k_1 - 4) (\frac 32 - \al) - (2 k_2 - 4) (\frac 32 - \al)} \\ 
&\quad \times \E \bigg[ \int_{\T^3} \wick{W_N^{2 j - 2}} \jb{\nb}^{-2 \al} \wick{W_N^{2 k_1 - 1}} \jb{\nb}^{-2 \al} \wick{W_N^{2 k_2 - 1}} dx \bigg]
\end{split}
\label{YYYN}
\end{align}

\noi
for any $2 \leq k_1 \leq k_2 \leq m$, where we ignore the time variable for simplicity. We denote $\{ n_\l \}_{1 \leq \l \leq 2j - 2}$ to be the frequencies of the nodes of $\wick{W_N^{2 j - 2}}$, $\{ n_\l^{(1)} \}_{1 \leq \l \leq 2 k_1 - 1}$ to be the frequencies of the nodes of $\wick{W_N^{2 k_1 - 1}}$, and $\{ n_\l^{(2)} \}_{1 \leq \l \leq 2 k_2 - 1}$ to be the frequencies of the nodes of $\wick{W_N^{2 k_2 - 1}}$. Recalling Definition~\ref{DEF:pair}, we see that \eqref{YYYN} is non-zero only if all frequencies are paired, so that we must have $2j - 2 \geq (2 k_2 - 1) - (2 k_1 - 1)$ and $2j - 2 \leq (2 k_1 - 1) + (2 k_2 - 1)$, or equivalently $k_2 - k_1 + 1 \leq j \leq k_1 + k_2$.

\smallskip \noi
\textbf{Case 1:} $k_2 - k_1 + 2 \leq j \leq k_1 + k_2 - 2$.

In this case, the only possibility (up to permutation) is that $\{ n_\l \}_{1 \leq \l \leq j + k_1 - k_2 - 1}$  are paired with $\{ n_\l^{(1)} \}_{1 \leq \l \leq j + k_1 - k_2 - 1}$, $\{ n_\l \}_{j + k_1 - k_2 \leq \l \leq 2j - 2}$ are paired with $\{ n_\l^{(2)} \}_{1 \leq \l \leq j - k_1 + k_2 - 1}$, and $\{ n_\l^{(1)} \}_{j + k_1 - k_2 \leq \l \leq 2 k_1 - 1}$ are paired with $\{ n_\l^{(2)} \}_{j - k_1 + k_2 \leq \l \leq 2 k_2 - 1}$. Thus, by Lemma~\ref{LEM:conv} several times along with the fact that $1 < \al \leq \frac 98$, we obtain the following contribution for the bound of $|\eqref{YYYN}|$:
\begin{align*}
&\sum_{n_1, \dots, n_{2j - 2} \in \Z^3 } \sum_{n_{j + k_1 - k_2}^{(1)}, \dots, n_{2 k_1 - 1}^{(1)} \in \Z^3} \bigg( \prod_{\l = 1}^{j + k_1 - k_2 - 1} \frac{1}{\jb{n_\l}^{2 \al + \frac{j + k_1 - k_2 - 2}{j + k_1 - k_2 - 1} (3 - 2 \al - \eps'')}} \bigg) \\
&\quad \times \bigg( \prod_{\l = j + k_1 - k_2}^{2j - 2} \frac{1}{\jb{n_\l}^{2 \al + \frac{j - k_1 + k_2 - 2}{j - k_1 + k_2 - 1} (3 - 2 \al - \eps'')}} \bigg) \bigg( \prod_{\l' = j + k_1 - k_2}^{2 k_1 - 1} \frac{1}{\jb{n_{\l'}^{(1)}}^{2 \al + \frac{- j + k_1 + k_2 - 2}{-j + k_1 + k_2} (3 - 2 \al - \eps'')}} \bigg) \\
&\quad \times \frac{1}{\jb{\sum_{\l = 1}^{j + k_1 - k_2 - 1} n_\l + \sum_{\l' = j + k_1 - k_2}^{2 k_1 - 1} n_{\l'}^{(1)}}^{2 \al}  \jb{\sum_{\l = j + k_1 - k_2}^{2j - 2} n_\l + \sum_{\l' = j + k_1 - k_2}^{2 k_1 - 1} n_{\l'}^{(1)}}^{2 \al}} \\
&\les \sum_{n'_1, n'_2, n'_3 \in \Z^3} \frac{1}{\jb{n'_1}^{2 \al - (j + k_1 - k_2 - 2) \eps''} \jb{n'_2}^{2 \al - (j - k_1 + k_2 - 2) \eps''} \jb{n'_3}^{4 \al - 3 - (- j + k_1 + k_2 - 2)\eps''}} \\
&\quad \times \frac{1}{\jb{n'_1 + n'_3}^{2 \al} \jb{n'_2 + n'_3}^{2 \al}} \\
&\les \sum_{n'_3 \in \Z^3} \frac{1}{\jb{n'_3}^{8 \al - 6 - (2j - 4) \eps''}} \cdot \frac{1}{\jb{n'_3}^{4 \al - 3 - (-j + k_1 + k_2 - 2) \eps''}} \\
&\les 1
\end{align*}

\noi
for some $\eps'' > 0$ small, as desired.

\smallskip \noi
\textbf{Case 2:} $j = k_1 + k_2 - 1$.

In this case, we must have $j \geq k_2 - k_1 + 3$. Similar as in Case 1, we obtain the following contribution for the bound of $|\eqref{YYYN}|$: 
\begin{align*}
&\sum_{n_1, \dots, n_{2j - 2} \in \Z^3 } \sum_{n_{j + k_1 - k_2}^{(1)}, \dots, n_{2 k_1 - 1}^{(1)} \in \Z^3} \bigg( \prod_{\l = 1}^{j + k_1 - k_2 - 1} \frac{1}{\jb{n_\l}^{2 \al + \frac{j + k_1 - k_2 - 3}{j + k_1 - k_2 - 1} (3 - 2 \al - \eps'')}} \bigg) \\
&\quad \times \bigg( \prod_{\l = j + k_1 - k_2}^{2j - 2} \frac{1}{\jb{n_\l}^{2 \al + \frac{j - k_1 + k_2 - 2}{j - k_1 + k_2 - 1} (3 - 2 \al - \eps'')}} \bigg) \bigg( \prod_{\l' = j + k_1 - k_2}^{2 k_1 - 1} \frac{1}{\jb{n_{\l'}^{(1)}}^{2 \al + \frac{- j + k_1 + k_2 - 1}{-j + k_1 + k_2} (3 - 2 \al - \eps'')}} \bigg) \\
&\quad \times \frac{1}{\jb{\sum_{\l = 1}^{j + k_1 - k_2 - 1} n_\l + \sum_{\l' = j + k_1 - k_2}^{2 k_1 - 1} n_{\l'}^{(1)}}^{2 \al}  \jb{\sum_{\l = j + k_1 - k_2}^{2j - 2} n_\l + \sum_{\l' = j + k_1 - k_2}^{2 k_1 - 1} n_{\l'}^{(1)}}^{2 \al}} \\
&\les \sum_{n'_1, n'_2, n'_3 \in \Z^3} \frac{1}{\jb{n'_1}^{4 \al - 3 - (j + k_1 - k_2 - 3) \eps''} \jb{n'_2}^{2 \al - (j - k_1 + k_2 - 2) \eps''} \jb{n'_3}^{2 \al - (- j + k_1 + k_2 - 1)\eps''}} \\
&\quad \times \frac{1}{\jb{n'_1 + n'_3}^{2 \al} \jb{n'_2 + n'_3}^{2 \al}} \\
&\les \sum_{n_3' \in \Z^3} \frac{1}{\jb{n_3'}^{10 \al - 9 - (2j - 5) \eps''}} \cdot \frac{1}{\jb{n_3'}^{2 \al - (-j + k_1 + k_2 - 1) \eps''}} \\
&\les 1
\end{align*}

\noi
for some $\eps'' > 0$ small, as desired.

\smallskip \noi
\textbf{Case 3:} $j = k_2 - k_1 + 1$.

In this case, we can assume without loss of generality that $\{ n_\l \}_{1 \leq \l \leq 2j - 2}$ are paired with $\{ n_\l^{(2)} \}_{1 \leq \l \leq j - k_1 + k_2 - 1}$ and $\{ n_\l^{(1)} \}_{1 \leq \l \leq 2 k_1 - 1}$ are paired with $\{ n_\l^{(2)} \}_{j - k_1 + k_2 \leq \l \leq 2 k_2 - 1}$. Thus, by Lemma~\ref{LEM:conv} several times along with the fact that $1 < \al \leq \frac 98$, we obtain the following contribution for the bound of $|\eqref{YYYN}|$:
\begin{align*}
&\sum_{n_1, \dots, n_{2j - 2} \in \Z^3} \sum_{n_1^{(1)}, \dots, n_{2 k_1 - 1}^{(1)} \in \Z^3} \bigg( \prod_{\l = 1}^{2j - 2} \frac{1}{\jb{n_\l}^{2 \al + \frac{2j - 4}{2j - 2} (3 - 2 \al - \eps'')}} \bigg) \bigg( \prod_{\l' = 1}^{2 k_1 - 1} \frac{1}{\jb{n_{\l'}^{(1)}}^{2 \al + \frac{2 k_1 - 3}{2 k_1 - 1} (3 - 2 \al - \eps'')}} \bigg) \\
&\quad \times \frac{1}{\jb{\sum_{\l' = 1}^{2 k_1 - 1} n_{\l'}^{(1)}}^{2 \al} \jb{\sum_{\l = 1}^{2j - 2} n_\l + \sum_{\l' = 1}^{2 k_1 - 1} n_{\l'}^{(1)}}^{2 \al}} \\
&\les \sum_{n'_1, n'_2 \in \Z^3} \frac{1}{\jb{n'_1}^{4 \al - 3 - (2j - 4) \eps''} \jb{n'_2}^{6 \al - 3 - (2k_1 - 3) \eps''} \jb{n'_1 + n'_2}^{2 \al}} \\
&\les 1
\end{align*}

\noi
for some $\eps'' > 0$ small, as desired.

\smallskip \noi
\textbf{Case 4:} $j = k_1 + k_2$.

In this case, we can assume without loss of generality that $\{n_\l\}_{1 \leq \l \leq 2 k_1 - 1}$ are paired with $\{ n_\l^{(1)} \}_{1 \leq \l \leq 2 k_1 - 1}$ and $\{n_\l\}_{2 k_1 \leq \l \leq 2 j - 2}$ are paired with $\{ n_\l^{(2)} \}_{1 \leq \l \leq 2 k_2 - 1}$. Thus, by Lemma~\ref{LEM:conv} several times along with the fact that $1 < \al \leq \frac 98$, we obtain the following contribution for the bound of $|\eqref{YYYN}|$:
\begin{align*}
&\sum_{n_1, \dots, n_{2j - 2} \in \Z^3} \bigg( \prod_{\l = 1}^{2 k_1 - 1} \frac{1}{\jb{n_\l}^{2 \al + \frac{2 k_1 - 3}{2 k_1 - 1} (3 - 2 \al - \eps'')}} \prod_{\l = 2 k_1}^{2j - 2} \frac{1}{\jb{n_\l}^{2 \al + \frac{2j - 2 k_1 - 3}{2j - 2 k_1 - 1} (3 - 2 \al - \eps'')}} \bigg) \\
&\quad \times \frac{1}{\jb{\sum_{\l = 1}^{2 k_1 - 1} n_\l}^{2 \al} \jb{\sum_{\l = 2 k_1}^{2j - 2} n_\l}^{2 \al}} \\
&\les \sum_{n'_1, n'_2 \in \Z^3} \frac{1}{\jb{n'_1}^{6 \al - 3 - (2 k_1 - 3) \eps''} \jb{n'_2}^{6 \al - 3 - (2j - 2 k_1 - 3) \eps''}} \\
&\les 1
\end{align*}

\noi
for some $\eps'' > 0$ small, as desired. Thus, we have finished the proof.
\end{proof}

Recalling $\wt{\mathcal{R}}_N^\dia$ in \eqref{RNtil}, by using \eqref{RNbdd1-1} in Lemma \ref{LEM:RNbdd1}, Lemma~\ref{LEM:RNbdd1.5}, Lemma \ref{LEM:RNbdd2}, Lemma~\ref{LEM:Yreg}~(i) and (iv), Lemma~\ref{LEM:RNsto1}, and Lemma~\ref{LEM:RNsto2}, we obtain the following useful estimate for sufficiently large $N \in \N$:
\begin{align}
\begin{split}
\E \bigg[ &\wt{\mathcal{R}}_N^\dia (W_N, \Ups_N - \ZZ_N) + \frac 12 \int_0^1 \big\| \dot{\Ups}^N (t) \big\|_{H_x^\al}^2 dt \bigg] \\
&\geq \frac 14 \E \big[ \| \Ups^N \|_{H^\al}^2 \big] + c \sum_{j = 2}^m N^{- (2j - 4) (\frac 32 - \al)} \E \big[ \| \Ups_N \|_{L^{2j}}^{2j} \big]  - C
\end{split}
\label{RNtil_est}
\end{align}

\noi
for some constants $C > 0$ and $0 < c \ll 1$ independent of $N$. We note that \eqref{RNtil_est} along with \eqref{BD_RNtil} already gives a proof of the uniform bound \eqref{L1bdd} in the statement of Theorem~\ref{THM:meas}.

\subsection{Tightness of the truncated Gibbs measures}
\label{SUB:tight}
In this subsection, we show the following tightness result of the truncated Gibbs measures $\{\varrho_N\}_{N \in \N}$ defined in \eqref{nuN}. 
\begin{proposition}
\label{PROP:tight}
Let $1 < \al \leq \frac 98$, $1 \leq p \leq \infty$, and $\eps > 0$ sufficiently small. Then, there exists $N_0 \in \N$ such that the set of truncated Gibbs measures $\{\varrho_N\}_{N \geq N_0}$ is tight as a sequence of measures on $W^{\al - \frac 32 - \eps, p} (\T^3)$.
\end{proposition}

\begin{proof}
Let $B_R \subset W^{\al - \frac 32 - \frac{\eps}{2}, p} (\T^3)$ be the closed ball in $W^{\al - \frac 32 - \frac{\eps}{2}, p} (\T^3)$ of radius $R > 0$ centered at the origin. By Rellich's compactness lemma, we know that $B_R$ is compact in $W^{\al - \frac 32 - \eps, p} (\T^3)$. Our goal is to show that given any $\eta > 0$, there exists $R = R(\eta) > 1$ such that
\begin{align}
\sup_{N \geq N_0} \varrho_N (B_R^c) < \eta.
\label{tight_goal}
\end{align}

Recalling $\mathcal{Z}_N$ from \eqref{defZN}, by \eqref{BD_RNtil}, \eqref{RNbdd1-2} in Lemma \ref{LEM:RNbdd1}, Lemma \ref{LEM:RNbdd1.5}, Lemma \ref{LEM:RNbdd2}, Lemma~\ref{LEM:Yreg}~(i), Lemma~\ref{LEM:RNsto1}, and Lemma~\ref{LEM:RNsto2}, we have
\begin{align}
\begin{split}
- \log \mathcal{Z}_N &= \inf_{\dot{\Ups}^N \in \mathbb{H}_a^\al} \E \bigg[ \wt{\mathcal{R}}_N^\dia (W_N, \Ups_N - \ZZ_N) + \frac 12 \int_0^1 \big\| \dot{\Ups}^N (t) \big\|_{H_x^\al}^2 dt \bigg] \\
&\leq \E \big[ \wt{\mathcal{R}}_N^\dia (W_N, 0 - \ZZ_N) \big] \leq C
\end{split}
\label{ZNupbdd}
\end{align}

\noi
for some constant $C > 0$ independent of $N \geq N_0$ for some $N_0 \in \mathbb{N}$ sufficiently large. Thus, we have a positive uniform lower bound for $\mathcal{Z}_N$ with $N \geq N_0$.

We now focus on showing \eqref{tight_goal}. Given $M > 1$, we let $F$ be a bounded smooth non-negative function such that
\begin{align}
F(u) =
\begin{cases}
M, & \text{if } \| u \|_{W^{\al - \frac 32 - \frac{\eps}{2}, p}} \leq \frac{R}{2}, \\
0, & \text{if } \| u \|_{W^{\al - \frac 32 - \frac{\eps}{2}, p}} > R.
\end{cases}
\label{FMu}
\end{align}

\noi
From \eqref{ZNupbdd}, we know that $Z_N$ is uniformly bounded from below, so that we have
\begin{align}
\varrho_N (B_R^c) \leq \mathcal{Z}_N^{-1} \int e^{- F(u) - \mathcal{R}_N^\dia (u)} d \mu_\al \les \int e^{- F(u) - \mathcal{R}_N^\dia (u)} d \mu_\al \deff \ft{\mathcal{Z}}_N.
\label{tight0}
\end{align}

\noi
where the underlying constant is independent of $N \geq N_0$. By the Bou\'e-Dupuis variational formula (Lemma \ref{LEM:BD})\footnote{Here, we applied Lemma \ref{LEM:BD} for $F$ on rough functions, which can be easily justified by a limiting argument. This will also be used at a later point.} and \eqref{BD_RNtil}, we have
\begin{align}
\begin{split}
- &\log \ft{\mathcal{Z}}_N \\
&= \inf_{\dot{\Ups}^N \in \mathbb{H}_a^\al} \E \bigg[ F (W + \Ups^N - \ZZ_N) + \wt{\mathcal{R}}_N^\dia (W_N, \Ups_N - \ZZ_N) + \frac 12 \int_0^1 \big\| \dot{\Ups}^N (t) \big\|_{H_x^\al}^2 dt \bigg].
\end{split}
\label{logZNhat}
\end{align}

\noi
By the Sobolev embedding, Lemma \ref{LEM:Yreg} (i), and Chebyshev's inequality, there exists a constant $c > 0$ such that
\begin{align}
\begin{split}
\mathbb{P} \Big( \big\| &W + \Ups^N - \ZZ_N \big\|_{W^{\al - \frac 32 - \frac{\eps}{2}, p}} > \frac{R}{2} \Big) \\
&\leq \mathbb{P} \Big( \| W - \ZZ_N \|_{W^{\al - \frac 32 - \frac{\eps}{2}, p}} > \frac{R}{4} \Big) + \mathbb{P} \Big( \| \Ups^N \|_{H^{\al}} > \frac{R}{4} \Big) \\
&\leq e^{- c R^{\frac{2}{2m - 1}}} + \frac{16}{R^2} \E \big[ \| \Ups^N \|_{H^\al}^2 \big]
\end{split}
\label{tight1}
\end{align}

\noi
uniformly in $N \in \N$ and $R \gg 1$. By setting $M = \frac{R^2}{64} \gg 1$ and using \eqref{FMu} and \eqref{tight1}, we obtain
\begin{align}
\begin{split}
\E \big[ F (W + \Ups^N - \ZZ_N) \big] &\geq M \E \Big[ \ind_{\big\{ \| W + \Ups^N - \ZZ_N \|_{W^{\al - \frac 32 - \frac{\eps}{2}, p}} \leq \frac{R}{2} \big\}} \big] \\
&\geq \frac{M}{2} - \frac{16M}{R^2} \E \big[ \| \Ups^N \|_{H^\al}^2 \big] \\
&\geq \frac{M}{2} - \frac 14 \E \big[ \| \Ups^N \|_{H^\al}^2 \big].
\end{split}
\label{tight2}
\end{align}

\noi
Thus, for $M > 0$ sufficiently large, by \eqref{logZNhat}, \eqref{tight2}, and \eqref{RNtil_est}, we obtain
\begin{align}
\begin{split}
- \log \ft{\mathcal{Z}}_N &\geq \frac{M}{2} - \frac 14 \E \big[ \| \Ups^N \|_{H^\al}^2 \big] \\
&\quad + \inf_{\dot{\Ups}^N \in \mathbb{H}_a^\al} \E \bigg[ \wt{\mathcal{R}}_N^\dia (W_N, \Ups_N - \ZZ_N) + \frac 12 \int_0^1 \big\| \dot{\Ups}^N (t) \big\|_{H_x^\al}^2 dt \bigg] \\
&\geq \frac{M}{2} - C
\end{split}
\label{tight3}
\end{align}

\noi
for some constant $C > 0$ independent of $N \geq N_0$. Thus, given $\eta > 0$, by letting $R = R(\eta) \gg 1$ sufficiently large and choosing $M = \frac{R^2}{64} \gg 1$, the bound \eqref{tight_goal} follows from \eqref{tight0} and \eqref{tight3}. This finishes the proof of tightness of $\{\varrho_N\}_{N \geq N_0}$.
\end{proof}

\begin{remark} \rm
\label{RMK:pos0}
In the proof of tightness of the truncated Gibbs measures $\{\varrho_N\}_{N \geq N_0}$ in Proposition~\ref{PROP:tight}, we only used the estimates \eqref{RNbdd1-1} and \eqref{RNbdd1-2} in Lemma~\ref{LEM:RNbdd1}, whose proof relies only on the weaker positivity condition \eqref{pos0}. This means that the weaker positivity condition \eqref{pos0} is enough to guarantee tightness of the truncated Gibbs measures $\{\varrho_N\}_{N \geq N_0}$. On the other hand, the estimate \eqref{RNbdd1-1.5} in Lemma~\ref{LEM:RNbdd1}, whose proof relies on the stronger positivity condition \eqref{pos}, will be used in the next subsection.
\end{remark}

\subsection{Convergence of the truncated Gibbs measures}
\label{SUB:meas_conv}
We are now ready to prove Theorem~\ref{THM:meas} in this subsection. By the tightness result in Proposition \ref{PROP:tight} and the Prokhorov theorem (Lemma \ref{LEM:prok}), there exists a weakly convergent subsequence of $\{\varrho_N\}_{N \in \N}$. To show that the whole sequence $\{\varrho_N\}_{N \in \N}$ converges weakly to the fractional $\Phi^4_3$-measure $\rho$ as defined in \eqref{rho}, we only need to show the following result.

\begin{proposition}
\label{PROP:meas_conv}
Let $1 < \al \leq \frac 98$. Let $\{ \varrho_{N_k} \}_{k \in \N}$ be any weakly convergent subsequence of the truncated Gibbs measures $\{\varrho_N\}_{N \in \N}$ defined in \eqref{nuN}, and we denote the limiting measure as~$\varrho$. Then, we have $\varrho = \rho$.
\end{proposition}

\begin{proof}
We recall from Section \ref{SEC:intro} that the fractional $\Phi^4_3$-measure is constructed as a weak limit of truncated $\Phi^4_3$-measures $\rho_N$ as $N \to \infty$. Moreover, $Z_N$ defined in \eqref{deffZN} converges as $N \to \infty$. We also observe that by considering the case when $\cj{a}_{1, N} N^{2(\frac 32 - \al)} = \frac{a}{2}$, $\cj{a}_{2, N} = \cj{a}_2$, and $\cj{a}_{j, N} = 0$ for all $3 \leq j \leq m$ and $N \in \N$, all estimates in Subsection \ref{SUB:meas} hold for $\rho_N$ correspondingly.

We recall $\mathcal{Z}_N$ from \eqref{defZN}. By \eqref{BD_RNtil}, \eqref{RNtil_est}, and \eqref{ZNupbdd}, we see that $\{\mathcal{Z}_N\}_{N \in \N}$ is bounded both from above and from below. Thus, by taking a further subsequence, we can assume that $\lim_{k \to \infty} \mathcal{Z}_{N_k}$ exists. Below, we also assume that $N_k$ and $N$ are sufficiently large.

\smallskip \noi
\textbf{Step 1:}
Let us first show that
\begin{align}
\lim_{k \to \infty} \mathcal{Z}_{N_k} = \lim_{N \to \infty} Z_N.
\label{limZN}
\end{align}

\noi
For convenience of indexing, we take a subsequence $\{N_k'\}_{k \in \N}$ such that $N_k \leq N_k'$ for all $k \in \N$. By \eqref{BD_RNtil}, we have
\begin{align}
- \log \mathcal{Z}_{N_k} &= \inf_{\dot{\Ups}^{N_k} \in \mathbb{H}_a^\al} \E \bigg[ \wt{\mathcal{R}}_{N_k}^\dia (W_{N_k}, \Ups_{N_k} - \ZZ_{N_k}) + \frac 12 \int_0^1 \big\| \dot{\Ups}^{N_k} (t) \big\|_{H_x^\al}^2 dt \bigg], \label{conv_inf} \\
- \log Z_{N_k'} &= \inf_{\dot{\Ups}^{N_k'} \in \mathbb{H}_a^\al} \E \bigg[ \wt{R}_{N_k'}^\dia (W_{N_k'}, \Ups_{N_k'} - \zeta_{N_k'}) + \frac 12 \int_0^1 \big\| \dot{\Ups}^{N_k'} (t) \big\|_{H_x^\al}^2 dt \bigg], \notag
\end{align}

\noi
where $\wt{\mathcal{R}}_{N}^\dia$ is as defined in \eqref{RNtil}, $\zeta_N = \pi_N \zeta^N = \pi_N \zeta^N (1)$ is defined by its time derivative
\begin{align}
\dot{\zeta}^N (t) \deff 4 \cj{a}_{2} (1 - \Dl)^{- \al} \wick{W_N^3}
\label{zeta}
\end{align}

\noi
with $\zeta^N (0) = 0$, and $\wt{R}_N^\dia$ is defined by
\begin{align}
\begin{split}
\wt{R}_N^\dia (W_N, \Ups_N - \zeta_N) &\deff \int_{\T^3} \frac{a}{2} (\Ups_N - \zeta_N)^2 dx + \int_{\T^3} a W_N (\Ups_N - \zeta_N) dx \\
&\quad + \int_{\T^3} \cj{a}_2 (\Ups_N - \zeta_N)^4 dx + \int_{\T^3} 4\cj{a}_2 W_N (\Ups_N - \zeta_N)^3 dx \\
&\quad + \int_{\T^3} 6\cj{a}_2 \wick{W_N^2} (\Ups_N - \zeta_N)^2 dx.
\end{split}
\label{defRNd2}
\end{align}

\noi
Given $\eta > 0$, let $\UUps^{N_k} \in \mathbb{H}_a^\al$ be such that
\begin{align}
- \log \mathcal{Z}_{N_k} \geq \E \bigg[ \wt{\mathcal{R}}_{N_k}^\dia (W_{N_k}, \UUps_{N_k} - \ZZ_{N_k}) + \frac 12 \int_0^1 \big\| \dot{\UUps}^{N_k} (t) \big\|_{H_x^\al}^2 dt \bigg] - \eta,
\label{conv_approx}
\end{align}

\noi
where $\UUps_{N_k} = \pi_{N_k} \UUps^{N_k}$. Then, by choosing $\Ups^{N_k'} = \UUps_{N_k}$, we obtain
\begin{align}
\begin{split}
- \log Z_{N_k'} + \log \mathcal{Z}_{N_k} &\leq \E \bigg[ \wt{R}_{N_k'}^\dia (W_{N_k'}, \UUps_{N_k} - \zeta_{N_k'}) + \frac 12 \int_0^1 \big\| \dot{\UUps}_{N_k} (t) \big\|_{H_x^\al}^2 dt \bigg] \\
&\quad - \E \bigg[ \wt{\mathcal{R}}_{N_k}^\dia (W_{N_k}, \UUps_{N_k} - \ZZ_{N_k}) + \frac 12 \int_0^1 \big\| \dot{\UUps}^{N_k} (t) \big\|_{H_x^\al}^2 dt \bigg] + \eta \\
&\leq \E \Big[ \wt{R}_{N_k'}^\dia (W_{N_k'}, \UUps_{N_k} - \zeta_{N_k'}) - \wt{\mathcal{R}}_{N_k}^\dia (W_{N_k}, \UUps_{N_k} - \ZZ_{N_k}) \Big] + \eta,
\end{split}
\label{log_diff}
\end{align}

\noi
where we used the fact that $\pi_{N_k'} \UUps_{N_k} = \UUps_{N_k}$ under the assumption $N_k' \geq N_k$. Note that
\begin{align}
\begin{split}
\wt{R}_{N_k'}^\dia &(W_{N_k'}, \UUps_{N_k} - \zeta_{N_k'}) - \wt{\mathcal{R}}_{N_k}^\dia (W_{N_k}, \UUps_{N_k} - \ZZ_{N_k})  \\
&= \int_{\T^3} \Big( \frac a2 (\UUps_{N_k} - \zeta_{N_k'})^2 - \cj{a}_{1, N_k} N_k^{2 (\frac 32 - \al)} (\UUps_{N_k} - \ZZ_{N_k})^2 \Big) dx \\
&\quad + \int_{\T^3} \Big( a W_{N_k'} (\UUps_{N_k} - \zeta_{N_k'}) - 2 \cj{a}_{1, N_k} N_k^{2 (\frac 32 - \al)} W_{N_k} (\UUps_{N_k} - \ZZ_{N_k}) \Big) dx \\
&\quad + \int_{\T^3} \Big( \cj{a}_2 (\UUps_{N_k} - \zeta_{N_k'})^4 - \cj{a}_{2, N_k} (\UUps_{N_k} - \ZZ_{N_k})^4 \Big) dx \\
&\quad + 4 \int_{\T^3} \Big( \cj{a}_2 W_{N_k'} (\UUps_{N_k} - \zeta_{N_k'})^3 - \cj{a}_{2, N_k} W_{N_k} (\UUps_{N_k} - \ZZ_{N_k})^3 \Big) dx \\
&\quad + 6 \int_{\T^3} \Big( \cj{a}_2 \wick{W_{N_k'}^2} (\UUps_{N_k} - \zeta_{N_k'})^2 - \cj{a}_{2, N_k} \wick{W_{N_k}^2} (\UUps_{N_k} - \ZZ_{N_k})^2 \Big) dx \\
&\quad - \int_{\T^3} \sum_{j = 3}^m \cj{a}_{j, N_k} N_k^{- (2j - 4) (\frac 32 - \al)} \sum_{\l = 0}^{2j - 2} \binom{2j}{\l} \wick{W_{N_k}^\l} (\UUps_{N_k} - \ZZ_{N_k})^{2j - \l} dx \\
&\deff \textup{I}_1 + \textup{I}_2 + \textup{I}_3 + \textup{I}_4 + \textup{I}_5 + \textup{I}_6.
\end{split}
\label{conv_main}
\end{align}

\noi
For $\textup{I}_6$, we use \eqref{RNbdd1-1.5} in Lemma \ref{LEM:RNbdd1}, Lemma \ref{LEM:RNbdd2}, 
Lemma~\ref{LEM:RNsto1}, and Lemma~\ref{LEM:RNsto2} to obtain
\begin{align}
\begin{split}
\E [ \textup{I}_6 ] &\leq C N_k^{- \eps'} + \dl N_k^{- \eps'} \Big( \E \big[ \| \UUps_{N_k} \|_{H^\al}^2 \big] + \E \big[ \| \UUps_{N_k} \|_{L^4}^4 \big] \Big) \\
&\quad + \dl \sum_{j = 3}^m N_k^{- (2j - 4) (\frac 32 - \al) - \eps'} \E \big[ \| \UUps_{N_k} \|_{L^{2j}}^{2j} \big] 
\end{split}
\label{I6}
\end{align}

\noi
for any $0 < \dl < 1$ and some constant $C > 0$, $\eps' > 0$ sufficiently small. 
As for $\textup{I}_1, \dots, \textup{I}_5$, we want to apply Lemma \ref{LEM:RNbdd1.5}, Lemma \ref{LEM:RNbdd2}, Lemma \ref{LEM:Yreg}, Lemma~\ref{LEM:RNsto1}, and Lemma~\ref{LEM:RNsto2} with some additional treatment. We note that they contain the following difference terms: $\frac a2 - \cj{a}_{1, N_k} N_k^{2 (\frac 32 - \al)}$, $\cj{a}_2 - \cj{a}_{2, N_k}$, $W_{N_k'} - W_{N_k}$, $\wick{W_{N_k'}^2} - \wick{W_{N_k}^2}$, and $\zeta_{N_k'} - \ZZ_{N_k}$. For $\frac a2 - \cj{a}_{1, N_k} N_k^{2 (\frac 32 - \al)}$, we use \eqref{a1N} to create a $N_k^{- \eps'}$ factor from this difference. For $\cj{a}_2 - \cj{a}_{2, N_k}$, we use \eqref{aj}, \eqref{ajN}, and \eqref{sigma_diff} to create a $N_k^{- \eps'}$ factor from this difference. For $W_{N_k'} - W_{N_k}$ and $\wick{W_{N_k'}^2} - \wick{W_{N_k}^2}$, in view of the proof of Lemma \ref{LEM:Yreg} (ii), we can also create a $N_k^{- \eps'}$ factor given $N_k' \geq N_k$. For $\zeta_{N_k'} - \ZZ_{N_k}$, we have
\begin{align}
\begin{split}
\zeta_{N_k'} - \ZZ_{N_k} &= \int_0^1 \bigg( \big( 4 \cj{a}_2 \jb{\nb}^{-2 \al} \wick{W_{N_k'}^3} - 4 \cj{a}_{2, N_k} \jb{\nb}^{-2 \al} \wick{W_{N_k}^3} \big) \\
&\quad - \sum_{j = 3}^m 2j \cj{a}_{j, N_k} N_k^{- (2j - 4) (\frac 32 - \al)} \jb{\nb}^{-2 \al} \wick{W_N^{2j - 1}} \bigg) dt
\end{split}
\label{zz_diff}
\end{align}

\noi
For the difference term on the right-hand-side of \eqref{zz_diff}, we can create a $N_k^{- \eps'}$ factor from $\cj{a}_2 - \cj{a}_{2, N_k}$ and $\wick{W_{N_k'}^3} - \wick{W_{N_k}^3}$ by the same reasoning before. For the summation term on the right-hand-side of \eqref{zz_diff}, we can create a $N_k^{- \eps'}$ factor in view of the open condition in Lemma~\ref{LEM:Yreg}~(i) or the extra $N^{\eps'}$ factor in Lemma~\ref{LEM:RNsto1} and Lemma~\ref{LEM:RNsto2}. To sum up, we obtain
\begin{align}
\E [\textup{I}_1 + \textup{I}_2 + \textup{I}_3 + \textup{I}_4 + \textup{I}_5] \leq C N_k^{- \eps'} + \dl N_k^{- \eps'} \Big( \E \big[ \| \UUps_{N_k} \|_{H^\al}^2 \big] + \E \big[ \| \UUps_{N_k} \|_{L^4}^4 \big] \Big)
\label{I12345}
\end{align}

\noi
for some constant $C > 0$ and $\dl > 0$ arbitrarily small. Thus, by \eqref{log_diff}, \eqref{conv_main}, \eqref{I6}, and \eqref{I12345}, we obtain
\begin{align}
\begin{split}
- \log Z_{N_k'} + \log \mathcal{Z}_{N_k} &\leq 2C N_k^{- \eps'} + 2 \dl N_k^{- \eps'} \Big( \E \big[ \| \UUps_{N_k} \|_{H^\al}^2 \big] + \E \big[ \| \UUps_{N_k} \|_{L^4}^4 \big] \Big) \\
&\quad + \dl \sum_{j = 3}^m N_k^{- (2j - 4) (\frac 32 - \al) - \eps'} \E \big[ \| \UUps_{N_k} \|_{L^{2j}}^{2j} \big] + \eta
\end{split}
\label{log_diff2}
\end{align}

\noi
By \eqref{RNtil_est}, \eqref{conv_approx}, \eqref{conv_inf}, and \eqref{ZNupbdd}, we have
\begin{align}
\begin{split}
2 \dl &\Big( \E \big[ \| \UUps_{N_k} \|_{H^\al}^2 \big] + \E \big[ \| \UUps_{N_k} \|_{L^{4}}^{4} \big] \Big) + \dl \sum_{j = 3}^m N_k^{- (2j - 4) (\frac 32 - \al)} \E \big[ \| \UUps_{N_k} \|_{L^{2j}}^{2j} \big] \\
&\leq C + C \E \bigg[ \wt{\mathcal{R}}_{N_k}^\dia (W_{N_k}, \UUps_{N_k} - \ZZ_{N_k}) + \frac 12 \int_0^1 \big\| \dot{\UUps}^{N_k} (t) \big\|_{H_x^\al}^2 dt \bigg] \\
&\leq C + C \eta + \E \big[ \wt{\mathcal{R}}_{N_k}^\dia (W_{N_k}, 0 - \ZZ_{N_k}) \big] \\
&\leq 3C.
\end{split}
\label{log_diff3}
\end{align}

\noi
Thus, from \eqref{log_diff2}, \eqref{log_diff3}, and the fact that $\eta > 0$ can be arbitrarily small, we obtain
\begin{align}
\lim_{k \to \infty} \mathcal{Z}_{N_k} \leq \lim_{k \to \infty} Z_{N_k'} = \lim_{N \to \infty} Z_N.
\label{limZN1}
\end{align}

We now consider the other direction. We pick a subsequence $\{ N_k'' \}_{k \in \N} \subset \{ N_k \}_{k \in \N}$ such that $N_k'' > N_k^{1 + \g}$ for some $\g > 0$ for all $k \in \N$. To show that \eqref{limZN} hold, in view of \eqref{limZN1}, we only need to show that
\begin{align}
\lim_{k \to \infty} \mathcal{Z}_{N_k''} \geq \lim_{k \to \infty} Z_{N_k}.
\label{limZN2}
\end{align}

\noi
Similar as above, given $\eta > 0$, we let $\UUps^{N_k} \in \mathbb{H}_a^\al$ be such that
\begin{align*}
-\log Z_{N_k} \geq \E \bigg[ \wt R_{N_k}^\dia (W_{N_k}, \UUps_{N_k} - \zeta_{N_k}) + \frac 12 \int_0^1 \big\| \dot{\UUps}^{N_k} (t) \big\|_{H_x^\al}^2 dt \bigg] - \eta,
\end{align*}

\noi
where $\UUps_{N_k} = \pi_{N_k} \UUps^{N_k}$. Then, we obtain
\begin{align}
\begin{split}
- \log &\mathcal{Z}_{N_k''} + \log Z_{N_k} \\
&\leq \E \Big[ \wt{\mathcal{R}}_{N_k''}^\dia (W_{N_k''}, \UUps_{N_k} - \ZZ_{N_k''}) - \wt R_{N_k}^\dia (W_{N_k}, \UUps_{N_k} - \zeta_{N_k}) \Big] + \eta \\
&= \E \bigg[ \int_{\T^3} \Big( \cj{a}_{1, N_k''} (N_k'')^{2 (\frac 32 - \al)} (\UUps_{N_k} - \ZZ_{N_k''})^2 - \frac a2 (\UUps_{N_k} - \zeta_{N_k})^2 \Big) dx \\
&\quad + \int_{\T^3} \Big( 2 \cj{a}_{1, N_k''} (N_k'')^{2 (\frac 32 - \al)} W_{N_k''} (\UUps_{N_k} - \ZZ_{N_k''}) - a W_{N_k} (\UUps_{N_k} - \zeta_{N_k}) \Big) dx \\
&\quad + \int_{\T^3} \Big( \cj{a}_{2, N_k''} (\UUps_{N_k} - \ZZ_{N_k''})^4 - \cj{a}_2 (\UUps_{N_k} - \zeta_{N_k})^4 \Big) dx \\
&\quad + 4 \int_{\T^3} \Big( \cj{a}_{2, N_k''} W_{N_k''} (\UUps_{N_k} - \ZZ_{N_k''})^3 - \cj{a}_2 W_{N_k} (\UUps_{N_k} - \zeta_{N_k})^3 \Big) dx \\
&\quad + 6 \int_{\T^3} \Big( \cj{a}_{2, N_k''} \wick{W_{N_k''}^2} (\UUps_{N_k''} - \ZZ_{N_k''})^2 - \cj{a}_2 \wick{W_{N_k}^2} (\UUps_{N_k} - \zeta_{N_k})^2 \Big) dx \\
&\quad + \int_{\T^3} \sum_{j = 3}^m \cj{a}_{j, N_k''} (N_k'')^{- (2j - 4) (\frac 32 - \al)} \sum_{k = 0}^{2j - 2} \binom{2j}{k} \wick{W_{N_k''}^k} (\UUps_{N_k} - \ZZ_{N_k''})^{2j - k} dx \bigg] + \eta \\
&\deff \E [ \II_1 + \II_2 + \II_3 + \II_4 + \II_5 + \II_6 ] + \eta.
\end{split}
\label{log2}
\end{align}

\noi
The treatment of $\II_1, \dots, \II_5$ is the same as above, and we can obtain
\begin{align}
\E [ \II_1 + \II_2 + \II_3 + \II_4 + \II_5 ] \leq C N_k^{- \eps'} + \dl N_k^{- \eps'} \Big( \E \big[ \| \UUps_{N_k} \|_{H^\al}^2 \big] + \E \big[ \| \UUps_{N_k} \|_{L^4}^4 \big] \Big)
\label{log3}
\end{align}

\noi
for some constant $C > 0$, $\eps' > 0$ sufficiently small, and $\dl > 0$ arbitrarily small. For $\II_6$, we use Lemma \ref{LEM:RNbdd2},  
Lemma~\ref{LEM:RNsto1}, Lemma~\ref{LEM:RNsto2}, and \eqref{RNbdd1-2} in Lemma~\ref{LEM:RNbdd1} (with $j \geq 3$) to obtain
\begin{align}
\begin{split}
\E [ \II_6 ] &\leq C (N_k'')^{- \eps'} + \dl (N_k'')^{- \eps'} \Big( \E \big[ \| \UUps_{N_k} \|_{H^\al}^2 \big] + \E \big[ \| \UUps_{N_k} \|_{L^4}^4 \big] \Big) \\
&\quad + C \sum_{j = 3}^m (N_k'')^{- (2j - 4) (\frac 32 - \al)} \E \big[ \| \UUps_{N_k} \|_{L^{2j}}^{2j} \big] \\
&\leq C N_k^{- \eps'} + \dl N_k^{- \eps'} \Big( \E \big[ \| \UUps_{N_k} \|_{H^\al}^2 \big] + \E \big[ \| \UUps_{N_k} \|_{L^4}^4 \big] \Big) \\
&\quad + C \sum_{j = 3}^m N_k^{- (2j - 4) (\frac 32 - \al) - \eps'} \E \big[ \| \UUps_{N_k} \|_{L^{2j}}^{2j} \big]
\end{split}
\label{log4}
\end{align}

\noi
where we used the fact that $N_k'' \geq N_k^{1 + \g}$ for some $\g > 0$ for all $k \in \N$. Combining \eqref{log2}, \eqref{log3}, \eqref{log4}, and \eqref{log_diff3}, we obtain the desired inequality \eqref{limZN2}.

\medskip \noi
\textbf{Step 2:}
We now show that $\varrho = \rho$, which will follow similarly from the argument in Step 1 above. In view of \eqref{limZN} and \eqref{ZNupbdd} (and an analogous bound for $- \log Z_N$), we know that it suffices to show that for every bounded Lipschitz continuous function $F: \mathcal{C}^{-100} (\T^3) \to \R$, we have
\begin{align*}
\lim_{k \to \infty} \mathcal{Z}_{N_k} \int \exp (F(u)) d \varrho_{N_k} (u) = \lim_{N \to \infty} Z_N \int \exp (F(u)) d \rho_N (u),
\end{align*}

\noi
which will follow from
\begin{align}
\lim_{k \to \infty} \mathcal{Z}_{N_k} \int \exp (F(u)) d \varrho_{N_k} (u) \leq \lim_{N \to \infty} Z_N \int \exp (F(u)) d \rho_N (u), \label{conv_goal1} \\
\lim_{k \to \infty} \mathcal{Z}_{N_k} \int \exp (F(u)) d \varrho_{N_k} (u) \geq \lim_{N \to \infty} Z_N \int \exp (F(u)) d \rho_N (u). \label{conv_goal2}
\end{align}

For \eqref{conv_goal1}, we take a subsequence $\{N_k'\}_{k \in \N}$ such that $N_k \leq N_k'$ for all $k \in \N$. By the Bou\'e-Dupuis variational formula (Lemma \ref{LEM:BD}) and \eqref{BD_RNtil}, we have
\begin{align}
\begin{split}
- &\log \bigg( \mathcal{Z}_{N_k} \int \exp (F(u)) d \varrho_{N_k} (u) \bigg) \\
&= \inf_{\dot{\Ups}^{N_k} \in \mathbb{H}_a^\al} \E \bigg[ -F (W + \Ups^{N_k} - \ZZ_{N_k}) + \wt{\mathcal{R}}_{N_k}^\dia (W_{N_k}, \Ups_{N_k} - \ZZ_{N_k}) \\
&\qquad + \frac 12 \int_0^1 \big\| \dot{\Ups}^{N_k} (t) \big\|_{H_x^\al}^2 dt \bigg],
\end{split}
\label{uni1}
\end{align}

\noi
and
\begin{align*}
- &\log \bigg( Z_{N_k'} \int \exp (F(u)) d \rho_{N_k'} (u) \bigg) \\
&= \inf_{\dot{\Ups}^{N_k'} \in \mathbb{H}_a^\al} \E \bigg[ -F (W + \Ups^{N_k'} - \zeta_{N_k'}) + \wt{R}_{N_k'}^\dia (W_{N_k'}, \Ups_{N_k'} - \zeta_{N_k'}) + \frac 12 \int_0^1 \big\| \dot{\Ups}^{N_k'} (t) \big\|_{H_x^\al}^2 dt \bigg].
\end{align*}

\noi
Given $\eta > 0$, let $\UUps^{N_k} \in \mathbb{H}_a^\al$ be such that
\begin{align}
\begin{split}
- &\log \bigg( \mathcal{Z}_{N_k} \int \exp (F(u)) d \varrho_{N_k} (u) \bigg) \\
&\geq \E \bigg[ -F (W + \UUps^{N_k} - \ZZ_{N_k}) + \wt{\mathcal{R}}_{N_k}^\dia (W_{N_k}, \UUps_{N_k} - \ZZ_{N_k}) + \frac 12 \int_0^1 \big\| \dot{\UUps}^{N_k} (t) \big\|_{H_x^\al}^2 dt \bigg] - \eta,
\end{split}
\label{uni2}
\end{align}

\noi
where $\UUps_{N_k} = \pi_{N_k} \UUps^{N_k}$. Then, by choosing $\Ups^{N_k'} = \UUps_{N_k}$, we obtain
\begin{align}
\begin{split}
&- \log \bigg( Z_{N_k'} \int \exp (F(u)) d \rho_{N_k'} (u) \bigg) + \log \bigg( \mathcal{Z}_{N_k} \int \exp (F(u)) d \varrho_{N_k} (u) \bigg) \\
&\leq \E \bigg[ -F (W + \UUps_{N_k} - \zeta_{N_k'}) + \wt{R}_{N_k'}^\dia (W_{N_k'}, \UUps_{N_k} - \zeta_{N_k'}) + \frac 12 \int_0^1 \big\| \dot{\UUps}_{N_k} (t) \big\|_{H_x^\al}^2 dt \bigg] \\
&\quad - \E \bigg[ -F (W + \UUps^{N_k} - \ZZ_{N_k}) + \wt{\mathcal{R}}_{N_k}^\dia (W_{N_k}, \UUps_{N_k} - \ZZ_{N_k}) + \frac 12 \int_0^1 \big\| \dot{\UUps}^{N_k} (t) \big\|_{H_x^\al}^2 dt \bigg] + \eta \\
&\leq \text{Lip} (F) \cdot \E \Big[ \big\| \pi_{N_k}^{\perp} \UUps^{N_k} + \zeta_{N_k'} - \ZZ_{N_k} \big\|_{\mathcal{C}^{-100}} \Big] \\
&\quad + \E \Big[ \wt{R}_{N_k'}^\dia (W_{N_k'}, \UUps_{N_k} - \zeta_{N_k'}) - \wt{\mathcal{R}}_{N_k}^\dia (W_{N_k}, \UUps_{N_k} - \ZZ_{N_k}) \Big] + \eta,
\end{split}
\label{log_diff4}
\end{align}

\noi
where $\pi_N^\perp = \Id - \pi_N$. For the second term on the right-hand-side of \eqref{log_diff4}, we can use the same steps from Step 1 to show that this term goes to 0 as $k \to \infty$. For the first term on the right-hand-side of \eqref{log_diff4}, we write
\begin{align}
\E \Big[ \big\| \pi_{N_k}^{\perp} \UUps^{N_k} + \zeta_{N_k'} - \ZZ_{N_k} \big\|_{\mathcal{C}^{-100}} \Big] \leq \E \Big[ \big\| \pi_{N_k}^{\perp} \UUps^{N_k} \big\|_{\mathcal{C}^{-100}} \Big] + \E \Big[ \big\| \zeta_{N_k'} - \ZZ_{N_k} \big\|_{\mathcal{C}^{-100}} \Big].
\label{log_diff5}
\end{align}

\noi
Note that the second term on the right-hand-side of \eqref{log_diff5} can be bounded easily by $C N_k^{- \eps'}$ for some $\eps' > 0$ in view of the discussion for dealing with \eqref{zz_diff}. For the first term on the right-hand-side of \eqref{log_diff5}, by the Besov embedding (Lemma~\ref{LEM:Bes}~(ii)), \eqref{RNtil_est}, \eqref{uni2}, \eqref{uni1}, the fact that $F$ is bounded, and \eqref{ZNupbdd}, we obtain
\begin{align}
\begin{split}
\E \Big[ &\big\| \pi_{N_k}^{\perp} \UUps^{N_k} \big\|_{\mathcal{C}^{-100}} \Big] \\
&\les N_k^{- \eps'} \E \Big[ \big\| \UUps^{N_k} \big\|_{H^\al}^2 \Big] \\
&\leq N_k^{- \eps'} \bigg( C + C \E \bigg[ \wt{\mathcal{R}}_{N_k}^\dia (W_{N_k}, \UUps_{N_k} - \ZZ_{N_k}) + \frac 12 \int_0^1 \big\| \dot{\UUps}^{N_k} (t) \big\|_{H_x^\al}^2 dt \bigg] \bigg) \\
&\leq N_k^{- \eps'} \Big( C + C \eta + \E \big[ \wt{\mathcal{R}}_{N_k}^\dia (W_{N_k}, 0 - \ZZ_{N_k}) \big] \Big) \\
&\leq 3 C N_k^{- \eps'}.
\end{split}
\label{log_diff6}
\end{align}

\noi
Thus, combining \eqref{log_diff4}, \eqref{log_diff5}, and \eqref{log_diff6}, we obtain
\begin{align*}
\lim_{k \to \infty} \mathcal{Z}_{N_k} \int \exp (F(u)) d \varrho_{N_k} (u) &\leq \lim_{k \to \infty} Z_{N_k'} \int \exp (F(u)) d \rho_{N_k'} (u) \\
&= \lim_{N \to \infty} Z_N \int \exp (F(u)) d \rho_N (u),
\end{align*}

\noi
which proves \eqref{conv_goal1}.

As for \eqref{conv_goal2}, we pick a subsequence $\{ N_k'' \}_{k \in \N} \subset \{ N_k \}_{k \in \N}$ such that $N_k'' > N_k^{1 + \g}$ for some $\g > 0$ for all $k \in \N$. Proceeding as above, we are able to obtain that given $\eta > 0$,
\begin{align}
\begin{split}
- &\log \bigg( \mathcal{Z}_{N_k''} \int \exp (F(u)) d \varrho_{N_k''} (u) \bigg) + \log \bigg( Z_{N_k} \int \exp (F(u)) d \rho_{N_k} (u) \bigg) \\
&\leq \text{Lip} (F) \cdot \E \Big[ \big\| \pi_{N_k}^\perp \UUps^{N_k} + \ZZ_{N_k''} - \zeta_{N_k} \big\|_{\mathcal{C}^{- 100}} \Big] \\
&\quad + \E \Big[ \wt{\mathcal{R}}_{N_k''}^\dia (W_{N_k''}, \UUps_{N_k} - \ZZ_{N_k''}) - \wt R_{N_k}^\dia (W_{N_k}, \UUps_{N_k} - \zeta_{N_k}) \Big] + \eta.
\end{split}
\label{log_diff7}
\end{align}

\noi
The second term on the right-hand-side of \eqref{log_diff7} goes to 0 as $k \to \infty$, which can be shown by the same steps from Step 1. The first term on the right-hand-side of \eqref{log_diff7} can be estimated similarly as above and also goes to 0 as $k \to \infty$. This proves \eqref{conv_goal2}, and so we have finished the proof of $\varrho = \rho$.
\end{proof}

\subsection{Necessity of the positivity condition}
\label{SUB:nonconv}
In this subsection, we prove Proposition~\ref{PROP:nonconv}, which shows that the positivity assumption \eqref{pos0} is an almost necessary condition for the boundedness of the densities for the truncated Gibbs measures $\{ \varrho_N \}_{N \in \N}$ defined in \eqref{nuN}. 
\begin{proof}[Proof of Proposition~\ref{PROP:nonconv}]
We are given $z' \in \R$ such that
\begin{align*}
\sum_{j = 2}^m \cj{a}_j (z')^{2 (j - 2)} < 0.
\end{align*}

\noi
By continuity, we may assume that $z' \neq 0$.
In view of the convergence of $\cj{a}_{j, N}$ to $\cj{a}_j$ for each $2 \leq j \leq m$, we have
\begin{align}
\sum_{j = 2}^m \cj{a}_{j, N} (z')^{2 (j - 2)} < 0
\label{npos}
\end{align}

\noi
for all sufficiently large $N \in \N$. Recalling \eqref{BD_RNtil} and \eqref{RNtil}, by letting $\dot{\Ups}^N = N^{\frac 32 - \al} z'$, we obtain
\begin{align}
- \log \big\| e^{- \mathcal{R}_N^\dia (u)} \big\|_{L^1 (\mu_\al)} \leq \E \bigg[ \wt{\mathcal{R}}_N^\dia (W_N, N^{\frac 32 - \al} z' - \ZZ_N) + \frac 12 (2 \pi)^3 N^{2 (\frac 32 - \al)} (z')^2 \bigg]
\label{npos1}
\end{align}

\noi
with
\begin{align}
\begin{split}
&\wt{\mathcal{R}}_N^\dia (W_N, N^{\frac 32 - \al} z' - \ZZ_N) \\
&= \sum_{j = 2}^m \sum_{k = 0}^{2 j - 2} \sum_{\l = 0}^{2 j - k} \binom{2j}{k} \binom{2j - k}{\l} (-1)^\l \cj{a}_{j, N} N^{- (k + \l - 4) (\frac 32 - \al)} (z')^{2j - k - \l} \int_{\T^3} \wick{W_N^k} \ZZ_N^{\l} dx \\
&\quad + 2 \cj{a}_{1, N} N^{2 (\frac 32 - \al)} \int_{\T^3} W_N  (N^{\frac 32 - \al} z' - \ZZ_N) dx + \cj{a}_{1, N} N^{2 (\frac 32 - \al)} \int_{\T^3} (N^{\frac 32 - \al} z' - \ZZ_N)^2 dx .
\end{split}
\label{npos2}
\end{align}

We first consider the summation term on the right-hand-side of \eqref{npos2}. Note that when $k = 0$ and $\l \geq 1$, from the product estimate (Lemma~\ref{LEM:Bes}~(iv)), \eqref{defZZ}, and Lemma~\ref{LEM:Yreg}~(i), we have
\begin{align}
N^{- (\l - 4) (\frac 32 - \al)} \E \bigg[ \int_{\T^3} \ZZ_N^\l dx \bigg] \les N^{3 (\frac 32 - \al)} \E \big[ \| \ZZ_N \|_{\mathcal{C}^{\frac 12}}^\l \big] = O (N^{3 (\frac 32 - \al)}).
\label{npos3}
\end{align}

\noi
When $k \geq 1$ and $\l = 0$, from Lemma~\ref{LEM:Yreg}~(iii), we have
\begin{align}
N^{- (k - 4) (\frac 32 - \al)} \E \bigg[ \int_{\T^3} \wick{W_N^k} dx \bigg] = 0.
\label{npos4}
\end{align}

\noi
When $k = 1$ and $\l \geq 1$, from duality, the embedding in Lemma~\ref{LEM:Bes}~(i), the product estimate in Lemma~\ref{LEM:Bes}~(iv), \eqref{defZZ}, and Lemma~\ref{LEM:Yreg}~(i), we have
\begin{align}
\begin{split}
N^{- (\l - 3) (\frac 32 - \al)} \E \bigg[ \int_{\T^3} W_N \ZZ_N^\l dx \bigg] 
&\les N^{- (\l - 3) (\frac 32 - \al)} \| W_N \|_{\mathcal{C}^{\al - \frac 32 - \eps}} \| \ZZ_N \|_{\mathcal{C}^{5 \al - \frac 92 - \eps}}^\l \\
&= O(N^{2 (\frac 32 - \al)}),
\end{split}
\label{npos5}
\end{align}

\noi
where we used the fact that $(\al - \frac 32 - \eps) + (5 \al - \frac 92 - \eps) > 0$ given $\al > 1$ and $\eps > 0$ sufficiently small. When $k \geq 2$ and $\l \geq 1$, we use Lemma~\ref{LEM:RNstoh} (for $\l = 1$) and Lemma~\ref{LEM:RNsto2} (for $\l \geq 2$) to obtain\footnote{As it appears, Lemma~\ref{LEM:RNsto2} covers the case when $k + \l$ is even. If $k + \l$ is odd, using the language in the proof of Lemma~\ref{LEM:RNsto2}, we see that the number of frequency nodes of the product $\wick{W_N^k} \ZZ_N^{\l}$ is odd. Since the expectation $\E [\int \wick{W_N^k} \ZZ_N^{\l}]$ is non-zero only if all frequencies in the product $\wick{W_N^k} \ZZ_N^{\l}$ are paired and it is impossible to pair up an odd number of frequencies, the expectation $\E [\int \wick{W_N^k} \ZZ_N^{\l}]$ must be 0 in this case.}
\begin{align}
N^{- (k + \l - 4) (\frac 32 - \al)} \E \bigg[ \int_{\T^3} \wick{W_N^k} \ZZ_N^{\l} dx \bigg] = O(N^{2 (\frac 32 - \al) - \eps'})
\label{npos6}
\end{align}

\noi
for any $\eps' > 0$ arbitrarily small. For the second and the third term on the right-hand-side of \eqref{npos2}, we can use \eqref{a1N} and similar steps as above to obtain
\begin{align}
\begin{split}
&\cj{a}_{1, N} N^{2 (\frac 32 - \al)} \E \bigg[ \int_{\T^3} W_N  (N^{\frac 32 - \al} z' - \ZZ_N) dx \bigg] = O (N^{\frac 32 - \al}) , \\
&\cj{a}_{1, N} N^{2 (\frac 32 - \al)} \E \bigg[ \int_{\T^3} (N^{\frac 32 - \al} z' - \ZZ_N)^2 dx \bigg] = O (N^{2 (\frac 32 - \al)}) .
\end{split}
\label{npos7}
\end{align}

Thus, combining \eqref{npos1}, \eqref{npos2}, \eqref{npos3}, \eqref{npos4}, \eqref{npos5}, \eqref{npos6}, and \eqref{npos7} and then using \eqref{npos}, we obtain
\begin{align*}
- &\log \big\| e^{- \mathcal{R}_N^\dia (u)} \big\|_{L^1 (\mu_\al)} \\
&\leq (2 \pi)^3 \sum_{j = 2}^m \cj{a}_{j, N} N^{4 (\frac 32 - \al)} (z')^{2j} + O (N^{3 (\frac 32 - \al)}) \\
&< - C N^{4 (\frac 32 - \al)}
\end{align*}

\noi
for some constant $C > 0$ and for all $N \in \N$ sufficiently large. This finishes the proof of Proposition~\ref{PROP:nonconv}.
\end{proof}

\section{Representation of the Gibbs measures and the fractional parabolic $\Phi_3^4$-model}
\label{SEC:heat}
In this section, our goal is to prove Theorem~\ref{THM:rep}, which provides good representations of both the truncated Gibbs measure $\varrho_N$ in \eqref{nuN} and the fractional $\Phi_3^4$-measure $\rho$ in \eqref{rho}. For this purpose, we study the Langevin dynamics for the measures $\varrho_N$ and $\rho$, which are given respectively by the stochastic fractional nonlinear heat equation on $\T^3$:
\begin{align}
\partial_t X_N + (1 - \Dl)^\al X_N + \pi_N V_N' (\pi_N X_N) = \sqrt{2} \xi^h
\label{SNLHN}
\end{align}

\noi
and its formal limit
\begin{align}
\partial_t X + (1 - \Dl)^\al X + a X + 4 \cj{a}_2 \wick{X^3} \, = \sqrt{2} \xi^h.
\label{SNLH}
\end{align}

\noi
Here, $\xi^h$ is a (two-sided) space-time white noise on $\R \times \T^3$ independent from $\xi$ in the equation \eqref{macro} (the superscript ``$h$'' stands for ``heat'') and $V_N$ is the nonlinearity given by \eqref{defVN}. To avoid confusion, we denote by $(\Om^h, \F^h, \mathbb{P}^h)$ the underlying probability space for $\xi^h$. The Wick ordering $\wick{X^3}$ in \eqref{SNLH} will be defined in a more precise manner in following subsections; see \eqref{defX3} below. We refer to \eqref{SNLHN} as the fractional parabolic macroscopic model and \eqref{SNLH} as the fractional parabolic $\Phi_3^4$-model.

\subsection{Global well-posedness of the fractional parabolic $\Phi_3^4$-model}
\label{SUB:SNLH}
In this subsection, we aim to prove global well-posedness for the fractional parabolic $\Phi_3^4$-model \eqref{SNLH}. For this purpose, we consider the following frequency truncated fractional parabolic $\Phi_3^4$-model:
\begin{align}
\dt X_N^\dagger + (1 - \Dl)^\al X_N^\dagger + a \pi_N X_N^\dagger + 4 \cj{a}_2 \pi_N \wick{(\pi_N X_N^\dagger)^3} \, = \sqrt{2} \xi^h.
\label{SNLHNd}
\end{align}

Let us define the following energy function:
\begin{align*}
E_N (u) = \int_{\T^3} \frac 12 |\jb{\nb}^\al u|^2 + \frac{a}{2} \wick{ (\pi_N u)^2 } + \, \cj{a}_2 \wick{ (\pi_N u)^4 } dx.
\end{align*}

\noi
Note that the truncated equation \eqref{SNLHNd} can be written as
\begin{align*}
\dt X_N^\dagger = - E_N' (X_N^\dagger) + \sqrt{2} \xi^h ,
\end{align*}

\noi
where $E_N'$ is the Fr\'echet derivative of $E_N$ with respect to the $L^2 (\T^3)$ norm.
In addition, the energy $E_N$ yields a truncated version of the fractional $\Phi_3^4$-measure given exactly by $\rho_N$ in \eqref{rhoN}. Thus, by classical results (see, for example, \cite{KS, Oks, DaP}), we have the following lemma.

\begin{lemma}
\label{LEM:inv_h}
Let $\al > 1$ and $N \in \N$. Then, the truncated fractional parabolic $\Phi_3^4$-model \eqref{SNLHNd} is almost surely globally well-posed with the random initial data distributed by the truncated Gibbs measure $\rho_N$ in \eqref{rhoN}. Moreover, $\rho_N$ is invariant under the resulting dynamics of \eqref{SNLHNd}.
\end{lemma}

We now consider the situation as $N \to \infty$. 
Let us first define some auxiliary terms. We let $\<1>^h$ be a stationary solution (i.e. with law constant in time) to the linear stochastic heat equation
\begin{align*}
\dt \<1>^h + (1 - \Dl)^\al \<1>^h = \sqrt{2} \xi^h.
\end{align*}

\noi
More precisely, for $t \geq 0$, we define
\begin{align}
\<1>^h (t) \deff \sqrt{2} \int_{- \infty}^t e^{- (t - t') \jb{\nb}^{2 \al}} d W^h (t') = \sum_{n \in \Z^3} \frac{\sqrt{2}}{(2 \pi)^{\frac 32}} \int_{- \infty}^t e^{- (t - t') \jb{n}^{2 \al}} d B_n^h (t') e^{in \cdot x},
\label{lin_h}
\end{align}

\noi
where $W^h$ is a cylindrical Wiener process on $L^2 (\T^3)$ given by
\begin{align*}
W^h (t) = \frac{1}{(2 \pi)^{\frac 32}} \sum_{n \in \Z^3} B_n^h (t) e^{i n \cdot x}
\end{align*}

\noi
with $B_n^h (t) = (2 \pi)^{- 3/2} \jb{\xi^h, \ind_{[0, t]} \cdot e^{in \cdot x}}_{t, x}$. Here, $\jb{\cdot, \cdot}_{t, x}$ denotes the duality pairing on $\R_+ \times \T^3$ and $\ind_{[0, t]} = - \ind_{[t, 0]}$ for $t < 0$. Note that $\{B_n^h\}_{n \in \Z^3}$ is a family of Gaussian processes satisfying the three properties at the beginning of Subsection~\ref{SUB:mult}. For any $N \in \N$, we denote $\<1>^h_N = \pi_N \<1>^h$. By independence and the Ito isometry, we obtain that
\begin{align}
\E \big[ |\<1>^h_N (t, x)|^2 \big] = \frac{1}{(2 \pi)^3} \sum_{\substack{n \in \Z^3 \\ |n| \leq N}} \frac{1}{\jb{n}^{2 \al}} = \s_N
\label{lin_hvar}
\end{align}

\noi
for any $t \geq 0$ and $x \in \T^3$, where $\s_N$ is as defined in \eqref{sigma}. For each $k \in \N$, we define the Wick power
\begin{align}
\begin{split}
\wick{(\<1>^h_N)^k} (t) &\deff H_k \big( \<1>^h_N (t); \s_N \big) \\
&\,= \sum_{\substack{n \in \Z^3 \\ |n| \leq kN}} \sum_{\substack{n_1 + \cdots + n_k = n \\ |n_1|, \dots, |n_k| \leq N}} \frac{2^{\frac k2}}{(2 \pi)^{\frac{3 k}{2}}} \\
&\quad \times \int_{- \infty}^t \prod_{\l = 1}^k e^{- (t - t_\l') \jb{n_\l}^{2 \al}} d B_{n_1}^h (t_1') \cdots d B_{n_k}^h (t_k') e^{in \cdot x},
\end{split}
\label{lin_hk}
\end{align}

\noi
where $H_k$ is the Hermite polynomial as given be \eqref{hermite} and the equality in the second line follows from Lemma~\ref{LEM:sto_ito}~(i). In particular, we denote
\begin{align}
\<2>_N^{\,h} \deff \, \wick{ (\<1>_N^h)^2 } \quad \text{and} \quad \<3>_N^{\,h} \deff \, \wick{ (\<1>_N^h)^3 }.
\label{quad_h}
\end{align}

\noi
We then define the cubic stochastic object $\<30>_N^{\,h}$ as
\begin{align}
\<30>_N^{\,h} (t) \deff \pi_N \int_{- \infty}^t e^{- (t - t') \jb{\nb}^{2 \al}} \<3>_N^{\,h} (t') dt',
\label{cub_h}
\end{align}

\noi
with $t \geq 0$.

We now proceed with the following second order expansion
\begin{align}
X_N^\dagger = \<1>^h - 4 \cj{a}_2 \<30>_N^{\,h} + Y_N^\dagger,
\label{sec_XN}
\end{align}

\noi
where $Y_N^\dagger$ is a remainder term satisfying
\begin{align}
\begin{cases}
\dt Y_N^\dagger + (1 - \Dl)^\al Y_N^\dagger \\
\qquad = - a \big( \<1>_N^h - 4 \cj{a}_2 \<30>_N^{\,h} + \pi_N Y_N^\dagger \big) - 12 \cj{a}_2 \pi_N \Big( \<2>_N^{\,h} \big( \pi_N Y_N^\dagger - 4 \cj{a}_2 \<30>_N^{\,h} \big) \Big) \\
\qquad \hphantom{=}- 12 \cj{a}_2 \pi_N \Big( \<1>_N^h \big( \pi_N Y_N^\dagger - 4 \cj{a}_2 \<30>_N^{\,h} \big)^2 \Big) - 4 \cj{a}_2 \pi_N \Big( \big( \pi_N Y_N^\dagger - 4 \cj{a}_2 \<30>_N^{\,h} \big)^3 \Big) \\
Y_N^\dagger|_{t = 0} = X_{0, N}^\dagger - \<1>^h (0) + 4 \cj{a}_2 \<30>_N^{\,h} (0),
\end{cases}
\label{SNLHYNd1}
\end{align}

\noi
where $X_{0, N}^\dagger = X_N^\dagger (0)$. Most of the terms on the right-hand-side of \eqref{SNLHNd} can be treated by deterministic estimates, except for $\<2>_N^{\,h} \cdot \<30>_N^{\,h}$. To deal with this term, we introduce the following quintic stochastic object:
\begin{align}
\<320>_N^{h} (t) \deff \pi_N \I^h \big( \<2>_N^{\,h} \cdot \<30>_N^{\,h} \big) (t) \deff \pi_N \int_0^t e^{- (t - t') \jb{\nb}^{2 \al}} \big( \<2>_N^{\,h} (t') \<30>_N^{\,h} (t') \big) dt'
\label{quin_h}
\end{align}

\noi
with $t \geq 0$. Then, we can write the \eqref{SNLHYNd1} in the following Duhamel formulation:
\begin{align}
\begin{split}
Y_N^\dagger &= e^{- t \jb{\nb}^{2 \al}}  Y_{0, N} - a \I^h \big( \<1>_N^h - 4 \cj{a}_2 \<30>_N^{\,h} + \pi_N Y_N^\dagger \big) + 48 \cj{a}_2^2 \<320>_N^h \\
&\quad - 12 \cj{a}_2 \pi_N \I^h \big( \<2>_N^{\,h} \cdot \pi_N Y_N^\dagger \big) 
- 12 \cj{a}_2 \pi_N \I^h \big( \<1>_N^h (\pi_N Y_N^\dagger)^2 \big) \\
&\quad + 96 \cj{a}_2^2 \pi_N \I^h \big( \<1>_N^h \cdot \<30>_N^{\,h} \cdot \pi_N Y_N^\dagger \big) 
- 192 \cj{a}_2^3 \pi_N \I^h \big( \<1>_N^h (\<30>_N^{\,h})^2 \big) \\
&\quad - 4 \cj{a}_2 \pi_N \I^h \big( (\pi_N Y_N^\dagger)^3 \big) + 48 \cj{a}_2^2 \pi_N \I^h \big( \<30>_N^{\,h} (\pi_N Y_N^\dagger)^2 \big) 
\\
&\quad - 192 \cj{a}_2^3 \pi_N \I^h \big( (\<30>_N^{\,h})^2 \pi_N Y_N^\dagger \big) 
+ 256 \cj{a}_2^4 \pi_N \I^h \big( (\<30>_N^{\,h})^3 \big),
\end{split}
\label{SNLHYNd}
\end{align}

\noi
where $Y_{0, N} = X_{0, N}^\dagger - \<1>^h (0) + \<30>_N^{\,h} (0)$.

We now establish the regularity properties of the above stochastic objects.
\begin{lemma}
\label{LEM:sto_h}
Let $1 < \al \leq \frac 98$, $\eps > 0$ be sufficiently small, and $T \geq 1$. We consider the symbols $Z_N$ and $Z$ along with their regularity $s \in \R$ as follows:
\begin{enumerate}
\item[(i)] $Z_N = \<1>_N^h$, $Z = \<1>^h$ as in \eqref{lin_h}, $s = \al - \frac 32 - \eps$; \smallskip

\item[(ii)] $Z_N = \<2>_N^{\,h}$ as in \eqref{quad_h}, $Z = \<2>^{\,h}$, $s = 2\al - 3 - \eps$; \smallskip

\item[(iii)] $Z_N = \<30>_N^{\,h}$ as in \eqref{cub_h}, $Z = \<30>^{\,h}$, $s = 5 \al - \frac 92 - \eps$; \smallskip

\item[(iv)] $Z_N = \<320>_N^h$ as in \eqref{quin_h}, $Z = \<320>^h$, $s = 4 \al - 3 - \eps$.
\end{enumerate} 

\noi
Then, $\{Z_N\}_{N \in \N}$ is $\mathbb{P}^h$-almost surely a Cauchy sequence in the space $C([0, T]; \C^s (\T^3))$, and we denote the limit by $Z$. We also have the tail estimate
\begin{align}
\mathbb{P}^h \big( \| Z_N \|_{C_T \C_x^s} > \ld \big) \leq C(T) e^{- c(T) \ld^\be}
\label{tail_h}
\end{align}

\noi
for any $\ld > 0$ and some constants $C(T), c(T), \be > 0$, uniformly in $N \in \N \cup \{\infty\}$, with the understanding that $Z_\infty = Z$. Moreover, there exists small $\ta > 0$ such that
\begin{align}
\mathbb{P}^h \big( N_2^\ta \| Z_{N_1} - Z_{N_2} \|_{C_T \C_x^s} > \ld \big) \leq C(T) e^{- c(T) \ld^\be}
\label{taildiff_h}
\end{align}

\noi
for any $\ld > 0$ and some constants $C(T), c(T), \be > 0$, uniformly in $N_1 \geq N_2 \geq 1$.
\end{lemma}

\begin{proof}
(i) We recall the difference operator $\dl_{t_0}$ in \eqref{dh}. For any $n \in \Z^3$ and $t, t+t_0 \in [0, T]$, by using the Ito isometry (Lemma~\ref{LEM:sto_ito}~(ii)) and the mean value theorem, we have
\begin{align*}
\E \Big[ \big| \ft{\<1>_N^h} (t, n) \big|^2 \Big] \leq \jb{n}^{- 2 \al}
\end{align*}

\noi
and
\begin{align*}
\E \Big[ \big| \dl_{t_0} \ft{\<1>_N^h} (t, n) \big|^2 \Big] \les \jb{n}^{- 2 \al + 2 \al \eps'} |t_0|^{\eps'}.
\end{align*}

\noi
for $\eps' > 0$ arbitrarily small. Also, for $N_1 \geq N_2 \geq 1$, we have
\begin{align*}
\E \Big[ \big| \ft{\<1>_{N_1}^h} (t, n) - \ft{\<1>_{N_2}^h} (t, n) \big|^2 \Big] \leq \ind_{\{|n| > N_2 \}}\jb{n}^{- 2 \al} \leq N_2^{- \ta} \jb{n}^{- 2 \al + \ta}
\end{align*}

\noi
and
\begin{align*}
\E \Big[ \big| \dl_{t_0} \ft{\<1>_{N_1}^h} (t, n) - \dl_{t_0} \ft{\<1>_{N_2}^h} (t, n) \big|^2 \Big] \les N_2^{- \g} \jb{n}^{- 2 \al + \g + 2 \al \eps'} |t_0|^{\eps'}
\end{align*}

\noi
for $\ta, \eps' > 0$ arbitrarily small. Thus, by Lemma~\ref{LEM:reg}, we obtain the desired convergence property and the tail estimates for $\<1>^h$.

\smallskip \noi
(ii) For simplicity below, we only prove the uniform-in-$N$ variance bound, since the other estimates are similar to those in part (i). For any $n \in \Z^3$ and $t \in [0, T]$, by \eqref{quad_h}, the Ito isometry (Lemma~\ref{LEM:sto_ito}~(ii)), and Lemma~\ref{LEM:conv}~(i) along with $1 < \al \leq \frac 98$, we have
\begin{align*}
\E \Big[ \big| \ft{\<2>_N^{\,h}} (t, n) \big|^2 \Big] \les \sum_{\substack{n_1, n_2 \in \Z^3 \\ n_1 + n_2 = n}} \frac{1}{\jb{n_1}^{2 \al} \jb{n_2}^{2 \al}} \les \jb{n}^{- 4 \al + 3},
\end{align*}

\noi
so that Lemma~\ref{LEM:reg} gives the desired properties for $\<2>^{\,h}$.

\smallskip \noi
(iii) For any $n \in \Z^3$ and $t \in [0, T]$, by \eqref{cub_h} and Fubini's theorem (Lemma~\ref{LEM:sto_fub}),\footnote{In view of the exponential decay in time of the heat kernel, the application of Fubini's theorem here on unbounded intervals can be easily justified by using an approximation (by bounded intervals) argument. This fact will also be used in the proof of Lemma~\ref{LEM:sto_h2} below.} we have
\begin{align}
\begin{split}
\ft{\<30>_N^{\,h}} (t, n) &= 2^{\frac 32} \ind_{\{|n| \leq N\}} \sum_{\substack{|n_1|, |n_2|, |n_3| \leq N \\ n_1 + n_2 + n_3 = n}} \int_{(- \infty, t]^3} \int_{\max \{s_1, s_2, s_3\}}^t e^{- (t - t') \jb{n}^{2 \al}} \\
&\quad \times \prod_{\l = 1}^3 e^{- (t' - s_\l) \jb{n_\l}^{2 \al}} d t' d B_{n_1}^0 (s_1) d B_{n_2}^0 (s_2) d B_{n_3}^0 (s_3)
\end{split}
\label{cub_hn}
\end{align}
 
\noi
in $L^2 (\Om^h)$, so that by the Ito isometry (Lemma~\ref{LEM:sto_ito}~(ii)),
\begin{align*}
\E \Big[ \big| \ft{\<30>_N^{\,h}} (t, n) \big|^2 \Big] &\les \sum_{\substack{|n_1|, |n_2|, |n_3| \leq N \\ n_1 + n_2 + n_3 = n}} \int_{(- \infty, t]^3} \bigg( \int_{\max \{s_1, s_2, s_3\}}^t e^{- (t - t') \jb{n}^{2 \al}} \\
&\quad \times \prod_{\l = 1}^3 e^{- (t' - s_\l) \jb{n_\l}^{2 \al}} d t'  \bigg)^2 d s_1 d s_2 d s_3 
\end{align*}

\noi
Since
\begin{align*}
\int_{\max \{s_1, s_2, s_3\}}^t e^{- (t - t') \jb{n}^{2 \al}} \prod_{\l = 1}^3 e^{- (t' - s_\l) \jb{n_\l}^{2 \al}} d t' \les \frac{1}{\max\{ \jb{n_1}^{2 \al}, \jb{n_2}^{2 \al}, \jb{n_3}^{2 \al} \}},
\end{align*}

\noi
we obtain from Lemma~\ref{LEM:conv}~(i) twice along with $1 < \al \leq \frac 98$ that
\begin{align*}
\E \Big[ \big| \ft{\<30>_N^{\,h}} (t, n) \big|^2 \Big] &\les \sum_{\substack{|n_1|, |n_2|, |n_3| \leq N \\ n_1 + n_2 + n_3 = n}} \frac{1}{\max\{ \jb{n_1}^{2 \al}, \jb{n_2}^{2 \al}, \jb{n_3}^{2 \al} \}} \\
&\quad \times \int_{- \infty}^t e^{- (t - t') \jb{n}^{2 \al}} \int_{( - \infty, t']^3} \prod_{\l = 1}^3 e^{- (t' - s_\l) \jb{n_\l}^{2 \al}} ds_1 ds_2 ds_3 dt' \\
&\les \sum_{\substack{n_1, n_2, n_3 \in \Z^3 \\ n_1 + n_2 + n_3 = n}} \frac{1}{\jb{n}^{4 \al} \jb{n_1}^{2 \al} \jb{n_2}^{2 \al} \jb{n_3}^{2 \al}} \\
&\les \jb{n}^{- 10 \al + 6},
\end{align*}

\noi
so that Lemma~\ref{LEM:reg} gives the desired properties for $\<30>^{\,h}$.

\smallskip \noi
(iv) By \eqref{quin_h} and Lemma~\ref{LEM:Schau}, we only need to show the properties of 
\begin{align}
\<32>_N^{h} \deff \pi_N \big( \<2>_N^{\,h} \cdot \<30>_N^{\,h} \big)
\label{heat32}
\end{align}

\noi
in the space $C_T \mathcal{C}^{2 \al - 3 - \frac{\eps}{2}}$. Let $n \in \Z^3$ and $t \in [0, T]$. From \eqref{quad_h}, \eqref{cub_hn}, and the product formula (Lemma~\ref{LEM:sto_prod}), we can decompose
\begin{align*}
\ft{\<32>_N^h} (t, n) = \ft{\<32>_N^h}^{(5)} (t, n) + \ft{\<32>_N^h}^{(3)} (t, n) + \ft{\<32>_N^h}^{(1)} (t, n),
\end{align*}

\noi
where $\ft{\<32>_N^h}^{(k)} \in \H_k$ for $k = 1, 3, 5$ with $\H_k$ being the $k$th Wiener chaos defined in Subsection~\ref{SUB:sto}.

For $k = 5$, by the Ito isometry (Lemma~\ref{LEM:sto_ito}~(ii)), \eqref{jen}, the computations in part~(ii) and part~(iii), and Lemma~\ref{LEM:conv}~(i) several times along with $1 < \al \leq \frac 98$, we obtain
\begin{align*}
\E \Big[ &\Big| \ft{\<32>_N^h}^{(5)} (t, n) \Big|^2 \Big] \\
&\les \sum_{\substack{|n_1|, \dots, |n_5| \leq N \\ n_1 + \cdots + n_5 = n}} \int_{( - \infty, t]^5} e^{- 2 (t - s_1) \jb{n_1}^{2 \al}} e^{- 2 (t - s_2) \jb{n_2}^{2 \al}} \\
&\quad \times \bigg( \int_{\max \{s_3, s_4, s_5\}}^t e^{- (t - t') \jb{n_3 + n_4 + n_5}^{2 \al}} \prod_{\l = 3}^5 e^{- (t' - s_\l) \jb{n_\l}^{2 \al}} d t'  \bigg)^2 d s_1 d s_2 d s_3 d s_4 d s_5 \\
&\les \sum_{\substack{n_1, \dots, n_5 \in \Z^3 \\ n_1 + \cdots + n_5 = n}} \frac{1}{\max \{ \jb{n_3}^{2 \al}, \jb{n_4}^{2 \al}, \jb{n_5}^{2 \al} \} } \bigg( \prod_{\l = 1}^5 \frac{1}{\jb{n_\l}^{2 \al}} \bigg) \frac{1}{ \jb{n_3 + n_4 + n_5}^{2 \al} }  \\
&\les \sum_{\substack{n', n'' \in \Z^3 \\ n' + n'' = n}} \frac{1}{\jb{n'}^{4 \al - 3} \jb{n''}^{10 \al - 6}} \\
&\les \jb{n}^{- 4 \al + 3 + \frac{\eps}{2}},
\end{align*}

\noi
which satisfies the desired properties in view of Lemma~\ref{LEM:reg}.

For $k = 3$, by symmetry we may assume that $n_1 + n_3 = 0$, so that
\begin{align*}
\ft{\<32>_N^h}^{(3)} (t, n) &= C \ind_{\{|n| \leq N\}} \sum_{\substack{|n_2|, |n_4|, |n_5| \leq N \\ n_2 + n_4 + n_5 = n}} \sum_{\substack{n_1 \in \Z^3 \\ |n_1| \leq N}} \int_{(- \infty, t]^3} e^{- (t - s_2) \jb{n_2}^{2 \al}} \\
&\quad \times \int_{- \infty}^t e^{- (t - s_1) \jb{n_1}^{2 \al}} \int_{\max \{s_1, s_4, s_5\}}^t e^{- (t - t') \jb{n_1 + n_4 + n_5}^{2 \al}} \\
&\quad \times \prod_{\l = 1, 4, 5} e^{- (t' - s_\l) \jb{n_\l}^{2 \al}} d t' d s_1 d B_{n_2}^h (s_2) d B_{n_4}^h (s_4) d B_{n_5}^h (s_5)
\end{align*}

\noi
for some constant $C > 0$.
Thus, by the Ito isometry (Lemma~\ref{LEM:sto_ito}~(ii)), \eqref{jen}, and similar computations as in part (ii) and (iii) along with Lemma~\ref{LEM:conv}~(i) several times along with $1 < \al \leq \frac 98$, we obtain
\begin{align*}
\E \Big[ &\Big| \ft{\<32>_N^h}^{(3)} (t, n) \Big|^2 \Big] \\
&\les \sum_{\substack{|n_2|, |n_4|, |n_5| \leq N \\ n_2 + n_4 + n_5 = n}} \int_{(- \infty, t]^3} e^{- 2 (t - s_2) \jb{n_2}^{2 \al}} \bigg( \sum_{\substack{n_1 \in \Z^3 \\ |n_1| \leq N}}  \int_{- \infty}^t e^{- (t - s_1) \jb{n_1}^{2 \al}}  \\
&\quad \times \int_{\max \{s_1, s_4, s_5\}}^t e^{- (t - t') \jb{n_1 + n_4 + n_5}^{2 \al}} \prod_{\l = 1, 4, 5} e^{- (t' - s_\l) \jb{n_\l}^{2 \al}} d t' d s_1 \bigg)^2 d s_2 d s_4 d s_5 \\
&\les \sum_{\substack{|n_2|, |n_4|, |n_5| \leq N \\ n_2 + n_4 + n_5 = n}} \bigg( \sum_{n_1 \in \Z^3} \frac{1}{\max \{ \jb{n_1}^{2 \al}, \jb{n_4}^{2 \al}, \jb{n_5}^{2 \al}  \} \jb{n_1}^{2 \al}} \bigg) \\
&\quad \times \bigg( \sum_{n_1 \in \Z^3} \frac{1}{\jb{n_1}^{2 \al} \jb{n_2}^{2 \al} \jb{n_4}^{2 \al} \jb{n_5}^{2 \al} \jb{n_1 + n_4 + n_5}^{2 \al}}  \bigg) \\
&\les \sum_{\substack{n_2, n_4, n_5 \in \Z^3 \\ n_2 + n_4 + n_5 = n}} \frac{1}{\jb{n_2}^{2 \al} \jb{n_4}^{2 \al} \jb{n_5}^{2 \al} \jb{n_4 + n_5}^{8 \al - 6 - \eps}} \\
&\les \jb{n}^{- 2 \al + \eps},
\end{align*}

\noi
which satisfies the desired properties in view of Lemma~\ref{LEM:reg}.

For $k = 1$, by symmetry we may assume that $n_1 + n_3 = 0$ and $n_2 + n_4 = 0$, so that
\begin{align*}
\ft{\<32>_N^h}^{(1)} (t, n) &= C \ind_{\{|n| \leq N\}} \sum_{\substack{n_1, n_2 \in \Z^3 \\ |n_1|, |n_2| \leq N}} \int_{- \infty}^t \int_{(- \infty, t]^2} e^{- (t - s_1) \jb{n_1}^{2 \al}} e^{- (t - s_2) \jb{n_2}^{2 \al}} \\
&\quad \times \int_{\max \{ s_1, s_2, s_5 \}}^t e^{- (t - t') \jb{n_1 + n_2 + n}^{2 \al}} \prod_{\l = 1, 2, 5} e^{- (t' - s_\l) \jb{n_\l}^{2 \al}} dt' d s_1 d s_2 d B_n^0 (s_5)
\end{align*}

\noi
for some constant $C > 0$.
Thus, by the Ito isometry (Lemma~\ref{LEM:sto_ito}~(ii)) and similar computations as in part (ii) and (iii) along with Lemma~\ref{LEM:conv}~(i) several times along with $1 < \al \leq \frac 98$, we obtain
\begin{align*}
\E \Big[ &\Big| \ft{\<32>_N^h}^{(1)} (t, n) \Big|^2 \Big] \\
&\les \int_{- \infty}^t \bigg( \sum_{\substack{n_1, n_2 \in \Z^3 \\ |n_1|, |n_2| \leq N}} \int_{( - \infty, t ]^2} e^{- (t - s_1) \jb{n_1}^{2 \al}} e^{- (t - s_2) \jb{n_2}^{2 \al}} \int_{\max \{ s_1, s_2, s_5 \}}^t e^{- (t' - s_5) \jb{n}^{2 \al}} \\ 
&\quad \times  e^{- (t - t') \jb{n_1 + n_2 + n}^{2 \al}} \prod_{\l = 1, 2} e^{- (t' - s_\l) \jb{n_\l}^{2 \al}} dt' d s_1 d s_2 \bigg)^2 d s_5 \\
&\les \bigg( \sum_{n_1, n_2 \in \Z^3} \frac{1}{\max \{ \jb{n_1}^{2 \al}, \jb{n_2}^{2 \al}, \jb{n}^{2 \al} \} \jb{n_1}^{2 \al} \jb{n_2}^{2 \al} } \bigg) \\
&\quad \times \bigg( \sum_{n_1, n_2 \in \Z^3} \frac{1}{\jb{n_1}^{2 \al} \jb{n_2}^{2 \al} \jb{n}^{2 \al} \jb{n_1 + n_2 + n}^{2 \al}} \bigg) \\
&\les \jb{n}^{-  8 \al + 6},
\end{align*}

\noi
which satisfies the desired properties in view of Lemma~\ref{LEM:reg}. This finishes the proof of part~(iv).
\end{proof}

We now consider uniform-in-$N$ local well-posedness  for the equation \eqref{SNLHYNd} for $Y_N^\dagger$. 
\begin{lemma}
\label{LEM:LWP_h}
Let $1 < \al \leq \frac 98$, $\eps > 0$ be sufficiently small, $T \geq 1$, and $N \in \N$. Suppose that we have the enhanced data set
\begin{align}
\pmb{Z}_N \deff \big( \<1>_N^h, \<2>_N^{\,h}, \<30>_N^{\,h}, \<320>_N^{h} \big) \in \mathcal{Z}_T^{\al, \eps},
\label{defZN1}
\end{align}

\noi
where
\begin{align}
\mathcal{Z}_T^{\al, \eps} \deff C_T \mathcal{C}^{\al - \frac 32 - \eps} \times C_T \mathcal{C}^{2\al - 3 - \eps} \times C_T \mathcal{C}^{5\al - \frac 92 - \eps} \times C_T \mathcal{C}^{4\al - 3 - \eps}.
\label{defZTae}
\end{align}

\smallskip \noi
\textup{(i)} Given any $Y_{0, N} \in \mathcal{C}^{\al - \frac 32 - \eps}$, there exist $0 < T_0 \leq 1$ \textup{(}depending only on $\| \pmb{Z}_N \|_{\mathcal{Z}_T^{\al, \eps}}$ and $\| Y_{0, N} \|_{\mathcal{C}^{\al - \frac 32 - \eps}}$\textup{)} and a unique solution $Y_N^\dagger$ to the equation \eqref{SNLHYNd} on $[0, T_0]$ satisfying $Y_N^\dagger |_{t = 0} = Y_{0, N}$ in the class $\mathcal{Y}_{T_0}^{\al, \eps}$ defined by
\begin{align}
\| Y \|_{\mathcal{Y}_{T_0}^{\al, \eps}} \deff \| Y \|_{C_{T_0} \mathcal{C}_x^{\al - \frac 32 - \eps}} + \sup_{0 \leq t \leq T_0} t^{- \frac{\al}{2} + \frac 54 + \eps} \| Y(t) \|_{\mathcal{C}^{1 + \eps}}.
\label{defYtae}
\end{align}

\noi
Furthermore, we have the bound
\begin{align*}
\| Y_N^\dagger \|_{\mathcal{Y}_{T_0}^{\al, \eps}} \les \| Y_{0, N} \|_{\mathcal{C}^{\al - \frac 32 - \eps}} + \| \pmb{Z}_N \|_{\mathcal{Z}_T^{\al, \eps}}^3 ,
\end{align*}

\noi
and solution $Y_N^\dagger$ depends continuously on $\pmb{Z}_N$ in $\mathcal{Z}_T^{\al, \eps}$ and on $Y_{0, N}$ in $\mathcal{C}^{\al - \frac 32 - \eps}$.

\smallskip \noi
\textup{(ii)} Given any $Y_{0, N} \in \mathcal{C}^{1 + \eps}$, there exist $0 < T_0 \leq 1$ \textup{(}depending only on $\| \pmb{Z}_N \|_{\mathcal{Z}_T^{\al, \eps}}$ and $\| Y_{0, N} \|_{\mathcal{C}^{1 + \eps}}$\textup{)} and a unique solution $Y_N^\dagger$ to the equation \eqref{SNLHYNd} on $[0, T_0]$ satisfying $Y_N^\dagger |_{t = 0} = Y_{0, N}$ in the class $C([0, T_0]; \mathcal{C}^{1 + \eps} (\T^3))$. Furthermore, we have the bound
\begin{align*}
\| Y_N^\dagger \|_{C_{T_0} \mathcal{C}_x^{1 + \eps}} \les \| Y_{0, N} \|_{\mathcal{C}^{1 + \eps}} + \| \pmb{Z}_N \|_{\mathcal{Z}_T^{\al, \eps}}^3.
\end{align*}
\end{lemma}

\begin{proof}
(i) By the interpolation estimate in Lemma~\ref{LEM:Bes}~(iii), we have
\begin{align}
\begin{split}
\sup_{0 \leq t \leq T_0} t^{- \frac{\al}{2} + 1} \| Y (t) \|_{\C^{\frac 12 - \eps}} 
&\les \sup_{0 \leq t \leq T_0} \| Y (t) \|_{\C^{\al - \frac 32 - \eps}}^{\frac{1 + 4 \eps}{- 2 \al + 5 + 4 \eps}} \big( t^{- \frac{\al}{2} + \frac 54 + \eps} \| Y (t) \|_{\C^{1 + \eps}} \big)^{\frac{- 2 \al + 4}{- 2 \al + 5 + 4 \eps}} \\
&\leq \| Y \|_{\mathcal{Y}_{T_0}^{\al, \eps}} 
\end{split}
\label{interp1}
\end{align}

\noi
and
\begin{align}
\begin{split}
\sup_{0 \leq t \leq T_0} t^{- \frac{\al}{2} + \frac 34 + \eps} \| Y (t) \|_{\C^{\eps}} 
&\les \sup_{0 \leq t \leq T_0} \| Y (t) \|_{\C^{\al - \frac 32 - \eps}}^{\frac{2}{- 2 \al + 5 + 4 \eps}} \big( t^{- \frac{\al}{2} + \frac 54 + \eps} \| Y (t) \|_{\C^{1 + \eps}} \big)^{\frac{- 2 \al + 3 + 4 \eps}{- 2 \al + 5 + 4 \eps}} \\
&\leq \| Y \|_{\mathcal{Y}_{T_0}^{\al, \eps}} .
\end{split}
\label{interp2}
\end{align}

\noi
By using the Schauder estimate (Lemma~\ref{LEM:Schau}) (since we always have $0 < t \leq 1$, we use $t^{- \frac{s_1 - s_2}{2 \alpha}} \leq t^{- \frac{s_1 - s_2}{2}}$ for simplicity), the product estimate (Lemma~\ref{LEM:Bes}~(iv)), \eqref{interp1}, and \eqref{interp2}, we obtain the following estimates in the $C_{T_0} \mathcal{C}^{\al - \frac 32 - \eps}$-norm (below we omit some similar steps):
\begin{align*}
&\big\| e^{- t \jb{\nb}^{2 \al}} Y_{0, N} \big\|_{C_{T_0} \mathcal{C}_x^{\al - \frac 32 - \eps}} \les \| Y_{0, N} \|_{\mathcal{C}^{\al - \frac 32 - \eps}},
\\
&\big\| \I^h \big( \<1>_N^h - \<30>_N^{\,h} + \pi_N Y_N^\dagger \big) \big\|_{C_{T_0} \mathcal{C}_x^{\al - \frac 32 - \eps}} \les T_0 \Big( \| \pmb{Z}_N \|_{\mathcal{Z}_T^{\al, \eps}} + \| Y_N^\dagger \|_{\mathcal{Y}_{T_0}^{\al, \eps}} \Big),
\\
&\big\| \<320>_N^h \big\|_{C_{T_0} \mathcal{C}_x^{\al - \frac 32 - \eps}} \les \| \pmb{Z}_N \|_{\mathcal{Z}_T^{\al, \eps}},
\\
&\big\| \I^h \big( \<2>_N^{\,h} \cdot \pi_N Y_N^\dagger \big) \big\|_{C_{T_0} \mathcal{C}_x^{\al - \frac 32 - \eps}}  \\
&\quad \les \sup_{0 \leq t \leq T_0} \int_0^{t} \frac{1}{(t - t')^{- \frac{\al}{2} + \frac 34}} \| \<2>_N^{\,h} (t') \cdot \pi_N Y_N^\dagger (t') \|_{\mathcal{C}^{2 \al - 3 - \eps}} dt'  \\
&\quad \les \sup_{0 \leq t \leq T_0} \int_0^{t} \frac{1}{(t - t')^{- \frac{\al}{2} + \frac 34} (t')^{- \frac{\al}{2} + \frac 54 + \eps}} dt' \cdot \| \<2>_N^{\,h} \|_{C_T \mathcal{C}_x^{2 \al - 3 - \eps}} \cdot \sup_{0 \leq t \leq T_0} t^{- \frac{\al}{2} + \frac 54 + \eps} \| Y_N^\dagger (t) \|_{\mathcal{C}^{1 + \eps}} \\
&\quad \les T_0^{\al - 1 - \eps} \| \pmb{Z}_N \|_{\mathcal{Z}_T^{\al, \eps}} \| Y_N^\dagger \|_{\mathcal{Y}_{T_0}^{\al, \eps}},
\\
&\big\| \I^h \big( \<1>_N^h (\pi_N Y_N^\dagger)^2 \big) \big\|_{C_{T_0} \mathcal{C}_x^{\al - \frac 32 - \eps}} \\
&\quad \les \sup_{0 \leq t \leq T_0} \int_0^{t} \| \<1>_N^h (t') \cdot \pi_N Y_N^\dagger (t') \cdot \pi_N Y_N^\dagger (t') \|_{\mathcal{C}^{\al - \frac 32 - \eps}} dt' \\
&\quad \les \sup_{0 \leq t \leq T_0} \int_0^{t} \frac{1}{(t')^{- \al + 2}} dt' \cdot \| \<1>_N^h \|_{C_T \mathcal{C}_x^{\al - \frac 32 - \eps}} \cdot \Big( \sup_{0 \leq t \leq T_0} t^{- \frac{\al}{2} + 1} \| Y_N^\dagger (t) \|_{\mathcal{C}^{\frac 12 - \eps}} \Big)^2 \\
&\quad \les T_0^{\al - 1} \| \pmb{Z}_N \|_{\mathcal{Z}_T^{\al, \eps}} \| Y_N^\dagger \|_{\mathcal{Y}_{T_0}^{\al, \eps}}^2, 
\\
&\big\| \I^h \big( \<1>_N^h \cdot \<30>_N^{\,h} \cdot \pi_N Y_N^\dagger \big) \big\|_{C_{T_0} \mathcal{C}_x^{\al - \frac 32 - \eps}} \les T_0^{\frac{\al}{2}} \| \pmb{Z}_N \|_{\mathcal{Z}_T^{\al, \eps}}^2 \| Y_N^\dagger \|_{\mathcal{Y}_{T_0}^{\al, \eps}},
\\
&\big\| \I^h \big( \<1>_N^h ( \<30>_N^{\,h} )^2 \big) \big\|_{C_{T_0} \mathcal{C}_x^{\al - \frac 32 - \eps}} \les T_0 \| \pmb{Z}_N \|_{\mathcal{Z}_T^{\al, \eps}}^3,
\\
&\big\| \I^h \big( (\pi_N Y_N^\dagger)^3 \big) \big\|_{C_{T_0} \mathcal{C}_x^{\al - \frac 32 - \eps}} \\
&\quad \les \sup_{0 \leq t \leq T_0} \int_0^{t} \big\| (\pi_N Y_N^\dagger (t'))^3 \big\|_{\mathcal{C}^{\al - \frac 32 - \eps}} dt' \\
&\quad \les \sup_{0 \leq t \leq T_0} \int_0^{t} \frac{1}{(t')^{- \frac{3 \al}{2} + \frac 94 + 3 \eps}} dt' \cdot \Big( \sup_{0 \leq t \leq T_0} t^{- \frac{\al}{2} + \frac 34 + \eps} \| Y_N^\dagger (t) \|_{\mathcal{C}^\eps} \Big)^3 \\
&\quad \les T_0^{\frac{3 \al}{2} - \frac 54 - 3 \eps} \| Y_N^\dagger \|_{\mathcal{Y}_{T_0}^{\al, \eps}}^3,
\\
&\big\| \I^h \big( \<30>_N^{\,h} (\pi_N Y_N^\dagger)^2 \big) \big\|_{C_{T_0} \mathcal{C}_x^{\al - \frac 32 - \eps}} \les T_0^{\al - \frac 12 - 2 \eps} \| \pmb{Z}_N \|_{\mathcal{Z}_T^{\al, \eps}} \| Y_N^\dagger \|_{\mathcal{Y}_{T_0}^{\al, \eps}}^2,
\\
&\big\| \I^h \big( (\<30>_N^{\,h})^2 \pi_N Y_N^\dagger \big) \big\|_{C_{T_0} \mathcal{C}_x^{\al - \frac 32 - \eps}} \les T_0^{\frac{\al}{2} + \frac 14 - \eps} \| \pmb{Z}_N \|_{\mathcal{Z}_T^{\al, \eps}}^2 \| Y_N^\dagger \|_{\mathcal{Y}_{T_0}^{\al, \eps}},
\\
&\big\| \I^h \big( (\<30>_N^{\,h})^3 \big) \big\|_{C_{T_0} \mathcal{C}_x^{\al - \frac 32 - \eps}} \les T_0 \| \pmb{Z}_N \|_{\mathcal{Z}_T^{\al, \eps}}^3.
\end{align*}

\noi
and also the following estimates in the $\mathcal{C}^{1 + \eps}$-norm:
\begin{align*}
&\sup_{0 \leq t \leq T_0} t^{- \frac{\al}{2} + \frac 54 + \eps} \big\| e^{- t \jb{\nb}^{2 \al}} Y_{0, N} \big\|_{\mathcal{C}^{1 + \eps}} \les \| Y_{0, N} \|_{\mathcal{C}^{\al - \frac 32 - \eps}},
\\
&\sup_{0 \leq t \leq T_0} t^{- \frac{\al}{2} + \frac 54 + \eps} \big\| \I^h (\<1>_N^h) (t) \big\|_{\mathcal{C}^{1 + \eps}} \\
&\quad \les \sup_{0 \leq t \leq T_0} \int_0^t \frac{t^{- \frac{\al}{2} + \frac 54 + \eps}}{(t - t')^{- \frac{\al}{2} + \frac 54 + \eps}} \| \<1>_N^h (t') \|_{\mathcal{C}^{\al - \frac 32 - \eps}} dt' \\
&\quad \les T_0 \| \pmb{Z}_N \|_{\mathcal{Z}_T^{\al, \eps}},
\\
&\sup_{0 \leq t \leq T_0} t^{- \frac{\al}{2} + \frac 54 + \eps} \big\| \I^h \big( \<30>_N^{\,h} \big) (t) \big\|_{\mathcal{C}^{1 + \eps}} \les T_0^{2 \al - \frac 12} \| \pmb{Z}_N \|_{\mathcal{Z}_T^{\al, \eps}},
\\
&\sup_{0 \leq t \leq T_0} t^{- \frac{\al}{2} + \frac 54 + \eps} \big\| \I^h (\pi_N Y_N^\dagger) (t) \big\|_{\mathcal{C}^{1 + \eps}} \\
&\quad \les \sup_{0 \leq t \leq T_0} \int_0^t \frac{t^{- \frac{\al}{2} + \frac 54 + \eps}}{(t')^{- \frac{\al}{2} + \frac 54 + \eps}} dt' \cdot \sup_{0 \leq t \leq T_0} t^{- \frac{\al}{2} + \frac 54 + \eps} \| Y_N^\dagger (t) \|_{\mathcal{C}^{1 + \eps}} \\
&\quad \les T_0 \| Y_N^\dagger \|_{\mathcal{Y}_{T_0}^{\al, \eps}},
\\
&\sup_{0 \leq t \leq T_0} t^{- \frac{\al}{2} + \frac 54 + \eps} \big\| \<320>_N^h (t) \big\|_{\mathcal{C}^{1 + \eps}} \les T_0^{- \frac{\al}{2} + \frac 54 + \eps} \| \pmb{Z}_N \|_{\mathcal{Z}_T^{\al, \eps}},
\\
&\sup_{0 \leq t \leq T_0} t^{- \frac{\al}{2} + \frac 54 + \eps} \big\| \I^h \big( \<2>_N^{\,h} \cdot \pi_N Y_N^\dagger \big) (t) \big\|_{\mathcal{C}^{1 + \eps}}  \\
&\quad \les \sup_{0 \leq t \leq T_0} \int_0^{t} \frac{t^{- \frac{\al}{2} + \frac 54 + \eps}}{(t - t')^{- \al + 2 + \eps}} \| \<2>_N^{\,h} (t') \cdot \pi_N Y_N^\dagger (t') \|_{\mathcal{C}^{2 \al - 3 - \eps}} dt'  \\
&\quad \les \sup_{0 \leq t \leq T_0} \int_0^{t} \frac{t^{- \frac{\al}{2} + \frac 54 + \eps}}{(t - t')^{- \al + 2 + \eps} (t')^{- \frac{\al}{2} + \frac 54 + \eps}} dt' \cdot \| \<2>_N^{\,h} \|_{C_T \mathcal{C}_x^{2 \al - 3 - \eps}} \cdot \sup_{0 \leq t \leq T_0} t^{- \frac{\al}{2} + \frac 54 + \eps} \| Y_N^\dagger (t) \|_{\mathcal{C}^{1 + \eps}} \\
&\quad \les T_0^{\al - 1 - \eps} \| \pmb{Z}_N \|_{\mathcal{Z}_T^{\al, \eps}} \| Y_N^\dagger \|_{\mathcal{Y}_{T_0}^{\al, \eps}},
\\
&\sup_{0 \leq t \leq T_0} t^{- \frac{\al}{2} + \frac 54 + \eps} \big\| \I^h \big( \<1>_N^h (\pi_N Y_N^\dagger)^2 \big) (t) \big\|_{\mathcal{C}^{1 + \eps}} \\
&\quad \les \sup_{0 \leq t \leq T_0} \int_0^{t} \frac{t^{- \frac{\al}{2} + \frac 54 + \eps}}{(t - t')^{- \frac{\al}{2} + \frac 54 + \eps}} \| \<1>_N^h (t') \cdot \pi_N Y_N^\dagger (t') \cdot \pi_N Y_N^\dagger (t') \|_{\mathcal{C}^{\al - \frac 32 - \eps}} dt' \\
&\quad \les \sup_{0 \leq t \leq T_0} \int_0^{t} \frac{t^{- \frac{\al}{2} + \frac 54 + \eps}}{(t - t')^{- \frac{\al}{2} + \frac 54 + \eps} (t')^{- \al + 2}} dt' \cdot \| \<1>_N^h \|_{C_T \mathcal{C}_x^{\al - \frac 32 - \eps}} \cdot \Big( \sup_{0 \leq t \leq T_0} t^{- \frac{\al}{2} + 1} \| Y_N^\dagger (t) \|_{\mathcal{C}^{\frac 12 - \eps}} \Big)^2 \\
&\quad \les T_0^{\al - 1} \| \pmb{Z}_N \|_{\mathcal{Z}_T^{\al, \eps}} \| Y_N^\dagger \|_{\mathcal{Y}_{T_0}^{\al, \eps}}^2, 
\\
&\sup_{0 \leq t \leq T_0} t^{- \frac{\al}{2} + \frac 54 + \eps} \big\| \I^h \big( \<1>_N^h \cdot \<30>_N^{\,h} \cdot \pi_N Y_N^\dagger \big) (t) \big\|_{\mathcal{C}^{1 + \eps}} \les T_0^{\frac{\al}{2}} \| \pmb{Z}_N \|_{\mathcal{Z}_T^{\al, \eps}}^2 \| Y_N^\dagger \|_{\mathcal{Y}_{T_0}^{\al, \eps}},
\\
&\sup_{0 \leq t \leq T_0} t^{- \frac{\al}{2} + \frac 54 + \eps} \big\| \I^h \big( \<1>_N^h ( \<30>_N^{\,h} )^2 \big) (t) \big\|_{\mathcal{C}^{1 + \eps}} \les T_0 \| \pmb{Z}_N \|_{\mathcal{Z}_T^{\al, \eps}}^3,
\\
&\sup_{0 \leq t \leq T_0} t^{- \frac{\al}{2} + \frac 54 + \eps} \big\| \I^h \big( (\pi_N Y_N^\dagger)^3 \big) (t) \big\|_{\mathcal{C}^{1 + \eps}} \\
&\quad \les \sup_{0 \leq t \leq T_0} \int_0^{t} \frac{t^{- \frac{\al}{2} + \frac 54 + \eps}}{(t - t')^{\frac 12}} \big\| (\pi_N Y_N^\dagger (t'))^3 \big\|_{\mathcal{C}^{\eps}} dt' \\
&\quad \les \sup_{0 \leq t \leq T_0} \int_0^{t} \frac{t^{- \frac{\al}{2} + \frac 54 + \eps}}{(t - t')^{\frac 12} (t')^{- \frac{3 \al}{2} + \frac 94 + 3 \eps}} dt' \cdot \Big( \sup_{0 \leq t \leq T_0} t^{- \frac{\al}{2} + \frac 34 + \eps} \| Y_N^\dagger (t) \|_{\mathcal{C}^\eps} \Big)^3 \\
&\quad \les T_0^{\al - \frac 12 - 2 \eps} \| Y_N^\dagger \|_{\mathcal{Y}_{T_0}^{\al, \eps}}^3,
\\
&\sup_{0 \leq t \leq T_0} t^{- \frac{\al}{2} + \frac 54 + \eps} \big\| \I^h \big( \<30>_N^{\,h} (\pi_N Y_N^\dagger)^2 \big) (t) \big\|_{\mathcal{C}^{1 + \eps}} \les T_0^{\frac{\al}{2} + \frac 14 - \eps} \| \pmb{Z}_N \|_{\mathcal{Z}_T^{\al, \eps}} \| Y_N^\dagger \|_{\mathcal{Y}_{T_0}^{\al, \eps}}^2,
\\
&\sup_{0 \leq t \leq T_0} t^{- \frac{\al}{2} + \frac 54 + \eps} \big\| \I^h \big( (\<30>_N^{\,h})^2 \pi_N Y_N^\dagger \big) (t) \big\|_{\mathcal{C}^{1 + \eps}} \les T_0 \| \pmb{Z}_N \|_{\mathcal{Z}_T^{\al, \eps}}^2 \| Y_N^\dagger \|_{\mathcal{Y}_{T_0}^{\al, \eps}},
\\
&\sup_{0 \leq t \leq T_0} t^{- \frac{\al}{2} + \frac 54 + \eps} \big\| \I^h \big( (\<30>_N^{\,h})^3 \big) (t) \big\|_{\mathcal{C}^{1 + \eps}} \les T_0^{- \frac{\al}{2} + \frac 74 + \eps} \| \pmb{Z}_N \|_{\mathcal{Z}_T^{\al, \eps}}^3.
\end{align*}

\noi
By denoting the right-hand-side of \eqref{SNLHYNd} by $\G^h [Y_N^\dagger]$ and gathering the above estimates, we obtain
\begin{align*}
\| \G^h [Y_N^\dagger] \|_{\mathcal{Y}_{T_0}^{\al, \eps}} \les \| Y_{0, N} \|_{\mathcal{C}^{\al - \frac 32 - \eps}} + \| \pmb{Z}_N \|_{\mathcal{Z}_T^{\al, \eps}}^3 + T_0^\ta \| Y_N^\dagger \|_{\mathcal{Y}_{T_0}^{\al, \eps}}^3
\end{align*}

\noi
for some $\ta > 0$. We can also use similar steps to obtain the following difference estimate:
\begin{align*}
\| \G^h [Y_N^{\dagger, 1}] - \G^h [Y_N^{\dagger, 2}] \|_{\mathcal{Y}_{T_0}^{\al, \eps}} \les T_0^\ta \Big( \| \pmb{Z}_N \|_{\mathcal{Z}_T^{\al, \eps}}^2 + \| Y_N^{\dagger, 1} \|_{\mathcal{Y}_{T_0}^{\al, \eps}}^2 + \| Y_N^{\dagger, 2} \|_{\mathcal{Y}_{T_0}^{\al, \eps}}^2  \Big) \| Y_N^{\dagger, 1} - Y_N^{\dagger, 2} \|_{\mathcal{Y}_{T_0}^{\al, \eps}}.
\end{align*}

\noi
Thus, by taking
\begin{align*}
0 < T_0 \ll \frac{1}{\big( \| Y_{0, N}^\dagger \|_{\mathcal{C}^{\al - \frac 32 - \eps}}^2 + \| \pmb{Z}_N \|_{\mathcal{Z}_T^{\al, \eps}}^6 \big)^{\frac{1}{\ta}}},
\end{align*}

\noi
the local well-posedness statement follows from a standard contraction argument.

\smallskip \noi
(ii) Using similar and simpler steps in part (i), we obtain the following estimates:
\begin{align*}
&\big\| e^{- t \jb{\nb}^{2 \al}} Y_{0, N} \big\|_{C_{T_0} \mathcal{C}_x^{1 + \eps}} \les \| Y_{0, N} \|_{\mathcal{C}_x^{1 + \eps}},
\\
&\big\| \I^h (\<1>_N^h) \big\|_{C_{T_0}\mathcal{C}_x^{1 + \eps}} \les T_0^{\frac{\al}{2} - \frac 14 - \eps} \| \pmb{Z}_N \|_{\mathcal{Z}_T^{\al, \eps}},
\\
&\big\| \I^h \big( \<30>_N^{\,h} \big) \big\|_{C_{T_0} \mathcal{C}_x^{1 + \eps}} \les T_0^{\frac{5 \al}{2} - \frac 74 - \eps} \| \pmb{Z}_N \|_{\mathcal{Z}_T^{\al, \eps}},
\\
&\big\| \I^h (\pi_N Y_N^\dagger) \big\|_{C_{T_0} \mathcal{C}_x^{1 + \eps}} \les T_0 \| Y_N^\dagger \|_{C_{T_0} \mathcal{C}_x^{1 + \eps}},
\\
&\big\| \<320>_N^h \big\|_{C_{T_0} \mathcal{C}_x^{1 + \eps}} \les \| \pmb{Z}_N \|_{\mathcal{Z}_T^{\al, \eps}},
\\
&\big\| \I^h \big( \<2>_N^{\,h} \cdot \pi_N Y_N^\dagger \big) \big\|_{C_{T_0} \mathcal{C}_x^{1 + \eps}} \les T_0^{\al - 1 - \eps} \| \pmb{Z}_N \|_{\mathcal{Z}_T^{\al, \eps}} \| Y_N^\dagger \|_{C_{T_0} \mathcal{C}_x^{1 + \eps}},
\\
&\big\| \I^h \big( \<1>_N^h (\pi_N Y_N^\dagger)^2 \big) \big\|_{C_{T_0} \mathcal{C}_x^{1 + \eps}} \les T_0^{\frac{\al}{2} - \frac 14 - \eps} \| \pmb{Z}_N \|_{\mathcal{Z}_T^{\al, \eps}} \| Y_N^\dagger \|_{C_{T_0} \mathcal{C}_x^{1 + \eps}}^2, 
\\
&\big\| \I^h \big( \<1>_N^h \cdot \<30>_N^{\,h} \cdot \pi_N Y_N^\dagger \big) \big\|_{C_{T_0} \mathcal{C}_x^{1 + \eps}} \les T_0^{\frac{\al}{2} - \frac 14 - \eps} \| \pmb{Z}_N \|_{\mathcal{Z}_T^{\al, \eps}}^2 \| Y_N^\dagger \|_{C_{T_0} \mathcal{C}_x^{1 + \eps}},
\\
&\big\| \I^h \big( \<1>_N^h ( \<30>_N^{\,h} )^2 \big) \big\|_{C_{T_0} \mathcal{C}_x^{1 + \eps}} \les T_0^{\frac{\al}{2} - \frac 14 - \eps} \| \pmb{Z}_N \|_{\mathcal{Z}_T^{\al, \eps}}^3,
\\
&\big\| \I^h \big( (\pi_N Y_N^\dagger)^3 \big) \big\|_{C_{T_0} \mathcal{C}_x^{1 + \eps}} \les T_0 \| Y_N^\dagger \|_{C_{T_0} \mathcal{C}_x^{1 + \eps}}^3,
\\
&\big\| \I^h \big( \<30>_N^{\,h} (\pi_N Y_N^\dagger)^2 \big) \big\|_{C_{T_0} \mathcal{C}_x^{1 + \eps}} \les T_0^{\frac{5 \al}{2} - \frac 74 - \eps} \| \pmb{Z}_N \|_{\mathcal{Z}_T^{\al, \eps}} \| Y_N^\dagger \|_{C_{T_0} \mathcal{C}_x^{1 + \eps}}^2,
\\
&\big\| \I^h \big( (\<30>_N^{\,h})^2 \pi_N Y_N^\dagger \big) \big\|_{C_{T_0} \mathcal{C}_x^{1 + \eps}} \les T_0^{\frac{5 \al}{2} - \frac 74 - \eps} \| \pmb{Z}_N \|_{\mathcal{Z}_T^{\al, \eps}}^2 \| Y_N^\dagger \|_{C_{T_0} \mathcal{C}_x^{1 + \eps}},
\\
&\big\| \I^h \big( (\<30>_N^{\,h})^3 \big) \big\|_{C_{T_0} \mathcal{C}_x^{1 + \eps}} \les T_0^{\frac{5 \al}{2} - \frac 74 - \eps} \| \pmb{Z}_N \|_{\mathcal{Z}_T^{\al, \eps}}^3.
\end{align*}

\noi
Then, local well-posedness follows from a standard contraction argument as in part (i), but in the space $C([0, T_0]; \mathcal{C}^{1 + \eps} (\T^3))$.
\end{proof}

As mentioned in Subsection~\ref{SUB:rep_intro}, to prove global well-posedness for the fractional parabolic $\Phi^4_3$-model \eqref{SNLH}, we proceed in an approach similar to that in \cite{OOT2, LTzW} for hyperbolic equations. For this purpose, let us prove the following uniform bound of the solution $Y_N^\dagger$ to \eqref{SNLHYNd} with large probability, using the invariance of the truncated Gibbs measure $\vec \rho_N$ in Lemma~\ref{LEM:inv_h} and the local well-posedness argument in the proof of Lemma~\ref{LEM:LWP_h}. For later convenience, we let $\I^h_{t_0}$ be the heat Duhamel operator as in \eqref{quin_h} but with the time integral starting from the time $t_0$ (instead of~0). We will also use the following Duhamel formulation for $Y_N^\dagger$ starting from the time $t_0$:
\begin{align}
\begin{split}
Y_N^\dagger &= \G^h_{t_0, N} [Y_N^\dagger] \\
&\deff e^{- (t - t_0) \jb{\nb}^{2 \al}} ( Y_N^\dagger (t_0) ) - a \I^h_{t_0} \big( \<1>_N^h - \<30>_N^{\,h} + \pi_N Y_N^\dagger \big) + 48 \cj{a}_2^2 \I^h_{t_0} \big( \<2>_N^{\,h} \cdot \<30>_N^{\,h} \big) \\
&\quad - 12 \cj{a}_2 \pi_N \I^h_{t_0} \big( \<2>_N^{\,h} \cdot \pi_N Y_N^\dagger \big) 
- 12 \cj{a}_2 \pi_N \I^h_{t_0} \big( \<1>_N^h (\pi_N Y_N^\dagger)^2 \big) \\
&\quad + 96 \cj{a}_2^2 \pi_N \I^h_{t_0} \big( \<1>_N^h \cdot \<30>_N^{\,h} \cdot \pi_N Y_N^\dagger \big) 
- 192 \cj{a}_2^3 \pi_N \I^h_{t_0} \big( \<1>_N^h (\<30>_N^{\,h})^2 \big) \\
&\quad - 4 \cj{a}_2 \pi_N \I^h_{t_0} \big( (\pi_N Y_N^\dagger)^3 \big) + 48 \cj{a}_2^2 \pi_N \I^h_{t_0} \big( \<30>_N^{\,h} (\pi_N Y_N^\dagger)^2 \big) \\
&\quad - 192 \cj{a}_2^3 \pi_N \I^h_{t_0} \big( (\<30>_N^{\,h})^2 \pi_N Y_N^\dagger \big) 
+ 256 \cj{a}_2^4 \pi_N \I^h_{t_0} \big( (\<30>_N^{\,h})^3 \big),
\end{split}
\label{Duh_h}
\end{align}

\begin{lemma}
\label{LEM:unibdd_h}
Let $1 < \al \leq \frac 98$, $\eps > 0$ be sufficiently small, and $T \geq 1$. Given $N \in \N$, let $Y_N^\dagger = Y_{N}^\dagger (X_{0, N}^\dagger, \pmb{Z}_{N})$ be the solution to \eqref{SNLHYNd}, where $\pmb{Z}_N$ is as defined in \eqref{defZN1} and the element $X_{0, N}^\dagger$ in the initial data is distributed by the truncated Gibbs measure $\rho_N$ in \eqref{rhoN}. Then, given any $\eta > 0$ and $0 < T_0 \leq 1$, there exist $C_0 = C_0 (T, T_0, \eta) \gg 1$ such that
\begin{align}
\rho_N \otimes \mathbb{P}^h \big( \| Y_N^\dagger \|_{C_{[T_0, T]} \mathcal{C}_x^{1 + \eps}} > C_0 \big) < \eta,
\label{YNbdd_goal}
\end{align}

\noi
uniformly in $N \in \N$.
\end{lemma}
\begin{proof}
We prove this bound by the following three steps.

\smallskip \noi
\textbf{Step 1:}
Let us first show that
\begin{align}
\rho_N \big( \| u \|_{\C^{\al - \frac 32 - \eps}} > K_0 \big) < C e^{- K_0}
\label{initbdd}
\end{align}

\noi
for any sufficiently large $K_0 \gg 1$ and some constant $C > 0$. By using the Bou\'e-Dupuis variational formula (Lemma~\ref{LEM:BD}) and analogous estimates to \eqref{BD_RNtil}, \eqref{RNtil_est}, and \eqref{ZNupbdd} (see also the beginning of the proof of Proposition~\ref{PROP:meas_conv}), we can bound
\begin{align}
\begin{split}
- &\log \int \exp \big( \| u \|_{\mathcal{C}^{\al - \frac 32 - \eps}} \big) d \rho_N (u) \\
&\geq \inf_{\dot{\Ups}^N \in \mathbb{H}_a^\al} \E \bigg[ - \| W + \Dr \|_{\mathcal{C}^{\al - \frac 32 - \eps}} + \wt{R}_N^\dia (W_N, \Ups_N - \zeta_N) \\
&\quad + \frac 12 \int_0^1 \big\| \dot{\Ups}^N (t) \big\|_{H^\al}^2 dt \bigg] + \log Z_N \\
&\geq \inf_{\dot{\Ups}^N \in \mathbb{H}_a^\al} \E \Big[ - \| W + \Dr \|_{\mathcal{C}^{\al - \frac 32 - \eps}} + \frac 14 \| \Ups^N \|_{H^\al}^2 \Big] - C'
\end{split}
\label{u0Nbdd1}
\end{align}

\noi
for some constant $C' > 0$, where we recall  that $\Law (W) = \Law (W (1)) = \mu_\al$ (defined in \eqref{mu_al}) from \eqref{defWt}, $\Dr = \Ups^N - \zeta_N$ with $\zeta_N$ being defined in \eqref{zeta}, $\Ups_N = \pi_N \Ups^N$, and $\wt R_N^\dia$ defined in \eqref{defRNd2}. By Lemma~\ref{LEM:Yreg}~(i), we have
\begin{align}
\E \big[ \| W \|_{\mathcal{C}^{\al - \frac 32 - \eps}} \big] + \E \big[ \| \zeta_N \|_{\mathcal{C}^{\al - \frac 32 - \eps}} \big] \les 1.
\label{u0Nbdd2}
\end{align}

\noi 
Also, by the Besov embedding (Lemma~\ref{LEM:Bes}~(ii)) and Young's inequality, we have
\begin{align}
\| \Ups^N \|_{\mathcal{C}^{\al - \frac 32 - \eps}} \leq C'' \| \Ups^N \|_{H^\al} \leq \dl \| \Ups^N \|_{H^\al}^2 + C(\dl)
\label{u0Nbdd3}
\end{align}

\noi
for some $\dl > 0$ arbitrarily small and some constants $C'', C(\dl) > 0$ independent of $N$. Thus, the bound \eqref{initbdd} then follows from Chebyshev's inequality, \eqref{u0Nbdd1}, \eqref{u0Nbdd2}, and \eqref{u0Nbdd3}.

\smallskip \noi
\textbf{Step 2:} 
Recalling the enhanced data set $\pmb{Z}_N$ in \eqref{defZN1} and the $\mathcal{Z}_T^{\al, \eps}$-norm in \eqref{defZTae} and using the tail estimate \eqref{tail_h} in Lemma~\ref{LEM:sto_h}, we have
\begin{align}
\mathbb{P}^h \big( \| \pmb{Z}_N \|_{\mathcal{Z}_T^{\al, \eps}} > K_1 \big) < e^{- C(T) K_1^\be}.
\label{ZNbdd}
\end{align}

\noi
for any large $K_1 \gg 1$ and some constants $C(T), \be > 0$.
Thus, from the local well-posedness argument in Lemma~\ref{LEM:LWP_h}~(i) along with \eqref{initbdd} for the initial data $X_{0, N}^\dagger$ and \eqref{ZNbdd}, by letting 
\begin{align*}
T_1 = \frac{\kappa}{(K_0^2 + K_1^6)^{\frac{1}{\ta}}}
\end{align*}

\noi
for some small constants $0 < \kappa, \ta \ll 1$, we have
\begin{align}
\rho_N \otimes \mathbb{P}^h \big( \| Y_N^\dagger \|_{\mathcal{Y}_{T_1}^{\al, \eps}} > C' (K_0 + K_1^3) \big) < C e^{- K_0} + e^{- C(T) K_1^\be}
\label{YNbdd1}
\end{align}

\noi
and in particular
\begin{align}
\rho_N \otimes \mathbb{P}^h \big( \| Y_N^\dagger (T_1) \|_{\mathcal{C}_x^{1 + \eps}} > C' T_1^{\frac{\al}{2} - \frac 54 - \eps} (K_0 + K_1^3) \big) < C e^{- K_0} + e^{- C(T) K_1^\be}
\label{YNbdd1-1}
\end{align}

\noi
for some constant $C' > 0$ (which can change from line to line below).

We now work on the event
\begin{align}
\| \pmb{Z}_N \|_{\mathcal{Z}_T^{\al, \eps}} \leq K_1 \quad \text{and} \quad \| Y_N^\dagger (T_1) \|_{\mathcal{C}^{1 + \eps}} \leq C' T_1^{\frac{\al}{2} - \frac 54 - \eps} (K_0 + K_1^3).
\label{event1}
\end{align}

\noi
Let us then show that for any $\eta' > 0$, there exists $C_1 = C_1 (T, \eta') > 0$ such that
\begin{align}
\rho_N \otimes \mathbb{P}^h \Big( \| Y_N^\dagger \|_{C_{[T_1, T]} \mathcal{C}_x^{\frac 12 + \eps}} > C_1 \Big) < \eta'.
\label{YNbdd0}
\end{align}

\noi
By the local well-posedness argument in Lemma~\ref{LEM:LWP_h}~(ii) along with \eqref{event1}, by letting
\begin{align*}
T_2 = \frac{\kappa'}{(K_0^2 + K_1^6)^{\frac{1}{\ta}}}
\end{align*}

\noi
for some small constants $0 < \kappa', \ta \ll 1$, we have
\begin{align}
\| Y_N^\dagger \|_{C_{[T_1, T_1 + T_2]} \mathcal{C}_x^{1 + \eps}} \leq C' T_1^{\frac{\al}{2} - \frac 54 - \eps} (K_0 + K_1^3).
\label{YNbdd2}
\end{align}

\noi
Let us now consider the term $\pi_N \I^h_{T_1} ( \ind_{[T_1, T_1 + T_2]} \wick{(\pi_N X_N^\dagger)^3} )$, where $X_N^\dagger$ is the solution for \eqref{SNLHNd}. By using the second order expansion \eqref{sec_XN}, we have
\begin{align*}
\pi_N &\I^h_{T_1} \big( \ind_{[T_1, T_1 + T_2]} \wick{(\pi_N X_N^\dagger)^3} \big) \\
&= \pi_N \I^h_{T_1} ( \ind_{[T_1, T_1 + T_2]} \<3>_N^{\,h}) - 12 \cj{a}_2 \pi_N \I^h_{T_1} \big( \ind_{[T_1, T_1 + T_2]} \<2>_N^{\,h} \cdot \<30>_N^{\,h} \big) \\
&\quad + 3 \pi_N \I^h_{T_1} \big( \ind_{[T_1, T_1 + T_2]} \<2>_N^{\,h} \cdot \pi_N Y_N^\dagger \big) + 3 \pi_N \I^h_{T_1} \big( \ind_{[T_1, T_1 + T_2]} \<1>_N^h (\pi_N Y_N^\dagger)^2 \big) \\
&\quad - 24 \cj{a}_2 \pi_N \I^h_{T_1} \big( \ind_{[T_1, T_1 + T_2]} \<1>_N^h \cdot \<30>_N^{\,h} \cdot \pi_N Y_N^\dagger \big) + 48 \cj{a}_2^2 \pi_N \I^h_{T_1} \big( \ind_{[T_1, T_1 + T_2]} \<1>_N^h (\<30>_N^{\,h})^2 \big) \\
&\quad + \pi_N \I^h_{T_1} \big( \ind_{[T_1, T_1 + T_2]} (\pi_N Y_N^\dagger)^3 \big) - 12 \cj{a}_2 \pi_N \I^h_{T_1} \big( \ind_{[T_1, T_1 + T_2]} \<30>_N^{\,h} (\pi_N Y_N^\dagger)^2 \big) \\
&\quad + 48 \cj{a}_2^2 \pi_N \I^h_{T_1} \big( \ind_{[T_1, T_1 + T_2]} (\<30>_N^{\,h})^2 \pi_N Y_N^\dagger \big) - 64 \cj{a}_2^3 \pi_N \I^h_{T_1} \big( \ind_{[T_1, T_1 + T_2]} (\<30>_N^{\,h})^3 \big).
\end{align*}

\noi
By similar estimates in Lemma~\ref{LEM:LWP_h}~(ii), \eqref{event1}, and \eqref{YNbdd2} along with the fact that the time restriction $\ind_{[T_1, T_1 + T_2]}$ does not affect the regularity of each term,\footnote{\label{foot} In view of the proof of Lemma~\ref{LEM:sto_h}, the term $\I^h_{T_1} (\<3>_N^{\,h})$ can be treated in the same way as $\<30>_N^{\,h}$. Similarly, the term $\I^h_{T_1} (\<2>_N^{\,h} \cdot \<30>_N^{\,h})$ can be treated in the same way as $\<320>_N^h$.} we obtain
\begin{align}
\big\| \pi_N \I^h_{T_1} \big( \ind_{[T_1, T_1 + T_2]} \wick{(\pi_N X_N^\dagger)^3} \big) \big\|_{C_{[T_1, T]} \mathcal{C}_x^{\frac 12 + \eps}} \leq C(T_1, K_0, K_1) 
\label{IXNbdd}
\end{align}

\noi
with
\begin{align*}
C(T_1, K_0, K_1) \deff C' T_1^{\frac{3\al}{2} - \frac{15}{4} - 3\eps} (K_0^3 + K_1^9) .
\end{align*}

We now let $J = \frac{T - T_1}{T_2}$, where we can further shrink $T_2$ so that $J$ is an integer. By a change of variable in time, the invariance statement in Lemma~\ref{LEM:inv_h}, and \eqref{IXNbdd} under the event \eqref{event1} along with \eqref{ZNbdd} and \eqref{YNbdd1}, we have
\begin{align*}
&\rho_N \otimes \mathbb{P}^h \Big( \big\| \pi_N \I^h_{T_1} \big( \wick{(\pi_N X_N^\dagger)^3} \big) \big\|_{C_{[T_1, T]} \mathcal{C}_x^{\frac 12 + \eps}} > J C(T_1, K_0, K_1) \Big) \\
&\leq \rho_N \otimes \mathbb{P}^h \bigg( \sum_{j = 0}^{J - 1} \big\| \pi_N \I^h_{T_1 + jT_2} \big( \ind_{[T_1 + jT_2, T_1 + (j + 1)T_2]} \wick{(\pi_N X_N^\dagger)^3} \big) \big\|_{C_{[T_1, T]} \mathcal{C}_x^{\frac 12 + \eps}} > J C(T_1, K_0, K_1) \bigg) \\
&\leq \sum_{j = 0}^{J - 1} \rho_N \otimes \mathbb{P}^h \Big( \big\| \pi_N \I^h_{T_1} \big( \ind_{[T_1, T_1 + T_2]} \wick{(\pi_N X_N^\dagger (\cdot + j T_2))^3} \big) \big\|_{C_{[T_1, T]} \mathcal{C}_x^{\frac 12 + \eps}} > C(T_1, K_0, K_1) \Big) \\
&= J \cdot \rho_N \otimes \mathbb{P}^h \Big( \big\| \pi_N \I^h_{T_1} \big( \ind_{[T_1, T_1 + T_2]} \wick{(\pi_N X_N^\dagger (\cdot))^3} \big) \big\|_{C_{[T_1, T]} \mathcal{C}_x^{\frac 12 + \eps}} > C(T_1, K_0, K_1) \Big) \\
&< J \big( C e^{- K_0} + 2 e^{- C(T) K_1^\be} \big). 
\end{align*}

\noi
Since we have
\begin{align*}
J \leq \frac{T (K_0^2 + K_1^6)^{\frac{1}{\ta}}}{\kappa'}
\end{align*}

\noi
and also the fact that $T_1$ depends only on $K_0$ and $K_1$, we can take $K_0 \sim K_1 \gg 1$ to be sufficiently large (depending only on $T$ and $\eta'$) so that there exists $C_1' = C_1' (T, \eta') > 0$ such that
\begin{align}
\rho_N \otimes \mathbb{P}^h \Big( \big\| \pi_N \I^h_{T_1} \big( \wick{(\pi_N X_N^\dagger)^3} \big) \big\|_{C_{[T_1, T]} \mathcal{C}_x^{\frac 12 + \eps}} > C_1' \Big) < \frac{\eta'}{3}.
\label{IXNbdd2}
\end{align}

\noi
By using similar and easier steps, we also obtain that there exists $C_1'' = C_1'' (T, \eta') > 0$ such that
\begin{align}
\rho_N \otimes \mathbb{P}^h \Big( \big\| \pi_N \I^h_{T_1} \big( X_N^\dagger \big) \big\|_{C_{[T_1, T]} \mathcal{C}_x^{\frac 12 + \eps}} > C_1'' \Big) < \frac{\eta'}{3}.
\label{IXNbdd3}
\end{align}

\noi
By using the Duhamel formulation \eqref{Duh_h}, we have
\begin{align*}
Y_N^\dagger (t) &= e^{- (t - T_1) \jb{\nb}^{2 \al}} (Y_N^\dagger (T_1)) - a \pi_N \I_{T_1}^h \big( X_N^\dagger \big) (t) \\
&\quad - 4 \cj{a}_2 \pi_N \I^h_{T_1} \big( \wick{(\pi_N X_N^\dagger)^3} \big) (t) + 4 \cj{a}_2 \pi_N \I^h_{T_1} (\<3>_N^{\,h}) (t)
\end{align*} 

\noi
for any $t \geq T_1$. Thus, the bound \eqref{YNbdd0} then follows from Lemma~\ref{LEM:Schau}, \eqref{YNbdd1-1} with $K_0 \sim K_1 \gg 1$ sufficiently large, \eqref{IXNbdd2}, \eqref{IXNbdd3}, and \eqref{ZNbdd} with $K_1 \gg 1$ sufficiently large (see also footnote~\ref{foot}).

\smallskip \noi
\textbf{Step 3:} 
We are now ready to prove the desired bound \eqref{YNbdd_goal}. Below, we work on the events
\begin{align}
\begin{split}
&\| \pmb{Z}_N \|_{\mathcal{Z}_T^{\al, \eps}} \leq K_1, \\ 
&\| Y_N^\dagger (T_1) \|_{\mathcal{C}^{1 + \eps}} \leq C' T_1^{\frac{\al}{2} - \frac 54 - \eps} (K_0 + K_1^3), \\  
&\| Y_N^\dagger \|_{C_{[T_1, T]} \mathcal{C}_x^{\frac 12 + \eps}} \leq C_1.
\end{split}
\label{event2}
\end{align}

\noi
Let $0 < T_3 \leq 1$. By using the Duhamel formulation \eqref{Duh_h} and slightly modifying the estimates in Lemma~\ref{LEM:LWP_h}~(ii) along with \eqref{event2}, we obtain
\begin{align}
\begin{split}
\| Y_N^\dagger \|_{C_{[T_1, T_1 + T_3]} \mathcal{C}_x^{1 + \eps}} &\les \| Y_N^\dagger (T_1) \|_{\mathcal{C}^{1 + \eps}} + \| \pmb{Z}_N \|_{\mathcal{Z}_T^{\al, \eps}}^3 \\
&\quad + T_2^\ta \Big( 1 + \| Y_N^\dagger \|_{C_{[T_1, T_1 + T_3]} \mathcal{C}_x^{\frac 12 + \eps}}^2 \Big) \| Y_N^\dagger \|_{C_{[T_1, T_1 + T_3]} \mathcal{C}_x^{1 + \eps}} \\
&\leq C' T_1^{\frac{\al}{2} - \frac 54 - \eps} (K_0 + K_1^3) + K_1^3 + T_3^\ta (1 + C_1^2)  \| Y_N^\dagger \|_{C_{[T_1, T_1 + T_3]} \mathcal{C}_x^{1 + \eps}}
\end{split}
\label{YN_lin}
\end{align}

\noi
for some $\ta > 0$. Here, when there are more than one $Y_N^\dagger$'s in a nonlinear estimate, we put one $Y_N^\dagger$ in the $\mathcal{C}^{1 + \eps}$ space and put other $Y_N^\dagger$'s in the $\mathcal{C}^{\frac 12 + \eps}$ space. Thus, from \eqref{YN_lin}, by choosing $T_3 > 0$ small enough such that $T_3^\ta (1 + C_1^2) \ll 1$, we obtain the bound
\begin{align*}
\| Y_N^\dagger \|_{C_{[T_1, T_1 + T_3]} \mathcal{C}_x^{1 + \eps}} \les T_1^{\frac{\al}{2} - \frac 54 - \eps} (K_0 + K_1^3) .
\end{align*}

\noi
Note that we can choose $T_3$ in a way such that $T_3$ depends only on $C_1$. Thus, we can iterate the above steps on subsequent intervals $[T_1 + T_3, T_1 + 2 T_3]$, $[T_1 + 2 T_3, T_1 + 3 T_3]$, ..., so that we obtain the bound
\begin{align}
\| Y_N^\dagger \|_{C_{[T_1, T]} \mathcal{C}_x^{1 + \eps}} \leq C_2,
\label{YNbdd3}
\end{align}

\noi
where $C_2 > 0$ is a constant depending on $T$, $T_1$, $K_0$, $K_1$, and $C_1$. From Step 2 above, we know that $C_2$ depends only on $T$ and $\eta'$ given by \eqref{YNbdd0}.

From \eqref{YNbdd3} under the event \eqref{event2} along with \eqref{ZNbdd}, \eqref{YNbdd1-1}, and \eqref{YNbdd0} (with $K_0$ and $K_1$ sufficiently large), we obtain 
\begin{align}
\rho_N \otimes \mathbb{P}^h \big( \| Y_N^\dagger \|_{C_{[T_1, T]} \mathcal{C}_x^{1 + \eps}} > C_2 \big) < 2 \eta'.
\label{YNbdd4}
\end{align}

\noi
Therefore, by taking $\eta' = \frac{\eta}{4}$, we obtain the desired bound \eqref{YNbdd_goal} from \eqref{YNbdd4} and \eqref{YNbdd1} (with $K_0$ and $K_1$ sufficiently large) with a separate discussion on the cases $T_0 \leq T_1$ and $T_0 \geq T_1$.
\end{proof}

We will also need the following stability estimate.
\begin{lemma}
\label{LEM:stb_h}
Let $1 < \al \leq \frac 98$, $\eps > 0$ be sufficiently small, $T \geq 1$, and $0 < T_0 \leq 1$. Let $K \gg 1$ and $C_0 \gg 1$. For $j = 1, 2$, let $Y_{N_j}^\dagger$ be the solution to the equation \eqref{SNLHYNd} with $N = N_j$ on $[0, T]$ with the enhanced data set $\pmb{Z}_{N_j}$ defined in \eqref{defZN1}. Then, there exist a large $N_0 = N_0 (T, K, C_0) \in \N$ and a small $\kappa_0 = \kappa_0 (T, K, C_0) > 0$ such that the following statement holds true. Given $N_1 \geq N_2 \geq N_0$ and suppose that
\begin{align}
\| \pmb{Z}_{N_1} \|_{\mathcal{Z}_T^{\al, \eps}} \leq K
\label{ZN1_bdd}
\end{align}

\noi
and
\begin{align}
\| Y_{N_1}^\dagger \|_{C_{[T_0, T]} \mathcal{C}_x^{1 + \eps}} \leq C_0,
\label{YN1_bdd}
\end{align}

\noi
where the $\mathcal{Z}_T^{\al, \eps}$-norm is as defined in \eqref{defZTae}. Furthermore, suppose that
\begin{align}
\| \pmb{Z}_{N_1} - \pmb{Z}_{N_2} \|_{\mathcal{Z}_T^{\al, \eps}} \leq N_2^{-\ta}
\label{Z_diff}
\end{align}

\noi
for some $\ta > 0$, and
\begin{align}
\| Y_{N_1}^\dagger (T_0) - Y_{N_2}^\dagger (T_0) \|_{\C^{1 + \eps}} \leq \kappa
\label{Xinit_diff}
\end{align}

\noi
for some $0 \leq \kappa \leq \kappa_0$.
Then, we have the difference estimate
\begin{align}
\| Y_{N_1}^\dagger - Y_{N_2}^\dagger \|_{C_{[T_0, T]} \mathcal{C}_x^{1 + \eps}} \leq C(T, K, C_0) (\kappa + N_2^{- \ta'})
\label{YN_diff}
\end{align}

\noi
for some constant $C (T, K, C_0) > 0$ and some $\ta' > 0$.
\end{lemma}

\begin{proof}
Following the idea in \cite[Proposition~6.8]{OOT2}, we define the norm
\begin{align*}
\| u \|_{C_{[T_0, T]}^\ld \mathcal{C}_x^s} \deff \| e^{- \ld (t - T_0)} u \|_{C_{[T_0, T]} \mathcal{C}_x^s}
\end{align*}

\noi
for any $T > 0$, $\ld \geq 1$, and $s \in \R$. Note that we have the obvious bound
\begin{align}
\| u \|_{C_{[T_0, T]}^\ld \mathcal{C}_x^s} \leq \| u \|_{C_{[T_0, T]} \mathcal{C}_x^s}.
\label{CTld1}
\end{align}

By \eqref{Z_diff} and \eqref{ZN1_bdd}, we have
\begin{align}
\| \pmb{Z}_{N_2} \|_{\mathcal{Z}_T^{\al, \eps}} \leq K + 1.
\label{ZN2_bdd}
\end{align}

\noi
We also assume that
\begin{align}
\| Y_{N_2}^\dagger \|_{C_{[T_0, T]} \mathcal{C}_x^{1 + \eps}} \leq 4 C_0.
\label{YN2_bdd}
\end{align}

\noi
By setting 
\begin{align*}
\dl Y^\dagger = Y_{N_2}^\dagger - Y_{N_1}^\dagger,
\end{align*}

\noi
we have
\begin{align*}
Y_{N_2}^\dagger = \dl Y^\dagger + Y_{N_1}^\dagger.
\end{align*}

\noi
Thus, in view of the Duhamel formulation \eqref{Duh_h}, we have
\begin{align}
\dl Y^\dagger = \G^h_{T_0, N_2} [ \dl Y^\dagger + Y_{N_1}^\dagger ] - \G^h_{T_0, N_1} [ Y_{N_1}^\dagger ].
\label{stb_Duh}
\end{align}


\noi
Note that we can write
\begin{align*}
e^{- \ld (t - T_0)} \dl Y^\dagger (t) = e^{- \ld (t - T_0)} \textup{I}_1 (t) + e^{- \ld (t - T_0)} \textup{I}_2 (t) + e^{- \ld (t - T_0)} \textup{I}_3 (t),
\end{align*}

\noi
where each term in $\textup{I}_1$ contains $Y_{N_1}^\dagger (T_0) - Y_{N_2}^\dagger (T_0)$ or the difference of one of the elements in the enhanced data sets $\pmb{Z}_{N_1}$ and $\pmb{Z}_{N_2}$, each term in $\textup{I}_2$ contains the difference $\pi_{N_1} - \pi_{N_2}$, and each term in $\textup{I}_3$ contains $\dl Y^\dagger$. We would like to modify the estimates in the proof of Lemma~\ref{LEM:LWP_h}~(ii) to deal with $\textup{I}_1$, $\textup{I}_2$, and $\textup{I}_3$.
For each term in $\textup{I}_2$, we produce a factor $N_2^{- \ta'}$ for some $\ta' > 0$ by losing some small amount of spatial regularity. For each term in $\textup{I}_3$, we put a $e^{- \ld (t' - T_0)}$ factor on one of $\dl Y^\dagger$ and then use the estimate (with $0 < \g < 1$, $p > 1$, $p \g < 1$, and $\frac{1}{p} + \frac{1}{p'} = 1$)
\begin{align}
e^{- \ld (t - T_0)} \int_{T_0}^t \frac{e^{\ld (t' - T_0)}}{(t - t')^{\g}} dt' \leq e^{- \ld (t - T_0)} \| e^{\ld (t' - T_0)} \|_{L^{p'}_{t'} ([T_0, t])} \| (t - t')^{- \g} \|_{L^p_{t'} ([T_0, t])} \les_T \ld^{- \frac{1}{p'}}
\label{CTld2}
\end{align}

\noi
for all $T_0 \leq t \leq T$, which follows from H\"older's inequality. Thus, from the estimates in the proof of Lemma~\ref{LEM:LWP_h}~(ii) along with \eqref{CTld1}, \eqref{ZN1_bdd}, \eqref{YN1_bdd}, \eqref{Z_diff}, \eqref{Xinit_diff}, \eqref{ZN2_bdd}, \eqref{YN2_bdd}, and \eqref{CTld2}, we obtain
\begin{align}
\| \dl Y^\dagger \|_{C_{[T_0, T]}^\ld \mathcal{C}_x^{1 + \eps}} \leq C (T, K, C_0) (\kappa + N_2^{- \ta'}) + C(T, K, C_0) \ld^{- \ta''}  \| \dl Y^\dagger \|_{C_{[T_0, T]}^\ld \mathcal{C}_x^{1 + \eps}}
\label{stb1}
\end{align}

\noi
for some constant $C(T, K, C_0) > 0$ (which may vary from line to line below) and $0 < \ta'' \ll 1$. We choose $\ld = \ld (T, K, C_0) \gg 1$ to be large enough such that
\begin{align*}
C(T, K, C_0) \ld^{- \ta''} \leq \frac 12,
\end{align*}

\noi
so that we obtain
\begin{align}
\| \dl Y^\dagger \|_{C_{[T_0, T]} \mathcal{C}_x^{1 + \eps}} \leq e^{\ld T} \| \dl Y^\dagger \|_{C_{[T_0, T]}^\ld \mathcal{C}_x^{1 + \eps}} \leq C (T, K, C_0) (\kappa + N_2^{- \ta'}) .
\label{YN_dif1}
\end{align}

It remains to verify the bound \eqref{YN2_bdd}. We assume for the sake of contradiction that there exists $T_0 < T_* < T$ such that $T_*$ is the largest number such that the bound
\begin{align*}
\| Y_{N_2}^\dagger \|_{C_{[T_0, T_*]} \mathcal{C}_x^{1 + \eps}} \leq 2 C_0
\end{align*}

\noi
holds. Then, by continuity, we know that there exists $T_1 > 0$ with $T_* + T_1 \leq T$ such that
\begin{align*}
\| Y_{N_2}^\dagger \|_{C_{[T_0, T_* + T_1]} \mathcal{C}_x^{1 + \eps}} \leq 4 C_0,
\end{align*}

\noi
so that the above argument that leads to \eqref{YN_dif1} gives
\begin{align}
\| \dl Y^\dagger \|_{C_{[T_0, T_* + T_1]} \mathcal{C}_x^{1 + \eps}} \leq C (T, K, C_0) (\kappa + N_2^{- \ta'}) .
\label{YN_dif2}
\end{align}

\noi
Then, from \eqref{YN_dif2} and \eqref{YN1_bdd}, we can take $\kappa = \kappa (T, K, C_0) \ll 1$ to be small enough and $N = N (T, K, C_0) \gg 1$ to be large enough such that
\begin{align*}
\| Y_{N_2}^\dagger \|_{C_{[T_0, T_* + T_1]} \mathcal{C}_x^{1 + \eps}} 
\leq \| \dl Y^\dagger \|_{C_{[T_0, T_* + T_1]} \mathcal{C}_x^{1 + \eps}} + \| Y_{N_1}^\dagger \|_{C_{[T_0, T]} \mathcal{C}_x^{1 + \eps}} \leq 2 C_0,
\end{align*}

\noi
which contradicts the definition of $T_*$. Thus, we must have $T_* = T$, and so we have finished the proof.
\end{proof}

In view of the weak convergence of the truncated Gibbs measure $\rho_N$ in \eqref{rhoN} to the limiting Gibbs measure $\rho$ (see \eqref{rho}), it follows from the Skorokhod representation theorem (Lemma~\ref{LEM:skorok}) that there exist a probability space $(\Om', \mathcal{F}', \PP')$ and random variables $X_{0, N}^\dagger$ and $X_0^\dagger$ such that
\begin{align}
\Law (X_{0, N}^\dagger) = \rho_N \quad \text{and} \quad \Law (X_0^\dagger) = \rho,
\label{defX0d}
\end{align}

\noi
and $X_{0, N}^\dagger$ converges $\PP'$-almost surely to $X_0^\dagger$ in $\C^{\al - \frac 32 - \eps} (\T^3)$. \footnote{To avoid the issue that the space $\mathcal{C}^{\al - \frac 32 - \eps} (\T^3)$ is non-separable, one can view the underlying metric space of the Gibbs measures $\rho_N$ and $\rho$ as $W^{\al - \frac 32 - \frac{\eps}{2}, p} (\T^3)$ with $\frac{6}{\eps} < p < \infty$ (see the proof of Proposition~\ref{PROP:tight}), on which one can apply the Skorokhod representation theorem. The almost sure convergence in the space $\mathcal{C}^{\al - \frac 32 - \eps} (\T^3)$ then follows from the embedding properties in Lemma~\ref{LEM:Bes}.}

We are now ready to show almost sure global well-posedness of the fractional parabolic $\Phi_3^4$-model \eqref{SNLH} and invariance of the Gibbs measure $\rho$ in \eqref{rho} under the flow of \eqref{SNLH}. 
\begin{proposition}
\label{PROP:GWP_h}
Let $1 < \al \leq \frac 98$ and $\eps > 0$. Let $X_{0, N}^\dagger$ and $X_0^\dagger$ be random variables on the probability space $(\Om', \F', \PP')$ as given by \eqref{defX0d}. Then, given any $T \geq 1$, the solution $X_N^\dagger$ to the truncated fractional parabolic $\Phi_3^4$-model \eqref{SNLHNd} with initial data $X_{0, N}^\dagger$ converges $\PP' \otimes \PP^h$-almost surely to some non-trivial stochastic process $X$ with $X (0) = X_0^\dagger$ in $C([0, T]; \mathcal{C}^{\al - \frac 32 - \eps} (\T^3))$ as $N \to \infty$. Furthermore, we have $\Law (X (t)) = \rho$ for any $t \in \R_+$.
\end{proposition}

\begin{proof}
We prove the proposition in the following two steps.

\smallskip \noi
\textbf{Step 1:} Let us first show almost sure global well-posedness for the fractional parabolic $\Phi^4_3$-model \eqref{SNLH}. Our goal is to show that for any $T \geq 1$ and $\kappa > 0$, we have
\begin{align*}
\lim_{N_0 \to \infty} \PP' \otimes \mathbb{P}^h \Big( \| X_{N_1}^\dagger - X_{N_2}^\dagger \|_{C_T \mathcal{C}_x^{\al - \frac 32 - \eps}} \leq \kappa \text{ for all } N_1 \geq N_2 \geq N_0 \Big) = 1.
\end{align*}

\noi
Note that it suffices to show
\begin{align*}
\lim_{N_0 \to \infty} \lim_{N \to \infty} \PP' \otimes \mathbb{P}^h \Big( \| X_{N_1}^\dagger - X_{N_2}^\dagger \|_{C_T \mathcal{C}_x^{\al - \frac 32 - \eps}} \leq \kappa \text{ for all } N \geq N_1 \geq N_2 \geq N_0 \Big) = 1.
\end{align*}

\noi
By using the second order expansion \eqref{sec_XN} and Lemma~\ref{LEM:sto_h} for the $\<30>^{\,h}$ term, we only need to show
\begin{align}
\begin{split}
\lim_{N_0 \to \infty} \lim_{N \to \infty} \PP' \otimes \mathbb{P}^h \Big( \| Y_{N_1}^\dagger - Y_{N_2}^\dagger \|_{C_T \mathcal{C}_x^{\al - \frac 32 - \eps}} \leq \kappa \text{ for all } N \geq N_1 \geq N_2 \geq N_0 \Big) = 1,
\end{split}
\label{inv_h2}
\end{align}

\noi
where $Y_{N_j}^\dagger = Y_{N_j}^\dagger (X_{0, N_j}^\dagger, \pmb{Z}_{N_j})$ is the solution to the equation \eqref{SNLHYNd} with enhanced data set $\pmb{Z}_{N_j}$ defined in \eqref{defZN1}, $j = 1, 2$.

Let $\eta > 0$ be arbitrary. 
From \eqref{initbdd}, we know that there exists $K_1 = K_1 (\eta)$ such that for any $N \in \N$,
\begin{align}
\PP' \big( \| X_{0, N}^\dagger \|_{\C^{\al - \frac 32 - \eps}} > K_1 \big) < \frac{\eta}{16} .
\label{inv_data00}
\end{align}

\noi
In view of the $\PP'$-almost sure convergence of $X_{0, N}^\dagger$ to $X_0^\dagger$ in $\C^{\al - \frac 32 - \eps} (\T^3)$, we know that for any $\kappa_1 > 0$ small, there exists $N_0 = N_0 (\kappa_1, \eta) \gg 1$ such that for any $N \geq N_0$,
\begin{align}
\PP' \big( \| X_{0, N}^\dagger - X_{0, N'}^\dagger \|_{\C^{\al - \frac 32 - \eps}} > \kappa_1 \text{ for some } N \geq N' \geq N_0 \big) < \frac{\eta}{16} .
\label{inv_data0}
\end{align}

\noi
Recalling the $\mathcal{Z}_T^{\al, \eps}$-norm in \eqref{defZTae} and applying the tail estimates in Lemma~\ref{LEM:sto_h}, we deduce that there exist $K_2 = K_2 (T, \eta) \gg 1$ sufficiently large, $N_0 = N_0 (T, \eta) \gg 1$ sufficiently large, and small $\ta > 0$ such that for any $N \geq N_0$,
\begin{align}
\mathbb{P}^h \big( \| \pmb{Z}_N \|_{\mathcal{Z}_T^{\al, \eps}} > K_2 \big) < \frac{\eta}{16}
\label{inv_data1}
\end{align}

\noi
and
\begin{align}
\mathbb{P}^h \big( \| \pmb{Z}_N - \pmb{Z}_{N'} \|_{\mathcal{Z}_T^{\al, \eps}} > (N')^{- \ta} \text{ for some } N \geq N' \geq N_0 \big) < \frac{\eta}{16}.
\label{inv_data2}
\end{align}

\noi
From the continuous dependence in Lemma~\ref{LEM:LWP_h}~(i) along with \eqref{inv_data00}, \eqref{inv_data0}, \eqref{inv_data1}, and \eqref{inv_data2}, we know that for any $\kappa_2 > 0$ small, there exist $0 < T_0 = T_0 (K_1, K_2) = T_0 (T, \eta) \leq 1$ and $N_0 = N_0 (\kappa_2, T, \eta) \gg 1$ sufficiently large such that for any $N \geq N_0$,
\begin{align*}
\mathbb{P}' \otimes \mathbb{P}^h \big( \| Y_N^\dagger - Y_{N'}^\dagger \|_{\mathcal{Y}_{T_0}^{\al, \eps}} > \kappa_2 \text{ for some } N \geq N' \geq N_0 \big) < \frac{\eta}{4} ,
\end{align*}

\noi
where the $\mathcal{Y}_{T_0}^{\al, \eps}$-norm is defined in \eqref{defYtae}. In particular, we have
\begin{align}
\mathbb{P}' \otimes \mathbb{P}^h \Big( \| Y_N^\dagger - Y_{N'}^\dagger \|_{C_{T_0} \C_x^{\al - \frac 32 - \eps}} > \kappa_2 \text{ for some } N \geq N' \geq N_0 \Big) < \frac{\eta}{4}
\label{inv_Ydiff1}
\end{align}

\noi
and
\begin{align}
\mathbb{P}' \otimes \mathbb{P}^h \Big( \| Y_N^\dagger (T_0) - Y_{N'}^\dagger (T_0) \|_{\C^{1 + \eps}} > T_0^{\frac{\al}{2} - \frac 54 - \eps} \kappa_2 \text{ for some } N \geq N' \geq N_0 \Big) < \frac{\eta}{4} .
\label{inv_Ydiff2}
\end{align}

By Lemma~\ref{LEM:unibdd_h}, there exists $C_0 = C_0 (T, \eta) \gg 1$ such that
\begin{align}
\PP' \otimes \mathbb{P}^h \big( \| Y_N^\dagger \|_{C_{[T_0, T]} \mathcal{C}_x^{1 + \eps}} > C_0 \big) < \frac{\eta}{8}
\label{inv_Rbdd}
\end{align}

\noi 
for any $N \geq N_0$. Thus, for any $\kappa > 0$, by Lemma~\ref{LEM:stb_h} along with \eqref{inv_Ydiff2} with $\kappa_2 = \kappa_2 (\kappa, T, K_2, C_0) = \kappa_2 (\kappa, T, \eta) > 0$ sufficiently small, \eqref{inv_data1}, \eqref{inv_data2}, and \eqref{inv_Rbdd}, there exists $N_0 = N_0 (\kappa, T, \eta) \gg 1$ sufficiently large such that for any $N \geq N_0$,
\begin{align*}
\PP' \otimes \mathbb{P}^h \big( \| Y_N^\dagger - Y_{N'}^\dagger \|_{C_{[T_0, T]} \mathcal{C}_x^{1 + \eps}} > \tfrac 12 \kappa \text{ for some } N \geq N' \geq N_0 \big) < \frac{\eta}{2} .
\end{align*}

\noi
This shows that for any $\kappa > 0$ and $N \geq N_0$, we have
\begin{align*}
\PP' \otimes \mathbb{P}^h \big( \| Y_{N_1}^\dagger - Y_{N_2}^\dagger \|_{C_{[T_0, T]} \mathcal{C}_x^{1 + \eps}} > \kappa \text{ for some } N \geq N_1 \geq N_2 \geq N_0 \big) < \eta,
\end{align*}

\noi
which implies the desired bound \eqref{inv_h2}.

\smallskip \noi
\textbf{Step 2:} We now prove the invariance part. Let $t > 0$ and $F: \mathcal{C}^{-100} (\T^3) \to \R$ be bounded and Lipschitz. From Step 1, we let $X$ be the $\PP' \otimes \mathbb{P}^h$-almost sure limit of the truncated solutions $X_N^\dagger$. Our goal is to show that
\begin{align}
\int F (X (t)) d (\PP' \otimes \mathbb{P}^h) = \int F (u) d \rho (u).
\label{inv_hgoal}
\end{align}

\noi
By the dominated convergence theorem, Lemma~\ref{LEM:inv_h}, and the weak convergence of $\rho_N$ to $\rho$, we have
\begin{align}
\begin{split}
\int F (X (t)) d (\PP' \otimes \mathbb{P}^h) &= \lim_{N \to \infty} \int F (X_N^\dagger (t)) d (\PP' \otimes \mathbb{P}^h) \\
&= \lim_{N \to \infty} \int F (u) d \rho_N (u) \\
&= \int F (u) d \rho (u),
\end{split}
\label{inv_h3}
\end{align}

\noi
which gives the desired invariance result \eqref{inv_hgoal}.
\end{proof}

Note that from Lemma~\ref{LEM:sto_h}, we deduce that the limiting process $Y$ of $\{ Y_N^\dagger \}_{N \in \N}$ shown in the proof of Proposition~\ref{PROP:GWP_h} satisfies the following equation:
\begin{align}
\begin{split}
Y &= e^{- t \jb{\nb}^{2 \al}} \big( X_0 - \<1>^h (0) + 4 \cj{a}_2 \<30>^{\,h} (0) \big) - a \I^h \big( \<1>^h - 4 \cj{a}_2 \<30>^{\,h} + Y \big) + 48 \cj{a}_2^2 \<320>^h \\
&\quad - 12 \cj{a}_2 \I^h (\<2>^{\,h} Y) - 12 \cj{a}_2 \I^h (\<1>^h Y^2) + 96 \cj{a}_2^2 \I^h \big( \<1>^h \<30>^{\,h} Y \big) - 192 \cj{a}_2^3 \I^h \big( \<1>^h (\<30>^{\,h})^2 \big) \\
&\quad - 4 \cj{a}_2 \I^h (Y^3) + 48 \cj{a}_2^2 \I^h \big( \<30>^{\,h} Y^2 \big) - 192 \cj{a}_2^3 \I^h \big( (\<30>^{\,h})^2 Y \big) + 256 \cj{a}_2^4 \I^h \big( (\<30>^{\,h})^3 \big),
\end{split}
\label{SNLHY}
\end{align}

\noi
where $X_0$ is distributed by the fractional $\Phi^4_3$-measure $\rho$ in \eqref{rho}. The equation \eqref{SNLHY} along with the second order expansion
\begin{align}
X = \<1>^h - 4 \cj{a}_2 \<30>^{\,h} + Y
\label{sec_exph}
\end{align} 
then gives a meaning to the formal equation \eqref{SNLH}, where the Wick ordering $\wick{X^3}$ is defined as
\begin{align}
\begin{split}
\wick{X^3} \, &\deff \<3>^{\,h} - 12 \cj{a}_2 \<32>^{h} + 3 \<2>^{\,h} Y + 3 \<1>^h Y^2 - 24 \cj{a}_2 \<1>^h \<30>^{\,h} Y + 48 \cj{a}_2^2 \<1>^h (\<30>^{\,h})^2 \\
&\quad + Y^3 - 12 \cj{a}_2 \<30>^{\,h} Y^2 + 48 \cj{a}_2^2 (\<30>^{\,h})^2 Y - 64 \cj{a}_2^3 (\<30>^{\,h})^3,
\end{split}
\label{defX3}
\end{align}

\noi
where $\<32>^h$ is the limit of $\<32>_N^h$ given by \eqref{heat32} as $N \to \infty$.

Before moving on, let us show the following a priori bound for the process $Y$ with large probability, which is a limiting version of Lemma~\ref{LEM:unibdd_h}. For convenience of notations, let us define the limiting enhanced data set 
\begin{align}
\pmb{Z} \deff (\<1>^h, \<2>^{\,h}, \<30>^{\, h}, \<320>^h).
\label{defZlim}
\end{align}
\begin{proposition}
\label{PROP:Y_bdd}
Let $1 < \al \leq \frac 98$, $\eps > 0$ be sufficiently small, and $T \geq 1$. Let $Y = Y (X_0, \pmb{Z})$ be the solution to \eqref{SNLHY} with the enhanced data set $\pmb{Z}$ defined in \eqref{defZlim} and $X_0$ distributed by the Gibbs measure $\rho$ in \eqref{rho}. Then, given any $\eta > 0$, there exist $0 < T_0 = T_0 (T, \eta) \leq 1$ and $C_0 = C_0 (T, \eta) \gg 1$ such that
\begin{align}
\rho \otimes \mathbb{P}^h \big( \| Y \|_{\mathcal{Y}_{T_0}^{\al, \eps}} + \| Y \|_{C_{[T_0, T]} \mathcal{C}_x^{1 + \eps}} > C_0 \big) < \eta.
\label{Y_bdd}
\end{align}
\end{proposition}

\begin{proof}
Let $X_{0, N}^\dagger$ and $X_0^\dagger$ be random variables on the probability space $(\Om', \F', \PP')$ as given by \eqref{defX0d}.
Then, from \eqref{initbdd}, we know that there exists $K_1 = K_1 (\eta)$ such that for any $N \in \N$, 
\begin{align}
\PP' \big( \| X_{0, N}^\dagger \|_{\C^{\al - \frac 32 - \eps}} > K_1 \big) < \frac{\eta}{8}.
\label{P'0}
\end{align}

\noi
Also, in view of the $\PP'$-almost sure convergence of $X_{0, N}^\dagger$ to $X_0^\dagger$ in $\C^{\al - \frac 32 - \eps} (\T^3)$, we know that for any $\kappa_1 > 0$ small, there exists $N_0 = N_0 (\kappa, \eta) \gg 1$ such that for any $N \geq N_0$, 
\begin{align}
\PP' \big( \| X_{0, N}^\dagger - X_0^\dagger \|_{\C^{\al - \frac 32 - \eps}} > \kappa \big) < \frac{\eta}{8}.
\label{P'1}
\end{align}

\noi
Given $N \in \N$, let $\pmb{Z}_N$ be the enhanced data set defined in \eqref{defZN1}. Recalling the $\mathcal{Z}_T^{\al, \eps}$-norm in \eqref{defZTae} and applying the tail estimates in Lemma~\ref{LEM:sto_h}, we deduce that there exist $K_2 = K_2 (T, \eta) \gg 1$ sufficiently large, $N_0 = N_0 (T, \eta) \gg 1$ sufficiently large, and small $\ta > 0$ such that for any $N \geq N_0$,
\begin{align}
\PP^h \big( \| \pmb{Z}_N \|_{\mathcal{Z}_T^{\al, \eps}} > K_2 \big) < \frac{\eta}{8}
\label{P'2}
\end{align}

\noi
and
\begin{align}
\PP^h \big( \| \pmb{Z}_N - \pmb{Z} \|_{\mathcal{Z}_T^{\al, \eps}} > N^{- \ta} \big) < \frac{\eta}{8}.
\label{P'3}
\end{align}

Given $N \in \N$, we let $Y_N^\dagger = Y_N^\dagger (X_{0, N}^\dagger , \pmb{Z}_N)$ be the solution to the equation \eqref{SNLHNd} with enhanced data set $\pmb{Z}_N$ in \eqref{defZN1}, which exists globally in time thanks to Lemma~\ref{LEM:inv_h}. From the continuous dependence in Lemma~\ref{LEM:LWP_h} along with \eqref{P'0}, \eqref{P'1}, \eqref{P'2}, and \eqref{P'3}, we know that for any $\kappa_2 > 0$ small, there exist $0 < T_0 = T_0 (K_1, K_2) = T_0 (T, \eta) \leq 1$ and $N_0 = N_0 (\kappa_2, T, \eta) \gg 1$ sufficiently large such that for any $N \geq N_0$,
\begin{align}
\PP' \otimes \PP^h \Big( \| Y_N^\dagger - Y  \|_{\mathcal{Y}_{T_0}^{\al, \eps}} > \kappa_2 \Big) < \frac{\eta}{2} ,
\label{P'Y1}
\end{align}

\noi
where $Y = Y(X_0^\dagger, \pmb{Z})$ satisfies \eqref{SNLHY} with $X_0 = X_0^\dagger$.
In particular, we have
\begin{align}
\PP' \otimes \PP^h \Big( \| Y_N^\dagger (T_0) - Y (T_0) \|_{\C^{1 + \eps}} > T_0^{\frac{\al}{2} - \frac 54 - \eps} \kappa_2 \Big) < \frac{\eta}{2} .
\label{P'Y}
\end{align}

\noi
From Lemma~\ref{LEM:unibdd_h}, we know that there exists $C_1 = C_1 (T, \eta) \gg 1$ such that
\begin{align}
\PP' \otimes \PP^h \big( \| Y_N^\dagger \|_{\mathcal{Y}_{T_0}^{\al, \eps}} + \| Y_N^\dagger \|_{C_{[T_0, T]} \C_x^{1 + \eps}} > C_1 \big) < \frac{\eta}{8}.
\label{P'4}
\end{align}

\noi
for any $N \geq N_0$. Thus, combining \eqref{P'2}, \eqref{P'3}, \eqref{P'Y}, and \eqref{P'4} and using a slight variant of Lemma~\ref{LEM:stb_h}, we know that there exist $C = C (T, K_2, C_1) = C (T, \eta) > 0$ and $\ta' > 0$ such that for any $N \geq N_0$,
\begin{align}
\PP' \otimes \PP^h \big( \| Y_N^\dagger - Y \|_{C_{[T_0, T]} \C_x^{1 + \eps}} > C (\kappa_2 + N^{- \ta'}) \big) < \frac{7\eta}{8} .
\label{P'5}
\end{align}

\noi
The desired a priori bound \eqref{Y_bdd} then follows from \eqref{P'Y1}, \eqref{P'4}, and \eqref{P'5}.
\end{proof}

\begin{remark} \rm
\label{RMK:heat_wave}
One can compare the global well-posedness argument for the fractional parabolic $\Phi^4_3$-model \eqref{SNLH} and that in \cite[Section~6]{LTzW} for the fractional hyperbolic $\Phi^4_3$-model. Compared to the hyperbolic model, the parabolic model has stronger smoothing properties, so that the number of stochastic terms that require stochastic analysis tools is fewer. Also, again thanks to the stronger smoothing properties, for the parabolic model one can directly solve the equation (for the remainder term $Y_N^\dagger$) using the Gibbs measure initial data; in the hyperbolic model, however, one has to incorporate the Gibbs measure initial data into stochastic objects $\<1>$, $\<30>$, $\<320>$, ..., which result in more complicated analysis in establishing the regularities for the stochastic objects (see \cite[Proposition~6.4]{LTzW}).

An interesting observation is that the time regularity is used in both the parabolic setting and the hyperbolic setting, although in different manners. In the parabolic model, the time regularity is mainly used to trade for spatial regularity, which is manifested in the use of Lemma~\ref{LEM:Schau}. Because of this, a separate analysis with less loss of time regularity is required if one wants to extend the local-in-time solution to a global-in-time solution (see Lemma~\ref{LEM:unibdd_h}). In the hyperbolic model, however, one has gain of time regularity (via the $X^{s, b}$-spaces) in addition to the gain of spatial regularity, which allows one to treat the local-in-time theory and the globalization argument in a slightly more unified manner. Nevertheless, in both models, the time regularity can be effectively used in the stability estimates (Lemma~\ref{LEM:stb_h} in this paper and Proposition~6.5 in \cite{LTzW}), where we lose some slight amount of time regularity to obtain a decay from the exponential time weight.
\end{remark}

\subsection{Convergence of the fractional parabolic macroscopic model}
\label{SUB:heat_wu}
In this section, we study the fractional parabolic macroscopic model \eqref{SNLHN}. Similar to the truncated fractional parabolic $\Phi_3^4$-model \eqref{SNLHNd}, the equation \eqref{SNLHN} can be written as
\begin{align*}
\dt X_N = - \mathcal{E}_N' (X_N) + \sqrt{2} \xi^h,
\end{align*}

\noi
where the energy $\mathcal{E}_N$ is defined by
\begin{align*}
\mathcal{E}_N (u) = \int_{\T^3} \frac 12 |\jb{\nb}^\al u|^2 + V_N (\pi_N u) dx.
\end{align*}

\noi
Note that the energy $\mathcal{E}_N$ yields the truncated Gibbs measure $\varrho_N$ in \eqref{nuN}. Thus, we obtain the following lemma similar to Lemma~\ref{LEM:inv_h}.
\begin{lemma}
\label{LEM:inv_h2}
Let $\al > 1$ and $N \in \N$. Then, the parabolic macroscopic model \eqref{SNLHN} is almost surely globally well-posed with the random initial data distributed by the truncated Gibbs measure $\varrho_N$ in \eqref{nuN}. Moreover, $\varrho_N$ is invariant under the resulting dynamics of \eqref{SNLHN}.
\end{lemma}

We now consider the situation as $N \to \infty$. In view of the weak convergence of the truncated Gibbs measure $\varrho_N$ in \eqref{nuN} to the limiting Gibbs measure $\rho$ in \eqref{rho} (Theorem~\ref{THM:meas}), it follows from the Skorokhod representation theorem (Lemma~\ref{LEM:skorok}) that there exist a probability space $(\Om_0, \F_0, \mathbb{P}_0)$ and random variables $X_{0, N}$ with $N \in \N$ and $\Law (X_{0, N}) = \varrho_N$ such that $X_{0, N}$ converges $\mathbb{P}_0$-almost surely to some $X_0$ with $\Law (X_0) = \rho$ in $\mathcal{C}^{\al - \frac 32 - \eps} (\T^3)$.

For each $N \in \N$, we equip the equation \eqref{SNLHN} with initial data $X_{0, N}$ given above. 
As in Subsection~\ref{SUB:SNLH}, we want to proceed with a second order expansion similar to \eqref{sec_XN}, but we need to extend the definition of the cubic stochastic object $\<30>_N^{\,h}$. Similar to \eqref{cubic}, we define the following extended cubic stochastic object in the heat setting:
\begin{align}
Z_N^{(3)} (t) \deff \sum_{j = 2}^m 2j \cj{a}_{j, N} N^{- (2j - 4) (\frac 32 - \al)} \pi_N \int_{- \infty}^t e^{- (t - t') \jb{\nb}^{2 \al}} \big( \wick{(\<1>_N^h (t'))^{2j - 1}} \big) d t'
\label{cub_h2}
\end{align}

\noi
for $t \geq 0$. We proceed with the following second order expansion
\begin{align}
X_N = \<1>^h - Z_N^{(3)} + Y_N,
\label{sec_expN}
\end{align}

\noi
where $\<1>^h$ is the linear stochastic object in \eqref{lin_h} and $Y_N$ is a remainder term satisfying
\begin{align}
\begin{split}
&\dt Y_N + (1 - \Dl)^\al Y_N = - 2 \cj{a}_{1, N} N^{2 (\frac 32 - \al)} \big( \<1>_N^{h} - Z_N^{(3)} + \pi_N Y_N \big) \\
&\quad - \sum_{j = 2}^m \sum_{\l = 1}^{2j - 1} \sum_{k = 0}^\l C_{j, \l, k} \cj{a}_{j, N} N^{(2j - 4) (\al - \frac 32)} \wick{(\<1>_N^{h})^{2j - 1 - \l}} (- Z_N^{(3)})^k (\pi_N Y_N)^{\l - k}
\end{split}
\label{SNLHYN1}
\end{align}

\noi
where $C_{j, \l, k} = 2j \binom{2j - 1}{\l} \binom{\l}{k}$. To deal with the terms on the right-hand-side of \eqref{SNLHYN1} we define the following extended quadratic stochastic object
\begin{align}
z_N^{(2)} (t) \deff \sum_{j = 2}^m C_{j, 1, 0} \cj{a}_{j, N} N^{(2j - 4) (\al - \frac 32)} \wick{( \<1>_N^h )^{2j - 2}}
\label{quad_h2}
\end{align}

\noi
and the following extended quintic stochastic object
\begin{align}
Z_N^{(5)} (t) \deff \sum_{j = 2}^m C_{j, 1, 1} \cj{a}_{j, N} N^{(2j - 4) (\al - \frac 32)} \pi_N \I^h \big( \wick{(\<1>_N^h)^{2j - 2}} Z_N^{(3)} \big) (t)
\label{quin_h2}
\end{align}

\noi
for $t \geq 0$, where $\I^h$ is as defined in \eqref{quin_h}. Note that we deliberately use the little case $z$ in \eqref{quad_h2} to emphasize that there is no operator $\I^h$ in it.

Let us first establish the regularity properties of the above stochastic terms in the following lemma.
\begin{lemma}
\label{LEM:sto_h2}
Let $1 < \al \leq \frac 98$, $\eps > 0$ be sufficiently small, and $T \geq 1$. We consider the following stochastic objects $Z_N$ and $Z$ (some of which are defined in Lemma~\ref{LEM:sto_h}) along with their regularity $s$ as follows:
\begin{enumerate}
\item[(i)] $Z_N = \<1>^h_N = \pi_N \<1>^h$, $Z = \<1>^h$, $s = \al - \frac 32 - \eps$;  \smallskip

\item[(ii)] $Z_N = N^{k (\al - \frac 32)} \wick{( \<1>^h_N )^{k + 1}}$ with $k \geq 1$, $Z = 0$, $s = \al - \frac 32 - \eps$;  \smallskip

\item[(iii)] $Z_N = z_N^{(2)}$ as in \eqref{quad_h2}, $Z = 12 \cj{a}_2 \<2>^{\,h}$, $s = 2 \al - 3 - \eps$; \smallskip

\item[(iv)] $Z_N = Z_N^{(3)}$ as in \eqref{cub_h2}, $Z = 4 \cj{a}_2 \<30>^{\,h}$, $s = 5 \al - \frac 92 - \eps$;  \smallskip

\item[(v)] $Z_N = Z_N^{(5)}$ as in \eqref{quin_h2}, $Z = 48 \cj{a}_2^2 \<320>^{h}$, $s = 4 \al - 3- \eps$.
\end{enumerate}

\noi
Then, $Z_N$ converges almost surely to $Z$ in the space $C ([0, T]; \C^s (\T^3))$ as $N \to \infty$. We also have the tail estimate
\begin{align}
\mathbb{P}^h \big( \| Z_N \|_{C_T \C_x^s} > \ld \big) \leq C(T) e^{- c (T) \ld^\be}
\label{tail_h2}
\end{align}

\noi
for any $\ld > 0$ and some constants $C(T), c(T), \be > 0$, uniformly in $N \in \N \cup \{\infty\}$, with the understanding that $Z_\infty = Z$. Moreover, there exists small $\ta > 0$ such that
\begin{align}
\mathbb{P}^h \big( N_2^\ta \| Z_{N_1} - Z_{N_2} \|_{C_T \C_x^s} > \ld \big) \leq C(T) e^{- c (T) \ld^\be}
\label{taild_h2}
\end{align}

\noi
for any $\ld > 0$ and some constants $C(T), c(T), \be > 0$, uniformly in $N_1 \geq N_2 \geq 1$.
\end{lemma}
\begin{proof}
As in the proof of Lemma~\ref{LEM:sto_h}, we only focus on proving the uniform-in-$N$ variance bound.

\smallskip \noi
(i) This is the same as Lemma~\ref{LEM:sto_h}~(i).

\smallskip \noi
(ii) As in the proof of Lemma~\ref{LEM:sto_h}, we only focus on the estimate at a fixed time $t \in [0, T]$, since the continuity in time can be obtained by using similar time difference estimate as in the proof of Lemma~\ref{LEM:sto_h}~(i). 

For any $n \in \Z^3$ and $t \in [0, T]$, from \eqref{lin_hk}, the Ito isometry (Lemma~\ref{LEM:sto_ito}~(ii)), and Lemma~\ref{LEM:conv}~(ii) along with $1 < \al \leq \frac 98$, we obtain
\begin{align*}
\E \Big[ &\big| N^{k (\al - \frac 32) + \eps'} \F_x \big( \wick{( \<1>^h_N )^{k + 1}} \big) (t, n) \big|^2 \Big] \\
&\les N^{k (2 \al - 3) + 2 \eps'} \sum_{\substack{ |n_1|, \dots, |n_{k + 1}| \leq N \\ n_1 + \cdots + n_{k + 1} = n}} \prod_{\l = 1}^{k + 1} \frac{1}{\jb{n_\l}^{2 \al}} \\
&\les \sum_{\substack{n_1, \dots, n_{k + 1} \in \Z^3 \\ n_1 + \cdots + n_{k + 1} = n}} \prod_{\l = 1}^{k + 1} \frac{1}{\jb{n_\l}^{2 \al + \frac{k}{k + 1} (3 - 2 \al) - \frac{2 \eps'}{k + 1}}} \\
&\les \jb{n}^{- 2 \al + 2 \eps'}
\end{align*}

\noi
for any $\eps' > 0$ arbitrarily small. Also, from the above computation, for any $N_1 \geq N_2 \geq 1$, we have
\begin{align*}
\E \Big[ &\big| N_1^{k (\al - \frac 32)} \F_x \big( \wick{( \<1>^h_{N_1} )^{k + 1}} \big) - N_2^{k (\al - \frac 32)} \F_x \big( \wick{( \<1>^h_{N_2} )^{k + 1}} \big) (t, n) \big|^2 \Big] \\
&\les N_1^{-2 \eps'} \E \Big[ \big| N_1^{k (\al - \frac 32) + \eps'} \F_x \big( \wick{( \<1>^h_{N_1} )^{k + 1}} \big) (t, n) \big|^2 \Big] \\
&\quad + N_2^{-2 \eps'} \E \Big[ \big| N_2^{k (\al - \frac 32) + \eps'} \F_x \big( \wick{( \<1>^h_{N_2} )^{k + 1}} \big) (t, n) \big|^2 \Big] \\
&\les N_2^{-2 \eps'} \jb{n}^{- 2 \al + 2 \eps'}
\end{align*}

\noi
for any $\eps' > 0$ arbitrarily small. Thus, from Lemma~\ref{LEM:reg}, we obtain the desired regularity properties of $N^{k (\al - \frac 32)} \wick{ (\<1>_N^h)^{k + 1} }$.

\smallskip \noi
(iii) From \eqref{quad_h2}, when $j = 2$, we have $C_{2, 1, 0} = 12$, so that the estimate for this term is the same as that in Lemma~\ref{LEM:sto_h}~(ii) with the additional observation that $\cj{a}_{2, N} \to \cj{a}_2$ as $N \to \infty$ in view of \eqref{aj}, \eqref{ajN}, and \eqref{sigma_diff}. When $j \geq 3$, for any $n \in \Z^3$ and $t \in [0, T]$, we use similar steps from part (ii) to obtain
\begin{align*}
\E \Big[ &\big| N^{(2j - 4) (\al - \frac 32) + \eps'} \F_x \big( \wick{ (\<1>_N^h)^{2j - 2} } \big) (t, n) \big|^2 \Big] \\
&\les N^{(2j - 4) (2 \al - 3) + 2 \eps'} \sum_{\substack{|n_1|, \dots, |n_{2j - 2}| \leq N \\ n_1 + \cdots + n_{2j - 2} = n}} \prod_{\l = 1}^{2j - 2} \frac{1}{\jb{n_\l}^{2 \al}} \\
&\les \sum_{\substack{n_1, \dots, n_{2j - 2} \in \Z^3 \\ n_1 + \cdots + n_{2j - 2} = n}} \prod_{\l = 1}^{2j - 2} \frac{1}{\jb{n_\l}^{2 \al + \frac{2j - 4}{2j - 2} (3 - 2 \al) - \frac{2 \eps'}{2j - 2}}} \\
&\les \jb{n}^{- 4 \al + 3 + 2 \eps'}
\end{align*}

\noi
for any $\eps' > 0$ arbitrarily small. Thus, from Lemma~\ref{LEM:reg}, we obtain the desired regularity properties of $z_N^{(2)}$.

\smallskip \noi
(iv) Similar to part (iii), the term in \eqref{cub_h2} with $j = 2$ can be treated in the same way as in Lemma~\ref{LEM:sto_h}~(iii) with the additional observation that $\cj{a}_{2, N} \to \cj{a}_2$ as $N \to \infty$ in view of \eqref{aj}, \eqref{ajN}, and \eqref{sigma_diff}. Let $n \in \Z^3$ and $t \in [0, T]$. When $j \geq 3$, by \eqref{cub_h2} and Fubini's theorem (Lemma~\ref{LEM:sto_fub}), the $n$th Fourier coefficient of each term in \eqref{cub_h2} is essentially given by
\begin{align}
\begin{split}
&N^{- (2j - 4) (\frac 32 - \al)} \ft{Z_{N, j}^{(3)}} (t, n) \\
&\quad \deff  N^{- (2j - 4) (\frac 32 - \al)} \ind_{\{|n| \leq N\}} \sum_{\substack{ |n_1|, \dots, |n_{2j - 1}| \leq N \\ n_1 + \cdots + n_{2j - 1} = n}} \int_{(-\infty, t]^{2j - 1}} \int_{\max\{s_1, \dots, s_{2j - 1}\}}^t e^{- (t - t') \jb{n}^{2 \al}} \\
&\quad \quad \times \prod_{\l = 1}^{2j - 1} e^{- (t' - s_\l) \jb{n_\l}^{2 \al}} dt' dB_{n_1}^h (s_1) \cdots dB_{n_{2j - 1}}^h (s_{2j - 1})
\end{split}
\label{Z3Nj}
\end{align}

\noi
in $L^2 (\Om^h)$, where we omit the constant coefficient. Thus, by the Ito isometry (Lemma~\ref{LEM:sto_ito}~(ii)), similar computations in Lemma~\ref{LEM:sto_h}~(iii), and Lemma~\ref{LEM:conv}~(ii) along with $1 < \al \leq \frac 98$, we obtain
\begin{align*}
\E \Big[ &\big| N^{- (2j - 4) (\frac 32 - \al)} \ft{Z_{N, j}^{(3)}} (t, n) \big|^2 \Big] \\
&\les N^{- (2j - 4) (3 - 2 \al)} \sum_{\substack{ |n_1|, \dots, |n_{2j - 1}| \leq N \\ n_1 + \cdots + n_{2j - 1} = n}} \int_{(- \infty, t]^{2j - 1}} \bigg( \int_{\max \{ s_1, \dots, s_{2j - 1} \}}^t e^{- (t - t') \jb{n}^{2 \al}} \\
&\quad \times \prod_{\l = 1}^{2j - 1} e^{- (t' - s_\l) \jb{n_\l}^{2 \al}} dt'  \bigg)^2 d s_1 \cdots d s_{2j - 1} \\
&\les N^{- (2j - 4) (3 - 2 \al)} \sum_{\substack{ |n_1|, \dots, |n_{2j - 1}| \leq N \\ n_1 + \cdots + n_{2j - 1} = n}} \frac{1}{\max \{ \jb{n_1}^{2 \al}, \dots, \jb{n_{2j - 1}}^{2 \al} \}} \\
&\quad \times \bigg( \prod_{\l = 1}^{2j - 1} \frac{1}{\jb{n_\l}^{2 \al}} \bigg) \frac{1}{\jb{n}^{2 \al}} \\
&\les N^{- \eps'} \frac{1}{\jb{n}^{4 \al}} \sum_{\substack{n_1, \dots, n_{2j - 1} \in \Z^3 \\ n_1 + \cdots + n_{2j - 1} = n}} \prod_{\l = 1}^{2j - 1} \frac{1}{\jb{n_\l}^{2 \al + \frac{2j - 4}{2j - 1} (3 - 2 \al) - \frac{\eps'}{2j - 1}}} \\
&\les N^{- \eps'} \jb{n}^{- 10 \al + 6 + \eps'}
\end{align*}

\noi
for any $\eps' > 0$ arbitrarily small. Thus, from Lemma~\ref{LEM:reg}, we obtain the desired regularity properties of $Z_N^{(3)}$.

\smallskip \noi
(v) From \eqref{quin_h2} and Lemma~\ref{LEM:Schau}, we need to show the properties of
\begin{align*}
z_{N, j, k}^{(5)} &\deff N^{- (2j + 2k - 8) (\frac 32 - \al) + \eps'}  \pi_N \big( \wick{ (\<1>_N^h)^{2j - 2} } Z_{N, k}^{(3)} \big)
\end{align*}

\noi
in the space $C_T \mathcal{C}^{2 \al - 3 - \frac{\eps}{2}}$, where $Z_{N, k}^{(3)}$ is as defined in \eqref{Z3Nj}. Here, we let $\eps' = 0$ if $j = 2$ and $k = 2$; otherwise, we let $\eps' > 0$ be small. We denote $n_1, \dots, n_{2j - 2}$ as the frequencies of the nodes of $\wick{ (\<1>_N^h)^{2j - 2} }$ and denote $n_{2j - 1}, \dots, n_{2j + 2k - 3}$ as the frequencies of the nodes of $Z_{N, k}^{(3)}$. From the product formula in Lemma~\ref{LEM:sto_prod}, we need to discuss the number of pairs between $\wick{ (\<1>_N^h)^{2j - 2} }$ and $Z_{N, k}^{(3)}$.

If there are no pairs, by the Ito isometry (Lemma~\ref{LEM:sto_ito}~(ii)), \eqref{jen}, and Lemma~\ref{LEM:conv}~(ii) and then (i) along with $1 < \al \leq \frac 98$, we obtain the following contribution for the variance of $z_{N, j, k}^{(5)}$ at frequency $n$:
\begin{align*}
&\sum_{\substack{|n_1|, \dots, |n_{2j + 2k - 3}| \leq N \\ n_1 + \cdots n_{2j + 2k - 3} = n}} \int_{( - \infty, t ]^{2j + 2k - 3}} N^{- (2j + 2k - 8) (3 - 2 \al) + 2 \eps'} \prod_{\l = 1}^{2j - 2} e^{- 2 (t - s_\l) \jb{n_\l}^{2 \al}} \\
&\quad \quad \times \bigg( \int_{\max \{ s_{2j - 1}, \dots, s_{2j + 2k - 3} \}}^t e^{- (t - t') \jb{ \sum_{\l' = 2j - 1}^{2j + 2k - 3} n_{\l'} }^{2 \al}} \\
&\quad \quad \times \prod_{\l' = 2j - 1}^{2j + 2k - 3} e^{- (t' - s_{\l'}) \jb{n_{\l'}}^{2 \al}} dt' \bigg)^2 ds_{1} \cdots ds_{2j + 2k - 3} \\
&\quad \les \sum_{\substack{n_1, \dots, n_{2j + 2k - 3} \in \Z^3 \\ n_1 + \cdots + n_{2j + 2k - 3} = n}} \frac{1}{\max \{ \jb{n_{2j - 1}}^{2 \al}, \dots, \jb{n_{2j + 2k - 3}}^{2 \al} \} \jb{\sum_{\l' = 2j - 1}^{2j + 2k - 3} n_{\l'}}^{2 \al}} \\
&\quad \quad \times \prod_{\l = 1}^{2j - 2} \frac{1}{\jb{n_\l}^{2 \al + \frac{2j - 4}{2j - 2} (3 - 2 \al - \eps'')}} \prod_{\l' = 2j - 1}^{2j + 2k - 3} \frac{1}{\jb{n_{\l'}}^{2 \al + \frac{2k - 4}{2k - 1} (3 - 2 \al - \eps'')}} \\
&\quad \les \sum_{\substack{n', n'' \in \Z^3 \\ n' + n'' = n}} \frac{1}{\jb{n'}^{4 \al - 3 - (2j - 4) \eps''} \jb{n''}^{10 \al - 6 - (2k - 4)\eps''}} \\
&\quad \les \jb{n}^{- 4 \al + 3 + (2j - 4) \eps''}
\end{align*}

\noi
for some $\eps'' > 0$ small. This gives the desired properties in view of Lemma~\ref{LEM:reg}.

If there are pairs, recalling Definition~\ref{DEF:pair}, we let $\mathcal P \neq \varnothing$ be a pairing of $\{ 1, \dots, 2j + 2k - 3 \}$ which respects the partition $S_1 \deff \{ 1, \dots, 2j - 2 \}$ and $S_2 \deff \{ 2j - 1, \dots, 2j + 2k - 3 \}$. By the Ito isometry (Lemma~\ref{LEM:sto_ito}~(ii)) and \eqref{jen}, we know that our main task is to estimate the following term:
\begin{align}
\begin{split}
&N^{- (2j + 2k - 8) (3 - 2 \al) + 2 \eps'} \sum_{\substack{(n_\l)_{\l \notin \mathcal P}, |n_\l| \leq N \\ \sum_{\l \notin \mathcal P} n_\l = n}} \int_{( - \infty, t ]^{2j + 2k - 3 - 2 |\mathcal P|}} \Big( \prod_{\l \notin \mathcal{P} \wedge \l \in S_1} e^{- 2 (t - s_\l) \jb{n_\l}^{2 \al}} \Big) \\
&\quad \times \bigg( \sum_{\substack{(n_{\l'})_{\l' \in \mathcal{P} \wedge \l' \in S_2} \\ |n_{\l'}| \leq N}} \ind_{\substack{(n_1, \dots, n_{2j + 2k - 3}) \\ \text{admissible}}} \int_{(- \infty, t]^{|\mathcal P|}} \Big( \prod_{\l' \in \mathcal{P} \wedge \l' \in S_2} e^{- (t - s_{\l'}) \jb{n_{\l'}}^{2 \al}} \Big)  \\
&\quad \times \int_{\max_{\l' \in S_2} \{ s_{\l'} \}}^t e^{- (t - t') \jb{\sum_{\l' \in S_2} n_{\l'}}^{2 \al}} \prod_{\l' \in S_2} e^{- (t' - s_{\l'}) \jb{n_{\l'}}^{2 \al}} dt' d (s_\l)_{\l \in \mathcal{P} \wedge \l \in S_2} \bigg)^2 d (s_\l)_{\l \notin \mathcal P} \\
&\les N^{- (2j + 2k - 8) (3 - 2 \al) + 2 \eps'} \sum_{\substack{(n_\l)_{\l \notin \mathcal P}, |n_\l| \leq N \\ \sum_{\l \notin \mathcal P} n_\l = n}} \bigg( \sum_{\substack{(n_{\l'})_{\l' \in \mathcal{P} \wedge \l' \in S_2} \\ |n_{\l'}| \leq N}} \frac{1}{\max_{\l' \in S_2} \{ \jb{n_{\l'}}^{2 \al} \}}  \\
&\quad \times  \prod_{\l' \in \mathcal{P} \wedge \l' \in S_2} \frac{1}{\jb{n_{\l'}}^{2 \al}} \bigg) \bigg( \sum_{\substack{(n_{\l'})_{\l' \in \mathcal{P} \wedge \l' \in S_2} \\ |n_{\l'}| \leq N}} \frac{1}{\jb{\sum_{\l' \in S_2} n_{\l'}}^{2 \al}} \prod_{\l \in S_2 \vee \l \notin \mathcal{P} } \frac{1}{\jb{n_\l}^{2 \al}} \bigg).
\end{split}
\label{5jk}
\end{align}

\noi
We consider the following two cases.

\smallskip \noi
\textbf{Case 1:} $|\mathcal P| = 1$.

By Lemma~\ref{LEM:conv}~(i) and (ii) several times along with $1 < \al \leq \frac 98$, we obtain the following contribution for the bound of \eqref{5jk}:
\begin{align*}
&\sum_{\substack{(n_\l)_{\l \notin \mathcal P} \\ \sum_{\l \notin \mathcal P} n_\l = n}} \bigg( \sum_{ (n_{\l'})_{\l' \in \mathcal{P} \wedge \l' \in S_2} } \frac{1}{\max_{\l' \in S_2} \{ \jb{n_{\l'}}^{2 \al} \}} \prod_{\l' \in \mathcal{P} \wedge \l' \in S_2} \frac{1}{\jb{n_{\l'}}^{2 \al}} \bigg)  \bigg( \sum_{ (n_{\l'})_{\l' \in \mathcal{P} \wedge \l' \in S_2} } \frac{1}{\jb{\sum_{\l' \in S_2} n_{\l'}}^{2 \al}}  \\
&\quad \times \prod_{ \l' \in \mathcal{P} \wedge \l' \in S_2 } \frac{1}{\jb{n_{\l'}}^{2 \al}} \prod_{ \l' \notin \mathcal{P} \wedge \l' \in S_2 } \frac{1}{\jb{n_{\l'}}^{2 \al + \frac{2k - 4}{2k - 2} (3 - 2 \al - \eps'')}}  \prod_{\l \notin \mathcal{P} \wedge \l \in S_1} \frac{1}{\jb{n_\l}^{2 \al + \frac{2j - 4}{2j - 3} (3 - 2 \al - \eps'')}} \bigg) \\
&\les \sum_{\substack{(n_\l)_{\l \notin \mathcal P} \\ \sum_{\l \notin \mathcal P} n_\l = n}} \frac{1}{\jb{\sum_{\l' \notin \mathcal{P} \wedge \l' \in S_2} n_{\l'}}^{8 \al - 6}} \prod_{ \l' \notin \mathcal{P} \wedge \l' \in S_2} \frac{1}{\jb{n_{\l'}}^{2 \al + \frac{2k - 4}{2k - 2} (3 - 2 \al - \eps'')}}  \\
&\quad \times \prod_{\l \notin \mathcal{P} \wedge \l \in S_1} \frac{1}{\jb{n_\l}^{2 \al + \frac{2j - 4}{2j - 3} (3 - 2 \al - \eps'')}} \\
&\les \sum_{\substack{n', n'' \in \Z^3 \\ n' + n'' = n}} \frac{1}{\jb{n'}^{12 \al - 9 - (2k - 4) \eps''} \jb{n''}^{2 \al - (2j - 4) \eps''}} \\
&\les \jb{n}^{- 2 \al + (2j - 4) \eps''}
\end{align*}

\noi
for some small $\eps'' > 0$. This gives the desired properties in view of Lemma~\ref{LEM:reg}.

\smallskip \noi
\textbf{Case 2:} $|\mathcal P| \geq 2$.

By Lemma~\ref{LEM:conv}~(ii) several times and then (i) along with $1 < \al \leq \frac 98$, we obtain the following bound for \eqref{5jk}:
\begin{align*}
&\sum_{\substack{(n_\l)_{\l \notin \mathcal P} \\ \sum_{\l \notin \mathcal P} n_\l = n}} \bigg( \sum_{ (n_{\l'})_{\l' \in \mathcal{P} \wedge \l' \in S_2} } \frac{1}{\max_{\l' \in S_2} \{ \jb{n_{\l'}}^{2 \al} \}} \prod_{\l' \in \mathcal{P} \wedge \l' \in S_2} \frac{1}{\jb{n_{\l'}}^{2 \al + \frac{|\mathcal P| - 2}{|\mathcal P|} (3 - 2 \al - \eps'')}} \bigg) \\
&\quad \times \bigg( \sum_{ (n_{\l'})_{\l' \in \mathcal{P} \wedge \l' \in S_2} } \frac{1}{\jb{\sum_{\l' \in S_2} n_{\l'}}^{2 \al}} \prod_{ \l' \in \mathcal{P} \wedge \l' \in S_2 } \frac{1}{\jb{n_{\l'}}^{2 \al + \frac{|\mathcal P| - 2}{|\mathcal P|} (3 - 2 \al - \eps'')}} \\
&\quad \quad \times \prod_{ \l' \notin \mathcal{P} \wedge \l' \in S_2} \frac{1}{\jb{n_{\l'}}^{2 \al + \frac{2k - 2 - |\mathcal P|}{2k - 1 - |\mathcal P|} (3 - 2 \al - \eps'')}}  \prod_{\l \notin \mathcal{P} \wedge \l \in S_1} \frac{1}{\jb{n_\l}^{3 - \eps''}} \bigg) \\
&\les \sum_{\substack{(n_\l)_{\l \notin \mathcal P} \\ \sum_{\l \notin \mathcal P} n_\l = n}} \frac{1}{\jb{\sum_{\l' \notin \mathcal{P} \wedge \l' \in S_2} n_{\l'}}^{12 \al - 12 - 2 (|\mathcal P| - 2) \eps''}} \prod_{ \l' \notin \mathcal{P} \wedge \l' \in S_2} \frac{1}{\jb{n_{\l'}}^{2 \al + \frac{2k - 2 - |\mathcal P|}{2k - 1 - |\mathcal P|} (3 - 2 \al - \eps'')}}  \\
&\quad \times \prod_{\l \notin \mathcal{P} \wedge \l \in S_1} \frac{1}{\jb{n_\l}^{3 - \eps''}} \\
&\les \sum_{\substack{n', n'' \in \Z^3 \\ n' + n'' = n}} \frac{1}{\jb{n'}^{14 \al - 12 - (2k + |\mathcal P| - 6) \eps''} \jb{n''}^{3 - (2j - 2 - |\mathcal P|) \eps''}} \\
&\les \jb{n}^{-14 \al + 12 + (2j + 2k - 8) \eps''} 
\end{align*}

\noi
for some small $\eps'' > 0$. This gives the desired properties in view of Lemma~\ref{LEM:reg}.
\end{proof}

We now show a key stability estimate of the equation \eqref{SNLHYN1} for $Y_N$. Let us write out the Duhamel formulation of the equation \eqref{SNLHYN1}:
\begin{align}
\begin{split}
Y_N &= e^{- t \jb{\nb}^{2 \al}} \big( X_{0, N} - \<1>^h (0) + Z_N^{(3)} (0) \big) - 2 \cj{a}_{1, N} N^{2 (\frac 32 - \al)} \I^h \big( \<1>_N^{h} - Z_N^{(3)} + \pi_N Y_N \big) \\
&\quad + Z_N^{(5)} - \pi_N \I^h \big( z_N^{(2)} \cdot \pi_N Y_N \big) - \sum_{j = 2}^m \sum_{\l = 2}^{2j - 1} \sum_{k = 0}^\l C_{j, \l, k} \cj{a}_{j, N} N^{(2j - 4) (\al - \frac 32)} \\
&\qquad \times \pi_N \I^h \big( \wick{(\<1>_N^{h})^{2j - 1 - \l}} (- Z_N^{(3)})^k (\pi_N Y_N)^{\l - k} \big).
\end{split}
\label{SNLHYN}
\end{align}

\noi
We also define the enhanced data sets
\begin{align}
\cj{\pmb{Z}}_N \deff \big( \<1>_N^h, N^{\al - \frac 32} \wick{( \<1>_N^h )^2}, \dots, N^{(2m - 4) (\al - \frac 32)} \wick{( \<1>_N^h )^{2m - 3}}, z_N^{(2)}, Z_N^{(3)}, Z_N^{(5)} \big),
\label{defZN2}
\end{align}
\begin{align}
\cj{\pmb{Z}} \deff \big( \<1>^h, 0, \dots, 0, 12 \cj{a}_2  \<2>^{\,h}, 4 \cj{a}_2 \<30>^{\,h},48 \cj{a}_2^2 \<320>^h \big)
\label{defZ2}
\end{align}

\noi
and the space
\begin{align}
\cj{\mathcal{Z}}_T^{\al, \eps} \deff ( C_T \mathcal{C}^{\al - \frac 32 - \eps} )^{\otimes (2m - 3)} \times C_T \mathcal{C}^{2\al - 3 - \eps} \times C_T \mathcal{C}^{5\al - \frac 92 - \eps} \times C_T \mathcal{C}^{4\al - 3 - \eps}.
\label{defZTae2}
\end{align}

\begin{lemma}
\label{LEM:stb_h2}
Let $1 < \al \leq \frac 98$, $\eps > 0$ be sufficiently small, and $T \geq 1$. Let $K_0, K_1 \gg 1$. Then, there exist $N_0 = N_0 (T, K, C_0) \in \N$ sufficiently large and $0 < \kappa_0 = \kappa_0 (T, K, C_0) \leq 1$ sufficiently small such that the following holds true. Suppose that
\begin{align}
\| \cj{\pmb{Z}} \|_{\cj{\mathcal{Z}}_T^{\al, \eps}} \leq K_0
\label{Z_bdd2}
\end{align}

\noi
with the $\cj{\mathcal{Z}}_T^{\al, \eps}$-norm defined in \eqref{defZTae2} and the limiting equation \eqref{SNLHY} with the enhanced data set $\cj{\pmb{Z}}$ has a solution $Y = Y (X_0, \cj{\pmb{Z}})$ on $[0, T]$, satisfying
\begin{align}
\| Y \|_{\mathcal{Y}_{T'}^{\al, \eps}} + \| Y \|_{C_{[T', T]} \mathcal{C}_x^{1 + \eps}} \leq K_1
\label{Y_bdd2}
\end{align}

\noi
for some $0 < T' \leq 1$ with the $\mathcal{Y}_{T'}^{\al, \eps}$-norm defined in \eqref{defYtae}.
We also let $N \geq N_0$ and $Y_N = Y_N (X_{0, N}, \cj{\pmb{Z}}_N)$ be the solution to the equation \eqref{SNLHYN} on $[0, T]$ with the truncated enhanced data set $\cj{\pmb{Z}}_N$ and suppose that
\begin{align}
\| \cj{\pmb{Z}}_{N} - \cj{\pmb{Z}} \|_{\cj{\mathcal{Z}}_T^{\al, \eps}} \leq N^{- \ta}.
\label{Z_diff2}
\end{align}

\noi
and
\begin{align}
\| X_{0, N} - X_0 \|_{\mathcal{C}^{\al - \frac 32 - \eps}} \leq \kappa
\label{X0_diff}
\end{align}

\noi
for some $0 < \ta \ll 1$ and $0 < \kappa \leq \kappa_0$. Then, there exists $0 < T_0 = T_0 (K_0, K_1) \leq T' \leq 1$ small such that
\begin{align}
\| Y_{N} - Y \|_{C_{[T_0, T]} \mathcal{C}_x^{1 + \eps}} \leq C(T, K_0, K_1) (\kappa + N^{- \ta'})
\label{YN_diff2}
\end{align}

\noi
for some constant $C (T, K_0, K_1) > 0$ and some $\ta' > 0$.
\end{lemma}

\begin{proof}
From \eqref{Z_bdd2} and \eqref{Z_diff2}, we have the following uniform bound for the truncated enhanced data set:
\begin{align}
\| \cj{\pmb{Z}}_N \|_{\cj{\mathcal{Z}}_T^{\al, \eps}} \leq K + 1.
\label{ZN_bdd2}
\end{align}

By using the Duhamel formulations \eqref{SNLHY} and \eqref{SNLHYN}, we perform the following grouping for estimating $Y_N - Y$:
\begin{align*}
Y_N - Y = \sum_{j = 1}^{5} \textup{I}_j,
\end{align*}

\noi
where
\begin{align*}
\textup{I}_1 &\deff e^{- t \jb{\nb}^{2 \al}} (X_{0, N} - X_0) + e^{- t \jb{\nb}^{2 \al}} \big( Z_N^{(3)} (0) - 4 \cj{a}_2 \<30>^{\,h} (0) \big), \\
\textup{I}_2 &\deff - 2 \cj{a}_{1, N} N^{2 (\frac 32 - \al)} \I^h \big( \<1>_N^h - Z_N^{(3)} + \pi_N Y_N \big) + a \I^h \big( \<1>^h - 4 \cj{a}_2 \<30>^{\,h} + Y \big), \\
\textup{I}_3 &\deff \big( Z_N^{(5)} - 48 \cj{a}_2^2 \<320>^h \big) 
- \Big( \pi_N \I^h \big( z_N^{(2)} \cdot \pi_N Y_N \big) - 12 \cj{a}_2 \I^h (\<2>^{\,h} Y) \Big) \\
&\quad - \Big( 12 \cj{a}_{2, N} \pi_N \I^h (\<1>^h (\pi_N Y_N)^2) - 12 \cj{a}_2 \I^h (\<1>^h Y^2) \Big) \\
&\quad + \Big( 96 \cj{a}_{2, N} \pi_N \I^h (\<1>_N^h Z_N^{(3)} \cdot \pi_N Y_N) - 96 \cj{a}_2^2 \I^h (\<1>^h \<30>^{\,h} Y) \Big) \\
&\quad - \Big( 192 \cj{a}_{2, N} \pi_N \I^h \big( \<1>_N^h (Z_N^{(3)})^2 \big) - 192 \cj{a}_2^3 \I^h \big( \<1>^h (\<30>^{\,h})^2 \big) \Big) \\
&\quad - \Big( 4 \cj{a}_{2, N} \pi_N \I^h \big( (\pi_N Y_N)^3 \big) - 4 \cj{a}_2 \I^h (Y^3) \Big) \\
&\quad + \Big( 48 \cj{a}_{2, N} \pi_N \I^h \big( Z_N^{(3)} (\pi_N Y_N)^2 \big) - 48 \cj{a}_2^2 \I^h \big( \<30>^{\,h} Y^2 \big) \Big) \\
&\quad - \Big( 192 \cj{a}_{2, N} \pi_N \I^h \big( (Z_N^{(3)})^2 \pi_N Y_N \big) - 192 \cj{a}_2^3 \I^h \big( (\<30>^{\,h})^2 Y \big) \Big) \\
&\quad + \Big( 256 \cj{a}_{2, N} \pi_N \I^h \big( (Z_N^{(3)})^3 \big) - 256 \cj{a}_2^4 \I^h \big( (\<30>^{\,h})^3 \big) \Big), \\
\textup{I}_4 &\deff - \sum_{j = 3}^m C_{j, 2, 0} \cj{a}_{j, N} N^{(2j - 4) (\al - \frac 32)} \pi_N \I^h \big( \wick{(\<1>_N^h)^{2j - 3}} (\pi_N Y_N)^2 \big), \\
\textup{I}_5 &\deff - \sum_{j = 3}^m \sum_{\l = 3}^{2j - 1} \sum_{k = 0}^\l C_{j, \l, k} \cj{a}_{j, N} N^{(2j - 4) (\al - \frac 32)} \pi_N \I^h \big( (- Z_N^{(3)})^k \wick{(\<1>_N^h)^{2j - 1 - \l}} (\pi_N Y_N)^{\l - k} \big).
\end{align*}

\noi
For $\textup{I}_1$, $\textup{I}_2$, $\textup{I}_3$, and $\textup{I}_4$, we slightly modify the estimates in the proof of Lemma~\ref{LEM:LWP_h}~(i) in the following way: for $X_{0, N} - X_0$ we use \eqref{X0_diff}; for $2 \cj{a}_{1, N} N^{2 (\frac 32 - \al)} - a$ we use \eqref{a1N} to create a factor $N^{- \ta'}$ for some $\ta' > 0$; for $\<1>_N^h - \<1>^h$, $z_N^{(2)} - \<2>^{\,h}$, $Z_N^{(3)} - \<30>^{\,h}$, and $N^{(2j - 4) (\al - \frac 32)} \wick{(\<1>_N^h)^{2j - 3}}$, we use \eqref{Z_diff2}; for $\cj{a}_{2, N} - \cj{a}_2$ we use \eqref{aj}, \eqref{ajN}, and \eqref{sigma_diff} to create a factor $N^{- \ta'}$; for the appearance of $\pi_N^\perp = \Id - \pi_N$, we can also produce a factor $N^{- \ta'}$ by losing some small amount of spatial regularity. Thus, by \eqref{Z_bdd2}, \eqref{Y_bdd2}, and \eqref{ZN_bdd2}, we obtain that for any $0 < T_0 \leq T'$,
\begin{align}
\begin{split}
\| &\textup{I}_1 \|_{\mathcal{Y}_{T_0}^{\al, \eps}} + \| \textup{I}_2 \|_{\mathcal{Y}_{T_0}^{\al, \eps}} + \| \textup{I}_3 \|_{\mathcal{Y}_{T_0}^{\al, \eps}} + \| \textup{I}_4 \|_{\mathcal{Y}_{T_0}^{\al, \eps}} \\
&\les (\kappa + N^{- \ta'}) K_0^3 K_1^3 \Big( 1 + \| Y_N \|_{\mathcal{Y}_{T_0}^{\al, \eps}}^3 \Big) + T_0^{\dl} K_0^2 \Big( K_1^2 + \| Y_N \|_{\mathcal{Y}_{T_0}^{\al, \eps}}^2 \Big) \| Y_N - Y \|_{\mathcal{Y}_{T_0}^{\al, \eps}}.
\end{split}
\label{stb_h1}
\end{align}

\noi
For $\textup{I}_5$, we only consider the case $k = 0$, since the case $k \geq 1$ is simpler. Let us write
\begin{align*}
\textup{I}_5 = \textup{I}_5^1 + \textup{I}_5^2,
\end{align*}

\noi
where $\textup{I}_5^1$ denotes the contribution by $\l = 3$ and $\textup{I}_5^2$ denotes the contribution by $4 \leq \l \leq 2j - 1$. Then, by the Schauder estimate (Lemma~\ref{LEM:Schau}), the product estimate (Lemma~\ref{LEM:Bes}~(iv)), Bernstein's inequality, \eqref{Z_diff2}, and \eqref{interp2}, we have
\begin{align}
\begin{split}
\| \textup{I}_5^1 \|_{\mathcal{Y}_{T_0}^{\al, \eps}} &\les \sum_{j = 3}^m \sup_{0 \leq t \leq T_0} \int_0^t \bigg( \frac{1}{(t')^{- \frac{3 \al}{2} + \frac 94 + 3 \eps}} + \frac{t^{- \frac{\al}{2} + \frac 54 + \eps}}{(t - t')^{\frac 12 + \eps} (t')^{- \frac{3 \al}{2} + \frac 94 + 3 \eps}} \bigg) dt' \\
&\quad \times N^{(2j - 4) (\al - \frac 32)} \| \wick{(\<1>_N^h)^{2j - 4}} \|_{C_{T_0} \mathcal{C}_x^{-\eps}} \Big( \sup_{0 \leq t \leq T_0} t^{- \frac{\al}{2} + \frac 34 + \eps} \| Y_N (t) \|_{\mathcal{C}^\eps} \Big)^3 \\
&\les T_0^{\dl} N^{- \ta'} \| Y_N \|_{\mathcal{Y}_{T_0}^{\al, \eps}}^3,
\end{split}
\label{stb_h2}
\end{align}

\noi
and
\begin{align}
\begin{split}
\| \textup{I}_5^2 \|_{\mathcal{Y}_{T_0}^{\al, \eps}} 
&\les \sum_{j = 3}^m \sum_{\l = 4}^{2j - 1} \sup_{0 \leq t \leq T_0} \int_0^t \bigg( \frac{1}{(t')^{- 2 \al + 3 + 4 \eps}} + \frac{t^{- \frac{\al}{2} + \frac 54 + \eps}}{(t - t')^{\frac 12 + \eps} (t')^{- 2 \al + 3 + 4 \eps}} \bigg) dt' \\
&\quad \times N^{(2j - 1 - \ell) (\al - \frac 32)} \| \wick{(\<1>_N^h)^{2j - 1 - \l}} \|_{C_{T_0} \mathcal{C}_x^{-\eps}} \Big( \sup_{0 \leq t \leq T_0} t^{- \frac{\al}{2} + \frac 34 + \eps} \| Y_N (t) \|_{\mathcal{C}^\eps} \Big)^4 \\
&\quad \times N^{(\ell - 4) (\al - \frac 32 - 2 \eps)} \| \pi_N Y_N \|_{C_{T_0} \mathcal{C}_x^\eps}^{\l - 4} \cdot N^{\al - \frac 32 + 2 (\ell - 4) \eps} \\
&\les T_0^{\dl} N^{- \ta'} \Big( 1 + \| Y_N \|_{\mathcal{Y}_{T_0}^{\al, \eps}}^{2m - 1} \Big),
\end{split}
\label{stb_h3}
\end{align}

\noi
for some $\dl > 0$ and $\ta' > 0$, where we used
\begin{align*}
\al - \frac 32 + 2 (\ell - 4) \eps < 0
\end{align*}

\noi
given $1 < \al \leq \frac 98$ and $\eps > 0$ sufficiently small. Combining \eqref{stb_h1}, \eqref{stb_h2}, and \eqref{stb_h3}, we obtain
\begin{align}
\begin{split}
\| Y_N - Y \|_{\mathcal{Y}_{T_0}^{\al, \eps}} &\leq C_1 (\kappa + N^{- \ta'}) K_0^3 K_1^3 \Big( K_0^{2m - 1} + \| Y_N \|_{\mathcal{Y}_{T_0}^{\al, \eps}}^{2m - 1} \Big) \\
&\quad + C_1 T_0^\dl K_0^2 \Big( K_1^2 + \| Y_N \|_{\mathcal{Y}_{T_0}^{\al, \eps}}^2 \Big) \| Y_N - Y \|_{\mathcal{Y}_{T_0}^{\al, \eps}}
\end{split}
\label{stb_hdiff}
\end{align}

\noi
for some constant $C_1 \geq 1$. We now choose $T_0 = T_0 (K_0, K_1) > 0$ small in such a way that $17 C_1 T_0^\dl K_0^2 K_1^2 = \frac 12$. We also note that the function $t \mapsto \| u \|_{\mathcal{Y}_t^{\al, \eps}}$ is continuous for any fixed $u \in \mathcal{Y}_t^{\al, \eps}$. Thus, by choosing $\kappa = \kappa (K_0, K_1) > 0$ to be sufficiently small and $N = N (K_0, K_1) \gg 1$ to be suffciently large and using a standard continuity argument along with \eqref{stb_hdiff} and \eqref{Y_bdd2}, we obtain the bound
\begin{align*}
\| Y_N \|_{\mathcal{Y}_{T_0}^{\al, \eps}} \leq 4 K_1,
\end{align*}

\noi
and so
\begin{align}
\| Y_N (T_0) - Y (T_0) \|_{\mathcal{C}_x^{1 + \eps}} \leq C'(K_0, K_1) (\kappa + N^{- \ta'})
\label{stb_hdiff2}
\end{align}

\noi
for some constant $C' (K_0, K_1) > 0$, and we can assume that $C' (K_0, K_1) \geq 4^{2m - 1} T_0^{-2m + 1} K_0^{2m + 2} K_1^{2m + 2}$.

We now start from time $T_0$. Let $T_1 > 0$. We can use similar (and much easier) argument as above to modify the estimates in the proof of Lemma~\ref{LEM:LWP_h}~(ii), and so we obtain 
\begin{align}
\begin{split}
\| &Y_N - Y \|_{C_{[T_0, T_0 + T_1]} \mathcal{C}_x^{1 + \eps}} \\
&\leq C_2 C' (K_0, K_1) (\kappa + N^{- \ta'}) + C_2 (\kappa + N^{- \ta'}) K_0^3 K_1^3 \Big( K_0^{2m - 1} + \| Y_N \|_{C_{[T_0, T]} \mathcal{C}_x^{1 + \eps}}^{2m - 1} \Big) \\
&\quad + C_2 T_1^\dl K_0^2 \Big( K_1^2 + \| Y_N \|_{C_{[T_0, T]} \mathcal{C}_x^{1 + \eps}}^2 \Big) \| Y_N - Y \|_{C_{[T_0, T_0 + T_1]} \mathcal{C}_x^{1 + \eps}}
\end{split}
\label{stb_hdiff3}
\end{align}

\noi
for some constant $C_2 > 0$. Let us first assume that
\begin{align}
\| Y_N \|_{C_{[T_0, T]} \mathcal{C}_x^{1 + \eps}} \leq 4 T_0^{-1} K_1.
\label{stb_hbdd}
\end{align}

\noi
Then, from \eqref{stb_hdiff3} and \eqref{stb_hbdd}, by taking $T_1 = T_1 (K_0, K_1) > 0$ to be small enough such that $17 C_2 T_1^\dl T_0^{-2} K_0^2 K_1^2 = \frac 12$, we obtain 
\begin{align*}
\| Y_N - Y \|_{C_{[T_0, T_0 + T_1]} \mathcal{C}_x^{1 + \eps}} \leq 4 C_2 C' (K_0, K_1) (\kappa + N^{- \ta'}).
\end{align*}

\noi
Recall that $T_0$ depends only on $K_0$ and $K_1$, so we know that $T_1$ depends only on $K_0$ and $K_1$. Thus, we can iterate the above argument to obtain
\begin{align}
\| Y_N - Y \|_{C_{[T_0, T]} \mathcal{C}_x^{1 + \eps}} \leq 4^J C_2 C' (K_0, K_1) (\kappa + N^{- \ta'}),
\label{stb_hdiff4}
\end{align}

\noi
where $J \in \N$ satisfies $J \geq \frac{T - T_0}{T_1}$ so that $J$ depends only on $K_0$ and $K_1$. The estimates \eqref{stb_hdiff4} and \eqref{stb_hdiff2} then gives the desired \eqref{YN_diff2}.

It remains to verify the bound \eqref{stb_hbdd}. Assume that $T_* < T$ is the largest number such that
\begin{align*}
\| Y_N \|_{C_{[T_0, T_*]} \mathcal{C}_x^{1 + \eps}} \leq 2 T_0^{-1} K_1.
\end{align*}

\noi
Then, by continuity, there exists $t > 0$ such that $T_* + t \leq T$ and
\begin{align*}
\| Y_N \|_{C_{[T_0, T_* + t]} \mathcal{C}_x^{1 + \eps}} \leq 4 T_0^{-1} K_1.
\end{align*}

\noi
Thus, by the same iterative argument above, we obtain the estimate \eqref{stb_hdiff4} with $T$ replaced by $T_* + t$. Thus, we obtain from \eqref{Y_bdd2} that
\begin{align*}
\| Y_N \|_{C_{[T_0, T_* + t]} \mathcal{C}_x^{1 + \eps}} &\leq \| Y_N - Y \|_{C_{[T_0, T_* + t]} \mathcal{C}_x^{1 + \eps}} + \| Y \|_{C_{[T_0', T_* + t]} \mathcal{C}_x^{1 + \eps}} \\
&\leq 4^J C_2 C' (K_0, K_1) (\kappa + N^{- \ta'}) + T_0^{\frac{\al}{2} - \frac 54 - \eps} K_1 \\
&\leq 2 T_0^{-1} K_1
\end{align*}

\noi
given $\al > 1$, as long as we take $\kappa = \kappa (T, K_0, K_1) > 0$ to be sufficiently small and $N = N (T, K_0, K_1) \gg 1$ to be sufficiently large. This contradicts the definition of $T_*$, so that we must have $T_* = T$. This finishes the proof. 
\end{proof}

\subsection{Construction of random variables distributed by the Gibbs measures}
\label{SUB:rep}
Our main task in this subsection is to prove Theorem~\ref{THM:rep}.
\begin{proof}[Proof of Theorem \ref{THM:rep}]
We take the following probability space:
\begin{align*}
(\Sigma, \mathcal{G}, \mathbb{Q}) = (\Om_0 \times \Om^h, \mathcal{F}_0 \otimes \mathcal{F}^h, \mathbb{P}_0 \otimes \mathbb{P}^{h})
\end{align*}

\noi
and also the random variables:
\begin{align*}
z_0 = \<1>^h (1), \quad v_{0, N} = - Z_N^{(3)} (1) + Y_N (1), \quad v_0 = - 4 \cj{a}_2 \<30>^{\,h} (1) + Y (1), 
\end{align*}

\noi
where $\<1>^h$ is defined in \eqref{lin_h}, $Z_N^{(3)}$ is defined in \eqref{cub_h2}, $\<30>^{\,h}$ is defined in Lemma~\ref{LEM:sto_h}, $Y_N$ is the solution to the equation \eqref{SNLHYN}, and $Y$ is the solution to the equation \eqref{SNLHY}.

From \eqref{lin_hvar}, we can easily verify that $\Law (z_0) = \Law (\<1>^h) = \mu_\al$, where $\mu_\al$ is as defined in \eqref{mu_al}. By \eqref{sec_expN} and Lemma~\ref{LEM:inv_h2}, we have
\begin{align*}
\Law (z_0 + v_{0, N}) = \Law \big( \<1>^h (1) - Z_N^{(3)} (1) + Y_N (1) \big) = \Law (X_N (1)) = \varrho_N.
\end{align*}

\noi
By \eqref{sec_exph} and Proposition~\ref{PROP:GWP_h}, we have
\begin{align*}
\Law (z_0 + v_{0}) = \Law \big( \<1>^h (1) - 4 \cj{a}_2 \<30>^{\,h} (1) + Y (1) \big) = \Law (X (1)) = \rho.
\end{align*}

It remains to show the convergence in probability of $v_{0, N}$ to $v_0$ in $\mathcal{C}^{\al - \frac 12 + \eps}$ for some $\eps > 0$. From Lemma~\ref{LEM:sto_h2} and the fact that $5 \al - \frac 92 - \eps > \al - \frac 12 + \eps$ given $\eps > 0$ sufficiently small, we see that $Z_N^{(3)} (1)$ converges $\mathbb{P}^h$-almost surely to $4 \cj{a}_2 \<30>^{\,h} (1)$ in $\mathcal{C}^{\al - \frac 12 + \eps}$. Also, given any $0 < \eta \ll 1$, from Lemma~\ref{LEM:stb_h2}, Lemma~\ref{LEM:sto_h2}, Proposition~\ref{PROP:Y_bdd}, and the fact that $X_{0, N}$ converges $\mathbb{P}_0$-almost surely to $X_0$ (see the beginning of Subsection~\ref{SUB:heat_wu}), we obtain
\begin{align*}
\mathbb{Q} &\big( \| Y_N (1) - Y (1) \|_{\mathcal{C}^{\al - \frac 12 + \eps}} > C (T, K_0, K_1) (\kappa + N^{- \ta'}) \big) \\
&\leq \mathbb{P}^h \big( \| \cj{\pmb{Z}} \|_{\cj{\mathcal{Z}}_T^{\al, \eps}} > K_0 \big) + \mathbb{Q} \big( \| Y \|_{\mathcal{Y}_{T'}^{\al, \eps}} + \| Y \|_{C_{[T', T]} \mathcal{C}_x^{1 + \eps}} > K_1 \big) \\
&\quad + \mathbb{P}^h \big( \| \cj{\pmb{Z}}_N - \cj{\pmb{Z}} \|_{\cj{\mathcal{Z}}_T^{\al, \eps}} > N^{- \ta} \big) + \mathbb{P}_0 \big( \| X_{0, N} - X_0 \|_{\mathcal{C}^{\al - \frac 32 - \eps}} > \kappa \big) \\
&< \eta
\end{align*}

\noi
for any $\kappa > 0$ sufficiently small and some $\ta, \ta' > 0$, where $K_0 = K_0 (T, \eta) \gg 1$ is sufficiently large, $0 < T' = T' (T, \eta) \leq 1$ is the $T_0$ given in the statement of Proposition~\ref{PROP:Y_bdd}, $K_1 = K_1 (T, \eta) \gg 1$ is sufficiently large, $N \geq N_0$ with $N_0 \gg 1$ sufficiently large, and $C (T, K_0, K_1) > 0$ is the constant as in the statement of Lemma~\ref{LEM:stb_h2}. This finishes the convergence part of Theorem~\ref{THM:rep}. The bound \eqref{v0bdd_goal} follows directly from Lemma~\ref{LEM:sto_h} for $\<30>^{\,h}$ and Proposition~\ref{PROP:Y_bdd}.
\end{proof}

\section{Dynamical weak universality}

In this section, we prove Theorem \ref{THM:dyn}, convergence of the solution $u_N$ of the macroscopic model \eqref{macro}.

\subsection{On the truncated dynamics}

In this subsection, we briefly go over global well-posedness of the macroscopic model
\begin{align}
\dt^2 u_N + \dt u_N + (1 - \Dl)^\al u_N + \pi_N V_N' (\pi_N u_N) = \sqrt{2} \xi
\label{SdNLWmain}
\end{align}

\noi
and the invariance of the truncated Gibbs measure $\vec \varrho_N = \varrho_N \otimes \mu_0$ as defined in \eqref{nuN} and \eqref{mu_al} under the flow of~\eqref{SdNLWmain}.

Let $N \in \N$. Let us write the truncated dynamics \eqref{SdNLWmain} as a superposition of the deterministic fractional NLW:
\begin{align}
\dt^2 u_N + (1 - \Dl)^\al u_N + \pi_N V_N' (\pi_N u_N) = 0
\label{fNLW}
\end{align}

\noi
and the Ornstein-Uhlenbeck process for $\dt u_N$:
\begin{align}
\dt (\dt u_N) = - \dt u_N + \sqrt{2} \xi.
\label{OU}
\end{align}

\noi
We observe that the truncated Gibbs measure $\vec \varrho_N$ is formally invariant under the dynamics of \eqref{SdNLWmain}. This is essentially due to the fact that $\vec \varrho_N$ is invariant under the dynamics of the fractional NLW \eqref{fNLW} and the white noise measure $\mu_0$ on $\dt u_N$ (and hence $\vec \varrho_N$ on $\vec u_N$) is invariant under the Ornstein-Uhlenbeck dynamics \eqref{OU}. Then, we can apply Bourgain's invariant measure argument \cite{Bour94} to obtain the following almost sure global well-posedness result for the macroscopic model \eqref{SdNLWmain} and invariance of the truncated Gibbs measure $\vec \varrho_N$. For a proof, see \cite[Lemma 9.3]{OOT1}.
\begin{lemma}
\label{LEM:invN}
Let $\al > 1$, $\eps > 0$, and $N \in \N$. Then, the macroscopic model \eqref{SdNLWmain} is almost surely globally well-posed with respect to the random initial data distributed by the truncated Gibbs measure $\vec \varrho_N = \varrho_N \otimes \mu_0$. Moreover, $\vec \varrho_N$ is invariant under the resulting dynamics of~\eqref{SdNLWmain}. 
\end{lemma}

\subsection{Some notations on the stochastic convolution}
In this subsection, we introduce some more notations for convenience. 

Before proceeding, let us be more precise on the probability space that we are going to work on. We recall that the space-time white noise $\xi$ is independent with the random initial data distributed by the Gibbs measures. Let $(\Sigma, \mathcal{G}, \mathbb{Q})$ be the underlying probability space for the random initial data given by Theorem~\ref{THM:rep}. We then view $\xi$ as a space-time white noise on a new probability space $(\Om^w, \F^w, \mathbb{P}^w)$ (the superscript ``$w$'' stands for ``wave''), and we study our equation \eqref{macro} on the following probability space:
\begin{align*}
(\Sigma \times \Om^w, \mathcal{G} \otimes \F^w, \mathbb{Q} \otimes \mathbb{P}^w).
\end{align*}

Let us denote $\vec z_0 = (z_0, z_1)$, where $z_0$ is the random variable on $(\Sigma, \mathcal{G}, \mathbb{Q})$ provided by Theorem~\ref{THM:rep} with $\Law (z_0) = \mu_\al$ (see \eqref{mu_al}) and $z_1$ is a random variable on $(\Sigma, \mathcal{G}, \mathbb{Q})$ with $\Law (z_1) = \mu_0$ independent with $z_0$. From now on, we focus on the following version of the stochastic convolution for the linear stochastic damped wave equation:
\begin{align}
\<1> (t; \vec z_0, \o) \deff S(t) \vec z_0 + \sqrt{2} \int_0^t \mathcal{D} (t - t') dW^w (t', \o),
\label{sto}
\end{align}

\noi
for $\o \in \Om^w$, where $\mathcal{D}(t)$ is defined in \eqref{defD}, $S(t)$ is defined in \eqref{defS}, and $W^w$ is a cylindrical Wiener process on $L^2 (\T^3)$ given by
\begin{align}
W^w (t) = \frac{1}{(2 \pi)^{\frac 32}} \sum_{n \in \Z^3} B_n^w (t) e^{i n \cdot x}
\label{defWw}
\end{align}

\noi
with $B_n^w (t) = (2 \pi)^{-3/2} \jb{\xi, \ind_{[0, t]} \cdot e^{in \cdot x}}_{t, x}$. Here, $\jb{\cdot, \cdot}_{t, x}$ denotes the duality pairing on $\R_+ \times \T^3$. Note that $\{B_n^w\}_{n \in \Z^3}$ is a family of Gaussian processes satisfying the three properties at the beginning of Subsection~\ref{SUB:mult}. 
For simplicity, we may omit the dependence on $\vec z_0$ and/or $\o$ if it is clear from the context.

Given $N \in \N$, we write $\<1>_N = \pi_N \<1>$. A tedious but direct computation shows that
\begin{align*}
\E [ |\<1>_N (t, x)|^2 ] = \sum_{\substack{n \in \N \\ |n| \leq N}} \frac{1}{\jb{n}^{2 \al}} = \s_N
\end{align*}

\noi
for each $t \geq 0$ and $x \in \T^3$, where $\s_N$ is as defined in \eqref{sigma}. Given an integer $k \geq 1$, we define the Wick ordered product
\begin{align}
\wick{\<1>_N^k (\vec z_0, \o)} \, \deff H_k (\<1>_N (\vec z_0, \o); \s_N),
\label{wick}
\end{align}

\noi
where $H_k$ is the $k$th Hermite polynomial defined in \eqref{hermite}.

Recall that we have an explicit expression for $z_0$ given by the proof of Theorem~\ref{THM:rep} in Subsection~\ref{SUB:rep} (see \eqref{lin_h}):
\begin{align*}
z_0 (x) = \sum_{n \in \Z^3} \frac{\sqrt{2}}{(2 \pi)^{\frac 32}} \int_{- \infty}^1 e^{- (1 - s) \jb{n}^{2 \al}} \, d B_n^0 (s) e^{in \cdot x}, 
\end{align*}

\noi
where we denote $B_n^0 = B_n^h$ for notational convenience at a later point.
In order to treat $z_0$ and $z_1$ in a uniform manner, we write
\begin{align*}
z_1 (x) = \sum_{n \in \Z^3} \frac{1}{(2 \pi)^{\frac 32}} \int_{0}^1 1 \, d B_n^1 (s) e^{in \cdot x},
\end{align*}

\noi
where $\{B_n^1\}_{n \in \Z^3}$ is a family of Gaussian process on $(\Sigma, \mathcal{G}, \mathbb{Q})$ independent with $z_0$, $v_{0, N}$, and $v_0$ provided by Theorem~\ref{THM:rep} and satisfying the three properties at the beginning of Subsection~\ref{SUB:mult}.
We also denote $B_n^2 = B_n^w$ for each $n \in \Z^3$, where $B_n^w$ is given by \eqref{defWw}. Thus, from \eqref{sto}, we can write
\begin{align}
\<1>_N (t, x) = \sum_{\substack{n \in \Z^3 \\ |n| \leq N}} \sum_{j \in \{ 0, 1, 2 \}} \int_{\R} \phi_{j, n} (t, s) d B_n^j (s) e^{in \cdot x},
\label{sto_phi}
\end{align}

\noi
where
\begin{align}
\begin{split}
\phi_{0, n} (t, s) &\deff \frac{\sqrt{2}}{(2 \pi)^{\frac 32}} e^{- \frac t2} \Big( \cos (t \jbb{n}) + \frac{\sin (t \jbb{n})}{2 \jbb{n}} \Big) e^{- (1 - s) \jb{n}^{2 \al}} \ind_{(- \infty, 1]} (s), \\
\phi_{1, n} (t, s) &\deff \frac{1}{(2 \pi)^{\frac 32}} e^{- \frac t2} \frac{\sin (t \jbb{n})}{\jbb{n}} \ind_{[0, 1]} (s), \\
\phi_{2, n} (t, s) &\deff \frac{\sqrt{2}}{(2 \pi)^{\frac 32}} e^{- \frac{t - s}{2}} \frac{\sin ((t - s) \jbb{n})}{\jbb{n}} \ind_{[0, t]} (s).
\end{split}
\label{def_phi}
\end{align}

\noi
The advantage of the above notations is that the Wick ordered product $\wick{\<1>_N^k}$ defined in \eqref{wick} can be written in the following neat form in view of Lemma~\ref{LEM:sto_ito}~(i) and the binomial expansion \eqref{her_binom}:
\begin{align}
\wick{\<1>_N^k} (t, x) = \sum_{\substack{n \in \Z^3 \\ |n| \leq kN}} \sum_{\substack{n_1 + \cdots + n_k = n \\ |n_1|, \dots, |n_k| \leq N}} \sum_{\substack{\iota_1, \dots, \iota_k \\ \in \{0, 1, 2\}}} \int_{\R^k} \prod_{\l = 1}^k \phi_{\iota_\l, n_\l} (t, s_\l) d B_{n_1}^{\iota_1} (s_1) \cdots d B_{n_k}^{\iota_k} (s_k) e^{in \cdot x}.
\label{sto_phi2}
\end{align}

\subsection{Full ansatz and enhanced data set}
\label{SUB:ansatz}

As mentioned in Subsection~\ref{SUB:intro_dyn}, we proceed with the following second order expansion:
\begin{align}
u_N = \<1> - \Xi_N^{(3)} + v_N,
\label{sec_exp2}
\end{align}

\noi
where $\Xi_N^{(3)}$ is the extended cubic stochastic term defined in \eqref{cubic} and $v_N$ is a remainder term with initial data $v_{0, N}$ given by Theorem~\ref{THM:rep}.

Let us now present our full ansatz. From the definition of $V_N$ in \eqref{defVN}, the second order expansion \eqref{sec_exp2}, and the binomial expansion \eqref{her_binom}, we have
\begin{align}
\begin{split}
V_N' &(\pi_N u_N) \\
&= 2 \cj{a}_{1, N} N^{2 (\frac 32 - \al)} \big( \<1>_N - \Xi_N^{(3)} + \pi_N v_N \big) \\
&\quad + \sum_{j = 2}^m 2j \cj{a}_{j, N} N^{(2j - 4) (\al - \frac 32)} \wick{\<1>_N^{2 j - 1}} \\
&\quad + \sum_{j = 2}^m 2j \cj{a}_{j, N} N^{(2j - 4) (\al - \frac 32)} \bigg( \sum_{\l = 1}^{2j - 1} \binom{2j - 1}{\l} :\! \<1>_N^{2j - 1 - \l} \!: \big( - \Xi_N^{(3)} + \pi_N v_N \big)^\l \bigg).
\end{split} 
\label{VNuN}
\end{align}

\noi
Using the definition of $\Xi_N^{(3)}$ in \eqref{cubic}, the second order expansion \eqref{sec_exp2}, and \eqref{VNuN}, we obtain the following equation for $v_N$:
\begin{align}
\begin{split}
&\dt^2 v_N + \dt v_N + (1 - \Dl)^\al v_N + 2 \cj{a}_{1, N} N^{2 (\frac 32 - \al)} \big( \<1>_N - \Xi_N^{(3)} + \pi_N v_N \big) \\
&\quad + \sum_{j = 2}^m 2j \cj{a}_{j, N} N^{(2j - 4) (\al - \frac 32)} \pi_N \bigg( \sum_{\l = 1}^{2j - 1} \binom{2j - 1}{\l} :\! \<1>_N^{2j - 1 - \l} \!: \big( - \Xi_N^{(3)} + \pi_N v_N \big)^\l \bigg) = 0
\end{split}
\label{SdNLWv1}
\end{align}

\noi
with initial data $\vec v_N |_{t = 0} = (v_{0, N}, 0)$.
Let us write \eqref{SdNLWv1} in the following Duhamel formulation starting at time $t_0 \geq 0$:
\begin{align}
\begin{split}
\vec v_N &= \vec S(\cdot) (v_{N} (t_0), \dt v_N (t_0)) - 2 \cj{a}_{1, N} N^{2 (\frac 32 - \al)} \pi_N \vec \I_{t_0} \big( \<1>_N - \Xi_N^{(3)} + \pi_N v_N \big) \\
&\quad - \sum_{j = 2}^m \sum_{\l = 1}^{2j - 1} \sum_{k = 0}^\l C_{j, \l, k} \cj{a}_{j, N}  N^{(2 j - 4) (\al - \frac 32)} \pi_N \vec \I_{t_0} \big( (- \Xi_N^{(3)})^k \wick{\<1>_N^{2j - 1 - \l}} (\pi_N v_N)^{\l - k} \big)
\end{split}
\label{SdNLWv2}
\end{align}

\noi
for some constants $C_{j, \l, k} > 0$, where $S(\cdot)$ is as defined in \eqref{defS}, $\I_{t_0}$ is the wave Duhamel operator as defined in \eqref{defI}, and their vector versions $\vec S (\cdot)$ and $\vec \I_{t_0}$ are defined in \eqref{vecSI}. For convenience, we define the following stochastic terms (including $\Xi_N^{(3)}$ for readers' convenience):
\begin{align}
\begin{split}
\Xi^{(3)}_N &\deff \sum_{j = 2}^m 2j \cj{a}_{j, N} N^{(2j - 4) (\al - \frac 32)} \pi_N \I \big( :\! \<1>_N^{2j - 1} \!: \big), \\
\Xi^{(2)}_{t_0, N} (u) &\deff \sum_{j = 2}^m C_{j, 1, 0} \cj{a}_{j, N} N^{(2j - 4) (\al - \frac 32)} \pi_N \I_{t_0} \big( \wick{\<1>_N^{2j - 2}} u \big), \\
\Xi^{(3, 1)}_{t_0, N} (u) &\deff \sum_{j = 2}^m C_{j, 2, 1} \cj{a}_{j, N} N^{(2j - 4) (\al - \frac 32)} \pi_N \I_{t_0} \big( \Xi_N^{(3)} \wick{\<1>_N^{2j - 3}} u \big), \\
\Xi^{(5)}_{t_0, N} &\deff \sum_{j = 2}^m C_{j, 1, 1} \cj{a}_{j, N} N^{(2j - 4) (\al - \frac 32)} \pi_N \I_{t_0} \big( \Xi_N^{(3)} \wick{\<1>_N^{2j - 2}} \big), \\
\Xi^{(7)}_{t_0, N} &\deff \sum_{j = 2}^m C_{j, 2, 2} \cj{a}_{j, N} N^{(2j - 4) (\al - \frac 32)} \pi_N \I_{t_0} \big( (\Xi_N^{(3)})^2 \wick{\<1>_N^{2j - 3}} \big).
\end{split}
\label{sto_all}
\end{align}

\noi
The superscripts in \eqref{sto_all} refer to the degree of their corresponding limiting stochastic objects in \eqref{sto_lim} below. Using the notations in \eqref{sto_all}, we can then rewrite \eqref{SdNLWv2} as
\begin{align}
\begin{split}
\vec v_N &= \vec S(\cdot) (v_{N} (t_0), \dt v_N (t_0)) - 2 \cj{a}_{1, N} N^{2 (\frac 32 - \al)} \pi_N \vec \I_{t_0} \big( \<1>_N - \Xi_N^{(3)} + \pi_N v_N \big) \\
&\quad - \vec \Xi_{t_0, N}^{(2)} (\pi_N v_N) + \vec \Xi_{t_0, N}^{(3, 1)} ( \pi_N v_N ) + \vec \Xi_{t_0, N}^{(5)} - \vec \Xi_{t_0, N}^{(7)} \\
&\quad - \sum_{j = 2}^m C_j \cj{a}_{j, N} N^{(2 j - 4) (\al - \frac 32)} \pi_N \vec \I_{t_0} \big( \wick{\<1>_N^{2j - 3}} (\pi_N v_N)^2 \big) \\
&\quad - \sum_{j = 2}^m \sum_{\l = 3}^{2j - 1} \sum_{k = 0}^\l C_{j, \l, k} \cj{a}_{j, N} N^{(2 j - 4) (\al - \frac 32)} \pi_N \vec \I_{t_0} \big( (- \Xi_N^{(3)})^k \wick{\<1>_N^{2j - 1 - \l}} (\pi_N v_N)^{\l - k} \big)
\end{split}
\label{SdNLWv3}
\end{align}

\noi
for some constants $C_j, C_{j, \l, k} > 0$, where $\vec \Xi_{t_0, N} = (\Xi_{t_0, N}, \dt \Xi_{t_0, N})$ (we also use this notation for the tree-structure objects below). For conciseness in later steps, we simply write
\begin{align*}
\Xi_N^{(2)} = \Xi_{\cdot, N}^{(2)}, ~ \Xi_N^{(3, 1)} = \Xi_{\cdot, N}^{(3, 1)}, ~ \Xi_N^{(5)} = \Xi_{\cdot, N}^{(5)}, ~ \Xi_N^{(7)} = \Xi_{\cdot, N}^{(7)}.
\end{align*}


In order to discuss the regularity properties of the above stochastic terms, we need to define two norms. Given $s, b \in \R$, $T > 0$, and $u_\cdot$ such that $u_{t_0} \in X_{[t_0, t_1]}^{s, b}$ for any $0 \leq t_0 < t_1 \leq T$, we define
\begin{align}
\| u_\cdot \|_{Y_T^{s, b}} \deff \sup_{0 \leq t_0 < t_1 \leq T} \| u_{t_0} \|_{X_{[t_0, t_1]}^{s, b}} .
\label{Ysb}
\end{align}

\noi
Also, given $s_1, s_2, b \in \R$, $\dl, T > 0$, and $S_\cdot$ such that $S_{t_0} \in \L (X_{[t_0, t_1]}^{s_1, b}; X_{[t_0, t_1]}^{s_2, b})$ for any $0 \leq t_0 < t_1 \leq T$, we define
\begin{align}
\| S_\cdot \|_{\L_T^{s_1, s_2, b, \dl}} \deff \sup_{0 \leq t_0 < t_1 \leq T} (t_1 - t_0)^{- \dl} \| S_{t_0} \|_{\L (X_{[t_0, t_1]}^{s_1, b}; X_{[t_0, t_1]}^{s_2, b})} .
\label{LXsb}
\end{align}
 
\noi
We also denote $\vec Y_T^{s, b} = Y_T^{s, b} \times Y_T^{s - \al, b}$ and $\vec \L_T^{\,s_1, s_2, b, \dl} = \L_T^{s_1, s_2, b, \dl} \times \L_T^{s_1, s_2 - \al, b, \dl}$.

To gather all stochastic terms, we define the truncated enhanced data set $\pmb{\Xi}_N$ as
\begin{align}
\pmb{\Xi}_N &\deff \big( \<1>_N, N^{\al - \frac 32} \wick{\<1>_N^2}, \dots, N^{(2m - 4) (\al - \frac 32)} \wick{\<1>_N^{2m - 3}}, \Xi_N^{(3)}, \vec \Xi_N^{(2)}, \vec \Xi_N^{(3, 1)}, \vec \Xi_N^{(5)}, \vec \Xi_N^{(7)} \big).
\label{enhN}
\end{align}

\noi
In view of \cite[Lemma~5.1]{LTzW}, we define the following stochastic objects:\footnote{With a slight abuse of notations, by writing $\I$, we actually mean $t_0 \mapsto \I_{t_0}$ in a few places. This simplifies the presentation in later argument.}
\begin{align}
\begin{split}
\<30> &\deff \lim_{N \to \infty} \pi_N \I (\wick{\<1>_N^3}) \quad \text{in } C_T W^{3\al - 3 - \eps, \infty} \cap X_T^{3 \al - 3 - \eps, \frac 12 + \dl}, \\
\If^{\<2>} (u) &= \If_\cdot^{\<2>} (u) \deff \lim_{N \to \infty} \pi_N \I (\wick{\<1>_N^2} u) \quad \text{in } \L_T^{\al - \frac 12 - \eps, \al - \frac 12 + \eps, \frac 12 + \dl, \dl}, \\
\If^{\<31>} (u) &= \If_\cdot^{\<31>} (u) \deff \lim_{N \to \infty} \pi_N \I \big( \pi_N \I (\wick{\<1>_N^3}) \<1>_N \cdot u \big) \quad \text{in } \L_T^{\al - \frac 12 - \eps, \al - \frac 12 + \eps, \frac 12 + \dl, \dl}, \\
\<320> &= \<320>_\cdot \deff \lim_{N \to \infty} \pi_N \I ( \pi_N \I (\wick{\<1>_N^3}) \wick{\<1>_N^2} ) \quad \text{in } Y_T^{\al - \frac 12, \frac 12 + \dl}, \\
\<70> &= \<70>_\cdot \deff \lim_{N \to \infty} \pi_N \I \big( (\pi_N \I (\wick{\<1>_N^3}))^2 \<1>_N \big) \quad \text{in } Y_T^{\al - \frac 12, \frac 12 + \dl}
\end{split}
\label{sto_lim}
\end{align}

\noi
for $\eps, \dl > 0$ sufficiently small, where the convergence holds almost surely.
We will show in the next subsection that $\pmb{\Xi}_N$ converges to the enhanced data set $\pmb{\Xi}$ defined by
\begin{align}
\begin{split}
\pmb{\Xi} \deff \big( \<1>, 0, \dots, 0, 4 \cj{a}_2 \<30>, 12 \cj{a}_2 \vec \If^{\,\<2>}, 96 \cj{a}_2^2 \vec \If^{\,\<31>},  48 \cj{a}_2^2 \vec{\<320>}, 192 \cj{a}_2^3 \vec{\<70>} \big)
\end{split}
\label{enh}
\end{align}

\noi
in the space
\begin{align}
\begin{split}
\mathcal{X}_T^{\al, \eps, \dl} &\deff (C_T W^{\al - \frac 32 - \eps, \infty})^{\otimes (2m - 3)} \times (C_T W^{3 \al - 3 - \eps, \infty} \cap X_T^{3 \al - 3 - \eps, \frac 12 + \dl} ) \\
&\quad \times \vec{\L}_T^{\,\al - \frac 12 - \eps, \al - \frac 12 + \eps, \frac 12 + \dl, \dl} \times \vec \L_T^{\,\al - \frac 12 - \eps, \al - \frac 12 + \eps, \frac 12 + \dl, \dl} \\
&\quad  \times \vec Y_T^{\al - \frac 12, \frac 12 + \dl} \times \vec Y_T^{\al - \frac 12, \frac 12 + \dl}.
\end{split}
\label{XTaed}
\end{align}

\subsection{Regularity of stochastic terms}
\label{SUB:reg}
Our goal in this subsection is to show that the truncated enhanced data set $\pmb{\Xi}_N$ in \eqref{enhN} converges almost surely to the enhanced data set $\pmb{\Xi}$ in \eqref{enh} in the space $\mathcal{X}_T^{\al, \eps, \dl}$ defined in \eqref{XTaed}. This is obtained via the following lemma.
\begin{lemma}
\label{LEM:data}
Let $1 < \al \leq \frac 98$, $\eps, \dl > 0$ be sufficiently small, and $T \geq 1$. We consider the following stochastic terms $\Xi_N$ and $\Xi$ along with their function space $B$ as follows:
\begin{enumerate}
\item[(i)] $\Xi_N = \<1>_N$, $\Xi = \<1>$, $B = C_T W^{\al - \frac 32 - \eps, \infty}$; \smallskip

\item[(ii)] $\Xi_N = N^{k ( \al - \frac 32 )} \wick{\<1>_N^{k + 1}}$ with $k \geq 1$, $\Xi = 0$, $B = C_T W^{\al - \frac 32 - \eps, \infty}$; \smallskip

\item[(iii)] $\Xi_N = \Xi_N^{(3)}$, $\Xi = 4 \cj{a}_2 \<30>$, $B = C_T W^{3 \al - 3 - \eps, \infty} \cap X_T^{3 \al - 3 - \eps, \frac 12 + \dl}$; \smallskip

\item[(iv)] $\Xi_N = \vec \Xi_N^{(2)}$, $\Xi = 12 \cj{a}_2 \vec \If^{\,\<2>}$, $B = \vec \L_T^{\,\al - \frac 12 - \eps, \al - \frac 12 + \eps, \frac 12 + \dl, \dl}$; \smallskip

\item[(v)] $\Xi_N = \vec \Xi_N^{(3, 1)}$, $\Xi = 96 \cj{a}_2^2 \vec \If^{\,\<31>}$, $B = \vec \L_T^{\,\al - \frac 12 - \eps, \al - \frac 12 + \eps, \frac 12 + \dl, \dl}$; \smallskip

\item[(vi)] $\Xi_N = \vec \Xi_N^{(5)}$, $\Xi = 48 \cj{a}_2^2 \vec{\<320>}$, $B = \vec Y_T^{\al - \frac 12, \frac 12 + \dl}$; \smallskip

\item[(vii)] $\Xi_N = \vec \Xi_N^{(7)}$, $\Xi = 192 \cj{a}_2^3 \vec{\<70>}$, $B = \vec Y_T^{\al - \frac 12, \frac 12 + \dl}$.
\end{enumerate}

\noi
Then, $\Xi_N$ converges $\mathbb{Q} \otimes \mathbb{P}^w$-almost surely to $\Xi$ in the space $B$ as $N \to \infty$. We also have the tail estimate
\begin{align}
\mathbb{Q} \otimes \mathbb{P}^w \big( \| \Xi_N \|_{B} > \ld \big) \leq C(T) e^{- c (T) \ld^\be}
\label{tail_w}
\end{align}

\noi
for any $\ld > 0$ and some constants $C(T), c(T), \be > 0$, uniformly in $N \in \N \cup \{ \infty \}$, with the understanding that $\Xi_\infty = \Xi$. Moreover, there exists small $\ta > 0$ such that
\begin{align}
\mathbb{Q} \otimes \mathbb{P}^w \big( N^\ta \| \Xi_{N} - \Xi \|_B > \ld \big) \leq C(T) e^{- c (T) \ld^\be}
\label{tail_w2}
\end{align}

\noi
for any $\ld > 0$ and some constants $C(T), c(T), \be > 0$, uniformly in $N \geq 1$.
\end{lemma}

\begin{proof}
Part~(i) has been proved in \cite[Lemma~5.1~(i)]{LTzW}; see also \cite[Lemma~3.1~(i)]{GKO2} and \cite[Lemma~2.3]{GKOT}. Part~(ii) was essentially proved in \cite[Lemma~7.1]{LTzW}, and the continuity in time can be shown similarly by the time difference estimate as in the proof of Lemma~\ref{LEM:sto_h}~(i). Part~(iii) follows from \cite[Lemma~7.3]{LTzW} with minor modifications. Part~(iv) follows from \cite[Lemma~7.2 and Lemma~5.2~(iii)]{LTzW}. Thus, we only focus on proving (v), (vi), and (vii). Also, from Lemma~\ref{LEM:nhomo}, we only focus on $\Xi_N^{(3, 1)}$, $\Xi_N^{(5)}$, and $\Xi_N^{(7)}$, as the regularity properties of their time derivatives follow from almost the same steps.

\smallskip \noi
(v) From \eqref{sto_all}, we need to show the regularity properties of the following term:
\begin{align*}
\Xi_{t_0, N, j_1, j_2}^{(3, 1)} (u) \deff N^{- (2 j_1 + 2 j_2 - 8) (\frac 32 - \al) + \eps'} \pi_N \I_{t_0} \big( \pi_N \I \big( \wick{\<1>_N^{2 j_1 - 1}} \big) \wick{\<1>_N^{2 j_2 - 3}} u \big),
\end{align*}

\noi
where $j_1 \geq 2$ and $j_2 \geq 2$. Here, we let $\eps' = 0$ if $j_1 = 2$ and $j_2 = 2$; otherwise, we let $\eps' > 0$ be small. For this purpose, let us consider the following term:
\begin{align*}
\xi_{N, j_1, j_2}^{(3, 1)} \deff N^{- (2 j_1 + 2 j_2 - 8) (\frac 32 - \al) + \eps'} \pi_N \I \big( \wick{\<1>_N^{2 j_1 - 1}} \big) \wick{\<1>_N^{2j_2 - 3}}
\end{align*}

\noi
in the space $L_T^q W^{\al - \frac 32 - \eps, r} (\T^3)$ for some $1 < q, r < \infty$. 
Once we have this, we then follow the same steps as \cite[Lemma~5.2~(vi)]{LTzW} to show the regularity properties for $\Xi_{\cdot, N, j_1, j_2}^{(3, 1)}$ in the space $\L_T^{\al - \frac 12 - \eps, \al - \frac 12 + \eps, \frac 12 + \dl, \dl}$, for which we omit details.

In view of the product formula in Lemma~\ref{LEM:prod}, we write
\begin{align*}
\xi_{N, j_1, j_2}^{(3, 1)} = \xi_{N, j_1, j_2}^{(3, 1)} - \E \big[ \xi_{N, j_1, j_2}^{(3, 1)} \big] + \E \big[ \xi_{N, j_1, j_2}^{(3, 1)} \big] \deff \xi_{N, j_1, j_2}^{(3, 1, *)} + \E \big[ \xi_{N, j_1, j_2}^{(3, 1)} \big],
\end{align*}

\noi
where $\E \big[ \xi_{N, j_1, j_2}^{(3, 1)} \big]$ is the term where all frequency nodes of $\I (\wick{\<1>_N^{2 j_1 - 1}})$ are paired with those of $\wick{\<1>_N^{2 j_2 - 3}}$ (recalling Definition~\ref{DEF:pair}).

We first deal with the $\xi_{N, j_1, j_2}^{(3, 1, *)}$ term. 
From Lemma~\ref{LEM:reg}, we aim to show
\begin{align}
\sum_{n \in \Z^3} \E \Big[ \big| \jb{n}^{\al - \frac 32 - \eps} \F_x \big( \xi_{N, j_1, j_2}^{(3, 1, *)} \big) (t, n) \big|^2 \Big] < \infty.
\label{cub-lin_goal}
\end{align}

\noi
We denote $n_1, \dots, n_{2j_1 - 1}$ as the nodes of the frequencies of $\I (\wick{\<1>_N^{2 j_1 - 1}})$ and denote $n_{2 j_1}, \dots, n_{2 j_1 + 2 j_2 - 4}$ as the nodes of the frequencies of $\wick{\<1>_N^{2 j_2 - 3}}$. From \eqref{sto_phi2}, \eqref{defI}, and Fubini's theorem (Lemma~\ref{LEM:sto_fub}), we have
\begin{align*}
\F_x \big( \I \big( \wick{\<1>_N^{2 j_1 - 1}} \big) \big) (t, n) &= \sum_{\substack{n_1 + \cdots + n_{2 j_1 - 1} = n \\ |n_1|, \dots, |n_{2 j_1 - 1}| \leq N}} \sum_{\substack{\iota_1, \dots, \iota_k \\ \in \{0, 1, 2\}}} \int_{\R^{2 j_1 - 1}} \int_0^t e^{- \frac{t - t'}{2}} \frac{\sin ( (t - t') \jbb{n} )}{\jbb{n}} \\
&\quad \times \prod_{\l = 1}^{2 j_1 - 1} \phi_{\iota_\l, n_\l} (t', s_\l) dt' d B_{n_1}^{\iota_1} (s_1) \cdots d B_{n_{2 j_1 - 1}}^{\iota_{2 j_1 - 1}} (s_{2 j_1 - 1})
\end{align*}

\noi
for any $t \in [0, T]$ and $n \in \Z^3$. From the product formula (Lemma~\ref{LEM:sto_prod}), we need to discuss whether there are pairs between $\I (\wick{\<1>_N^{2 j_1 - 1}})$ and $\wick{\<1>_N^{2j_2 - 3}}$.

We first consider the case when there are no pairs in $\xi_{N, j_1, j_2}^{(3, 1, *)}$. By dyadic decompositions, we can assume that $|n_\l| \sim N_\l$ for some dyadic numbers $1 \leq N_\l \leq N$ for all $1 \leq \l \leq 2 j_1 + 2 j_2 - 4$. By the Ito isometry (Lemma~\ref{LEM:sto_ito}~(ii)) and Jensen's inequality \eqref{jen}, we obtain the following contribution for \eqref{cub-lin_goal} (with all frequencies dyadically localized and the summation on $\iota_\l$'s omitted):
\begin{align}
\begin{split}
&N^{- (2 j_1 + 2 j_2 - 8) (3 - 2 \al) + 2 \eps'} \sum_{n \in \Z^3} \sum_{\substack{n_1, \dots, n_{2 j_1 + 2 j_2 - 4} \in \Z^3 \\ n_1 + \cdots + n_{2 j_1 + 2 j_2 - 4} = n}} \bigg( \prod_{\l = 1}^{2 j_1 + 2 j_2 - 4} \ind_{\{ |n_\l| \sim N_\l \}} \bigg) \jb{n}^{2 \al - 3 - 2 \eps}  \\ 
&\quad \times \int_{\R^{2 j_1 + 2 j_2 - 4}}  \bigg( \int_0^t e^{- \frac{t - t'}{2}} \frac{\sin ( (t - t') \jbb{\sum_{\l = 1}^{2j_1 - 1} n_\l} )}{\jbb{\sum_{\l = 1}^{2j_1 - 1} n_\l}} \prod_{\l = 1}^{2 j_1 - 1} \phi_{\iota_\l, n_\l} (t', s_\l) dt' \bigg)^2 \\
&\quad \times \prod_{\l' = 2 j_1}^{2 j_1 + 2 j_2 - 4} \phi_{\iota_{\l'}, n_{\l'}} (t, s_{\l'})^2  d s_1 \cdots d s_{2 j_1 + 2 j_2 - 4}
\end{split}
\label{cub-lin1}
\end{align}

\noi
We then expand the sine and cosine functions into complex exponentials. Thus, by integration in $t'$, integration in $s_\ell$'s, and a level set decomposition, we bound
\begin{align}
\begin{split}
\eqref{cub-lin1} &\les N^{- (2 j_1 + 2 j_2 - 8) (3 - 2 \al) + 2 \eps'} \sum_{n_1, \dots, n_{2 j_1 + 2 j_2 - 4} \in \Z^3} \bigg( \prod_{\l = 1}^{2 j_1 + 2 j_2 - 4} \ind_{\{ |n_\l| \sim N_\l \}} \jbb{n_\l}^{-2} \bigg) \\
&\quad \times \bigg\langle \sum_{\l = 1}^{2 j_1 + 2 j_2 - 4} n_\l \bigg\rangle^{2 \al - 3 - 2 \eps} \bigg\llbracket \sum_{\l = 1}^{2 j_1 - 1} n_\l \bigg\rrbracket^{- 2} \cdot \frac{1}{\jb{\kappa (\cj n)}^2} \\
&\les N^{- (2 j_1 + 2 j_2 - 8) (3 - 2 \al) + 2 \eps'} \sup_{\zeta \in \Z} \sum_{n_1, \dots, n_{2 j_1 + 2 j_2 - 4} \in \Z^3} \bigg( \prod_{\l = 1}^{2 j_1 + 2 j_2 - 4} \ind_{\{ |n_\l| \sim N_\l \}} \jbb{n_\l}^{-2} \bigg) \\
&\quad \times \bigg\langle \sum_{\l = 1}^{2 j_1 + 2 j_2 - 4} n_\l \bigg\rangle^{2 \al - 3 - 2 \eps}  \bigg\llbracket \sum_{\l = 1}^{2 j_1 - 1} n_\l \bigg\rrbracket^{- 2} \ind_{\{ |\kappa (\cj n) - \zeta| \leq 1 \}},
\end{split}
\label{cub-lin2}
\end{align}

\noi
where
\begin{align}
\kappa (\cj n) = \kappa (n_1, \dots, n_{2 j_1 - 1}) \deff \eps_{\text{sum}} \bigg\llbracket \sum_{\l = 1}^{2 j_1 - 1} n_\l \bigg\rrbracket + \sum_{\l = 1}^{2 j_1 - 1} \eps_\l \jbb{n_\l}
\label{kappa_cl}
\end{align}

\noi
with $\eps_{\text{sum}}, \eps_1, \dots, \eps_{2 j_1 - 1} \in \{ \pm 1 \}$. Here, we omitted the summation on $\eps_{\text{sum}}, \eps_1, \dots, \eps_{2 j_1 - 1}$ for simplicity. Thus, by using \eqref{cub-lin2} and the extended non-resonant counting estimate (Lemma~\ref{LEM:count_quin}~(ii)) and then summing over dyadic scales, we obtain the desired estimate.

We now consider the case when there are pairs in $\xi_{N, j_1, j_2}^{(3, 1, *)}$, but not all frequencies are paired. Recalling Definition~\ref{DEF:pair}, we let $\mathcal{P}$ be a pairing of $\{ 1, \dots, 2 j_1 + 2 j_2 - 4 \}$ which respects the partition $\mathcal{K}_1 \deff \{ 1, \dots, 2 j_1 - 1 \}$ and $\mathcal{K}_2 \deff \{ 2 j_1, \dots, 2 j_1 + 2 j_2 - 4 \}$. As before, we assume that $|n_\l| \sim N_\l$ for some dyadic numbers $1 \leq N_\l \leq N$ for all $1 \leq \l \leq 2 j_1 + 2 j_2 - 4$. By the Ito isometry (Lemma~\ref{LEM:sto_ito}~(ii)) and Jensen's inequality \eqref{jen}, we obtain the following contribution for \eqref{cub-lin_goal} (with all frequencies dyadically localized and the summation on $\iota_\l$'s omitted):
\begin{align}
\begin{split}
&N^{- (2 j_1 + 2 j_2 - 8) (3 - 2 \al) + 2 \eps'} \sum_{ (n_\l)_{\l \notin \mathcal P} } \bigg( \prod_{\l \notin \mathcal P} \ind_{\{|n_\l| \sim N_\l \}} \bigg) \bigg\langle \sum_{\l \notin \mathcal P} n_\l \bigg\rangle^{2 \al - 3 - 2 \eps} \\
&\quad \times \int_{\R^{2 j_1 + 2 j_2 - 4 - 2 |\mathcal P|}}  \bigg(  \sum_{(n_\l)_{\l \in \mathcal{P}}}  \ind_{\substack{(n_1, \dots, n_{2 j_1 + 2 j_2 - 4}) \\ \text{admissible}}}  \bigg( \prod_{\l \in \mathcal{P}} \ind_{\{|n_\l| \sim N_\l \}} \bigg)  \int_{\R^{2 |\mathcal P|}}  \int_0^t  e^{- \frac{t - t'}{2}}     \\
&\quad \times  \frac{ \sin ( (t - t') \jbb{\sum_{\l \in \mathcal{K}_1} n_\l} ) }{\jbb{\sum_{\l \in \mathcal{K}_1} n_\l}}  \prod_{\l \in \mathcal{K}_1}  \phi_{\iota_\l, n_\l} (t', s_\l)  dt' \cdot \prod_{\l' \in \mathcal{K}_2 \wedge \l' \in \mathcal{P}} \phi_{\iota_{\l'}, n_{\l'}} (t, s_{\l'})  d (s_\l)_{\l \in \mathcal{P}} \bigg)^2  \\
&\quad \times \prod_{\l' \in \mathcal{K}_2 \wedge \l' \notin \mathcal{P}} \phi_{\iota_{\l'}, n_{\l'}} (t, s_{\l'})^2  d (s_\l)_{\l \notin \mathcal P}.
\end{split}
\label{cub-lin3}
\end{align}

\noi
Here, if $\mathcal P$ has the form $\{ (\l_1, \l_2), (\l_3, \l_4), \dots, (\l_{2 q' - 1}, \l_{2q'}) \}$ for some integer $q' \geq 1$, we used the notation
\begin{align}
d (s_\l)_{\l \in \mathcal P} = \delta (s_{\l_1} - s_{\l_2}) \cdots \delta (s_{\l_{2 q' - 1}} - s_{\l_{2 q'}}) d s_{\l_1} \cdots d s_{\l_{2 q'}}.
\label{def_dsl}
\end{align}

\noi
By expanding the sine and cosine functions into complex exponentials, integrating in $t'$ and $s_\ell$'s, and using a level set decomposition, we can bound
\begin{align*}
&\eqref{cub-lin3} \\
&\les N^{- (2 j_1 + 2 j_2 - 8) (3 - 2 \al) + 2 \eps'} \sum_{(n_\l)_{\l \notin \mathcal P}} \bigg( \prod_{\l \notin \mathcal P} \ind_{\{ |n_\l| \sim N_\l \}} \jbb{n_\l}^{- 2} \bigg) \bigg\langle \sum_{\l \notin \mathcal P} n_\l \bigg\rangle^{2 \al - 3 - 2 \eps} \\
&\quad \times \bigg( \sum_{(n_\l)_{\l \in \mathcal P}}  \ind_{\substack{(n_1, \dots, n_{2 j_1 + 2 j_2 - 4}) \\ \text{admissible}}}  \bigg( \prod_{\l \in \mathcal P} \ind_{\{ |n_\l| \sim N_\l \}} \jbb{n_\l}^{-1} \bigg)  \bigg\llbracket \sum_{\l \in \mathcal{K}_1} n_\l \bigg\rrbracket^{-1}  \frac{1}{\jb{\kappa (\cj n)}}  \bigg)^2 \\
&\les N^{- (2 j_1 + 2 j_2 - 8) (3 - 2 \al) + 2 \eps'} \sum_{(n_\l)_{\l \notin \mathcal P}} \bigg( \prod_{\l \notin \mathcal P} \ind_{\{ |n_\l| \sim N_\l \}} \jbb{n_\l}^{- 2} \bigg) \bigg\langle \sum_{\l \notin \mathcal P} n_\l \bigg\rangle^{2 \al - 3 - 2 \eps} \\
&\quad \times \bigg( \sum_{\zeta \in \Z} \frac{1}{\jb{\zeta}}  \sum_{(n_\l)_{\l \in \mathcal P}}  \ind_{\substack{(n_1, \dots, n_{2 j_1 + 2 j_2 - 4}) \\ \text{admissible}}}  \bigg( \prod_{\l \in \mathcal P} \ind_{\{ |n_\l| \sim N_\l \}} \jbb{n_\l}^{-1} \bigg)  \bigg\llbracket \sum_{\l \in \mathcal{K}_1} n_\l \bigg\rrbracket^{-1} \ind_{\{ |\kappa (\cj n) - \zeta| \leq 1 \}}  \bigg)^2 ,
\end{align*}

\noi
so that by applying the Cauchy-Schwarz inequality in $\zeta$ followed by H\"older's inequality in $\zeta$, we get
\begin{align}
\begin{split}
&\eqref{cub-lin3} \\
&\les N^{- (2 j_1 + 2 j_2 - 8) (3 - 2 \al) + 2 \eps''} \sup_{\zeta \in \Z} \sum_{(n_\l)_{\l \notin \mathcal P}} \bigg( \prod_{\l \notin \mathcal P} \ind_{\{ |n_\l| \sim N_\l \}} \jbb{n_\l}^{- 2} \bigg) \bigg\langle \sum_{\l \notin \mathcal P} n_\l \bigg\rangle^{2 \al - 3 - 2 \eps} \\
&\quad \times \bigg( \sum_{(n_\l)_{\l \in \mathcal P}}  \ind_{\substack{(n_1, \dots, n_{2 j_1 + 2 j_2 - 4}) \\ \text{admissible}}}  \bigg( \prod_{\l \in \mathcal P} \ind_{\{ |n_\l| \sim N_\l \}} \jbb{n_\l}^{-1} \bigg) \bigg\llbracket \sum_{\l \in \mathcal{K}_1} n_\l \bigg\rrbracket^{-1}  \ind_{\{ |\kappa (\cj n) - \zeta| \leq 1 \}}  \bigg)^2
\end{split}
\label{cub-lin4}
\end{align}

\noi
for some $\eps'' > 0$ arbitrarily small, where $\kappa (\cj n)$ is as defined in \eqref{kappa_cl}. Thus, by using \eqref{cub-lin4} and the extended resonant counting estimates (\eqref{count_quin} and \eqref{count_quin+} in Lemma~\ref{LEM:count_quin2}) and then summing over dyadic scales, we obtain the desired estimate.

It remains to deal with $\E \big[ \xi_{N, j_1, j_2}^{(3, 1)} \big]$, but this is similar to the case of pairs with $|\mathcal{K}_1| = |\mathcal{K}_2| = |\mathcal{P}|$ in the last paragraph. The only essential change is that we need to use the counting estimate \eqref{count_quin-} in Lemma~\ref{LEM:count_quin2} at the end since all frequencies are in the pairing $\mathcal{P}$.

\smallskip \noi
(vi) From \eqref{sto_all}, we need to show the regularity properties of the following term:
\begin{align*}
\Xi_{t_0, N, j_1, j_2}^{(5)} \deff N^{- (2 j_1 + 2 j_2 - 8) (\frac 32 - \al) + \eps'} \I_{t_0} \big( \pi_N \I \big( \wick{\<1>_N^{2 j_1 - 1}} \big) \wick{\<1>_N^{2 j_2 - 2}} \big),
\end{align*}

\noi
where $j_1 \geq 2$ and $j_2 \geq 2$. Here, we let $\eps' = 0$ if $j_1 = 2$ and $j_2 = 2$; otherwise, we let $\eps' > 0$ be small. From \eqref{Ysb} and Lemma~\ref{LEM:nhomo}, it suffices to consider the term
\begin{align*}
\xi_{N, j_1, j_2}^{(5)} \deff N^{- (2 j_1 + 2 j_2 - 8) (\frac 32 - \al) + \eps'} \I \big( \wick{\<1>_N^{2 j_1 - 1}} \big) \wick{\<1>_N^{2 j_2 - 2}}
\end{align*}

\noi
in the space $X_T^{- \frac 12, - \frac 12 + \dl}$. In fact, by an interpolation with a trivial $X_T^{0, 0} = L_T^2 L_x^2$ bound, we only need to study $\xi_{N, j_1, j_2}^{(5)}$ in the space $X_T^{- \frac 12 - \eps, - \frac 12 - \dl}$ for $\eps, \dl > 0$ sufficiently small.

As in part (v), in view of the product formula (Lemma~\ref{LEM:sto_prod}), we need to discuss whether there are pairs between $\I ( \wick{\<1>_N^{2 j_1 - 1}} )$ and $\wick{\<1>_N^{2 j_2 - 2}}$. When there are pairs, in view of the embedding $L_T^2 H^{- \frac 12 - \eps} \subset X_T^{- \frac 12 - \eps, - \frac 12 - \dl}$, we can argue in the same way as in part (v) to obtain the desired estimate. Here, since $2 j_2 - 2 \geq 2$ and $2 j_1 - 1 \neq 2 j_2 - 2$, we can always apply the resonant counting estimate \eqref{count_quin} in Lemma~\ref{LEM:count_quin2}.

It remains to deal with the case when there are no pairs between $\I ( \wick{\<1>_N^{2 j_1 - 1}} )$ and $\wick{\<1>_N^{2 j_2 - 2}}$, and we denote this contribution by $\xi_{N, j_1, j_2}^{(5, *)}$. Note that for any $\tau \in \R$ and $n \in \Z^3$, by Fubini's theorem (Lemma~\ref{LEM:sto_fub}), we have
\begin{align*}
\F_{t, x} &\big( \xi_{N, j_1, j_2}^{(5, *)} \ind_{[0, T]} \big) (\tau, n) \\
&= N^{- (2 j_1 + 2 j_2 - 8) (\frac 32 - \al) + \eps'} \sum_{\substack{|n_1|, \dots, |n_{2 j_1 + 2 j_2 - 3}| \leq N \\ n_1 + \cdots + n_{2 j_1 + 2 j_2 - 3 = n}}}  \sum_{\substack{ \iota_1, \dots, \iota_{2 j_1 + 2 j_2 - 3} \\ \in \{0, 1, 2\} }}  \int_{\R^{2 j_1 + 2 j_2 - 3}} \int_0^T e^{- i t \tau}  \\
&\quad \times \int_0^t e^{- \frac{t - t'}{2}} \frac{\sin ( (t - t') \jbb{\sum_{\l = 1}^{2j_1 - 1} n_\l} )}{\jbb{\sum_{\l = 1}^{2j_1 - 1} n_\l}} \prod_{\l = 1}^{2 j_1 - 1} \phi_{\iota_\l, n_\l} (t', s_\l) dt'  \\
&\quad \times \prod_{\l' = 2 j_1}^{2 j_1 + 2 j_2 - 3} \phi_{\iota_{\l'}, n_{\l'}} (t, s_{\l'})
dt  d B_{n_1}^{\iota_1} (s_1) \cdots d B_{n_{2 j_1 + 2 j_2 - 3}}^{\iota_{2 j_1 + 2 j_2 - 3}} (s_{2 j_1 + 2 j_2 - 3})
\end{align*}

By the Wiener chaos estimate (Lemma~\ref{LEM:hyp}), the Ito isometry (Lemma~\ref{LEM:sto_ito}~(ii)), and \eqref{jen}, we have (with the summation of $\iota_\l$'s omitted)
\begin{align*}
\E \Big[ &\big\| \xi_{N, j_1, j_2}^{(5, *)} \big\|_{X_T^{- \frac 12 - \eps, - \frac 12 - \dl}}^p \Big]^{\frac 2p} \\
&\les \sum_{\eps_0 \in \{\pm 1\}} \E \Big[ \big\| \jb{n}^{- \frac 12 - \eps} \jb{\tau}^{- \frac 12 - \dl} \F_{t, x} \big( \xi_{N, j_1, j_2}^{(5, *)} \ind_{[0, T]} \big) (\tau - \eps_0 \jbb{n}, n) \big\|_{\l_n^2 L_\tau^2}^2 \Big] \\
&\les N^{- (2 j_1 + 2 j_2 - 8) (3 - 2 \al) + 2 \eps'} \sum_{\eps_0 \in \{\pm 1\}} \sum_{n \in \Z^3} \jb{n}^{-1 - 2 \eps} \sum_{\substack{|n_1|, \dots, |n_{2 j_1 + 2 j_2 - 3}| \leq N \\ n_1 + \cdots + n_{2 j_1 + 2 j_2 - 3 = n}}}  \\
&\quad  \int_\R \jb{\tau}^{-1 - 2 \dl} \int_{\R^{2 j_1 + 2 j_2 - 3}} \bigg( \int_0^T e^{- i t (\tau - \eps_0 \jbb{n})} \int_0^t e^{- \frac{t - t'}{2}} \frac{\sin ( (t - t') \jbb{\sum_{\l = 1}^{2j_1 - 1} n_\l} )}{\jbb{\sum_{\l = 1}^{2j_1 - 1} n_\l}}  \\
&\quad \times \prod_{\l = 1}^{2 j_1 - 1} \phi_{\iota_\l, n_\l} (t', s_\l) dt'  \prod_{\l' = 2 j_1}^{2 j_1 + 2 j_2 - 3} \phi_{\iota_{\l'}, n_{\l'}} (t, s_{\l'})
dt  \bigg)^2 d s_1 \cdots d s_{2 j_1 + 2 j_2 - 3} d \tau.
\end{align*}

\noi
As before, we expand the sine and cosine functions into complex exponentials. Thus, from \eqref{def_phi}, the integration in $t'$ in the above expression gives
\begin{align*}
\frac{e^{( \frac{2 j_1 - 2}{2} + i \kappa (\cj n) ) t} - 1}{\frac{2 j_1 - 2}{2} + i \kappa (\cj n)} C(s_1, \dots, s_{2 j_1 - 1}) ,
\end{align*}

\noi
where $\kappa (\cj n)$ is as defined in \eqref{kappa_cl} and $C(s_1, \dots, s_{2 j_1 - 1})$ is a bounded constant independent of $t$. Thus, after integration in $t$ and in $s_\ell$'s, we obtain (with the summation of $\eps_0$ omitted)
\begin{align}
\begin{split}
\E \Big[ &\big\| \xi_{N, j_1, j_2}^{(5, *)} \big\|_{X_T^{- \frac 12 - \eps, - \frac 12 - \dl}}^p \Big]^{\frac 2p} \\
&\les_T N^{- (2 j_1 + 2 j_2 - 8) (3 - 2 \al) + 2 \eps'}  \sum_{n \in \Z^3} \jb{n}^{-1 - 2 \eps} \sum_{\substack{|n_1|, \dots, |n_{2 j_1 + 2 j_2 - 3}| \leq N \\ n_1 + \cdots + n_{2 j_1 + 2 j_2 - 3 = n}}} \prod_{\l = 1}^{2 j_1 + 2 j_2 - 3} \jbb{n_\l}^{- 2} \\
&\quad \times \bigg\llbracket \sum_{\l = 1}^{2 j_1 - 1} n_\l \bigg\rrbracket^{-2} \cdot \frac{1}{\jb{\kappa (\cj n)}^2} \int_\R \jb{\tau}^{-1 - 2 \dl} \Big( \frac{1}{ \jb{\tau - \kappa_1 (\cj n)}^2 } + \frac{1}{ \jb{\tau - \kappa_2 (\cj n)}^2 } \Big) d \tau,
\end{split}
\label{quin1}
\end{align}

\noi
where
\begin{align*}
\kappa_1 (\cj n) &= \kappa_1 (n_1, \dots, n_{2 j_1 + 2 j_2 - 3}) \deff \eps_0 \bigg\llbracket \sum_{\l = 1}^{2 j_1 + 2 j_2 - 3} n_\l \bigg\rrbracket + \sum_{\l = 1}^{2 j_1 + 2 j_2 - 3} \eps_\l \jbb{n_\l}, \\
\kappa_2 (\cj n) &= \kappa_2 (n_1, \dots, n_{2 j_1 + 2 j_2 - 3}) \deff \eps_0 \bigg\llbracket \sum_{\l = 1}^{2 j_1 + 2 j_2 - 3} n_\l \bigg\rrbracket + \eps_{\text{sum}} \bigg\llbracket \sum_{\l = 1}^{2 j_1 - 1} n_\l \bigg\rrbracket + \sum_{\l' = 2 j_1}^{2 j_1 + 2 j_2 - 3} \eps_{\l'} \jbb{n_{\l'}}.
\end{align*}

\noi
with $\eps_{\text{sum}}, \eps_1, \dots, \eps_{2 j_1 + 2 j_2 - 3} \in \{ \pm 1 \}$. After dyadically localizing $|n_\l| \sim N_\l$ for some dyadic numbers $1 \leq N_\l \leq N$ for all $1 \leq \l \leq 2 j_1 + 2 j_3 - 3$, by level set decompositions on $\kappa$, $\kappa_1$, and $\kappa_2$ and Lemma~\ref{LEM:conv_int}, we obtain the following bound for \eqref{quin1} (with all frequencies dyadically localized and the summations on $\eps_\l$'s omitted):
\begin{align*}
&N^{- (2 j_1 + 2 j_2 - 8) (3 - 2 \al) + 2 \eps'}\sum_{n_1, \dots, n_{2 j_1 + 2 j_2 - 3} \in \Z^3} \bigg( \prod_{\l = 1}^{2 j_1 + 2 j_2 - 3} \ind_{\{ |n_\l| \sim N_\l \}} \jbb{n_\l}^{- 2} \bigg) \bigg\langle \sum_{\l = 1}^{2 j_1 + 2 j_2 - 3} n_\l \bigg\rangle^{-1 - 2 \eps}  \\
&\quad \quad \times \bigg\llbracket \sum_{\l = 1}^{2 j_1 - 1} n_\l \bigg\rrbracket^{-2} \cdot \frac{1}{\jb{\kappa (\cj n)}^2} \Big( \frac{1}{\jb{\kappa_1 (\cj n)}^{1 + 2 \dl}} + \frac{1}{\jb{\kappa_2 (\cj n)}^{1 + 2 \dl}} \Big) \\
&\quad \les  N^{- (2 j_1 + 2 j_2 - 8) (3 - 2 \al) + 2 \eps'} \sup_{\zeta, \zeta \in \Z}  \sum_{n_1, \dots, n_{2 j_1 + 2 j_2 - 3} \in \Z^3} \bigg( \prod_{\l = 1}^{2 j_1 + 2 j_2 - 3} \ind_{\{ |n_\l| \sim N_\l \}} \jbb{n_\l}^{- 2} \bigg)  \\
&\quad \quad \times \bigg\langle \sum_{\l = 1}^{2 j_1 + 2 j_2 - 3} n_\l \bigg\rangle^{-1 - 2 \eps}  \bigg\llbracket \sum_{\l = 1}^{2 j_1 - 1} n_\l \bigg\rrbracket^{-2} \ind_{\{ |\kappa (\cj n) - \zeta| \leq 1 \}} \cdot \big( \ind_{\{ |\kappa_1 (\cj n) - \zeta'| \leq 1 \}} + \ind_{\{ |\kappa_2 (\cj n) - \zeta'| \}} \big).
\end{align*}

\noi
Thus, by using the extended non-resonant counting estimate (Lemma~\ref{LEM:count_quin}~(i)) and then summing over dyadic scales, we obtain the desired estimate.

\smallskip \noi
(vii) From \eqref{sto_all}, \eqref{Ysb}, and Lemma~\ref{LEM:nhomo}, it suffices to consider the term
\begin{align*}
\xi_{N, j_1, j_2, j_3}^{(7)} \deff N^{- (2 j_1 + 2 j_2 + 2 j_3 - 12) (\frac 32 - \al) + \eps'} \I \big( \wick{\<1>_N^{2 j_1 - 1}} \big) \I \big( \wick{\<1>_N^{2 j_2 - 1}} \big) \wick{\<1>_N^{2 j_3 - 3}}
\end{align*}

\noi
in the space $L_T^2 H^{- \frac 12} (\T^3)$, where $j_1 \geq 2$, $j_2 \geq 2$, and $j_3 \geq 2$. Here, we let $\eps' = 0$ if $j_1 = 2$, $j_2 = 2$, and $j_3 = 2$; otherwise, we let $\eps' > 0$ be small. Let $\mathcal{K}_1 \deff \{ n_1, \dots, n_{2 j_1 - 1} \}$ be the frequencies of the nodes of $\I ( \wick{\<1>_N^{2 j_1 - 1}} )$, $\mathcal{K}_2 \deff \{ n_{2 j_1}, \dots, 2 j_1 + 2 j_2 - 2 \}$ be the frequencies of the nodes of $\I ( \wick{\<1>_N^{2 j_2 - 1}} )$, and $\mathcal{K}_3 \deff \{ 2 j_1 + 2 j_2 - 1, \dots, 2 j_1 + 2 j_2 + 2 j_3 - 5 \}$ be the frequencies of the nodes of $\wick{\<1>_N^{2 j_3 - 3}}$. Recalling Definition~\ref{DEF:pair}, we let $\mathcal{P}$ (which can be empty) be a pairing of $\{1, \dots, 2 j_1 + 2 j_2 + 2 j_3 - 5 \}$ which respects the partition $\mathcal{K}_1$, $\mathcal{K}_2$, and $\mathcal{K}_3$. 

As in part (v), we obtain the following formulation for the variance of $\xi_{N, j_1, j_2, j_3}^{(7)}$ in the space $H^{- \frac 12}$ at a fixed time $t \in [0, T]$:
\begin{align*}
&N^{- (2 j_1 + 2 j_2 + 2 j_3 - 12) (3 - 2 \al) + 2 \eps'} \sum_{(n_\l)_{\l \notin \mathcal P} } \bigg\langle \sum_{\l \notin \mathcal P} n_\l \bigg\rangle^{- 1} \\
&\quad \times \int_{\R^{2 j_1 + 2 j_2 + 2 j_3 - 5 - 2 |\mathcal P|}}  \bigg(  \sum_{(n_\l)_{\l \in \mathcal P}}  \ind_{\substack{(n_1, \dots, n_{2 j_1 + 2 j_2 + 2 j_3 - 5}) \\ \text{admissible}}}   \\
&\quad \times \int_{\R^{2 |\mathcal P|}}  \int_0^t e^{- \frac{t - t'_1}{2}} \frac{\sin ( (t - t'_1) \jbb{\sum_{\l \in \mathcal{K}_1} n_\l} )}{\jbb{\sum_{\l \in \mathcal{K}_1} n_\l}}  \prod_{\l \in \mathcal{K}_1} \phi_{\iota_\l, n_\l} (t'_1, s_\l) dt'_1  \\
&\quad \times \int_0^t e^{- \frac{t - t'_2}{2}} \frac{\sin ( (t - t'_2) \jbb{\sum_{\l' \in \mathcal{K}_2} n_{\l'}} )}{\jbb{\sum_{\l' \in \mathcal{K}_2} n_{\l'}}}  \prod_{\l' \in \mathcal{K}_2} \phi_{\iota_{\l'}, n_{\l'}} (t'_2, s_\l) dt'_2 \\
&\quad \times  \prod_{\l'' \in \mathcal{K}_3 \wedge \l'' \in \mathcal{P}} \phi_{\iota_{\l''}, n_{\iota_{\l''}}} (t, s_{\l''})  d (s_\l)_{\l \in \mathcal{P}} \bigg)^2 \prod_{\l'' \in \mathcal{K}_3 \wedge \l'' \notin \mathcal{P}} \phi_{\iota_{\l''}, n_{\iota_{\l''}}} (t, s_{\l''})^2  d (s_\l)_{\l \notin \mathcal P},
\end{align*}

\noi
where we used the notation $d (s_\l)_{\l \in \mathcal P}$ as in \eqref{def_dsl}. Then, as in part (v), we dyadically localize $|n_\l| \sim N_\l$, expand the sine and cosine functions into complex exponentials, and integrate in $t_1'$, $t_2'$, and $s_\ell$'s to obtain the following contribution (with all frequencies dyadically localized and the summations on $\eps_\l$'s omitted):
\begin{align*}
&N^{- (2 j_1 + 2 j_2 + 2 j_3 - 12) (3 - 2 \al) + 2 \eps'} \sum_{(n_\l)_{\l \notin \mathcal P} } \bigg\langle \sum_{\l \notin \mathcal P} n_\l \bigg\rangle^{- 1} \bigg( \sum_{(n_\l)_{\l \in \mathcal P}}  \ind_{\substack{(n_1, \dots, n_{2 j_1 + 2 j_2 + 2 j_3 - 5}) \\ \text{admissible}}}  \\  
&\quad \times \bigg( \prod_{\l = 1}^{2 j_1 + 2 j_2 + 2 j_3 - 5} \ind_{\{ |n_\l| \sim N_\l \}} \bigg)  \bigg\llbracket \sum_{\l \in \mathcal{K}_1} n_\l \bigg\rrbracket^{-1}  \bigg(  \prod_{\l \in \mathcal{K}_1} \jbb{n_\l}^{-1}  \bigg) \cdot \frac{1}{\jb{\kappa_{j_1} (\cj n)}}  \\
&\quad \times  \bigg\llbracket \sum_{\l' \in \mathcal{K}_2} n_{\l'} \bigg\rrbracket^{-1}  \bigg(  \prod_{\l' \in \mathcal{K}_2} \jbb{n_{\l'}}^{-1}  \bigg) \cdot \frac{1}{\jb{\kappa_{j_2} (\cj n)}}  \cdot \prod_{\l'' \in \mathcal{K}_3} \jbb{n_{\l''}}^{-1} \bigg)^2,
\end{align*}

\noi
where
\begin{align*}
\kappa_{j_1} (\cj n) &= \kappa (n_1, \dots, n_{2j_1 - 1}) \deff \eps_{\text{sum}}^1  \bigg\llbracket \sum_{\l \in \mathcal{K}_1} n_{\l} \bigg\rrbracket + \sum_{\l \in \mathcal{K}_1} \eps_\l^1 \jbb{n_\l}, \\
\kappa_{j_2} (\cj n) &= \kappa (n_{2j_1}, \dots, n_{2j_1 + 2 j_2 - 2}) \deff \eps_{\text{sum}}^2 \bigg\llbracket \sum_{\l' \in \mathcal{K}_2} n_{\l'} \bigg\rrbracket + \sum_{\l' \in \mathcal{K}_2} \eps_{\l'}^2 \jbb{n_{\l'}}
\end{align*}

\noi
with $\eps_{\text{sum}}^1, \eps_{\text{sum}}^2, \eps_{\l}^1, \eps_{\l'}^2 \in \{\pm 1\}$. The desired estimate then follows from level set decompositions of $\kappa_{j_1}$ and $\kappa_{j_2}$ and the extended septic counting estimate in Lemma~\ref{LEM:count_sep}.
\end{proof}

From Lemma~\ref{LEM:data}, we immediately obtain the following proposition.
\begin{proposition}
\label{PROP:data}
Let $1 < \al \leq \frac 98$, $\eps, \dl > 0$ be sufficiently small, and $T \geq 1$. Then, the truncated enhanced data set $\pmb{\Xi}_N$ in \eqref{enhN} converges $\mathbb{Q} \otimes \mathbb{P}^w$-almost surely to the enhanced data set $\pmb{\Xi}$ in \eqref{enh} in the space $\mathcal{X}_T^{\al, \eps, \dl}$ defined in \eqref{XTaed}. We also have the tail estimate
\begin{align*}
\mathbb{Q} \otimes \mathbb{P}^w \big( \| \pmb{\Xi}_{N} \|_{\mathcal{X}_T^{\al, \eps, \dl}} > \ld \big) \leq C(T) e^{- c (T) \ld^\be}
\end{align*}

\noi
for any $\ld > 0$ and some constants $C(T), c(T), \be > 0$, uniformly in $N \in \N \cup \{\infty\}$. Moreover, there exists small $\ta > 0$ such that
\begin{align*}
\mathbb{Q} \otimes \mathbb{P}^w \big( N^\ta \| \pmb{\Xi}_{N} - \pmb{\Xi} \|_{\mathcal{X}_T^{\al, \eps, \dl}} > \ld \big) \leq C(T) e^{- c (T) \ld^\be}
\end{align*}

\noi
for any $\ld > 0$ and some constants $C(T), c(T), \be > 0$, uniformly in $N \geq 1$.
\end{proposition}

\subsection{Stability estimate}
In this subsection, we establish a key stability estimate. We recall the equation \eqref{SdNLWv3} for $v_N$ starting at time $t_0 \geq 0$:
\begin{align}
\begin{split}
\vec v_N &= \vec \G_{t_0, N} [v_N] \\
&\deff \vec S(\cdot) (\vec v_{N} (t_0)) - 2 \cj{a}_{1, N} N^{2 (\frac 32 - \al)} \pi_N \vec \I_{t_0} \big( \<1>_N - \Xi_N^{(3)} + \pi_N v_N \big) \\
&\quad - \vec \Xi_{t_0, N}^{(2)} (\pi_N v_N) + \vec \Xi_{t_0, N}^{(3, 1)} ( \pi_N v_N ) + \vec \Xi_{t_0, N}^{(5)} - \vec \Xi_{t_0, N}^{(7)} \\
&\quad - \sum_{j = 2}^m C_j \cj{a}_{j, N} N^{(2 j - 4) (\al - \frac 32)} \pi_N \vec \I_{t_0} \big( \wick{\<1>_N^{2j - 3}} (\pi_N v_N)^2 \big) \\
&\quad - \sum_{j = 2}^m \sum_{\l = 3}^{2j - 1} \sum_{k = 0}^\l C_{j, \l, k} \cj{a}_{j, N} N^{(2 j - 4) (\al - \frac 32)} \pi_N \vec \I_{t_0} \big( (- \Xi_N^{(3)})^k \wick{\<1>_N^{2j - 1 - \l}} (\pi_N v_N)^{\l - k} \big),
\end{split}
\label{SdNLWv4a}
\end{align}

\noi
where $\vec v_N |_{t = 0} = (v_{0, N}, 0)$ with $v_{0, N}$ being given by Theorem~\ref{THM:rep} and $\Xi_{t_0, N}^{(2)}$, $\Xi_{t_0, N}^{(3)}$, $\Xi_{t_0, N}^{(3, 1)}$, $\Xi_{t_0, N}^{(5)}$, and $\Xi_{t_0, N}^{(7)}$ are defined in \eqref{sto_all}. We also recall the truncated enhanced data set $\pmb{\Xi}_N$ in \eqref{enhN} and the limiting enhanced data set $\pmb{\Xi}$ in \eqref{enh}.
In view of the convergence of $v_{0, N}$ to $v_0$ as stated in Theorem~\ref{THM:rep}, the convergence of $\pmb{\Xi}_N$ to $\pmb{\Xi}$ in Proposition~\ref{PROP:data}, \eqref{a1N}, and the convergence of $\cj a_{2, N}$ to $\cj a_2$, we expect $\vec v_N$ to converge to a limit $\vec v$ satisfying the equation
\begin{align}
\begin{split}
\vec v &= \vec \G_{t_0} [v] \\
&\deff \vec S(\cdot) (\vec v(t_0)) - a \vec \I_{t_0} \big( \<1> - 4 \cj{a}_2 \<30> + v \big) - 12 \cj{a}_2 \vec \If_{t_0}^{\,\<2>} (v) + 96 \cj{a}_2^2 \vec \If_{t_0}^{\,\<31>} ( v ) \\
&\quad + 48 \cj{a}_2^2 \vec{\<320>}_{t_0} - 192 \cj{a}_2^3 \vec{\<70>}_{t_0} - 12 \cj{a}_2 \vec \I_{t_0} (\<1> v^2) - 4 \cj{a}_2 \vec \I_{t_0} (v^3) \\
&\quad + 48 \cj{a}_2^2 \vec \I_{t_0} \big( \<30> v^2 \big) - 192 \cj{a}_2^3 \vec \I_{t_0} \big( \<30>^2 v \big) + 256 \cj{a}_2^4 \vec \I_{t_0} \big( \<30>^3 \big).
\end{split}
\label{SdNLWv4b}
\end{align}

\noi
where $\vec v |_{t = 0} = (v_0, 0)$

We now state and prove the following stability result.
\begin{proposition}
\label{PROP:stb}
Let $1 < \al \leq \frac 98$, $\eps, \dl > 0$ be sufficiently small, $1 \ll p < \infty$ be sufficiently large, $T > 0$, and $K_0, K_0', K_1, K_2 \geq 1$ with $K_0 + K_0' \ll K_1$. Then, there exist $0 < \kappa_0 = \kappa_0 (T, K_1) \leq 1$ sufficiently small and $N_0 = N_0 (T, K_1, K_2) \in \N$ sufficiently large such that the following holds true. 
Let $N \geq N_0$ and $\vec v_N = \vec v_N (v_{0, N}, \pmb{\Xi}_N)$ be the solution to the equation \eqref{SdNLWv4a} on $[0, T]$ with the truncated enhanced data set $\pmb{\Xi}_N$ and initial data $\vec v_N |_{t = 0} = (v_{0, N}, 0)$ and suppose that
\begin{align}
&\| \pmb{\Xi} \|_{\mathcal{X}_T^{\al, \eps, \dl}} \leq K_0,
\label{enh_bdd} \\
&\| v_0 \|_{H^{\al - \frac 12}} \leq K_0',
\label{init_bdd} \\
&\| \vec v_N \|_{\vec X_T^{\al - \frac 12, \frac 12 + \dl}} \leq K_1,
\label{vN_bddXT} \\
&\| v_N \|_{L_T^p W_x^{\al - \frac 32 - \eps, \infty}} \leq K_2,
\label{vN_bddW} \\
&\| \pmb{\Xi}_N - \pmb{\Xi} \|_{\mathcal{X}_T^{\al, \eps, \dl}} \leq N^{- \ta},
\label{enh_diff} \\
&\| v_{0, N} - v_0 \|_{H^{\al - \frac 12}} \leq \kappa
\label{init_diff}
\end{align}

\noi
for some $0 < \ta \ll 1$ and $0 < \kappa \leq \kappa_0$, where the $\mathcal{X}_T^{\al, \eps, \dl}$-norm is defined in \eqref{XTaed}. Then, there exists a solution $v = v (v_0, \pmb{\Xi})$ to the limiting equation \eqref{SdNLWv4b} on $[0, T]$ with the enhanced data set $\pmb{\Xi}$ and initial data $\vec v |_{t = 0} = (v_0, 0)$. Moreover, we have
\begin{align}
\| \vec v_N - \vec v \|_{\vec X_T^{\al - \frac 12, \frac 12 + \dl}} \leq C (T, K_1, K_2) (\kappa + N^{-\ta'})
\label{vNv_diff2}
\end{align}

\noi
and
\begin{align}
\| \vec v \|_{\vec X_T^{\al - \frac 12, \frac 12 + \dl}} \leq 2 K_1
\label{v_bddX}
\end{align}

\noi
for some constants $C (T, K_1, K_2) > 0$ and $\ta' > 0$.
\end{proposition}

\begin{proof}
We prove the proposition by the following three steps.

\smallskip \noi
{\bf Step 1:}
Let us first show that there exists $0 < T_0 = T_0 (K_1) \leq 1$ such that a unique solution $\vec v$ to \eqref{SdNLWv4b} exists in $X_{T_0}^{\al - \frac 12, \frac 12 + \dl}$ and that 
\begin{align}
\| \vec v \|_{\vec X_{T_0}^{\al - \frac 12, \frac 12 + \dl}} \leq C_0 K_1.
\label{v4K}
\end{align}

\noi
for some constant $C_0 > 0$ independent of $T_0$.

Given $0 < T_0 \leq 1$, by using the estimates from \cite[Proposition~5.4]{LTzW} along with the following additional estimate using Lemma~\ref{LEM:nhomo} and Lemma~\ref{LEM:time}:
\begin{align*}
\| \vec \I (\<1> - 4 \cj{a}_2 \<30> + v) \|_{\vec X_{T_0}^{\al - \frac 12, \frac 12 + \dl}} \les T_0^{\frac 12 - \dl} \Big( \| \<1> \|_{L_{T_0}^2 H_x^{- \frac 12}} + \| \<30> \|_{L_{T_0}^2 H_x^{- \frac 12}} + \| v \|_{X_{T_0}^{\al - \frac 12, \frac 12 + \dl}} \Big) ,
\end{align*}

\noi
we obtain
\begin{align*}
\| \vec \Gamma_0 [v] \|_{\vec X_{T_0}^{\al - \frac 12, \frac 12 + \dl}} \leq C T_0^\gamma \Big( \| v \|_{X_{T_0}^{\al - \frac 12, \frac 12 + \dl}}^3 + \| \pmb{\Xi} \|_{\mathcal{X}_T^{\al, \eps, \dl}}^3 \Big) + C \big( \| \pmb{\Xi} \|_{\mathcal{X}_T^{\al, \eps, \dl}} + \| v_0 \|_{H^{\al - \frac 12}} \big)
\end{align*}

\noi
and
\begin{align*}
\| &\vec \Gamma_0 [v_1] - \vec \Gamma_0 [v_2] \|_{\vec X_{T_0}^{\al - \frac 12, \frac 12 + \dl}} \\
&\leq C T_0^\gamma \Big( \| v_1 \|_{X_{T_0}^{\al - \frac 12, \frac 12 + \dl}}^2 + \| v_2 \|_{X_{T_0}^{\al - \frac 12, \frac 12 + \dl}}^2 + \| \pmb{\Xi} \|_{\mathcal{X}_T^{\al, \eps, \dl}}^2 \Big) \| v_1 - v_2 \|_{X_{T_0}^{\al - \frac 12, \frac 12 + \dl}} 
\end{align*}

\noi
for some constant $C > 0$ and $\gamma > 0$. 
Thus, a standard contraction argument along with \eqref{enh_bdd}, \eqref{init_bdd}, and the fact that $K_0 + K_0' \ll K_1$ shows local well-posedness of \eqref{SdNLWv4b} on $[0, T_0]$ for some $0 < T_0 = T_0 (K_1) \leq 1$ and the bound \eqref{v4K}.

\smallskip \noi
{\bf Step 2:}
We now show that with a possibly smaller value of $T_0 = T_0 (K_1) > 0$ from Step 1, we have the following improvement of the bound \eqref{v4K}:
\begin{align}
\| \vec v \|_{\vec X_{T_0}^{\al - \frac 12, \frac 12 + \dl}} \leq 2 K_1 .
\label{v2K}
\end{align}

From \eqref{enh_bdd} and \eqref{enh_diff}, we have the following uniform bound for the truncated enhanced data set:
\begin{align}
\| \pmb{\Xi}_N \|_{\mathcal{X}_T^{\al, \eps, \dl}} \leq K_0 + 1.
\label{enhN_bdd}
\end{align}

We set $\dl v_N = v - v_N$. From \eqref{SdNLWv4a} and \eqref{SdNLWv4b}, we see that $\dl v_N$ satisfies
\begin{align}
 \vec \dl v_N = \vec \G_0 [\dl v_N + v_N] - \vec \G_{0, N} [v_N]. 
\label{dvN}
\end{align}

\noi
Note that we have the following grouping:
\begin{align}
\vec \G_0 [\dl v_N + v_N] - \vec \G_{0, N} [v_N] = \vec S(t) (v_{0} - v_{0, N}, 0) + \sum_{j = 1}^{12} \vec{\textup{I}}_j,
\label{group}
\end{align}

\noi
where
\begin{align*}
\textup{I}_1 &\deff - a \I \big( \<1> - 4 \cj{a}_2 \<30> + \dl v_N + v_N \big) + 2 \cj{a}_{1, N} N^{2 (\frac 32 - \al)} \pi_N \I \big( \<1>_N - \Xi_N^{(3)} + \pi_N v_N \big), \\
\textup{I}_2 &\deff - 12 \cj{a}_2 \If_0^{\<2>} (\dl v_N + v_N) + \Xi_{0, N}^{(2)} (\pi_N v_N) , \\
\textup{I}_3 &\deff 96 \cj{a}_2^2 \If_0^{\<31>} (\dl v_N + v_N) - \Xi_{0, N}^{(3, 1)} (\pi_N v_N) , \\
\textup{I}_4 &\deff 48 \cj{a}_2^2 \<320>_0 - \Xi_{0, N}^{(5)} , \\
\textup{I}_5 &\deff - 192 \cj{a}_2^3 \<70>_0 + \Xi_{0, N}^{(7)} , \\
\textup{I}_6 &\deff - 12 \cj{a}_2 \I (\<1> (\dl v_N + v_N)^2) + 12 \cj{a}_{2, N} \pi_N \I ( \<1>_N (\pi_N v_N)^2 ) , \\
\textup{I}_7 &\deff - 4 \cj{a}_2 \I ( (\dl v_N + v_N)^3 ) + 4 \cj{a}_{2, N} \pi_N \I ( (\pi_N v_N)^3 ) , \\
\textup{I}_8 &\deff 48 \cj{a}_2^2 \I \big( \<30> (\dl v_N + v_N)^2 \big) - 12 \cj{a}_{2, N} \pi_N \I \big( \Xi_N^{(3)} (\pi_N v_N)^2 \big) , \\
\textup{I}_9 &\deff - 192 \cj{a}_2^3 \I \big( \<30>^2 (\dl v_N + v_N) \big) + 12 \cj{a}_{2, N} \pi_N \I \big( (\Xi_N^{(3)})^2 \pi_N v_N \big) , \\
\textup{I}_{10} &\deff 256 \cj{a}_2^4 \I \big( \<30>^3 \big) - 4 \cj{a}_{2, N} \pi_N \I \big( (\Xi_N^{(3)})^3 \big) , \\
\textup{I}_{11} &\deff \sum_{j = 3}^m C_j \cj{a}_{j, N} N^{(2j - 4) (\al - \frac 32)} \pi_N \I \big( \wick{\<1>_N^{2j - 3}} (\pi_N v_N)^2 \big), \\
\textup{I}_{12} &\deff \sum_{j = 3}^m \sum_{\l = 3}^{2j - 1} \sum_{k = 0}^\l C_{j, \l, k} \cj{a}_{j, N} N^{(2 j - 4) (\al - \frac 32)} 
\pi_N \I \big( (- \Xi_N^{(3)})^k \wick{\<1>_N^{2j - 1 - \l}} (\pi_N v_N)^{\l - k} \big).
\end{align*}

\noi
We now deal with the above terms by giving details for some of the terms and briefly discussing other terms. 

\smallskip \noi
\textbf{Term} $\textup{I}_1$: 
We write
\begin{align*}
\textup{I}_1 &= -a \I (\dl v_N) + \big( 2 \cj{a}_{1, N} N^{2 (\frac 32 - \al)} - a \big) \I \big( \<1> - 4 \cj{a}_2 \<30> + v_N \big) \\
&\quad - 2 \cj{a}_{1, N} N^{2 (\frac 32 - \al)} \pi_N^\perp \I \big( \<1> - 4 \cj{a}_2 \<30> + v_N \big) + 2 \cj{a}_{1, N} N^{2 (\frac 32 - \al)} \pi_N \I (\<1>_N - \<1>) \\
&\quad 
- 2 \cj{a}_{1, N} N^{2 (\frac 32 - \al)} \pi_N \I \big( \Xi_N^{(3)} - 4 \cj{a}_2 \<30> \big) - 2 \cj{a}_{1, N} N^{2 (\frac 32 - \al)} \pi_N \I ( \pi_N^\perp v_N )  \\
&\deff \textup{I}_{1, 1} + \textup{I}_{1, 2} + \textup{I}_{1, 3} + \textup{I}_{1, 4} + \textup{I}_{1, 5} + \textup{I}_{1, 6},
\end{align*}

\noi
where $\pi_N^\perp = \Id - \pi_N$. For $\textup{I}_{1, 1}$, by the inhomogeneous linear estimate in Lemma~\ref{LEM:nhomo} and the time localization estimate in Lemma~\ref{LEM:time}, we obtain
\begin{align}
\| \vec{\textup{I}}_{1, 1} \|_{\vec X_{T_0}^{\al - \frac 12, \frac 12 + \dl}} \les T_0^{\frac 12 - \dl} \| \dl v_N \|_{L_{T_0}^2 H_x^{- \frac 12}} \leq T_0^{\frac 12 - \dl} \| \dl v_N \|_{X_{T_0}^{\al - \frac 12, \frac 12 + \dl}} .
\label{I1-0}
\end{align}

\noi
For $\textup{I}_{1, 2}$, by \eqref{a1N}, the inhomogeneous linear estimate in Lemma~\ref{LEM:nhomo}, the time localization estimate in Lemma~\ref{LEM:time}, \eqref{enh_bdd}, and \eqref{vN_bddXT}, we obtain
\begin{align}
\begin{split}
\| \vec{\textup{I}}_{1, 2} \|_{\vec X_{T_0}^{\al - \frac 12, \frac 12 + \dl}} 
&\les N^{- \ta'} \Big( \| \<1> \|_{L_{T_0}^2 H_x^{- \frac 12}} + \| \<30> \|_{L_{T_0}^2 H_x^{- \frac 12}} + \| v_N \|_{X_{T_0}^{\al - \frac 12, \frac 12 + \dl}} \Big) \\
&\les N^{- \ta'} (K_0 + K_1)
\end{split}
\label{I1-1}
\end{align}

\noi
for some $\ta' > 0$. For $\textup{I}_{1, 3}$ and $\textup{I}_{1, 6}$, due to $\pi_N^\perp$, we may lose some slight amount of regularity to create a negative power of $N$. Thus, by using similar steps as in \eqref{I1-1}, we obtain
\begin{align}
\| \vec{\textup{I}}_{1, 3} \|_{\vec X_{T_0}^{\al - \frac 12, \frac 12 + \dl}} + \| \vec{\textup{I}}_{1, 6} \|_{\vec X_{T_0}^{\al - \frac 12, \frac 12 + \dl}} \les N^{- \ta'} (K_0 + K_1).
\label{I1-2}
\end{align}

\noi
For $\textup{I}_{1, 4}$ and $\textup{I}_{1, 5}$, we use \eqref{enh_diff} and similar steps in \eqref{I1-1} to obtain
\begin{align}
\| \vec{\textup{I}}_{1, 4} \|_{\vec X_{T_0}^{\al - \frac 12, \frac 12 + \dl}} + \| \vec{\textup{I}}_{1, 5} \|_{\vec X_{T_0}^{\al - \frac 12, \frac 12 + \dl}} \les N^{- \ta}.
\label{I1-3}
\end{align}

\noi
Combining \eqref{I1-0}, \eqref{I1-1}, \eqref{I1-2}, and \eqref{I1-3}, we obtain (assuming that $\ta' < \ta$, same below)
\begin{align}
\| \vec{\textup{I}}_1 \|_{\vec X_{T_0}^{\al - \frac 12, \frac 12 + \dl}} \les N^{- \ta'} (K_0 + K_1) + T_0^{\frac 12 - \dl} \| \dl v_N \|_{X_{T_0}^{\al - \frac 12, \frac 12 + \dl}}.
\label{I1_bdd}
\end{align}

\smallskip \noi
\textbf{Term} $\textup{I}_2$: We write
\begin{align*}
\textup{I}_2 = - 12 \cj{a}_2 \If^{\<2>} (\dl v_N) + \big( \Xi_N^{(2)} (v_N) - 12 \cj{a}_2 \If^{\<2>} (v_N) \big) - \Xi_N^{(2)} (\pi_N^\perp v_N).
\end{align*}

\noi
Thus, in the same spirit of $\textup{I}_1$ along with \eqref{vN_bddXT} and \eqref{enhN_bdd}, we obtain
\begin{align}
\| \vec{\textup{I}}_2 \|_{\vec X_{T_0}^{\al - \frac 12, \frac 12 + \dl}} \les N^{- \ta'} (K_0 K_1 + K_1) + T_0^\dl K_0 \| \dl v_N \|_{X_{T_0}^{\al - \frac 12, \frac 12 + \dl}}
\label{I2_bdd}
\end{align} 

\noi
for some $\ta' > 0$.

\smallskip \noi
\textbf{Term} $\textup{I}_3$: Similar to term $\textup{I}_2$, we obtain
\begin{align}
\| \vec{\textup{I}}_3 \|_{\vec X_{T_0}^{\al - \frac 12, \frac 12 + \dl}} \les N^{- \ta'} (K_0 K_1 + K_1) + T_0^{\dl} K_0 \| \dl v_N \|_{X_{T_0}^{\al - \frac 12, \frac 12 + \dl}}.
\label{I3_bdd}
\end{align}

\smallskip \noi
\textbf{Term} $\textup{I}_4$: From \eqref{enh_diff}, we obtain
\begin{align}
\| \vec{\textup{I}}_4 \|_{\vec X_{T_0}^{\al - \frac 12, \frac 12 + \dl}} \les N^{- \ta}.
\label{I4_bdd}
\end{align}

\smallskip \noi
\textbf{Term} $\textup{I}_5$: From \eqref{enh_diff}, we obtain
\begin{align}
\| \vec{\textup{I}}_5 \|_{\vec X_{T_0}^{\al - \frac 12, \frac 12 + \dl}} \les N^{- \ta}.
\label{I5_bdd}
\end{align}

\smallskip \noi
\textbf{Term} $\textup{I}_6$: In a similar manner of term $\textup{I}_1$, we can write
\begin{align*}
\textup{I}_6 &= - 12 \cj{a}_2 \I ( \<1> (\dl v_N)^2 ) - 24 \cj{a}_2 \I ( \<1> \cdot \dl v_N \cdot v_N ) + 12 (\cj{a}_{2, N} - \cj{a}_{2}) \I (\<1> v_N^2) \\
&\quad - 12 \cj{a}_{2, N} \pi_N^\perp \I (\<1> v_N^2) + 12 \cj{a}_{2, N} \pi_N \I \big( (\<1>_N - \<1>) v_N^2 \big) \\
&\quad - 12 \cj{a}_{2, N} \pi_N \I \big( \<1>_N \cdot v_N \cdot \pi_N^\perp v_N \big) - 12 \cj{a}_{2, N} \pi_N \I \big( \<1>_N \cdot \pi_N^\perp v_N \cdot \pi_N v_N \big).
\end{align*}

\noi
Note that by \eqref{aj}, \eqref{ajN}, and \eqref{sigma_diff}, we have
\begin{align}
|\cj{a}_{2, N} - \cj{a}_2| \les N^{-\ta'}
\label{I6help}
\end{align}

\noi
for some $\ta' > 0$. Also, from \eqref{vN_bddXT} and \eqref{v4K}, we have 
\begin{align}
\| \dl v_N \|_{X_{T_0}^{\al - \frac 12, \frac 12 + \dl}} \leq (C_0 + 1) K_1.
\label{dvN_bdd}
\end{align}

\noi
Thus, by the inhomogeneous linear estimate in Lemma~\ref{LEM:nhomo}, the time localization estimate in Lemma~\ref{LEM:time}, the trilinear estimate \eqref{str_mul1-2} in Lemma~\ref{LEM:Str_mul1}, \eqref{dvN_bdd}, \eqref{I6help},  \eqref{enh_diff}, \eqref{enhN_bdd}, and \eqref{vN_bddXT}, we obtain
\begin{align}
\begin{split}
\| \vec{\textup{I}}_6 \|_{\vec X_{T_0}^{\al - \frac 12, \frac 12 + \dl}} \les N^{-\ta'} K_0 K_1^2 + T_0^{\dl} K_0 K_1 \| \dl v_N \|_{X_{T_0}^{\al - \frac 12, \frac 12 + \dl}}.
\end{split}
\label{I6_bdd}
\end{align}

\smallskip \noi
\textbf{Term} $\textup{I}_7$: In a similar manner of term $\textup{I}_1$, we write
\begin{align*}
\textup{I}_7 &= - 4 \cj{a}_2 \I ( (\dl v_N)^3 ) - 12 \cj{a}_2 \I ( (\dl v_N)^2 v_N ) - 12 \cj{a}_2 \I ( \dl v_N (v_N)^2 ) \\
&\quad + 4 (\cj{a}_{2, N} - \cj{a}_2) \I (v_N^3) - 4 \cj{a}_{2, N} \pi_N^\perp \I (v_N^3) - 4 \cj{a}_{2, N} \pi_N \I (v_N^2 \cdot \pi_N^\perp v_N) \\
&\quad - 4 \cj{a}_{2, N} \pi_N \I (v_N \cdot \pi_N^\perp v_N \cdot \pi_N v_N) - 4 \cj{a}_{2, N} \pi_N \I (\pi_N^\perp v_N (\pi_N v_N)^2 ) .
\end{align*}

\noi
Thus, by the inhomogeneous linear estimate in Lemma~\ref{LEM:nhomo}, the time localization estimate in Lemma~\ref{LEM:time}, \eqref{str00} in Lemma~\ref{LEM:Str_mul2} (with $v \equiv 1$), the Strichartz estimate in Lemma~\ref{LEM:Str}, \eqref{dvN_bdd}, \eqref{I6help}, and \eqref{vN_bddXT}, we obtain
\begin{align}
\begin{split}
\| \vec{\textup{I}}_7 \|_{\vec X_{T_0}^{\al - \frac 12, \frac 12 + \dl}} \les N^{- \ta'} K_1^3  + T_0^\dl K_1^2 \| \dl v_N \|_{X_{T_0}^{\al - \frac 12, \frac 12 + \dl}}
\end{split}
\label{I7_bdd}
\end{align}

\noi
for some $\ta' > 0$.

\smallskip \noi
\textbf{Term} $\textup{I}_8$: In a similar manner of term $\textup{I}_1$, we write
\begin{align*}
\textup{I}_8 &= 48 \cj{a}_2^2 \I \big( \<30> (\dl v_N)^2 \big) + 96 \cj{a}_2^2 \I \big( \<30> \cdot \dl v_N \cdot v_N \big) + 12 (\cj{a}_{2} - \cj{a}_{2, N}) \I \big( 4 \cj{a}_2 \<30> v_N^2 \big) \\
&\quad + 12 \cj{a}_{2, N} \pi_N^\perp \I \big( 4 \cj{a}_2 \<30> v_N^2 \big) + 12 \cj{a}_{2, N} \pi_N \I \big( (4 \cj{a}_2 \<30> - \Xi_N^{(3)}) v_N^2 \big) \\
&\quad + 12 \cj{a}_{2, N} \pi_N \I \big( \Xi_N^{(3)} \cdot v_N \cdot \pi_N^\perp v_N \big) + 12 \cj{a}_{2, N} \pi_N \I \big( \Xi_N^{(3)} \cdot \pi_N^\perp v_N \cdot \pi_N v_N \big) .
\end{align*}

\noi
Thus, by the inhomogeneous linear estimate in Lemma~\ref{LEM:nhomo}, the time localization estimate in Lemma~\ref{LEM:time}, \eqref{str00} in Lemma~\ref{LEM:Str_mul2} (with $v \equiv 1$), the fact that $\eps < 3 \al - 3 - \eps$ for $\eps > 0$ sufficiently small, the Strichartz estimate in Lemma~\ref{LEM:Str}, \eqref{dvN_bdd}, \eqref{I6help}, \eqref{vN_bddXT}, \eqref{enh_bdd}, and \eqref{enh_diff}, we obtain
\begin{align}
\begin{split}
\| \vec{\textup{I}}_8 \|_{\vec X_{T_0}^{\al - \frac 12, \frac 12 + \dl}} \les N^{- \ta'} K_0  K_1^2 + T_0^\dl K_0 K_1 \| \dl v_N \|_{X_{T_0}^{\al - \frac 12, \frac 12 + \dl}} 
\end{split}
\label{I8_bdd}
\end{align}

\noi
for some $\ta' > 0$.

\smallskip \noi
\textbf{Term} $\textup{I}_9$: Similar to $\textup{I}_8$, we obtain
\begin{align}
\begin{split}
\| \vec{\textup{I}}_9 \|_{\vec X_{T_0}^{\al - \frac 12, \frac 12 + \dl}} \les N^{- \ta'} K_0^2  K_1 + T_0^\dl K_0^2 \| \dl v_N \|_{X_{T_0}^{\al - \frac 12, \frac 12 + \dl}}
\end{split}
\label{I9_bdd}
\end{align}

\noi
for some $\ta' > 0$.

\smallskip \noi
\textbf{Term} $\textup{I}_{10}$: Similar to $\textup{I}_8$, we obtain
\begin{align}
\begin{split}
\| \vec{\textup{I}}_{10} \|_{\vec X_{T_0}^{\al - \frac 12, \frac 12 + \dl}} \les N^{- \ta'} K_0^3
\end{split}
\label{I10_bdd}
\end{align}

\noi
for some $\ta' > 0$.

\smallskip \noi
\textbf{Term} $\textup{I}_{11}$: From \eqref{ajN} and \eqref{sigma_diff}, we know that $\{ \cj{a}_{j, N} \}_{N \in \N}$ is bounded for each $3 \leq j \leq m$. By the inhomogeneous linear estimate in Lemma~\ref{LEM:nhomo}, the trilinear estimate \eqref{str_mul1-2} in Lemma~\ref{LEM:Str_mul1}, \eqref{enh_diff}, and \eqref{vN_bddXT}, we obtain
\begin{align}
\begin{split}
\| \vec{\textup{I}}_{11} \|_{\vec X_{T_0}^{\al - \frac 12, \frac 12 + \dl}} &\les \sum_{j = 3}^m \big\| N^{(2j - 4) (\al - \frac 32)} \wick{\<1>_N^{2j - 3}} (\pi_N v_N)^2 \big\|_{X_{T_0}^{- \frac 12, - \frac 12 + \dl}} \\
&\les \sum_{j = 3}^m \big\| N^{(2j - 4) (\al - \frac 32)} \wick{\<1>_N^{2j - 3}} \big\|_{L_{T_0}^\infty W_x^{\al - \frac 32 - \eps, \infty}} \| v_N \|_{X_{T_0}^{\al - \frac 12, \frac 12 + \dl}}^2 \\
&\les N^{- \ta} K_1^2 .
\end{split}
\label{I11_bdd}
\end{align}

\smallskip \noi
\textbf{Term} $\textup{I}_{12}$: We write
\begin{align*}
\textup{I}_{12} = \textup{I}_{12, 1} + \textup{I}_{12, 2} + \textup{I}_{12, 3},
\end{align*}

\noi
where $\textup{I}_{12, 1}$ denotes the contribution by $\l = 3$, $\textup{I}_{12, 2}$ denotes the contribution by $4 \leq \l \leq 2j - 2$, and $\textup{I}_{12, 3}$ denotes the contribution by $\l = 2j - 1$.
For $\textup{I}_{12, 1}$, by the boundedness of $\{ \cj{a}_{j, N} \}_{N \in \N}$, the inhomogeneous linear estimate in Lemma~\ref{LEM:nhomo}, the multilinear estimate \eqref{str00} in Lemma~\ref{LEM:Str_mul2}, the fact that $\eps < 3 \al - 3 - \eps$ for $\eps > 0$ sufficiently small, the Strichartz estimate in Lemma~\ref{LEM:Str}, \eqref{enhN_bdd}, \eqref{vN_bddXT}, and \eqref{enh_diff}, we obtain
\begin{align}
\begin{split}
\| \vec{\textup{I}}_{12, 1} \|_{\vec X_{T_0}^{\al - \frac 12, \frac 12 + \dl}} 
&\les \sum_{j = 3}^m \sum_{k = 0}^3 \| \Xi_N^{(3)} \|_{L_{T_0}^4 W_x^{\eps, 4}}^k \| v_N \|_{L_{T_0}^4 W_x^{\eps, 4}}^{3 - k} \big\| N^{(2j - 4) (\al - \frac 32)} \wick{\<1>_N^{2j - 4}} \big\|_{L_{T_0}^\infty W_x^{- \eps, \infty}} \\
&\les \sum_{j = 3}^m \sum_{k = 0}^3 K_0^k  \| v_N \|_{X_{T_0}^{\al - \frac 12, \frac 12 + \dl}}^{3 - k} \big\| N^{(2j - 5) (\al - \frac 32)} \wick{\<1>_N^{2j - 4}} \big\|_{L_{T_0}^\infty W_x^{\al - \frac 32 - \eps, \infty}} \\
&\les N^{- \ta} K_0^3 K_1^3 .
\end{split}
\label{I12-1}
\end{align}

\noi
For $\textup{I}_{12, 2}$, we have either $k \geq 2$ or $\l - k \geq 2$, and we assume the slightly harder case $\l - k \geq 2$ without loss of generality. Let $k' = \min \{\l - k - 1, 3\}$. By the boundedness of $\{ \cj{a}_{j, N} \}_{N \in \N}$, the inhomogeneous linear estimate in Lemma~\ref{LEM:nhomo}, the multilinear estimate \eqref{str0} in Lemma~\ref{LEM:Str_mul2}, the Strichartz estimate in Lemma~\ref{LEM:Str}, \eqref{enhN_bdd}, \eqref{vN_bddXT}, \eqref{enh_diff}, and \eqref{vN_bddW}, we obtain
\begin{align}
\begin{split}
\| &\vec{\textup{I}}_{12, 2} \|_{\vec X_{T_0}^{\al - \frac 12, \frac 12 + \dl}} \\
&\les \sum_{j = 3}^m \sum_{\l = 4}^{2j - 2} \sum_{k = 0}^\l N^{(2j - 4) (\al - \frac 32)} \| \Xi_N^{(3)} \|_{L_{T_0}^r L_x^r}^{3 - k'} \| v_N \|_{L_{T_0}^r L_x^r}^{k' + \ta'} \| \wick{\<1>_N^{2j - 1 - \l}} \|_{L_{T_0}^\infty L_x^\infty} \\
&\quad \times \| \Xi_N^{(3)} \|_{L_{T_0}^\infty L_x^\infty}^{k + k' - 3} \| v_N \|_{L_{T_0}^p L_x^p}^{\l - k - k' - \ta'} \\
&\les \sum_{j = 3}^m \sum_{\l = 4}^{2j - 2} \sum_{k = 0}^\l N^{(\l - 2) (\al - \frac 32) - (1 + \l - k - k' - \ta') (\al - \frac 32 - \eps)} \| \Xi_N^{(3)} \|_{L_{T_0}^\infty W_x^{3 \al - 3 - \eps, \infty}}^{k} \\
&\quad \times \| v_N \|_{X_{T_0}^{\al - \frac 12, \frac 12 + \dl}}^{k' + \ta'} \big\| N^{(2j - 2 - \l) (\al - \frac 32)} \wick{\<1>_N^{2j - 1 - \l}} \big\|_{L_{T_0}^\infty W_x^{\al - \frac 32 - \eps, \infty}}  \| v_N \|_{L_{T_0}^p W_x^{\al - \frac 32 - \eps, p}}^{\l - k - k' - \ta'} \\
&\les N^{- \ta} K_0^{2m - 2} K_1^{3 + \ta'} K_2^{2m - 2} 
\end{split}
\label{I12-2}
\end{align}

\noi
for some $\ta' > 0$ small, $r > 4$ satisfying $\frac 12 > \frac 32 - \frac{\al + 3}{r}$, and $1 \ll p < \infty$ sufficiently large, where we used 
$$(\l - 2) \Big( \al - \frac 32 \Big) - (1 + \l - k - k' - \ta') \Big( \al - \frac 32 - \eps \Big) < 0$$ 
given $k + k' \geq 3$ and $\eps = \eps (\ta') > 0$ sufficiently small. For $\textup{I}_{12, 3}$, we use similar steps as in \eqref{I12-2} to obtain
\begin{align}
\begin{split}
\| &\vec{\textup{I}}_{12, 3} \|_{\vec X_{T_0}^{\al - \frac 12, \frac 12 + \dl}} \\
&\les \sum_{j = 3}^m \sum_{k = 0}^{2j - 1} N^{(2j - 4) (\al - \frac 32)} \| \Xi_N^{(3)} \|_{L_{T_0}^r L_x^r}^{3 - k'} \| v_N \|_{L_{T_0}^r L_x^r}^{k' + \ta'} \| \Xi_N^{(3)} \|_{L_{T_0}^\infty L_x^\infty}^{k + k' - 3}  \| v_N \|_{L_{T_0}^p L_x^p}^{2j - 1 - k - k' - \ta'} \\
&\les \sum_{j = 3}^m \sum_{k = 0}^{2j - 1} N^{(2j - 4) (\al - \frac 32) - (2j - 1 - k - k' - \ta') (\al - \frac 32 - \eps)} \| \Xi_N^{(3)} \|_{L_{T_0}^\infty W_x^{3 \al - 3 - \eps, \infty}}^{k} \\
&\quad \times \| v_N \|_{X_{T_0}^{\al - \frac 12, \frac 12 + \dl}}^{k' + \ta'} \| v_N \|_{L_{T_0}^p W_x^{\al - \frac 32 - \eps, p}}^{2j - 1 - k - k' - \ta'} \\
&\les N^{-\ta'} K_0^{2m - 1} K_1^{3 + \ta'} K_2^{2m - 1} 
\end{split}
\label{I12-3}
\end{align}

\noi
for some $\ta' > 0$. Thus, combining \eqref{I12-1}, \eqref{I12-2}, and \eqref{I12-3}, we obtain
\begin{align}
\| \vec{\textup{I}}_{12} \|_{\vec X_{T_0}^{\al - \frac 12, \frac 12 + \dl}} \les N^{- \ta'} K_0^{2m - 1} K_1^{3 + \ta'} K_2^{2m - 1} .
\label{I12_bdd}
\end{align}

\smallskip
By using \eqref{group}, the homogeneous linear estimate in Lemma~\ref{LEM:homo}, and \eqref{init_diff} and gathering all the above 12 estimates: \eqref{I1_bdd}, \eqref{I2_bdd}, \eqref{I3_bdd}, \eqref{I4_bdd}, \eqref{I5_bdd}, \eqref{I6_bdd}, \eqref{I7_bdd}, \eqref{I8_bdd}, \eqref{I9_bdd}, \eqref{I10_bdd}, \eqref{I11_bdd}, and \eqref{I12_bdd}, we obtain
\begin{align}
\begin{split}
\| \vec \dl v_N \|_{\vec X_{T_0}^{\al - \frac 12, \frac 12 + \dl}} &\leq C \kappa + C K_0^{2m - 1} K_1^{3 + \ta'} K_2^{2m - 1} N^{- \ta'} \\
&\quad + C T_0^\dl (K_0 K_1 + K_1^2) \| \dl v_N \|_{X_{T_0}^{\al - \frac 12, \frac 12 + \dl}}
\end{split}
\label{vN_diff}
\end{align}

\noi
for some constant $C > 0$. 
Since $K_0 \ll K_1$, we can take $T_0 = T_0 (K_1) > 0$ sufficiently small such that $C T_0^\dl (K_0 K_1 + K_1^2) \leq \frac 12$, so that \eqref{vN_diff} gives
\begin{align}
\| \vec \dl v_N \|_{\vec X_{T_0}^{\al - \frac 12, \frac 12 + \dl}} \leq 2 C \kappa + 2 C K_0^{2m - 1} K_1^{3 + \ta'} K_2^{2m - 1} N^{- \ta'} .
\label{dvN1-1}
\end{align}

\noi
Thus, by taking $\kappa$ to be sufficiently small and $N = N (K_1, K_2)$ to be sufficiently large, we get
\begin{align}
\| \vec \dl v_N \|_{\vec X_{T_0}^{\al - \frac 12, \frac 12 + \dl}} \leq K_1 .
\label{dvN1-2}
\end{align}

\noi
Thus, from \eqref{vN_bddXT} and \eqref{dvN1-2}, we obtain
\begin{align*}
\| \vec v \|_{\vec X_{T_0}^{\al - \frac 12, \frac 12 + \dl}} \leq \| \vec v_N \|_{\vec X_{T_0}^{\al - \frac 12, \frac 12 + \dl}} + \| \vec \dl v_N \|_{X_{T_0}^{\al - \frac 12, \frac 12 + \dl}} \leq 2 K_1 ,
\end{align*}

\noi
which is the desired bound \eqref{v2K}.

\smallskip \noi
{\bf Step 3:} 
Let $T_0' = \frac{T_0}{2}$. From \eqref{v2K} and Lemma~\ref{LEM:HsXsb}, we have
\begin{align}
\| \vec v (T_0') \|_{\vec H^{\al - \frac 12}} \leq C_1 \| \vec v \|_{\vec X_{T_0}^{\al - \frac 12, \frac 12 + \dl}} \leq 2 C_1 K_1 
\label{v0bdd2}
\end{align}

\noi
for some constant $C_1 \geq 1$. Thus, by using \eqref{v0bdd2} and similar argument in Step 1 with a possibly smaller value of $T_0 = T_0 (K_1) > 0$, we know that there exists a unique solution $v$ to \eqref{SdNLWv4b} on $[T_0', T_0' + T_0]$ and that
\begin{align}
\| \vec v \|_{\vec X_{[T_0', T_0' + T_0]}^{\al - \frac 12, \frac 12 + \dl}} \leq 4 C_1 K_1 .
\label{v4K2}
\end{align}

\noi
Also, from \eqref{dvN1-1} and Lemma~\ref{LEM:HsXsb}, we have
\begin{align}
\begin{split}
\| \vec v_N (T_0') - \vec v (T_0') \|_{\vec H^{\al - \frac 12}} &\leq C_1 \| \vec \dl v_N \|_{\vec X_{T_0}^{\al - \frac 12, \frac 12 + \dl}} \\
&\leq 2 C C_1 \kappa + 2 C C_1 K_0^{2m - 1} K_1^{3 + \ta'} K_2^{2m - 1} N^{- \ta'} .
\end{split}
\label{dvN0bdd2}
\end{align}

\noi
Thus, by using \eqref{v4K2} (needed for getting an analog of the bound \eqref{dvN_bdd} on $[T_0', T_0' + T_0]$), \eqref{dvN0bdd2}, and similar argument in Step 2, we get
\begin{align}
\| \vec \dl v_N \|_{\vec X_{[T_0', T_0' + T_0]}^{\al - \frac 12, \frac 12 + \dl}} 
&\leq 4 C^2 C_1 \kappa + (4 C^2 C_1 + 2 C) K_0^{2m - 1} K_1^{3 + \ta'} K_2^{2m - 1} N^{- \ta'} ,
\label{dvN2T}
\end{align}

\noi
so that using Lemma~\ref{LEM:glue} and the difference estimates \eqref{dvN1-1} and \eqref{dvN2T}, we obtain
\begin{align*}
\| \vec \dl v_N \|_{\vec X_{\frac 32 T_0}^{\al - \frac 12, \frac 12 + \dl}} 
&\leq (4 C^2 C_1 + 2 C) \kappa + (4 C^2 C_1 + 4 C) K_0^{2m - 1} K_1^{3 + \ta'} K_2^{2m - 1} N^{- \ta'} .
\end{align*}

\noi
Thus, since $K_0 \ll K_1$, by taking $\kappa$ to be sufficiently small and $N = N (K_1, K_2)$ to be sufficiently large, we get
\begin{align}
\| \vec \dl v_N \|_{\vec X_{\frac 32 T_0}^{\al - \frac 12, \frac 12 + \dl}} \leq K_1 ,
\label{dvN2-2}
\end{align}

\noi
so that from \eqref{vN_bddXT} and \eqref{dvN2-2}, we obtain the following extension of the bound \eqref{v2K}:
\begin{align*}
\| \vec v \|_{\vec X_{\frac 32 T_0}^{\al - \frac 12, \frac 12 + \dl}} \leq 2 K_1 .
\end{align*}

We repeat the above procedure $J = 2 [\frac{T}{T_0}]$ times, so that we get
\begin{align*}
\| \vec \dl v_N \|_{\vec X_T^{\al - \frac 12, \frac 12 + \dl}} \leq C (J) \kappa + C (J) K_0^{2m - 1} K_1^{3 + \ta'} K_2^{2m - 1} N^{- \ta'} 
\end{align*}

\noi
for some constant $C(J) > 0$, which gives the difference estimate \eqref{vNv_diff2}. Thus, since $T_0$ depends only on $K_1$ and $K_0 \ll K_1$, by taking $\kappa = \kappa (T, K_1)$ to be sufficiently small and $N = N (T, K_1, K_2)$ to be sufficiently large, we get
\begin{align}
\| \vec \dl v_N \|_{\vec X_{T}^{\al - \frac 12, \frac 12 + \dl}} \leq K_1 ,
\label{dvNT-2}
\end{align}

\noi
so that from \eqref{vN_bddXT} and \eqref{dvNT-2}, we obtain \eqref{v_bddX}.
\end{proof}

\begin{remark} \rm
One may compare the stability result for the parabolic macroscopic model in Lemma~\ref{LEM:stb_h2} and the one for the hyperbolic macroscopic model in Proposition~\ref{PROP:stb}. The key difference is that the hyperbolic model requires an additional a priori bound \eqref{vN_bddW} in an $L^\infty$-based norm. Due to the lack of $L^\infty$-based estimates for hyperbolic equations, this a priori bound cannot be obtained purely from the analysis of the dynamics but, at this point, can only be proved from the invariance point of view (see Lemma~\ref{LEM:uN_bdd} and Proposition~\ref{PROP:vN_bddW} below).
\end{remark}

\subsection{A priori bounds with large probability}
In this subsection, we prove some a priori bounds that are essential for proving the global-in-time convergence of $u_N$ satisfying \eqref{SdNLWmain}. We recall from Theorem~\ref{THM:rep} 
that $z_0$ and $v_{0, N}$ are random variables on $(\Sigma, \mathcal{G}, \mathbb{Q})$ such that
\begin{align*}
\Law (z_0 + v_{0, N}) = \varrho_N.
\end{align*}

\begin{lemma}
\label{LEM:uN_bdd}
Let $1 < \al \leq \frac 98$, $T > 0$, $\eps > 0$, and $1 \leq p < \infty$. Let $u_N$ be the solution to the macroscopic model \eqref{SdNLWmain} with $u_N |_{t = 0} = z_0 + v_{0, N}$. Then, given any $\eta > 0$, there exists $K = K (T, \eta) \gg 1$ such that
\begin{align}
\mathbb{Q} \otimes \mathbb{P}^w \Big( \| u_N \|_{L_T^p W_x^{\al - \frac 32 - \eps, \infty}} > K \Big) < \eta,
\label{uNbdd_goal}
\end{align}

\noi
uniformly in $N \in \N$.
\end{lemma}
\begin{proof}
By Chebyshev's inequality, Lemma~\ref{LEM:invN}, and the fact that $\Law (z_0 + v_{0, N}) = \varrho_N$, we have
\begin{align}
\begin{split}
\mathbb{Q} &\otimes \mathbb{P}^w \Big( \| u_N \|_{L_T^p W_x^{\al - \frac 32 - \eps, \infty}} > K \Big) \\
&\leq \frac{1}{K^p} \int_{\Sigma \times \Om^w} \| u_N \|_{L_T^p W_x^{\al - \frac 32 - \eps, \infty}}^p d (\mathbb{Q} \otimes \mathbb{P}^w) \\
&\leq \frac{1}{K^p} \bigg\| \bigg( \int_{\Sigma \times \Om^w} \| u_N (t) \|_{W^{\al - \frac 32 - \eps, \infty}}^p d (\mathbb{Q} \otimes \mathbb{P}^w) \bigg)^{\frac{1}{p}} \bigg\|_{L_T^p}^p \\
&= \frac{1}{K^p} \bigg\| \bigg( \int_{\Sigma \times \Om^w} \| z_0 + v_{0, N} \|_{W^{\al - \frac 32 - \eps, \infty}}^p d (\mathbb{Q} \otimes \mathbb{P}^w) \bigg)^{\frac{1}{p}} \bigg\|_{L_T^p}^p \\
&= \frac{1}{K^p} \bigg\| \bigg( \int \| u \|_{W^{\al - \frac 32 - \eps, \infty}}^p d \varrho_N (u) \bigg)^{\frac{1}{p}} \bigg\|_{L_T^p}^p.
\end{split}
\label{uNbdd1}
\end{align}

\noi 
By the Bou\'e-Dupuis variational formula (Lemma~\ref{LEM:BD}), \eqref{BD_RNtil}, \eqref{RNtil_est}, and \eqref{ZNupbdd}, we have
\begin{align}
\begin{split}
- &\log \int \| u \|_{W^{\al - \frac 32 - \eps, \infty}}^p d \varrho_N (u) \\
&\geq - \log \int \exp \Big( \| u \|_{W^{\al - \frac 32 - \eps, \infty}} \Big) d \varrho_N (u) - C(p) \\
&\geq \inf_{\dot{\Ups}^N \in \mathbb{H}_a^\al} \E \bigg[ - \| W + \Dr \|_{W^{\al - \frac 32 - \eps, \infty}} + \wt{\mathcal{R}}^\dia_N (W_N, \Ups_N - \ZZ_N) \\
&\quad + \frac 12 \int_0^1 \big\| \dot{\Ups}^N (t) \big\|_{H^\al}^2 dt \bigg] + \log \mathcal{Z}_N - C(p) \\
&\geq \inf_{\dot{\Ups}^N \in \mathbb{H}_a^\al} \E \Big[ - \| W + \Dr \|_{W^{\al - \frac 32 - \eps, \infty}} + \frac 14 \| \Ups^N \|_{H^\al}^2 \Big] - C(p)
\end{split}
\label{uNbdd2}
\end{align}

\noi
for some constant $C(p) > 0$, where we recall from \eqref{defWt} that $\Law (W) = \Law (W (1)) = \mu_\al$ (defined in \eqref{mu_al}), $\Dr = \Ups^N - \ZZ_N$ with $\ZZ_N$ being defined in \eqref{defZZ}, and $\Ups_N = \pi_N \Ups^N$. By Lemma~\ref{LEM:Yreg}~(i) and the embedding in Lemma~\ref{LEM:Bes}~(i), we have
\begin{align}
\E \big[ \| W \|_{W^{\al - \frac 32 - \eps, \infty}} \big] + \E \big[ \| \ZZ_N \|_{W^{\al - \frac 32 - \eps, \infty}} \big] \leq C
\label{uNbdd3}
\end{align}

\noi
for some constant $C > 0$ independent of $N$. Also, by the Sobolev embedding and Cauchy's inequality, we have
\begin{align}
\| \Ups^N \|_{W^{\al - \frac 32 - \eps, \infty}} \leq C \| \Ups^N \|_{H^\al} \leq \frac 14 \| \Ups^N \|_{H^\al}^2 + C
\label{uNbdd4}
\end{align}

\noi
for some constant $C > 0$ independent of $N$. Combining \eqref{uNbdd2}, \eqref{uNbdd3}, and \eqref{uNbdd4}, we obtain the bound
\begin{align}
\int \| u \|_{W^{\al - \frac 32 - \eps, \infty}}^p d \varrho_N (u) \les 1,
\label{uNbdd5}
\end{align}

\noi
where the underlying constant is independent of $N$. The desired estimate \eqref{uNbdd_goal} then follows from \eqref{uNbdd1} and \eqref{uNbdd5}.
\end{proof}

As a consequence of Lemma~\ref{LEM:uN_bdd}, we obtain the following $L^\infty$-a priori bound for the solution $v_N$ to the equation \eqref{SdNLWv4a}.
\begin{proposition}
\label{PROP:vN_bddW}
Let $1 < \al \leq \frac 98$, $T > 0$, $\eps > 0$, and $1 \leq p < \infty$. Let $v_N$ be the solution to the equation \eqref{SdNLWv4a} on $[0, T]$. Then, given any $\eta > 0$, there exists $K = K (T, \eta) \gg 1$ such that
\begin{align*}
\mathbb{Q} \otimes \mathbb{P}^w \Big( \| v_N \|_{L_T^p W_x^{\al - \frac 32 - \eps, \infty}} > K \Big) < \eta,
\end{align*}

\noi
uniformly in $N \in \N$.
\end{proposition}
\begin{proof}
In view of the second order expansion \eqref{sec_exp2}, the proof follows immediately from Lemma~\ref{LEM:uN_bdd} and the tail estimates \eqref{tail_w} in Lemma~\ref{LEM:data} for $\<1>$ and $\Xi_N^{(3)}$.
\end{proof}

We also have the following $X_T^{s,b}$-a priori bound for the solution $v_N$ to the equation \eqref{SdNLWv4a}.
\begin{proposition}
\label{PROP:vN_bddX}
Let $1 < \al \leq \frac 98$, $T > 0$, and $\dl > 0$ be sufficiently small. Let $\vec v_N$ be the solution to the equation \eqref{SdNLWv4a} on $[0, T]$. Then, given any $\eta > 0$, there exists $K = K (T, \eta) \gg 1$ such that
\begin{align*}
\mathbb{Q} \otimes \mathbb{P}^w \Big( \| \vec v_N \|_{\vec X_T^{\al - \frac 12, \frac 12 + \dl}} > K \Big) < \eta,
\end{align*}

\noi
uniformly in $N \in \N$.
\end{proposition}

\begin{proof}
The steps are similar to those in \cite[Subsection~6.4]{LTzW}, and so we will be very brief. From Proposition~\ref{PROP:data}, we may assume that the truncated enhanced data set $\pmb{\Xi}_N$ in \eqref{enhN} equipped with the $\mathcal{X}_T^{\al, \eps, \dl}$-norm in \eqref{XTaed} satisfies
\begin{align}
\| \pmb{\Xi}_N \|_{\mathcal{X}_T^{\al, \eps, \dl}} \leq K_0
\label{bddxi}
\end{align}

\noi
for some constant $K_0 \geq 1$.
From Theorem~\ref{THM:rep} and Lemma~\ref{LEM:Bes}~(i), we may assume that the initial data $v_{0, N}$ satisfies
\begin{align}
\| v_{0, N} \|_{H^{\al - \frac 12}} \leq K_0'
\label{bddinit}
\end{align}

\noi
for some constant $K_0' \geq 1$.
From Proposition~\ref{PROP:vN_bddW}, we may assume that 
\begin{align}
\| v_N \|_{L_T^p W_x^{\al - \frac 32 - \eps, \infty}} \leq K_2
\label{bddW}
\end{align}

\noi
for $1 \ll p < \infty$ sufficiently large and $K_2 \geq 1$ being a constant. 

We recall that $v_N$ satisfies the equation \eqref{SdNLWv4a}. Let $0 < T_0 \leq 1$ be chosen later. By the linear estimates in Lemma~\ref{LEM:homo} and Lemma~\ref{LEM:nhomo}, the time localization estimate in Lemma~\ref{LEM:time}, the multilinear estimates similar to the treatment of $\textup{I}_6$ -- $\textup{I}_{12}$ in the proof of Proposition~\ref{PROP:stb}, \eqref{bddxi}, \eqref{bddinit}, and \eqref{bddW}, we obtain
\begin{align}
\| \vec v_N \|_{\vec X_{T_0}^{\al - \frac 12, \frac 12 + \dl}} \leq C K_0^{2 m - 1} K_0' + C T_0^\dl K_0^{2 m - 1} K_2^{2 m - 1} \| v_N \|_{X_{T_0}^{\al - \frac 12, \frac 12 + \dl}}^{3 + \ta'}
\label{vNnl}
\end{align}

\noi
for some $C > 0$ and $\ta' > 0$. Using a standard continuity argument along with Lemma~\ref{LEM:time_cty}, we see that if we take $T_0 = T_0 (K_0, K_0', K_2) \ll 1$ to be sufficiently small, we obtain the local-in-time bound
\begin{align}
\| \vec v_N \|_{\vec X_{T_0}^{\al - \frac 12, \frac 12 + \dl}} \leq 2 C K_0^{2m - 1} K_0'.
\label{vNnl2}
\end{align}

\noi
We also note from \eqref{VNuN} that $\pi_N \I (V'_N (\pi_N u_N))$ contains all of the terms on the right-hand-side of \eqref{SdNLWv4a} (with $t_0 = 0$) with an additional $\Xi_N^{(3)}$ term. 
Thus, by using similar estimates that lead to \eqref{vNnl} and the bound \eqref{vNnl2}, 
we obtain the local-in-time nonlinear smoothing estimate
\begin{align*}
\big\| \pi_N \I \big( \ind_{[0, T_0]} V_N' (\pi_N u_N) \big) \big\|_{L_T^\infty W_x^{3 \al - 3 - \eps, \infty} + X_T^{\al - \frac 12, \frac 12 + \dl}} \leq 16 C^5 K_0^{10 m - 5} (K_0')^4 K_2^{2m - 1},
\end{align*}

\noi
where, as in the proof of \cite[Proposition~6.5]{LTzW}, the appearance of $\ind_{[0, T_0]}$ does not affect any of the multilinear estimates or the regularity properties of stochastic objects.
Then, using the invariance property in Lemma~\ref{LEM:invN} and proceeding as in the proof of \cite[Proposition~6.5]{LTzW}, we obtain the global-in-time nonlinear smoothing estimate
\begin{align}
\big\| \pi_N \I \big( V_N' (\pi_N u_N) \big) \big\|_{L_T^\infty W_x^{3 \al - 3 - \eps, \infty} + X_T^{\al - \frac 12, \frac 12 + \dl}} \leq C_0
\label{ns_bdd1}
\end{align}

\noi
for some constant $C_0 = C_0 (T) > 0$ with large probability.

We now proceed as in the proof of \cite[Proposition~6.6]{LTzW}. From \eqref{ns_bdd1}, the second order expansion \eqref{sec_exp2}, and \eqref{bddxi} (for $\Xi_N^{(3)}$), we get the following uniform bound:
\begin{align}
\| v_N \|_{L_T^\infty W_x^{3 \al - 3 - \eps, \infty} + X_T^{\al - \frac 12, \frac 12 + \dl}} \leq C_0 + K_0.
\label{ns_bdd2}
\end{align}

\noi
Let $I = [t_0, t_1] \subset [0, T]$ be a closed interval with $0 < |I| \leq 1$. 
From the Strichartz estimate in Lemma~\ref{LEM:Str} and \eqref{ns_bdd2}, we deduce that for any $r \geq 1$ satisfying $\al - \frac 12 > \frac 32 - \frac{\al + 3}{r}$, we have the following useful bound:
\begin{align}
\| v_N \|_{L_I^r W_x^{\eps, r}} \les \| v_N \|_{L_T^\infty W_x^{3 \al - 3 - \eps, \infty} + X_T^{\al - \frac 12, \frac 12 + \dl}} \leq C_0 + K_0
\label{ns_bdd3}
\end{align}

\noi
for $\eps > 0$ sufficiently small. We then use the Duhamel formulation \eqref{SdNLWv4a} again to estimate $\| \vec v_N \|_{\vec X_I^{\al - \frac 12, \frac 12 + \dl}}$ by performing linear and multilinear estimates. Note that most of the terms can be estimated in the same way as those in \cite[Proposition~6.6]{LTzW} along with \eqref{ns_bdd3}, except for the following two terms:
\begin{align*}
\textup{I}_{11} &= \sum_{j = 3}^m C_j \cj{a}_{j, N} N^{(2j - 4) (\al - \frac 32)} \pi_N \I_{t_0} \big( \wick{\<1>_N^{2j - 3}} (\pi_N v_N)^2 \big), \\
\textup{I}_{12} &= \sum_{j = 3}^m \sum_{\l = 3}^{2j - 1} \sum_{k = 0}^\l C_{j, \l, k} \cj{a}_{j, N} N^{(2 j - 4) (\al - \frac 32)} 
\pi_N \I_{t_0} \big( (- \Xi_N^{(3)})^k \wick{\<1>_N^{2j - 1 - \l}} (\pi_N v_N)^{\l - k} \big).
\end{align*}

\noi
For $\textup{I}_{11}$, by the fact that $\{ \cj{a}_{j, N} \}_{N \in \N}$ is bounded for each $3 \leq j \leq m$ (using \eqref{ajN} and \eqref{sigma_diff}), the inhomogeneous linear estimate in Lemma~\ref{LEM:nhomo}, the time localization estimate in Lemma~\ref{LEM:time}, the trilinear estimate \eqref{str_mul1-1} in Lemma~\ref{LEM:Str_mul1}, \eqref{bddxi}, and \eqref{ns_bdd3}, we obtain
\begin{align*}
\| \vec{\textup{I}}_{11} \|_{\vec X_I^{\al - \frac 12, \frac 12 + \dl}} &\les |I|^\dl \sum_{j = 3}^m \big\| N^{(2j - 4) (\al - \frac 32)} \wick{\<1>_N^{2j - 3}} (\pi_N v_N)^2 \big\|_{X_I^{- \frac 12, - \frac 12 + \dl}} \\
&\les |I|^\dl \sum_{j = 3}^m \big\| N^{(2j - 4) (\al - \frac 32)} \wick{\<1>_N^{2j - 3}} \big\|_{L_T^\infty W_x^{\al - \frac 32 - \eps, \infty}} \| v_N \|_{L_I^r L_x^r} \| v_N \|_{X_I^{\al - \frac 12, \frac 12 + \dl}} \\
&\leq |I|^\dl C_0 K_0^2 \| v_N \|_{X_I^{\al - \frac 12, \frac 12 + \dl}},
\end{align*}

\noi
where $r \geq \frac{10}{3}$ satisfies $\al - \frac 12 > \frac 32 - \frac{\al + 3}{r}$. For $\textup{I}_{12}$, we use similar steps as in \eqref{I12-1}, \eqref{I12-2}, and \eqref{I12-3} (with the use of the Strichartz estimate in Lemma~\ref{LEM:Str} replaced by \eqref{ns_bdd3}) along with \eqref{bddxi} and \eqref{bddW} to obtain
\begin{align*}
\| \vec{\textup{I}}_{12} \|_{\vec X_I^{\al - \frac 12, \frac 12 + \dl}} &\les |I|^\dl  K_0^{4} \Big( 1 + \| v_N \|_{L_T^\infty W_x^{3 \al - 3 - \eps, \infty} + X_T^{\al - \frac 12, \frac 12 + \dl}}^{3} \Big) \\
&\quad + |I|^\dl K_0^{2 m - 1} K_2^{2m - 2} \Big( 1 + \| v_N \|_{L_T^\infty W_x^{3 \al - 3 - \eps, \infty} + X_T^{\al - \frac 12, \frac 12 + \dl}}^{3 + \ta'} \Big) \\
&\quad + |I|^\dl K_0^{2 m - 1} K_2^{2 m - 1} \Big( 1 + \| v_N \|_{L_T^\infty W_x^{3 \al - 3 - \eps, \infty} + X_T^{\al - \frac 12, \frac 12 + \dl}}^{3 + \ta'} \Big) \\
&\les |I|^\dl K_0^{2m} K_2^{2m - 1} (C_0 K_0)^{3 + \ta'}
\end{align*}

\noi
for some $\ta' > 0$. Thus, gathering all the estimates, we obtain
\begin{align*}
\| \vec v_N \|_{\vec X_I^{\al - \frac 12, \frac 12 + \dl}} &\leq C \| \vec v_N (t_0) \|_{\vec H^{\al - \frac 12}} + C(T) C_0^4 K_0^{2m + 4} K_2^{2m - 1} \\
&\quad + C(T) |I|^\dl C_0 K_0^2 \| v_N \|_{X_I^{\al - \frac 12, \frac 12 + \dl}}
\end{align*}

\noi
for some constants $C, C (T) > 0$. By letting $|I|$ be small enough, we get
\begin{align*}
\| \vec v_N \|_{\vec X_I^{\al - \frac 12, \frac 12 + \dl}} \leq 2 C \| \vec v_N (t_0) \|_{\vec H^{\al - \frac 12}} + 2 C(T) C_0^4 K_0^{2m + 4} K_2^{2m - 1} .
\end{align*}

\noi
Then, as in the proof of \cite[Proposition~6.6]{LTzW}, by using a simple iteration along with Lemma~\ref{LEM:HsXsb} and 
applying the gluing lemma in Lemma~\ref{LEM:glue} to concatenate all time intervals, we obtain the desired $\vec X_T^{\al - \frac 12, \frac 12 + \dl}$-a priori bound for $\vec v_N$.
\end{proof}

\subsection{Proof of Theorem \ref{THM:dyn}}

In this subsection, we prove Theorem~\ref{THM:dyn}, the main result in this paper. 
Let us first prove the following proposition.

\begin{proposition}
\label{PROP:conv_vN}
Let $1 < \al \leq \frac 98$, $T > 0$, and $\eps, \dl > 0$ be sufficiently small. Let $v_{0, N}$ and $v_0$ be given by Theorem~\ref{THM:rep}. Let $\pmb{\Xi}_N$ be the truncated enhanced data set in \eqref{enhN} and let $\pmb{\Xi}$ be the enhanced data set in \eqref{enh}. Let $\vec v_N = \vec v_N (v_{0, N}, \pmb{\Xi}_N)$ be the solution to the equation \eqref{SdNLWv4a} with the truncated enhanced data set $\pmb{\Xi}_N$ and initial data $\vec v_N |_{t = 0} = (v_{0, N}, 0)$. Then, $\vec v_N$ converges in probability to a process $\vec v$ in the space $\vec X_T^{\al - \frac 12, \frac 12 + \dl}$, which is the solution to the limiting equation \eqref{SdNLWv4b}. Moreover, for any $\eta > 0$, there exists $C = C(T, \eta) > 0$ such that
\begin{align}
\mathbb{Q} \otimes \mathbb{P}^w \Big( \| \vec v \|_{\vec X_T^{\al - \frac 12, \frac 12 + \dl}} > C \Big) < \eta.
\label{v_bddXT}
\end{align}
\end{proposition}

\begin{proof}
Let $0 < \eta, \g \ll 1$ be arbitrary. Let $p \gg 1$ be as given by Proposition~\ref{PROP:stb}. Given $K_0, K_0', K_1, K_2 > 0$, $N \in \N$, $\ta > 0$, $\kappa > 0$, we define the following subsets of $\Sigma \times \Om^w$:
\begin{align*}
A_1 (K_0) &\deff \big\{ \| \pmb{\Xi} \|_{\mathcal{X}_T^{\al, \eps, \dl}} \leq K_0 \big\}, \\
A_2 (K_0') &\deff \big\{ \| v_{0} \|_{H^{\al - \frac 12}} \leq K_0' \big\}, \\
A_3 (N, K_1) &\deff \Big\{ \| \vec v_N \|_{\vec X_T^{\al - \frac 12, \frac 12 + \dl}} \leq K_1 \Big\}, \\
A_4 (N, K_2) &\deff \Big\{ \| v_N \|_{L_T^p W_x^{\al - \frac 32 - \eps, \infty}} \leq K_2 \Big\}, \\
A_5 (N) &\deff \big\{ \| \pmb{\Xi}_N - \pmb{\Xi} \|_{\mathcal{X}_T^{\al, \eps, \dl}} < N^{- \ta} \big\}, \\
A_6 (N, \kappa) &\deff \big\{ \| v_{0, N} - v_0 \|_{H^{\al - \frac 12}} < \kappa \big\}.
\end{align*}

\noi
By the tail estimates in Proposition \ref{PROP:data} with a fixed small enough $\ta > 0$, we have
\begin{align}
\mathbb{Q} \otimes \mathbb{P}^w (A_1 (K_0)^c) + \mathbb{Q} \otimes \mathbb{P}^w (A_5 (N)^c) < \frac{\eta}{5}
\label{final1}
\end{align}

\noi
given sufficiently large $K_0 \gg 1$ and $N \geq N_0$ with $N_0 \gg 1$ large enough. 
By the tail bound in Theorem~\ref{THM:rep}, we have
\begin{align}
\mathbb{Q} \otimes \mathbb{P}^w (A_2 (K_0')^c) < \frac{\eta}{5}
\label{final1.1}
\end{align}

\noi
given sufficiently large $K_0' \gg 1$ and $N \geq N_0$ with $N_0 \gg 1$ large enough. By Proposition~\ref{PROP:vN_bddX}, we have
\begin{align}
\mathbb{Q} \otimes \mathbb{P}^w (A_3 (N, K_1)^c) < \frac{\eta}{5}
\label{final2}
\end{align}

\noi
given sufficiently large $K_1 \gg 1$. We may further enlarge the value of $K_1$ so that $K_0 + K_0' \ll K_1$. By Proposition~\ref{PROP:vN_bddW}, we have
\begin{align}
\mathbb{Q} \otimes \mathbb{P}^w (A_4 (N, K_2)^c) < \frac{\eta}{5}
\label{final3}
\end{align}

\noi
given sufficiently large $K_2 \gg 1$ and $N \geq N_0$ with $N_0 \gg 1$ large enough. Let $C_0 (T, K_1, K_2) > 0$ and $\ta' > 0$ be the constants given by Proposition~\ref{PROP:stb} and let $0 < \kappa = \kappa (T, K_1, K_2) \ll 1$ be small enough and $N_0 = N_0 (T, K_1, K_2) \gg 1$ be large enough such that
\begin{align}
C_0 (T, K_1, K_2) (\kappa + N^{- \ta'}) \leq \g
\label{finalh}
\end{align}

\noi
for any $N \geq N_0$. By Theorem~\ref{THM:rep}, we have
\begin{align}
\mathbb{Q} \otimes \mathbb{P}^w (A_6 (N, \kappa)^c) < \frac{\eta}{5}
\label{final4}
\end{align}

\noi
for any $N \geq N_0$ with $N_0 \gg 1$ large enough.

By Proposition~\ref{PROP:stb}, \eqref{final1}, \eqref{final1.1}, \eqref{final2}, \eqref{final3}, \eqref{final4}, and \eqref{finalh}, we obtain
\begin{align*}
\mathbb{Q} \otimes \mathbb{P}^w \Big( \| \vec v_N - \vec v \|_{\vec X_T^{\al - \frac 12, \frac 12 + \dl}} \geq \g \Big) < \eta
\end{align*}

\noi
for any $0 < \eta, \g \ll 1$ and also the bound \eqref{v_bddXT}. This finishes the proof.
\end{proof}

We are now ready to prove our main result in this paper.

\begin{proof}[Proof of Theorem~\ref{THM:dyn}]
Let $T > 0$. As mentioned in Subsection~\ref{SUB:intro_dyn}, we proceed with the second order expansion \eqref{sec_exp2}:
\begin{align}
u_N = \<1> - \Xi_N^{(3)} + v_N ,
\label{uN_exp}
\end{align}

\noi
where $\<1>$ is given by \eqref{sto} and \eqref{sto_phi}, $\Xi_N^{(3)}$ is defined in \eqref{sto_all}, and $v_N$ satisfies the equation \eqref{SdNLWv4a}. From Proposition~\ref{PROP:conv_vN}, we know that $\vec v_N$ converges in probability to a process $\vec v$ in $\vec X_T^{\al - \frac 12, \frac 12 + \dl}$ for some $\dl > 0$ sufficiently small, where $v$ satisfies the equation \eqref{SdNLWv4b}. 

Using a slight variant of the stability estimate in \cite[Proposition~6.4]{LTzW} (see also the second part of \cite[Proposition~6.8]{OOT2}) along with the probabilistic bound \eqref{v_bddXT} and Proposition~\ref{PROP:data}, we may identify $\vec v$ as the limit in probability of the processes $\vec v_N^\dagger$ as $N \to \infty$, where $\vec v_N^\dagger$ satisfies the frequency truncated version of \eqref{SdNLWv4b}:
\begin{align*}
\vec v_N^\dagger &= \vec S(t) (v_0, 0) - a \pi_N \vec \I \big( \<1>_N - 4 \cj{a}_2 \<30>_N + \pi_N v_N^\dagger \big) - 12 \cj{a}_2 \vec \If^{\,\<2>_N} (\pi_N v_N^\dagger) + 96 \cj{a}_2^2 \vec \If^{\,\<31>_N} (\pi_N v_N^\dagger) \\
&\quad + 48 \cj{a}_2^2 \vec{\<320>}_N - 192 \cj{a}_2^3 \vec{\<70>}_N - 12 \cj{a}_2 \pi_N \vec \I \big( \<1>_N (\pi_N v_N^\dagger)^2 \big) - 4 \cj{a}_2 \pi_N \vec \I \big( (\pi_N v_N^\dagger)^3 \big) \\
&\quad + 48 \cj{a}_2^2 \pi_N \vec \I \big( \<30>_N (\pi_N v_N^\dagger)^2 \big) - 192 \cj{a}_2^3 \pi_N \vec \I \big( \<30>_N^2 \pi_N v_N^\dagger \big) + 256 \cj{a}_2^4 \pi_N \vec \I \big( \<30>_N^3 \big),
\end{align*}

\noi
where $\If^{\<2>_N}$, $\<30>_N$, $\If^{\<31>_N}$, $\<320>_N$, $\<70>_N$ are the frequency truncated versions of $\If^{\<2>}$, $\<30>$, $\If^{\<31>}$, $\<320>$, $\<70>$. Then, by letting
\begin{align}
u_N^\dagger = \<1> - 4 \cj{a}_2 \<30>_N + v_N^\dagger,
\label{uN_exp2} 
\end{align}

\noi
we see that $u_N^\dagger$ is a solution to the truncated fractional hyperbolic $\Phi^4_3$-model \eqref{SdfNLW_N} on $[0, T]$ with $\vec u_N^\dagger |_{t = 0} = (z_0 + v_0, z_1)$. In view of Theorem~\ref{THM:rep} and the fact that $\Law (z_1) = \mu_0$ is the white noise measure defined in \eqref{mu_al}, we have $\Law (\vec u_N^\dagger |_{t = 0}) = \vec \rho = \rho \otimes \mu_0$, where $\rho$ is the fractional $\Phi^4_3$-measure in \eqref{rho}.  Then, Theorem~\ref{THM:dyn} follows from the second order expansions \eqref{uN_exp} and \eqref{uN_exp2}, the convergence of $\vec u_N^\dagger$ to $\vec u$ as stated in Theorem~\ref{THM:GWP}, the convergence of $\Xi_N^{(3)}$ and $\<30>_N$ to $\<30>$ in Lemma~\ref{LEM:data}~(iv) and \eqref{sto_lim}, and the convergence of $\vec v_N$ and $\vec v_N^\dagger$ to $\vec v$ along with Lemma~\ref{LEM:HsXsb}.
\end{proof}

\begin{ackno}\rm
The author would like to thank Tadahiro Oh for the initial proposal for the study of the fractional hyperbolic $\Phi^4_3$-model, which leads to this work, and also for his help and support throughout this work. The author also thanks Leonardo Tolomeo, Nikolay Tzvetkov, and Yuzhao Wang for helpful discussions and suggestions. The author acknowledges the use of AI tools for improvement of the presentation and for correction of typographical errors. All mathematical arguments were verified and written entirely by the author. R.L. was supported by the European Research Council (grant no. 864138 ``SingStochDispDyn''). R.L. also acknowledge support from the Deutsche Forschungsgemeinschaft (DFG, German Research Foundation) under Germany's Excellence Strategy -- EXC-2047/1 -- 390685813. R.L. was also funded by the Deutsche Forschungsgemeinschaft (DFG, German Research Foundation) -- Project-ID 211504053 - SFB 1060. R.L. was also funded by the Deutsche Forschungsgemeinschaft (DFG, German Research Foundation) -- Project-ID 539309657 -- SFB 1720.
\end{ackno}

\end{document}